\documentclass[fontsize=11pt,english,headinclude,bibliography=totoc]{scrartcl} 
\usepackage{ucs}
\usepackage{ae}
\usepackage{microtype}
\usepackage[utf8x]{inputenc}
\usepackage[T1]{fontenc}
\usepackage[english]{babel}

\usepackage{amsmath}
\usepackage{amssymb}
\usepackage{amstext}
\usepackage{amsthm}

\usepackage{enumitem}

\usepackage{scrlayer-scrpage}
\usepackage{geometry}
\usepackage{xcolor}

\newcommand{\EDIT}[1]{{\color{black}#1}}
\newcommand{\EEDIT}[1]{{\color{black}#1}}

\usepackage{mathtools}
\usepackage{slashed}

\usepackage{mathscinet}
\usepackage[numbers]{natbib}

\usepackage{graphicx}
\usepackage{tikz-cd}
\usepackage{caption}

\usepackage{imakeidx}
\makeindex[intoc]

\usepackage{tabularx}

\usepackage{hyperref}
\usepackage[capitalize]{cleveref}

\newif\ifcomment
\newcommand{\comment}[1]{\ifcomment#1\fi}

\newlength{\currentparindent}
\makeatletter
\newcommand{\@minipagerestore}{\setlength{\parindent}{\currentparindent}}
\makeatother

\newcommand{\nospacepunct}[1]{\makebox[0pt][l]{#1}}

\usepackage{relsize}
\usepackage[bbgreekl]{mathbbol}
\usepackage{amsfonts}

\DeclareSymbolFontAlphabet{\mathbb}{AMSb}
\DeclareSymbolFontAlphabet{\mathbbl}{bbold}
\newcommand{\Prism}{{\mathlarger{\mathbbl{\Delta}}}}
\newcommand{\prism}{\mathbbl{\Delta}}

\DeclareMathOperator{\Ker}{Ker}

\DeclareMathOperator{\id}{id}

\DeclareMathOperator{\Aut}{Aut}

\DeclareMathOperator{\Hom}{Hom}
\DeclareMathOperator{\Ext}{Ext}
\DeclareMathOperator{\Spf}{Spf}
\DeclareMathOperator{\Spec}{Spec}
\DeclareMathOperator{\Spa}{Spa}
\DeclareMathOperator{\Pic}{Pic}
\DeclareMathOperator{\Cone}{Cone}
\DeclareMathOperator{\Mod}{Mod}
\DeclareMathOperator{\Perf}{Perf}
\DeclareMathOperator{\gr}{gr}
\DeclareMathOperator{\Fil}{Fil}
\DeclareMathOperator{\Rep}{Rep}
\DeclareMathOperator{\Fun}{Fun}
\DeclareMathOperator{\MF}{MF}

\DeclareMathOperator{\Vect}{Vect}
\DeclareMathOperator{\Coh}{Coh}
\DeclareMathOperator{\Loc}{Loc}

\DeclareMathOperator{\fib}{fib}
\DeclareMathOperator{\cofib}{cofib}
\DeclareMathOperator{\RHom}{RHom}

\DeclareMathOperator{\Map}{Map}
\DeclareMathOperator{\Rees}{Rees}

\let\Gauge\gauge
\DeclareMathOperator{\FGauge}{F-Gauge}
\DeclareMathOperator{\Kos}{Kos}
\DeclareMathOperator{\Stab}{Stab}

\DeclareMathOperator{\Ind}{Ind}
\DeclareMathOperator{\Sym}{Sym}
\DeclareMathOperator{\Tot}{Tot}
\DeclareMathOperator{\Crys}{Crys}
\DeclareMathOperator{\Isoc}{Isoc}

\newcommand{\colim}{\operatornamewithlimits{colim}}

\newcommand{\sHom}{\underline{\mathrm{Hom}}}
\newcommand{\sRHom}{\underline{\mathrm{RHom}}}
\newcommand{\sExt}{\underline{\mathrm{Ext}}}
\newcommand{\relSpec}{\underline{\mathrm{Spec}}}
\newcommand{\sEnd}{\underline{\mathrm{End}}}
\newcommand{\Der}{\mathcal{D}\kern -.5pt er}

\newcommand{\tensorL}{\otimes^\mathbb{L}}

\newcommand\A{\mathbb{A}}
\newcommand\F{\mathbb{F}}
\newcommand\G{\mathbb{G}}
\newcommand\Q{\mathbb{Q}}
\newcommand\V{\mathbb{V}}
\newcommand\Z{\mathbb{Z}}

\newcommand\T{\mathbb{T}}

\newcommand\D{\mathcal{D}}
\newcommand\DF{\mathcal{DF}}

\newcommand\C{\mathcal{C}}

\let\O\cO

\let\H\calH

\let\T\calT

\newcommand\dHod{\slashed{D}}

\newcommand\crys{\mathrm{crys}}
\newcommand\dR{\mathrm{dR}}
\newcommand\N{\mathrm{N}}
\newcommand\Syn{\mathrm{Syn}}
\newcommand\et{\mathrm{\acute{e}t}}
\newcommand\proet{\mathrm{pro\acute{e}t}}
\newcommand\refl{\mathrm{refl}}
\newcommand\adm{\mathrm{adm}}
\newcommand\op{\mathrm{op}}

\let\inf\Ainf
\newcommand\HT{\mathrm{HT}}
\newcommand\lisse{\mathrm{lisse}}
\newcommand\FM{\mathrm{FM}}
\newcommand\red{\mathrm{red}}
\newcommand\Hod{\mathrm{Hod}}

\newcommand\can{\mathrm{can}}

\newcommand\typ{\mathrm{typ}}
\newcommand\fl{\mathrm{fl}}
\newcommand\nilp{\mathrm{nilp}}

\let\d\derivative
\newcommand\incl{\mathrm{incl}}
\newcommand\conj{\mathrm{conj}}
\newcommand\an{\mathrm{an}}

\let\epsilon\varepsilon
\let\phi\varphi
\let\ol\overline
\let\ul\underline
\let\tensor\otimes

\let\cal\mathcal

\newtheorem{thm}{Theorem}[subsection]
\newtheorem{prop}[thm]{Proposition}
\newtheorem{lem}[thm]{Lemma}
\newtheorem{cor}[thm]{Corollary}

\theoremstyle{definition}
\newtheorem{defi}[thm]{Definition}

\newtheorem{rem}[thm]{Remark}

\newtheorem{exx}[thm]{Example}
\newenvironment{ex}
  {\pushQED{\qed}\exx}
  {\popQED\endexx}

\numberwithin{equation}{subsection}

\title{A stacky approach to $p$-adic Hodge theory}
\author{Maximilian Hauck}
\date{6th September 2024}

\begin{document}

\maketitle

\begin{abstract}
We use the stacky approach to $p$-adic cohomology theories recently developed by Drinfeld and Bhatt--Lurie to generalise known comparison theorems in $p$-adic Hodge theory so as to accommodate coefficients. More precisely, we establish a comparison between the rational crystalline cohomology of the special fibre and the rational $p$-adic étale cohomology of the arithmetic generic fibre of any proper $p$-adic formal scheme $X$ which allows for coefficients in any crystalline local system on the generic fibre of $X$; moreover, we also prove a comparison between the Nygaard filtration and the Hodge filtration for coefficients in an arbitrary gauge in the sense of Bhatt--Lurie. In the process, we develop a stacky approach to diffracted Hodge cohomology as introduced by Bhatt--Lurie, establish a version of the Beilinson fibre square of Antieau--Mathew--Morrow--Nikolaus with coefficients in the proper case and prove a comparison between syntomic cohomology and $p$-adic étale cohomology with coefficients in an arbitrary $F$-gauge. This work is the author's master's thesis at the University of Bonn.
\end{abstract}

\tableofcontents

\newpage

\section{Introduction}

Broadly speaking, the field of $p$-adic Hodge theory is concerned with studying cohomology theories for schemes over the \EEDIT{ring $\Z_p$ of $p$-adic integers}, where $p$ denotes a fixed prime throughout. For instance, one of its starting points is the following theorem comparing $p$-adic étale cohomology and \EDIT{crystalline cohomology}, conjectured by Fontaine in \cite{FontainePAdicHodgeTheory} and first proven by Faltings (see \cite{Faltings}):

\begin{thm}
\label{thm:intro-ccrys}
\EDIT{Let $X$ be a smooth and proper $\Z_p$-scheme.} For any $i\geq 0$, there is a natural isomorphism
\begin{equation*}
H^i_\et(X_{\ol{\Q}_p}, \Q_p)\tensor_{\Q_p} B_\crys\cong H^i_\crys(X_{\F_p})\tensor_{\Z_p} B_\crys
\end{equation*}
compatible with the filtrations, \EDIT{Frobenius endomorphisms} and Galois actions on either side. Here, $B_\crys$ denotes Fontaine's \EDIT{crystalline period ring} (see Section \ref{subsect:beilfibsq-crys} for a definition).
\end{thm}

In recent years, the field of $p$-adic Hodge theory has been given a new impetus by the groundbreaking discovery of \emph{prismatic cohomology} by Bhatt--Scholze, \EDIT{precursors of which can already be found in the papers \cite{BMS} and \cite{THHandPAdicHodgeTheory} of Bhatt--Morrow--Scholze}: In \cite{Prisms}, \EDIT{the authors} construct a new cohomology theory attached to $p$-adic formal schemes which specialises to many previously known $p$-adic cohomology theories such as étale cohomology, de Rham cohomology and crystalline cohomology as well as more recent ones such as $A_\inf$-cohomology as introduced in \cite{BMS} or $q$-de Rham cohomology as described in \cite{qDeformations}.

While the idea that one may compute the value of a cohomology theory attached to a scheme $X$ by instead computing the coherent cohomology of a suitably defined stack attached to $X$ goes back, in the case of de Rham cohomology, to work of Simpson in the 1990s, see \cite{Simpson} and \cite{Simpson2}, this approach has only been starting to be fully exploited in the course of the last few years with the formulation of prismatic cohomology in terms of stacks independently developed by Bhatt--Lurie and Drinfeld in \cite{APC}, \cite{PFS} and \cite{Prismatization}. Roughly speaking, \EDIT{similarly to how one can attach to any $p$-adic formal scheme $X$ its \emph{de Rham stack} $X^\dR$, which has the property that coherent cohomology of the structure sheaf $\O_{X^\dR}$ agrees with the ($p$-completed) de Rham cohomology of $X$ and whose construction we will recall in Section \ref{subsect:derham},} they functorially attach a stack $X^\prism$ to any \EDIT{such} $X$ with the feature that coherent cohomology of the structure sheaf $\O_{X^\prism}$ agrees with the (absolute) prismatic cohomology of $X$ \EEDIT{if $X$ is $p$-quasisyntomic (i.e.\ locally has the form $\Spf R$ for $p$-quasisyntomic rings $R$ in the sense of \cite[Def. C.6]{APC})}; correspondingly, the stack $X^\prism$ is called the \emph{prismatisation} of $X$. The upshot of this picture is that various statements about prismatic cohomology and related cohomology theories now admit a ``geometric'' formulation; \EDIT{for example}, the comparison between prismatic cohomology and de Rham cohomology from \cite[Thm. 5.4.2]{APC} can be reinterpreted as saying that, for any smooth $p$-adic formal scheme $X$, there is a functorial isomorphism
\begin{equation*}
(X_{p=0})^\prism\cong X^\dR\;.
\end{equation*}

\EDIT{In the case of de Rham cohomology, one can define a filtered refinement $X^{\dR, +}$ of the stack $X^\dR$ which also takes into account the Hodge filtration, see Section \ref{subsect:fildrstack} for details. In a similar spirit, as} prismatic cohomology is also naturally equipped with a filtration called the \emph{Nygaard filtration}, Bhatt--Lurie and Drinfeld also consider a filtered refinement $X^\N$ of the prismatisation of a $p$-adic formal scheme called the \emph{Nygaard-filtered prismatisation} of $X$. \EDIT{As it turns out, the stack $X^\N$ contains two open substacks
\begin{equation*}
j_\dR: X^\prism\hookrightarrow X^\N\;, \hspace{0.5cm} j_\HT: X^\prism\hookrightarrow X^\N
\end{equation*}
isomorphic to $X^\prism$.} Gluing them together, Bhatt--Lurie and Drinfeld obtain a third stack $X^\Syn$ attached to $X$, which is the so-called \emph{syntomification} of $X$. 

As the naming suggests, \EDIT{the stack $X^\Syn$ is connected to the \emph{syntomic cohomology} of $X$, the idea of which goes back to work of Fontaine--Messing in the 1980s, see \cite{FontaineMessing}; however, their definition only works well for rational coefficients and the first definition of syntomic cohomology also suited for integral coefficients was given by Bhatt--Morrow--Scholze in \cite{THHandPAdicHodgeTheory} via topological Hochschild homology. More precisely, for any \EEDIT{$n\in\Z$}, one can define \EEDIT{line} bundles \EEDIT{$\O\{n\}$} on $X^\Syn$, see \cref{ex:syntomification-bktwists}, and, if $X$ is $p$-quasisyntomic, they have the property that
\begin{equation*}
R\Gamma(X^\Syn, \O\{n\})\cong R\Gamma_\Syn(X, \Z_p(n))\;,
\end{equation*}
where the right-hand side denotes syntomic cohomology in weight $n$ in the sense of \EEDIT{Bhatt--Morrow--Scholze}; this will be proven in upcoming work of Bhatt--Lurie. In practice, syntomic cohomology is often useful as a link between $p$-adic étale cohomology and de Rham cohomology and we will see this connection play out in our proof of \cref{thm:intro-cryset}.}

Perhaps most importantly for us, the formulation of cohomology theories in terms of stacks also furnishes natural categories of coefficients for these cohomology theories: \EDIT{for instance,} we may define the prismatic cohomology of a $p$-adic formal scheme $X$ with coefficients in any \EEDIT{quasi-coherent complex} $E$ on $X^\prism$ simply as $R\Gamma(X^\prism, E)$. The most important such category of coefficients will be the category of \emph{$F$-gauges} on $X$ defined by
\begin{equation*}
\FGauge_\prism(X)\coloneqq \D(X^\Syn)\;.
\end{equation*}
The fact that (perfect) $F$-gauges play a central role in the theory is due to the fact that they seem to be some sort of ``universal coefficients'' for $p$-adic cohomology theories, \EDIT{i.e.\ they appear to be good candidates for a theory of $p$-adic motives}: namely, they admit various realisation functors towards more classical notions of coefficients for $p$-adic cohomology theories, e.g.\ a \emph{de Rham realisation}
\begin{equation*}
T_\dR: \Vect(X^\Syn)\rightarrow \Vect(X^\dR)\cong \Vect^{\nabla}(X)
\end{equation*}
towards the category of vector bundles on $X$ equipped with a flat connection, \EDIT{which arises via pullback from a naturally defined map
\begin{equation*}
i_\dR: X^\dR\rightarrow X^\N\;;
\end{equation*}
a filtered refinement
\begin{equation*}
T_{\dR, +}: \Vect(X^\Syn)\rightarrow \Vect(X^{\dR, +})\;,
\end{equation*}
where the right-hand side may roughly be described as filtered vector bundles on $X$ equipped with a flat connection satisfying Griffiths transversality, \EEDIT{see \cref{rem:fildrstack-vect}}, which again arises from a naturally defined map}
\begin{equation*}
\EDIT{
i_{\dR, +}: X^{\dR, +}\rightarrow X^\N\;;
}
\end{equation*}
an \emph{étale realisation} 
\begin{equation*}
T_\et: \Vect(X^\Syn)\rightarrow \Loc_{\Z_p}(X_\eta)
\end{equation*}
towards the category of pro-étale $\Z_p$-local systems on the generic fibre $X_\eta$ of $X$; \EDIT{and a \emph{crystalline realisation}
\begin{equation*}
T_\crys: \Vect(X^\Syn)\rightarrow \Vect((X_{p=0})^\Syn)\;,
\end{equation*}
which may be postcomposed with a natural functor
\begin{equation*}
\Vect((X_{p=0})^\Syn)\rightarrow \Isoc^\phi(X_{p=0}/\Z_p)
\end{equation*}
towards the category of \EDIT{$F$-isocrystals} on the special fibre $X_{p=0}$ of $X$.} The details of this picture are discussed in Section \ref{subsect:syntomification}.

It is the aim of this thesis to make use of these features of the stacky approach to $p$-adic cohomology theories and, in particular, syntomic cohomology to generalise several known comparison theorems between $p$-adic cohomology theories \EEDIT{so as} to accommodate coefficients. The main result of this type we prove is the following one, \EDIT{which compares} the crystalline cohomology of the special fibre of a smooth proper $p$-adic formal scheme $X$ with the étale cohomology of its generic fibre. \EDIT{Note that, compared to \cref{thm:intro-ccrys}, it is concerned with the \emph{arithmetic} generic fibre $X_\eta$ while \cref{thm:intro-ccrys} uses the \emph{geometric} generic fibre, i.e.\ the base change of $X_\eta$ to an algebraic closure of $\Q_p$. For the relevant definitions about crystalline local systems, we refer to Section \ref{subsect:review-cryslocsys}.}

\begin{thm}[\cref{thm:cryset-main}]
\label{thm:intro-cryset}
Let $X$ be a $p$-adic formal scheme \EEDIT{which is smooth and proper over $\Spf\Z_p$}. For any crystalline local system $L$ on $X_\eta$ with Hodge--Tate weights all at most $-i-1$ for some $i\geq 0$, let $\cal{E}$ be its associated $F$-isocrystal. Then there is a natural morphism
\begin{equation*}
R\Gamma_\crys(X_{p=0}, \cal{E})[\tfrac{1}{p}][-1]\rightarrow R\Gamma_\proet(X_\eta, L)[\tfrac{1}{p}]\;,
\end{equation*}
which induces an isomorphism
\begin{equation*}
\tau^{\leq i} R\Gamma_\crys(X_{p=0}, \cal{E})[\tfrac{1}{p}][-1]\cong\tau^{\leq i} R\Gamma_\proet(X_\eta, L)[\tfrac{1}{p}]
\end{equation*}
and an injection on $H^{i+1}$.
\end{thm}

This generalises a previous result of Colmez--Niziol who obtained essentially the same statement in the special case of $L=\Z_p(n)$ being a Tate twist, see \cite[Cor. 1.4]{ColmezNiziol}. Our strategy to prove this will be to first construct an $F$-gauge corresponding to any crystalline \EDIT{local system} building on the works \cite{GuoReinecke} of Guo--Reinecke and \cite{GuoLi} of Guo--Li; more precisely, we prove that:

\begin{thm}[\cref{thm:cryset-locsysfgauges}]
\label{thm:intro-locsysfgauges}
Let $X$ be a smooth $p$-adic formal scheme. Then there is a fully faithful embedding
\begin{equation*}
\Loc^\crys_{\Z_p}(X_\eta)\hookrightarrow \Perf(X^\Syn)
\end{equation*}
which is a weak right inverse to the étale realisation. The image of a crystalline local system $L$ under this embedding is called the \emph{associated $F$-gauge} $E$ of $L$ and has the property that $T_\dR(E)$ identifies with the associated $F$-isocrystal of $L$.
\end{thm}

Then the proof is divided into two parts: On the one hand, we show that, in the range which is relevant to us, étale cohomology with coefficients in a crystalline local system agrees with syntomic cohomology with coefficients in the associated $F$-gauge. More precisely, we obtain the following theorem:

\begin{thm}[\cref{thm:syntomicetale-mainfine}]
\label{thm:intro-syntomicetale}
Let $X$ be a smooth qcqs $p$-adic formal scheme. For any vector bundle $F$-gauge $E\in\Vect(X^\Syn)$ with Hodge--Tate weights all at most $-i-1$ for some $i\geq 0$, the natural morphism
\begin{equation*}
R\Gamma(X^\Syn, E)\rightarrow R\Gamma_\proet(X_\eta, T_\et(E))
\end{equation*}
induces an isomorphism
\begin{equation*}
\tau^{\leq i} R\Gamma(X^\Syn, E)\cong\tau^{\leq i} R\Gamma_\proet(X_\eta, T_\et(E))
\end{equation*}
and an injection on $H^{i+1}$.
\end{thm}

\comment{
Moreover, we also show that the above comparison between syntomic cohomology and étale cohomology holds in an even larger range of Hodge--Tate weights if the formal scheme $X$ admits a map to $\Spf\O_C$, see \cref{thm:syntomicetale-mainfineoc}.
}

On the other hand, in order to make the connection between syntomic cohomology and crystalline cohomology, we prove the following version of the Beilinson fibre square from \cite{BeilFibSq} with coefficients, which recovers the previously known result \cite[Thm. 6.17]{BeilFibSq} of Antieau--Mathew--Morrow--Nikolaus in the case of smooth and proper $p$-adic formal schemes upon plugging in a Breuil-Kisin twist $\O\{n\}$ for the $F$-gauge $E$:

\begin{thm}[\cref{cor:beilfibsq-coeffs}]
\label{thm:intro-beilfibsq}
Let $X$ be a $p$-adic formal scheme which is smooth and proper over $\Spf\Z_p$. For any perfect $F$-gauge $E\in\Perf(X^\Syn)$, there is a natural pullback square
\begin{equation*}
\begin{tikzcd}
R\Gamma(X^\Syn, E)[\frac{1}{p}]\ar[r]\ar[d] & R\Gamma((X_{p=0})^\Syn, T_\crys(E))[\frac{1}{p}]\ar[d] \\
R\Gamma(X^{\dR, +}, T_{\dR, +}(E))[\frac{1}{p}]\ar[r] & R\Gamma(X^\dR, T_\dR(E))[\frac{1}{p}]
\end{tikzcd}
\end{equation*}
\EDIT{in the category $\D(\Q_p)$.}
\end{thm}

We note here that this really also gives a new proof of the special case $E=\O\{n\}$ treated in \cite[Thm. 6.17]{BeilFibSq}: Indeed, while our argument uses results from \cite{GuoReinecke}, which in turn are proved using the Beilinson fibre square in the form from \cite[Thm. 6.17]{BeilFibSq}, Du--Liu--Moon--Shimizu have shown in \cite{DuLiuMoonShimizu} that one can also obtain the same results via a different route which does not make use of the Beilinson fibre square, see also \cite[Rem. 1.9]{GuoReinecke}. In particular, in contrast to the proof of Antieau--Mathew--Morrow--Nikolaus, our proof does not make use of topological Hochschild homology and only relies on methods from the theory of prismatic cohomology. Let us also remark here that, \EDIT{during the preparation of this thesis, we learnt that,} in currently still unpublished work, Lurie has recently found yet another proof of the Beilinson fibre square, which also makes use of the stacky formulation of syntomic cohomology, but still significantly differs from our proof.

Moreover, we also show that the pullback square from \cref{thm:intro-beilfibsq} actually comes from a commutative square of stacks 
\begin{equation}
\label{eq:intro-beilfibsq}
\begin{tikzcd}
\Z_p^\dR\ar[r]\ar[d] & \Z_p^{\dR, +}\ar[d] \\
\F_p^\Syn\ar[r] & \Z_p^\Syn\nospacepunct{\;,}
\end{tikzcd}
\end{equation}
which gives a new perspective on the Beilinson fibre square as a statement about the geometry of the stack $\Z_p^\Syn$: Indeed, one may reinterpret \cref{thm:intro-beilfibsq} as saying that, at least from the perspective of cohomology after inverting $p$, the stack $\Z_p^\Syn$ \EDIT{behaves like it} is glued together from the stacks $\F_p^\Syn$ and $\Z_p^{\dR, +}$. Indeed, we also prove that this interpretation is reflected on the level of derived categories:

\begin{thm}[\cref{thm:beilfibsq-categorical}]
\label{thm:intro-beilfibsqcat}
The functor
\begin{equation}
\label{eq:intro-beilfibsqcat}
\Perf(\Z_p^\Syn)[\tfrac{1}{p}]\rightarrow \DF(\Q_p)\times_{\D(\Q_p)} \D(\Mod^\phi(\Q_p))
\end{equation}
induced by the diagram (\ref{eq:intro-beilfibsq}) is a fully faithful embedding. Here, $\DF(\Q_p)$ denotes the filtered derived category of $\Q_p$-vector spaces and $\Mod^\phi(\Q_p)$ is the category of $\phi$-modules over $\Q_p$.
\end{thm}

\EDIT{In particular, note that the above is a categorical generalisation of \cref{thm:intro-beilfibsq}: indeed, for any $E\in\Perf(\Z_p^\Syn)$, one recovers the statement of \cref{thm:intro-beilfibsq} by computing $\RHom(\O, E)[\tfrac{1}{p}]$ via the right-hand side of (\ref{eq:intro-beilfibsqcat}).}

Our second main result is similar in spirit to our version of the Beilinson fibre square: Namely, we prove the existence of a natural pullback square relating the Nygaard filtration on prismatic cohomology with the Hodge filtration on de Rham cohomology.

\begin{thm}[\cref{cor:nygaardhodge-coeffs}]
\label{thm:intro-nygaardhodge}
Let $X$ be a $p$-adic formal scheme which \EDIT{is smooth and proper} over $\Spf\Z_p$. For any perfect gauge $E\in\Perf(X^\N)$, there is a natural pullback square
\begin{equation*}
\begin{tikzcd}
R\Gamma(X^\N, E)[\frac{1}{p}]\ar[r]\ar[d] & R\Gamma(X^\prism, j_\dR^*E)[\frac{1}{p}]\ar[d] \\
R\Gamma(X^{\dR, +}, i_{\dR, +}^*E)[\frac{1}{p}]\ar[r] & R\Gamma(X^\dR, i_\dR^*E)[\frac{1}{p}]\nospacepunct{\;.}
\end{tikzcd}
\end{equation*}
If the Hodge--Tate weights of $E$ are all at least \EEDIT{$-p$}, see \cref{defi:nygaardhodge-htweights}, then the statement already holds integrally.
\end{thm}

We note here that putting $E=\O\{n\}$ above recovers the previous result \cite[Prop. 5.5.12]{APC}, albeit only in the case where $X$ is proper. Again, the pullback square above actually comes from a commutative diagram
\begin{equation*}
\begin{tikzcd}
\Z_p^\dR\ar[r]\ar[d] & \Z_p^{\dR, +}\ar[d] \\
\Z_p^\prism\ar[r] & \Z_p^\N
\end{tikzcd}
\end{equation*} 
of the corresponding stacks. 

To prove \cref{thm:intro-nygaardhodge}, we reduce to an assertion on the level of the associated graded and then make use of the fact that we can explicitly describe the derived category of a certain closed substack $\Z_p^{\HT, c}\subseteq \Z_p^\N$ relevant for the corresponding statement on the associated graded level. To establish this description, we more generally develop a stacky formulation of the theory of \emph{diffracted Hodge cohomology} that was introduced in \cite[§4.7]{APC}: \EDIT{For any bounded $p$-adic formal scheme $X$, its diffracted Hodge cohomology $R\Gamma_\dHod(X)$ is a derived $p$-complete complex of $\Z_p$-modules, which can be regarded as a deformation of the de Rham cohomology of $X$. Indeed, if $X$ is smooth, after modding out $p$, the two agree up to Frobenius twists, see \cite[Rem. 4.7.18]{APC}; moreover, the complex $R\Gamma_\dHod(X)$ is naturally equipped with an ascending filtration $\Fil^\conj_\bullet R\Gamma_\dHod(X)$ called the \emph{conjugate filtration} with the property that, for smooth $X$, the associated graded pieces of the conjugate filtration on the diffracted Hodge cohomology of $X$ agree with the ones of the Hodge filtration on the de Rham cohomology of $X$ -- more precisely, they both identify with \EEDIT{the} Hodge cohomology \EEDIT{of $X$}, i.e., for any $n\in\Z$, we have
\begin{equation*}
\gr^\conj_n R\Gamma_\dHod(X)\cong R\Gamma(X, \widehat{\Omega}_{X/\Z_p}^n)[-n]\cong \gr^n_\Hod R\Gamma_\dR(X)\;,
\end{equation*}
where $\widehat{\Omega}_{X/\Z_p}$ denotes the $p$-completed cotangent sheaf of $X$ over $\Spf\Z_p$, as usual. Finally, the complex $R\Gamma_\dHod(X)$ is also equipped with an endomorphism $\Theta$ called the \emph{Sen operator}, which can be used to prove an integral refinement of the Deligne--Illusie theorem for smooth $X$ of dimension less than $p$, see \cite[Ex. 4.7.17, Rem. 4.7.18]{APC}. A brief review of the theory of diffracted Hodge cohomology is given in Section \ref{subsect:review-dhod}.} 

For any bounded $p$-adic formal scheme $X$, we functorially construct stacks $X^\dHod, X^{\dHod, c}$ and $X^{\HT, c}$ which geometrise \EDIT{the diffracted Hodge cohomology $R\Gamma_\dHod(X)$ of $X$ together with its conjugate filtration and the Sen operator.} More precisely, we show:

\begin{thm}[\cref{thm:fildhod-comparison}, \cref{thm:fildhod-comparisonfiltered}, \cref{thm:fildhod-comparisonsen}]
\label{thm:intro-fildhod}
Let $X$ be a bounded $p$-adic formal scheme and assume that $X$ is $p$-quasisyntomic and qcqs.
Then the following are true:
\begin{enumerate}[label=(\roman*)]
\item The pushforward of $\O_{X^\dHod}$ along the map $X^\dHod\rightarrow \Z_p^\dHod$ identifies with $R\Gamma_\dHod(X)$.
\item The pushforward of $\O_{X^{\dHod, c}}$ along the map $X^{\dHod, c}\rightarrow\Z_p^{\dHod, c}$ identifies with $\Fil_\bullet^\conj R\Gamma_\dHod(X)$ under the Rees equivalence \EDIT{(\cref{prop:rees-main})}.
\item Under the equivalence
\begin{equation*}
\D(\Z_p^{\HT, c})\cong \widehat{\D}_{\gr, D-\nilp}(\Z_p\{x, D\}/(Dx-xD-1))
\end{equation*}
\EDIT{from \cref{prop:fildhod-zpntzero}}, the underlying graded \EEDIT{$\Z_p[x]$-complex} of the pushforward of $\O_{X^{\HT, c}}$ along $X^{\HT, c}\rightarrow\Z_p^{\HT, c}$ identifies with $\Fil_\bullet^\conj R\Gamma_\dHod(X)$ under the Rees equivalence \EEDIT{and, under this identification, the operator 
\begin{equation*}
xD-i: \Fil_i^\conj R\Gamma_\dHod(X)\rightarrow \Fil_i^\conj R\Gamma_\dHod(X)
\end{equation*}
identifies with the Sen operator on $\Fil_i^\conj R\Gamma_\dHod(X)$ for all $i\in\Z$.} 
\end{enumerate}
\end{thm}

While the diffracted Hodge stack $X^\dHod$ has already been introduced in \cite[Constr. 3.8]{PFS}, albeit without an explicit mention of the comparison result (i) from \cref{thm:intro-fildhod}, to our current knowledge, neither the filtered refinement $X^{\dHod, c}$ nor the further refinement $X^{\HT, c}$ \EEDIT{appear} anywhere in the literature yet.

Finally, let us briefly discuss some open questions and possible directions of further research: First, note that \cref{thm:intro-beilfibsq} and consequently also \cref{thm:intro-cryset} only \EEDIT{apply} to smooth formal schemes $X$ which are proper over $\Spf\Z_p$. However, in reality, the properness assumption should not be necessary: Indeed, the previous \EDIT{results \cite[Cor. 1.4]{ColmezNiziol} of Colmez--Niziol and \cite[Thm. 6.17]{BeilFibSq} of Antieau--Mathew--Morrow--Nikolaus do} not need this additional assumption and the same is true for Lurie's new proof of the \EDIT{Beilinson fibre square} using the syntomification. However, at this time we are not sure if there is a way to remove this assumption also in our approach: It seems like one should be able to use a noetherian approximation argument to do this, \EDIT{i.e.\ write a quasi-coherent sheaf $E$ on $\Z_p^\Syn$ as the filtered colimit of its coherent subsheaves and then use the fact that filtered colimits commute with cohomology on $\Z_p^\Syn$ up to $p$-completion under favourable conditions, see \cite[Lem. 6.5.17]{FGauges}} -- however, to make this work, one needs to control the $p$-torsion measuring the failure of the statement of \cref{thm:intro-beilfibsq} to hold integrally, \EDIT{i.e., for any coherent sheaf $E$ on $\Z_p^\Syn$, one needs to bound the power of $p$ which annihilates the cofibre of the natural map
\begin{equation*}
\begin{tikzcd}[ampersand replacement=\&]
\cofib(R\Gamma(X^\Syn, E)\rightarrow R\Gamma((X_{p=0})^\Syn, T_\crys(E)))\ar[d] \\ \cofib(R\Gamma(X^{\dR, +}, T_{\dR, +}(E))\rightarrow R\Gamma(X^\dR, T_\dR(E)))
\end{tikzcd}
\end{equation*}
induced by the square (\ref{eq:intro-beilfibsq}), but} we currently do not know how to do \EDIT{this} \EEDIT{using our approach}. Analogous remarks also apply to the result from \cref{thm:intro-nygaardhodge}.

Second, even if one is not able to remove the properness hypothesis from \cref{thm:intro-beilfibsq} using our approach, one might still hope to be able to prove an analogous statement in the geometric case, i.e.\ for formal schemes $X$ which are \EEDIT{smooth and} proper over $\Spf\O_C$. This is interesting because it would most likely allow for a stacky proof of \EEDIT{the generalisation of the crystalline comparison theorem (\cref{thm:intro-ccrys}) to the case of coefficients in any crystalline local system on $X_\eta$ given in} \cite[Cor. C]{GuoReinecke} using the strategy from the proofs of \cite[Cor. 5.15, Prop. 5.20]{ColmezNiziol}.

\bigskip

\textbf{Outline of the thesis.} The thesis will be structured as follows: In Section \ref{sect:stacks}, we first review the basic definitions on stacks we are going to use and then move on to first discuss the construction of the de Rham stack and its Hodge-filtered variant in Section \ref{subsect:derham} and then introduce the prismatisation along with its Nygaard-filtered refinement as well as the syntomification and present their most important properties in Section \ref{subsect:prismsyn}. Subsequently, Section \ref{sect:conjdhod} will be devoted to developing the theory of the diffracted Hodge stack and its refinements, which we will then use in Section \ref{sect:nygaardhodge} to prove \cref{thm:intro-nygaardhodge}. In Section \ref{sect:beilfibsq}, we will then turn our attention to the Beilinson fibre square and first construct the commutative square (\ref{eq:intro-beilfibsq}) in Section \ref{subsect:beilfibsq-construction}, then use the following sections to prepare for our proof of \cref{thm:intro-beilfibsq}, which we will then carry out in Section \ref{subsect:beilfibsq-main}; finally, we will use \cref{thm:intro-beilfibsq} to relate our notion of syntomic cohomology to Fontaine--Messing syntomic cohomology and then prove \cref{thm:intro-beilfibsqcat}. After proving \cref{thm:intro-syntomicetale} in Section \ref{sect:syntomicetale} using our description of the reduced locus of the syntomification, which we introduce and discuss in sections \ref{subsect:syntomicetale-red} and \ref{subsect:syntomicetale-compred}, we use Section \ref{sect:cryset} to first \EDIT{review the main definitions surrounding crystalline local systems and subsequently move on to proving} \cref{thm:intro-locsysfgauges} and then \cref{thm:intro-cryset}. Finally, we devote Appendix \ref{sect:basechange} to the proof of some base change statements for various cartesian diagrams of stacks we will need throughout the thesis. In Appendix \ref{sect:finiteness}, we show that, for a $p$-adic formal scheme $X$ \EEDIT{which is smooth and proper over $\Spf\Z_p$}, pushforward along $X^\Syn\rightarrow\Z_p^\Syn$ preserves perfect $F$-gauges.

\bigskip

\textbf{Notations and conventions.} \EDIT{We assume knowledge of prismatic cohomology as introduced in \cite{Prisms} and of various other more standard $p$-adic cohomology theories such as de Rham cohomology, crystalline cohomology and $p$-adic (pro-)étale cohomology. Moreover, we freely make use of the theory of derived algebraic geometry as laid out in \cite{DAG} and of the theory of adic and, in particular, perfectoid spaces, see \cite{BerkeleyLectures} for an introduction; \EEDIT{however,} the reader will probably also get by without detailed knowledge of these two subjects. Furthermore, we use the language of $\infty$-categories in the style of Lurie throughout, see \cite{HTT}.

Unless explicitly stated otherwise, all our pullbacks and pushforwards are in the derived sense; i.e., when we write $f_*$ for a map $f: \cal{X}\rightarrow\cal{Y}$ of schemes/formal schemes/stacks, we really mean the derived pushforward $Rf_*: \cal{D}(\cal{X})\rightarrow\cal{D}(\cal{Y})$. Finally, when we speak of ``completeness'' of a module or complex with respect to some ideal, we will usually mean derived completeness as defined, for example, in \cite[Tag 091N]{Stacks} and distinguish any other usage of the term ``complete'' by speaking about ``classical completeness''. \EEDIT{For a ring $A$ and an ideal $I$,} we will use $\widehat{\D}(A)$ to denote the category of derived $I$-complete complexes of $A$-modules; \EEDIT{here, we omit the ideal $I$ from the notation as it will generally be clear from the context.} \EEDIT{Moreover, in the aforementioned situation, we denote the derived $I$-completion of a complex $M$ of $A$-modules by $M_I^\wedge$.}

Throughout, $p$ will denote a fixed prime and we fix an algebraic closure $\ol{\Q}_p$ of $\Q_p$, \EEDIT{whose completion we will denote by $C$.}
}

\bigskip

\textbf{Acknowledgements.} \EDIT{I heartily want to thank my advisor Guido Bosco for introducing me to the world of $p$-adic Hodge theory and suggesting this fascinating topic. I am very grateful for his continued support throughout my work on this thesis, for many long and fruitful discussions, his constant willingness to answer all of my questions and lots of helpful comments on an earlier version of this thesis. I also thank Bhargav Bhatt for suggesting that one may prove \cref{prop:beilfibsq-perfzpsyncrys} via a cohomology computation. \comment{Finally, I wish to express my immense gratitude to my parents who have always had my back and supported me in whichever way they could. I also want to deeply thank my girlfriend for always being there for me.} 
}

\newpage

\section{Stacks and $p$-adic cohomology theories}
\label{sect:stacks}

We present various stacks attached to a smooth $p$-adic formal scheme computing different cohomology theories. Throughout, we work with the fpqc topology on the category of commutative rings. In general, the \EDIT{idea} will be that, given a cohomology theory $R\Gamma_\typ(-)$ defined on a certain full subcategory $\C$ of $p$-adic formal schemes, we will functorially attach to each $X\in\C$ a stack $X^\typ$ with the property that, in good cases, coherent cohomology of the structure sheaf on the stack $X^\typ$ computes $R\Gamma_\typ(X)$, i.e.\
\begin{equation*}
R\Gamma(X^\typ, \O_{X^\typ})\cong R\Gamma_\typ(X)\;.
\end{equation*}
In this picture, comparison theorems and natural maps between cohomology theories \EDIT{should} translate into isomorphisms and maps between the corresponding stacks.

\subsection{Preliminaries on stacks}

\EDIT{Throughout, we will work with stacks in groupoids since the additional generality of stacks in $\infty$-groupoids is not really necessary for us. However, we remark that this choice does not affect any of the constructions or results; in particular, the outcome of the gluing procedure used to define the syntomification in \cref{defi:syntomification-def} does not depend on whether one works in the framework of stacks in groupoids or stacks in $\infty$-groupoids.}

\EDIT{
Most of the time, we will actually consider \emph{($p$-adic) formal stacks}, i.e.\ stacks $\cal{X}$ such that $\cal{X}(S)$ is empty whenever $S$ is not $p$-nilpotent. In general, any ring $R$ equipped with the $I$-adic topology for an ideal $I\subseteq R$ containing $p$ defines a formal stack $\Spf R$ via
\begin{equation*}
\Spf R\coloneqq \colim_n \Spec R/I^n\;.
\end{equation*}
Note that, if $R$ is classically $I$-complete, i.e.\ if $R\cong\lim_n R/I^n$, then the groupoid $(\Spf R)(S)$ identifies with the groupoid of continuous morphisms from $R$ to the discrete ring $S$, i.e.\ morphisms $R\rightarrow S$ killing some power of $I$.

Clearly, one can straightforwardly turn any stack $\cal{X}$ into a formal stack by passing to the stack $\cal{X}\times\Spf\Z_p$, whose fibre over a test ring $S$ is given by $\cal{X}(S)$ if $S$ is $p$-nilpotent and is empty otherwise. As we will soon work exclusively with formal stacks, however, we will often suppress the product with $\Spf\Z_p$ from the notation and restrict to the category of $p$-nilpotent rings $S$ as our test objects.
}

We \EDIT{now} recall the usual definition for properties of morphisms of stacks.

\begin{defi}
Let $\cal{P}$ be a property of morphisms of schemes which is stable under base change and fpqc-local on the target. Then we say that a morphism $f: \cal{X}\rightarrow\cal{Y}$ of stacks has property $\cal{P}$ if, for any morphism $\Spec S\rightarrow \cal{Y}$ from an affine scheme, the fibre product $\Spec S\times_{\cal{Y}} \cal{X}$ is a scheme and the base-changed morphism $\Spec S\times_{\cal{Y}} \cal{X}\rightarrow\Spec S$ has property $\cal{P}$.
\end{defi}

The corresponding definition for properties of stacks has to be modified slightly \EDIT{to accommodate the case of formal stacks:}

\begin{defi}
Let $\cal{P}$ be a property of \EDIT{formal} schemes which is fpqc-local. Then we say that a \EDIT{formal} stack $\cal{X}$ has property $\cal{P}$ if there is an fpqc cover from a formal scheme $\Spf R\rightarrow\cal{X}$ such that \EDIT{$\Spf R$} has property $\cal{P}$.
\end{defi}

Our conventions regarding quasi-coherent complexes on stacks will largely follow \cite{DAG}. Most importantly, we adopt the following definition:

\begin{defi}
Let $\cal{X}$ be a stack. For a ring $S$, denote by $\D(S)$ the derived $\infty$-category of $S$-modules and let $\Perf(S)$ be the full subcategory of perfect complexes; \EEDIT{moreover, let $\Vect(S)$ be the category of finite projective $S$-modules.} Then the \emph{$\infty$-category of quasi-coherent complexes} on $\cal{X}$ is defined by
\begin{equation*}
\D(\cal{X})\coloneqq\lim_{\Spec S\rightarrow\cal{X}} \D(S)\;,
\end{equation*}
\EDIT{where the limit is taken in the $\infty$-category of $\infty$-categories.} Similarly, the \emph{$\infty$-category of perfect complexes} on $\cal{X}$ is defined by
\begin{equation*}
\Perf(\cal{X})\coloneqq\lim_{\Spec S\rightarrow\cal{X}} \Perf(S)
\end{equation*}
and the \emph{$\infty$-category of vector bundles} on $\cal{X}$ is given by
\begin{equation*}
\Vect(\cal{X})\coloneqq\lim_{\Spec S\rightarrow\cal{X}} \Vect(S)\;.
\end{equation*}
\end{defi}

\EDIT{Note that these definitions make sense since the functor $R\mapsto \D(R)$ satisfies fpqc descent by \cite[Thm. 1.1.1]{DAGVIII} and being a perfect complex or a vector bundle are fpqc-local properties, respectively. In particular, observe that,} if $\cal{X}=\Spec R$ is an affine scheme, we have $\D(\cal{X})=\D(R)$. Perhaps most importantly for what follows, for a ring $R$ endowed with the $p$-adic topology having bounded $p^\infty$-torsion, the category $\D(\Spf R)$ identifies with the category $\widehat{\D}(R)$ of derived $p$-complete complexes of $R$-modules, see \cite[Prop. A.11]{DeltaRings}. 

The definition above also immediately furnishes for any morphism $f: \cal{X}\rightarrow\cal{Y}$ of stacks a pullback functor $f^*: \D(\cal{Y})\rightarrow\D(\cal{X})$, which recovers the usual derived pullback in the case of affine schemes. As this functor commutes with all colimits (and $\D(\cal{X}), \D(\cal{Y})$ are presentable), it admits a right-adjoint $f_*$ \EDIT{by the adjoint functor theorem}.

\begin{defi}
Let $\cal{X}$ be a stack. An object $F\in\D(\cal{X})$ is called \emph{connective} if $f^*F\in\D(S)$ is connective with respect to the standard $t$-structure for all $f: \Spec S\rightarrow\cal{X}$. The full subcategory of $\D(\cal{X})$ spanned by connective objects is denoted $\D^{\leq 0}(\cal{X})$.
\end{defi}

\EEDIT{Using the commutativity of pullbacks with colimits, one sees that the inclusion $\cal{D}^{\leq 0}(\cal{X})\rightarrow\cal{D}(\cal{X})$ commutes with colimits and hence admits a right-adjoint by the adjoint functor theorem; thus, by \cite[§1]{Aisles}, we conclude} that $\cal{D}^{\leq 0}(\cal{X})$ defines the connective part of a $t$-structure $(\D^{\leq 0}(\cal{X}), \D^{\geq 0}(\cal{X}))$ on the category $\D(\cal{X})$ for a stack $\cal{X}$. Note that, if $\cal{X}=\Spec R$ is an affine scheme, this agrees with the usual $t$-structure on $\D(R)$. Moreover, if $R$ is endowed with the $p$-adic topology and has bounded $p^\infty$-torsion, then we also recover the usual $t$-structure on $\widehat{\D}(R)$ by \cite[Ex. A.18]{DeltaRings}.

\subsection{De Rham stacks}
\label{subsect:derham}

The idea of the de Rham stack and its filtered refinement goes back to Simpson, see \cite[§5]{Simpson}. Here, we roughly follow the treatment of Bhatt in \cite[Ch. 2]{FGauges}. 

\subsubsection{Some preparations}

To begin explaining the construction, we denote by $\G_a$ the ring scheme whose functor of points is given by sending a ring $S$ to itself and which is represented by $\Spec\Z[t]$.

\begin{defi}
Let $\G_a^\sharp$ be the PD-hull of $\G_a$ at the origin, i.e.\ $\G_a^\sharp=\Spec\Z[t, \frac{t^2}{2!}, \frac{t^3}{3!}, \dots]$. It becomes a group scheme by virtue of the coalgebra structure on $\Z[t, \frac{t^2}{2!}, \frac{t^3}{3!}, \dots]$ given by
\begin{equation*}
\EDIT{
\begin{split}
\Z[t, \tfrac{t^2}{2!}, \tfrac{t^3}{3!}, \dots]&\rightarrow \Z[x, \tfrac{x^2}{2!}, \tfrac{x^3}{3!}, \dots]\tensor\Z[y, \tfrac{y^2}{2!}, \tfrac{y^3}{3!}, \dots] \\
\frac{t^n}{n!}&\mapsto \frac{(x+y)^n}{n!}=\sum_{i+j=n} \frac{x^i}{i!}\cdot \frac{y^j}{j!}\;.
\end{split}
}
\end{equation*}
\end{defi}

Thus, for a ring $S$, an $S$-valued point of $\G_a^\sharp$ is an element of $S$ together with a compatible system of divided powers, i.e.\ elements $s_1, s_2, \ldots\in S$ such that
\begin{equation}
\label{eq:div-powers}
s_ms_n=\binom{m+n}{n}s_{m+n}
\end{equation}
for all $m, n\geq 1$. Note that the canonical map $\Z[t]\hookrightarrow \Z[t, \frac{t^2}{2!}, \frac{t^3}{3!}, \dots]$ induces a morphism $\G_a^\sharp\rightarrow\G_a$. Moreover, there is a natural multiplicative action of $\G_a$ on $\G_a^\sharp$ induced by the coaction
\begin{equation*}
\EDIT{
\begin{split}
\Z[t, \tfrac{t^2}{2!}, \tfrac{t^3}{3!}, \dots]&\rightarrow \Z[t, \tfrac{t^2}{2!}, \tfrac{t^3}{3!}, \dots]\tensor \Z[x] \\
\frac{t^n}{n!}&\mapsto \frac{(tx)^n}{n!}=\frac{t^n}{n!}\cdot x^n\;.
\end{split}
}
\end{equation*}

For later use, we record a technical lemma about rings admitting no nontrivial $\G_a^\sharp$-torsors.

\begin{lem}
\label{lem:gasharp-acyclic}
There is a basis of the fpqc topology consisting of $\G_a^\sharp$-acyclic rings (i.e.\ rings for which all $\G_a^\sharp$-torsors are trivial).
\end{lem}
\begin{proof}
As $\G_a^\sharp$ is countably presented, this follows from a variant of the proof of \cite[Lem. 5.3.1]{FlatPurity}, which we briefly \EDIT{outline here}: Namely, for any ring $S$, fix a set $\cal{S}$ of representatives \EEDIT{of} isomorphism classes of countably presented faithfully flat $S$-algebras and put 
\begin{equation*}
\EEDIT{S_1\coloneqq \bigotimes_{T\in\cal{S}} T\coloneqq \colim_{\cal{S}'\subseteq\cal{S}\text{ finite}} \bigotimes_{T\in\cal{S}'} T\;.}
\end{equation*}
Iterating this procedure, we obtain a tower of \EEDIT{faithfully flat} $S$-algebras $S_1\rightarrow S_2\rightarrow\dots$ and we put $S_\omega\coloneqq\colim_n S_n$. Continuing in this manner, we obtain $S$-algebras $S_\alpha$ for each countable ordinal $\alpha$ by transfinite induction and finally put $S_{\omega_1}\coloneqq\colim_\alpha S_\alpha$. 

\EEDIT{Now observe that the map $S\rightarrow S_{\omega_1}$ is a flat cover: Indeed, each $S\rightarrow \bigotimes_{T\in\cal{S}'} T$ for $\cal{S}'\subseteq\cal{S}$ is a flat cover since flat covers are stable under base change and composition; this then implies that $S_1$ is flat over $S$ as flatness is stable under filtered colimits. Moreover, note that any transition map in the colimit defining $S_1$ is also faithfully flat; in particular, all transition maps are injective. This means that, for any nonzero $S$-module $M$, we have 
\begin{equation*}
M\tensor_S S_1=\colim_{\cal{S}'\subseteq\cal{S}\text{ finite}} M\tensor_S \bigotimes_{T\in\cal{S}'} T\;,
\end{equation*}
where each term of the colimit is nonzero and each transition map is injective, hence $M\tensor_S S_1$ is nonzero itself and, consequently, $S_1$ is a faithfully flat $S$-algebra. Using similar arguments, one can now indeed deduce that $S\rightarrow S_{\omega_1}$ is a flat cover by transfinite induction.}

\EEDIT{Now note that} \EDIT{the ring $S_{\omega_1}$ admits no nonsplit countably presented flat cover $S_{\omega_1}\rightarrow R$: indeed, as $R$ is a countably presented $S_{\omega_1}$-algebra, it descends to a countably presented $S_\alpha$-algebra $\widetilde{R}$ for some countable ordinal $\alpha$ and $S_\alpha\rightarrow\widetilde{R}$ is faithfully flat since this can be checked after base change to $S_{\omega_1}$. However, by construction of $S_{\alpha+1}$, this implies that there is a morphism $\widetilde{R}\rightarrow S_{\alpha+1}\rightarrow S_{\omega_1}$ of $S_\alpha$-algebras, which yields the desired splitting of $S_{\omega_1}\rightarrow R$. In particular, we conclude that any $\G_a^\sharp$-torsor over $S_{\omega_1}$ is split, i.e.\ trivial.}
\end{proof}

\begin{defi}
For a stack $\cal{X}$ and a vector bundle $E\in\Vect(\cal{X})$, we define a stack $\V(E)$ whose functor of points is given by
\begin{equation*}
\EDIT{
\V(E)(S)=\{\text{pairs $(f: \Spec S\rightarrow\cal{X}, f^*E^\vee\rightarrow\O_{\Spec S})$}\}^{\cong}\;.
}
\end{equation*}
\end{defi}

Note that, for $\cal{X}$ a scheme, this recovers the usual construction of the vector bundle associated to a locally free sheaf \EDIT{of finite rank}. Moreover, similarly to the construction of $\G_a^\sharp$ from $\G_a$, we may construct the PD-hull $\V(E)^\sharp$ of the zero section of any vector bundle $\V(E)$ \EDIT{as the substack of $\V(E)$ given on points by
\begin{equation*}
\begin{split}
\V(E)^\sharp(S)=\{&\text{pairs $(f: \Spec S\rightarrow\cal{X}, f^*E^\vee\rightarrow\O_{\Spec S})$} \\
&\hspace{0.5cm}\text{such that $f^*E^\vee\rightarrow\O_{\Spec S}$ admits divided powers}\}^{\cong}\;.
\end{split}
\end{equation*}
Here, by divided powers for a map $s_1: f^*E^\vee\rightarrow\O_{\Spec S}$, we mean maps $s_n: f^*E^{\tensor n\vee}\rightarrow\O_{\Spec S}$ for $n\geq 2$ satisfying the relation (\ref{eq:div-powers}) for all $m, n\geq 1$.}

Before we can move on, we need to make a short technical digression on quasi-ideals as described in \cite[§1.3]{Prismatization}. For proofs and details, we refer the reader to \cite{RingGroupoid}. Namely, first consider a morphism $d: A\rightarrow B$ of abelian groups. Then $A$ acts on $B$ via $a.b\coloneqq b+d(a)$ and we denote the corresponding action groupoid, \EDIT{i.e.\ the category whose objects are the elements of $B$ and which has a morphism from $b_1$ to $b_2$ for every $a\in A$ with $a.b_1=b_2$}, by $\Cone(d)$. Then $\Cone(d)$ will inherit a symmetric monoidal structure from the group structures of $A$ and $B$. If $d$ is injective, the groupoid $\Cone(d)$ is actually a set and identifies with $B/A$, where the quotient is formed via $d$. Note that, under the identification of $\infty$-groupoids and anima, this agrees with the usual cofibre of the map $d$; \EDIT{albeit in slightly different language, this is shown in \cite[Exposé XVIII, §1.4]{SGA4}.}

\begin{defi}
Let $d: A\rightarrow B$ be a morphism of sheaves of abelian groups \EDIT{on the category of commutative rings}. We denote by $\Cone(d: A\rightarrow B)=\Cone(d)$ the stack obtained by sheafifying the assignment
\begin{equation*}
S\mapsto \Cone(A(S)\xrightarrow{d(S)} B(S))\;.
\end{equation*}
Note that $\Cone(d)$ is equipped with the structure of a 1-truncated animated abelian group stack.
\end{defi}

\begin{ex}
For a commutative group scheme $G$, we have $BG=\Cone(G\rightarrow 0)$.
\end{ex}

In certain situations, the stack $\Cone(d)$ from the above definition has even more structure. This is captured by the notion of a quasi-ideal:

\begin{defi}
Let $C$ be a commutative ring scheme. A morphism $d: I\rightarrow C$ of $C$-module schemes is called a \emph{quasi-ideal} if, for any test ring $S$ and any $x, y\in I(S)$, we have $d(x)\cdot y=d(y)\cdot x$. \EDIT{Equivalently, this means that the following diagram commutes:}
\begin{equation*}
\EDIT{
\begin{tikzcd}[ampersand replacement=\&]
 \& I\times C\ar[rd, "\mathrm{act}"] \& \\
I\times I\ar[ru, "\id\times d"] \ar[rd, "d\times\id", swap] \& \& I \\
\& C\times I\ar[ru, "\mathrm{act}", swap] \& 
\end{tikzcd}
}
\end{equation*}
\end{defi}

\begin{ex}
The morphism $\G_a^\sharp\rightarrow\G_a$ from above is a quasi-ideal: \EDIT{To see this, observe that, in our case, the top and bottom map from the above definition are induced by the maps
\begin{equation*}
\begin{split}
\Z[t, \tfrac{t^2}{2!}, \tfrac{t^3}{3!}, \dots]&\rightarrow \Z[x, \tfrac{x^2}{2!}, \tfrac{x^3}{3!}, \dots]\tensor\Z[y, \tfrac{y^2}{2!}, \tfrac{y^3}{3!}, \dots] \\
\frac{t^n}{n!}&\mapsto \frac{x^n}{n!}\cdot y^n
\end{split}
\end{equation*}
and 
\begin{equation*}
\begin{split}
\Z[t, \tfrac{t^2}{2!}, \tfrac{t^3}{3!}, \dots]&\rightarrow \Z[x, \tfrac{x^2}{2!}, \tfrac{x^3}{3!}, \dots]\tensor\Z[y, \tfrac{y^2}{2!}, \tfrac{y^3}{3!}, \dots] \\
\frac{t^n}{n!}&\mapsto x^n\cdot\frac{y^n}{n!}\;,
\end{split}
\end{equation*}
respectively, which are evidently the same.}
\end{ex}

Given a quasi-ideal $d: I\rightarrow C$, the stack $\Cone(d)$ acquires the structure of a 1-truncated animated ring stack, see \cite[Ch. 3]{RingGroupoid}.

\subsubsection{The de Rham stack}

Now we are in the position to define the de Rham stack. From now on, let $X$ be a bounded $p$-adic formal scheme; here, $X$ being bounded means that it locally has the form $\Spf R$ for $R$ having bounded $p^\infty$-torsion. We work with \EDIT{$p$-adic formal stacks}.

\begin{defi}
Consider the stack
\begin{equation*}
\G_a^\dR\coloneqq\Cone(\G_a^\sharp\xrightarrow{\can}\G_a)\;.
\end{equation*}
The \emph{de Rham stack} $X^\dR$ of $X$ is the stack defined by
\begin{equation*}
X^\dR(S)\coloneqq \Map(\Spec\G_a^\dR(S), X)\;,
\end{equation*}
where the mapping space is computed in derived algebraic geometry. If $X=\Spf R$ is affine, we also write $R^\dR$ in place of $X^\dR$.
\end{defi}

Note that the notation $\Spec\G_a^\dR(S)$ actually makes sense since $\G_a^\dR$ has the structure of a 1-truncated animated $\G_a$-algebra stack by our digression on quasi-ideals above. \EDIT{In other words,} $\G_a^\dR(S)$ is a 1-truncated animated $R$-algebra. 

\begin{ex}
\label{ex:drstack-perfd}
Let $R$ be a perfectoid ring and denote by $(A_\inf(R), I)=(\Prism_R, I)$ the initial object of the absolute prismatic site of $R$; i.e.\ $A_\inf(R)=W(R^\flat)$, where, as usual,
\begin{equation*}
R^\flat\coloneqq\lim_{x\mapsto x^p} R\cong\lim_{x\mapsto x^p} R/p
\end{equation*}
and $I$ is the kernel of the Fontaine map
\begin{equation*}
\theta: A_\inf(R)\rightarrow R\;, \hspace{0.5cm} \sum_k p^k[x_k]\mapsto \sum_k p^kx_k^\sharp\;;
\end{equation*}
here, for $x\in R^\flat=\lim_{x\mapsto x^p} R$, we denote by $x^\sharp$ the image of $x$ under the projection onto the first coordinate. \EDIT{As $I$ is principal, we may pick a generator $\xi\in A_\inf(R)$} and let
\begin{equation*}
A_\crys(R)\coloneqq A_\inf(R)[\tfrac{\xi^n}{n!}: n\geq 0]_{(p)}^\wedge
\end{equation*}
be the $p$-completed divided power envelope of $A_\inf(R)$ at $\xi$; more canonically, one may describe $A_\crys(R)$ as the $p$-completed divided power envelope of $A_\inf(R)$ at $I$. Let us that we are using the free divided power envelope here, i.e. $A_\inf(R)[\tfrac{\xi^n}{n!}: n\geq 0]$ denotes the quotient of the polynomial ring $A_\inf(R)[s_1, s_2, \dots]$ by the relations $s_1=\xi$ and (\ref{eq:div-powers}); in particular, this is not the same as the subring of $A_\inf(R)[\tfrac{1}{p}]$ generated by the elements $\tfrac{\xi^n}{n!}$ for $n\geq 0$.

We claim that
\begin{equation}
\label{eq:drstack-perfd}
R^\dR\cong\Spf A_\crys(R)\;,
\end{equation}
\EDIT{where $A_\crys(R)$ is endowed with the $(p, \xi)$-adic topology (note that this agrees with the $p$-adic topology due to the existence of divided powers for $\xi$).} Indeed, we may describe this isomorphism on $S$-valued points for $\G_a^\sharp$-acyclic $p$-nilpotent rings $S$ by \cref{lem:gasharp-acyclic}. Then $\G_a^\dR(S)=\Cone(\G_a^\sharp(S)\rightarrow S)$ and an $S$-valued point of \EDIT{$R^\dR$} is the datum of a morphism $R\rightarrow\Cone(\G_a^\sharp(S)\rightarrow S)$. We will shortly show that this yields a contractibly unique lift
\begin{equation*}
\begin{tikzcd}
A_\inf(R)\ar[r, dotted]\ar[d] & S\ar[d] \\
R=A_\inf(R)/I\ar[r] & \Cone(\G_a^\sharp(S)\rightarrow S)\nospacepunct{\;.}
\end{tikzcd}
\end{equation*}
The commutativity of the diagram then provides a system of divided powers for the image of $I$ in $S$ and hence induces a unique map $A_\crys(R)\rightarrow S$; i.e.\ we obtain an $S$-valued point of $\Spf A_\crys(R)$. \EDIT{Finally, one easily sees that the construction is reversible: an $S$-valued point of $\Spf A_\crys(R)$ furnishes a composite map $A_\inf(R)\rightarrow A_\crys(R)\rightarrow S$ together with a system of divided powers for the image of $I$ in $S$ and hence a factorisation of the composite $A_\inf(R)\rightarrow S\rightarrow\Cone(\G_a^\sharp(S)\rightarrow S)$ through $R=A_\inf(R)/I$, as desired. Thus, (\ref{eq:drstack-perfd}) is proved.}

To perform the missing deformation-theoretic argument, first observe that the image of $\G_a^\sharp(S)\rightarrow S$ is locally nilpotent (indeed, if $n$ is large enough so that $p^n=0$ in $S$, then $x^{p^n}=(p^n)!\cdot\frac{x^{p^n}}{(p^n)!}=0$ for any $x\in S$ admitting divided powers), hence $S\rightarrow\Cone(\G_a^\sharp(S)\rightarrow S)$ is surjective on $\pi_0$ with locally nilpotent kernel. As the cotangent complex $L\Omega_{R^\flat/\F_p}$ vanishes by virtue of $R^\flat$ being perfect (see e.g.\ \cite[Lem. 3.5]{PrismaticBhatt}), derived deformation theory\footnote{More precisely, we are using the following assertion: Let $A$ be an animated ring and $B$ an $A$-algebra such that $L\Omega_{B/A}$ vanishes. For any map of animated $A$-algebras $C'\rightarrow C$ which is surjective on $\pi_0$ with locally nilpotent kernel, any $A$-algebra map $B\rightarrow C$ lifts uniquely to an $A$-algebra map $B\rightarrow C'$. This roughly follows by combining a variant of the proof of \cite[Prop. 11.2.1.2]{SAG} with a noetherian approximation argument as in the proof of \cite[Cor. 5.2.10]{FlatPurity}.} shows that the composition 
\begin{equation*}
R^\flat=\lim_{x\mapsto x^p} R/p\rightarrow R/p\rightarrow \Cone(\G_a^\sharp(S)\rightarrow S)/^\mathbb{L} p\;,
\end{equation*}
where the first map is projection onto the first coordinate, lifts uniquely to a map $R^\flat\rightarrow S/^\mathbb{L} p$ and composition with projection onto $\pi_0$ yields a map $R^\flat\rightarrow S/p$ (conversely, $R^\flat\rightarrow S/^\mathbb{L} p$ can uniquely be recovered from this map as $S/^\mathbb{L} p\rightarrow S/p$ is an isomorphism on $\pi_0$). Since also $L\Omega_{W(R^\flat)/p^n/\Z/p^n}$ vanishes for any $n\geq 1$ (indeed, this may be checked after reduction mod $p$ \EDIT{by derived Nakayama}, but $L\Omega_{W(R^\flat)/p^n/\Z/p^n}\tensorL_{\Z/p^n} \F_p\cong L\Omega_{R^\flat/\F_p}=0$), the same argument shows that this map lifts uniquely to maps $A_\inf(R)=W(R^\flat)\rightarrow S/p^n$. As $S$ is $p$-nilpotent, taking $n$ sufficiently large, we obtain the desired map $A_\inf(R)\rightarrow S/p^n=S$.
\end{ex}

\begin{ex}
\label{ex:drstack-perf}
In the previous example, let $R=k$ be a perfect field of characteristic $p$. Then some elementary manipulations show that
\begin{equation*}
A_\crys(k)\cong (W(k)[x_1, x_2, \dots]/(px_1-p^p, px_2-x_1^p, \dots))^\wedge_{(p)}
\end{equation*}
and hence
\begin{equation*}
k^\dR\cong\Spf (W(k)[x_1, x_2, \dots]/(px_1-p^p, px_2-x_1^p, \dots))^\wedge_{(p)}\;. \qedhere
\end{equation*}
\end{ex}

In good cases, coherent cohomology on the de Rham stack of $X$ actually computes \EDIT{the} ($p$-completed) de Rham cohomology of $X$:

\begin{thm}
\label{thm:drstack-comparison}
Let $X$ be a \EEDIT{smooth qcqs} $p$-adic formal scheme and write $\pi_{X^\dR}: X^\dR\rightarrow \Spf\Z_p$. Then $\H_\dR(X)\coloneqq \pi_{X^\dR, *}\O_{X^\dR}$ identifies with $R\Gamma_\dR(X)$.
\end{thm}
\begin{proof}
This is easily deduced from \cite[Thm. 2.5.6]{FGauges}.
\end{proof}

Motivated by this result, we make the following definition:

\begin{defi}
Let $X$ be a \EEDIT{smooth qcqs} $p$-adic formal scheme. For a quasi-coherent complex $E\in\D(X^\dR)$, we define the \emph{de Rham cohomology} of $X$ with coefficients in $E$ as
\begin{equation*}
\EDIT{
R\Gamma_\dR(X, E)\coloneqq R\Gamma(X^\dR, E)=\pi_{X^\dR, *}(E)\;.
}
\end{equation*}
\end{defi}

\begin{rem}
\label{rem:drstack-vect}
In fact, the above notion of coefficients for de Rham cohomology recovers a more classical such notion: Namely, vector bundles on $X^\dR$ are the same as vector bundles on $X$ equipped with a flat connection \EDIT{having locally nilpotent $p$-curvature after reduction mod $p$ (see \cite[§5]{Katz} for a definition of the notion of $p$-curvature)} -- this can be proved using \cite[Rem. 2.5.7]{FGauges} and a similar calculation as in \cref{lem:fildhod-complexeszpntzero}.
\end{rem}

\subsubsection{The Rees equivalence}

We now want to bring the Hodge filtration into the picture. For this, we need a stack whose category of quasi-coherent complexes consists of filtered complexes over a chosen base ring $R$. More precisely, we want a stack whose category of quasi-coherent complexes is equivalent to the following category:

\begin{defi}
For a ring $R$, the \emph{filtered derived category} $\DF(R)$ is defined as the functor category $\Fun((\Z, \geq), \D(R))$, where we identify the poset $(\Z, \geq)$ with the corresponding category (i.e.\ the category with objects $n\in\Z$ and exactly one arrow $m\rightarrow n$ whenever $m\geq n$). For an object $F\in\DF(R)$, we write $\ul{F}\coloneqq\colim_i F(i)$ and $\Fil^i\ul{F}\coloneqq F(i)$.
\end{defi}

Our exposition will mostly follow \cite[Sect. 2.2.1]{FGauges}. Recall that $\DF(R)$ comes with a symmetric monoidal structure given by
\begin{equation*}
(F\tensor G)(n)=\colim_{i+j\geq n} F(i)\tensor G(j)
\end{equation*}
and that it admits an exact, symmetric monoidal and colimit-preserving functor $\gr_{\Fil}^\bullet$ given by 
\begin{equation*}
F\mapsto \gr_{\Fil}^\bullet\ul{F}\coloneqq\bigoplus_{i\in\Z} \gr_{\Fil}^i \ul{F}\coloneqq\bigoplus_{i\in\Z} \cofib(\Fil^{i+1}\ul{F}\rightarrow \Fil^i\ul{F})
\end{equation*}
to the $\infty$-category $\D_{\gr}(R)\coloneqq\Fun(\Z, \D(R))$ of graded objects in $\D(R)$; here, we identify the set $\Z$ with the corresponding discrete category. We call $\gr_{\Fil}^\bullet\ul{F}$ the \emph{associated graded} of the filtered object $F$. Moreover, the stable $\infty$-category $\DF(R)$ admits a standard $t$-structure whose (co-)connective objects are exactly those $F$ such that $F(i)$ is (co-)connective for each $i$. Finally, an object $F\in\DF(R)$ is called \emph{complete} if $\lim_i F(i)=0$.

To approach our goal of finding a stack whose category of quasi-coherent complexes is $\DF(R)$, we first consider the classifying stack $B\G_m$ of the group scheme $\G_m=\Spec R[t, t^{-1}]$ whose functor of points sends a ring $S$ to the group of units $S^\times$. As a $\G_m$-torsor $T$ on a scheme $S$ is uniquely of the form $T=\relSpec_S(\bigoplus_{i\in\Z} L^{\tensor i})$ for a line bundle $L\in\Pic(S)$, the stack $B\G_m$ classifies line bundles and comes equipped with a universal line bundle $\O_{B\G_m}(1)\in\Pic(B\G_m)$, whose $n$-th powers we denote by $\O_{B\G_m}(n)$ \EDIT{for $n\geq 0$ while $\O_{B\G_m}(-n)$ refers to the dual of $\O_{B\G_m}(n)$, as usual}. As a $\G_m$-equivariant structure on a complex $M\in\D(R)$ amounts to a grading on $M$, we furthermore see that there is an equivalence of categories
\begin{equation}
\label{eq:rees-dbgm}
\D(B\G_m)\cong \D_{\gr}(R)\;,
\end{equation}
which we normalise in such a way that $\O_{B\G_m}(1)$ corresponds to the $R$-module $R$ placed in degree $-1$; see \cite{Moulinos} for details.

Finally, the stack we seek is the quotient $\A^1/\G_m$; \EEDIT{here, the $\G_m$-action on $\A^1=\Spec R[t]$ is given by placing $t$ in grading degree \EEDIT{$-1$} (recall that giving a $\G_m$-action on an affine scheme is equivalent to specifying a grading on its global sections); note that this differs from the sign convention used in \cite{FGauges}, but our convention has the advantage of removing a change of sign which is otherwise necessary in part (3) of \cref{prop:rees-main}.} To describe the functor of points of $\A^1/\G_m$, we observe that for a $\G_m$-torsor $T=\relSpec_S(\bigoplus_{i\in\Z} L^{\tensor i})$ over $S$, where $L\in\Pic(S)$, the required $\G_m$-equivariant morphism $T\rightarrow \A^1$ is equivalent to the datum of a global section of \EEDIT{$L^{-1}$}, i.e.\ of a linear map \EEDIT{$L\rightarrow S$}. A linear map \EDIT{$L\rightarrow S$} for some $L\in\Pic(S)$ is called a \emph{generalised Cartier divisor} and thus $\A^1/\G_m$ is the moduli stack of generalised Cartier divisors. As such, it comes equipped with a universal generalised Cartier divisor
\begin{equation*}
\EEDIT{
t: \O_{\A^1/\G_m}(1)\rightarrow \O_{\A^1/\G_m}
}
\end{equation*}
and we write $\O_{\A^1/\G_m}(n)$ for the $n$-th power of $\O_{\A^1/\G_m}(1)$; \EEDIT{we warn the reader that, due to our chosen sign convention, the line bundle we denote $\O_{\A^1/\G_m}(1)$ actually identifies with the line bundle which is denoted $\O_{\A^1/\G_m}(-1)$ in \cite{FGauges}.} Note that there is a closed immersion $B\G_m=\{0\}/\G_m\subseteq\A^1/\G_m$ and, under this map, we have \EEDIT{$\O_{\A^1/\G_m}(1)|_{B\G_m}=\O_{B\G_m}(1)$}. Finally, observe that there is an equivalence of categories
\begin{equation*}
\D(\A^1/\G_m)\cong\Mod_{R[t]}(\D_{\gr}(R))\;,
\end{equation*}
where we view $R[t]$ as an object of $\D_{\gr}(R)$ via the grading above. This is due to the fact that a complex $M\in\D(R[t])$ is the same as an $R[t]$-module object in $\D(R)$ by \cite[Thm.s 7.1.2.13, 7.1.3.1]{HA} and that a $\G_m$-equivariant structure on $M$ then amounts to a grading compatible with the action of $R[t]$. \EEDIT{Under this equivalence, the line bundle $\O_{\A^1/\G_m}(1)$ corresponds to the graded $R[t]$-module $tR[t]$.}

\begin{prop}
\label{prop:rees-main}
The Rees construction described by carrying a filtered object $F\in\DF(R)$ to the graded $R[t]$-module object
\begin{equation*}
\Rees(F)\coloneqq \bigoplus_{i\in\Z} \Fil^i\ul{F}\cdot t^{-i}
\end{equation*}
defines a symmetric monoidal equivalence
\begin{equation*}
\D(\A^1/\G_m)\cong\DF(R)\;.
\end{equation*}
This equivalence has the following features:
\begin{enumerate}[label=(\arabic*)]
\item It is $t$-exact for the standard $t$-structures.

\item Restriction to the open substack $\Spec R=\G_m/\G_m\subseteq\A^1/\G_m$ corresponds to forgetting the filtration, i.e.\ to the functor $F\mapsto \ul{F}$.

\item Restriction to the closed substack $B\G_m=\{0\}/\G_m\subseteq\A^1/\G_m$ corresponds to passage to the associated graded, i.e.\ to the functor $\gr_{\Fil}^\bullet$.

\item Tensoring with \EEDIT{$\O(n)$} corresponds to shifting the filtration by $n$, i.e.\ to the functor $F\mapsto F\{n\}$ with $\Fil^i \ul{F\{n\}}=\Fil^{i+n} \ul{F}$.

\item It matches completeness of an object in $\DF(R)$ with derived $t$-completeness in $\D(\A^1/\G_m)$.
\end{enumerate}
\end{prop}
\begin{proof}
See \cite[Prop. 2.2.6]{FGauges} or \cite{Moulinos}.
\end{proof}

We shortly want to discuss the analogue of the result above over the base $\Spf\Z_p$. \EDIT{More precisely,} we take $R=\Z$ and consider $\A^1/\G_m\times\Spf\Z_p$ and $B\G_m\times\Spf\Z_p$.

\begin{cor}
There are symmetric monoidal equivalences
\begin{equation*}
\begin{split}
\D(B\G_m\times\Spf\Z_p)&\cong\widehat{\D}_{\gr}(\Z_p)\coloneqq\Fun(\Z, \widehat{\D}(\Z_p)) \\
\D(\A^1/\G_m\times\Spf\Z_p)&\cong\widehat{\DF}(\Z_p)\coloneqq\Fun((\Z, \geq), \widehat{\D}(\Z_p))
\end{split}
\end{equation*}
with the same features as in \cref{prop:rees-main}.
\end{cor}
\begin{proof}
We only prove the first equivalence, the second one is similar. As any morphism $\Spec S\rightarrow B\G_m\times\Spf\Z_p$ factors through $B\G_m\times\Spec\Z/p^n\Z$ for some $n$, using (\ref{eq:rees-dbgm}) \EDIT{and the universal property of the limit}, we obtain equivalences
\begin{equation*}
\begin{split}
\D(B\G_m\times\Spf\Z_p)&\cong\lim_n \D(B\G_m\times\Spec\Z/p^n\Z)\cong\lim_n \Fun(\Z, \D(\Z/p^n\Z)) \\
&\cong\Fun(\Z, \lim_n \D(\Z/p^n\Z))\cong\Fun(\Z, \widehat{\D}(\Z_p))\;,
\end{split}
\end{equation*}
as desired. The claimed features are directly inherited from the corresponding properties of the equivalences over the bases $\Spec\Z/p^n\Z$.
\end{proof}

In fact, one can check that a similar argument shows that, for any stack $\cal{X}$, we have 
\begin{equation*}
\D(\A^1/\G_m\times\cal{X})\cong\Fun((\Z, \geq), \D(\cal{X}))\;.
\end{equation*}
Finally, let us remark that everything of course works completely analogously in the case of increasing filtrations instead of decreasing filtrations. In this case, one just considers $\A^1/\G_m$ for $\A^1=\Spec R[u]$ with $u$ being placed in degree \EEDIT{$1$}.

\subsubsection{The Hodge-filtered de Rham stack}
\label{subsect:fildrstack}

We are finally ready to define the Hodge-filtered de Rham stack associated to a $p$-adic formal scheme $X$. As this should capture the Hodge filtration on de Rham cohomology, in view of the results of the previous section, it is only natural that it will be a stack over $\A^1/\G_m$. Throughout, we work over $\Spf\Z_p$.

\begin{defi}
Over $\A^1/\G_m$, the canonical map \EEDIT{$\V(\O(1))^\sharp\rightarrow\G_a^\sharp\rightarrow\G_a$} is a quasi-ideal and hence defines a 1-truncated animated $\G_a$-algebra stack
\begin{equation*}
\EEDIT{
\G_a^{\dR, +}\coloneqq\Cone(\V(\O(1))^\sharp\xrightarrow{\can}\G_a)
}
\end{equation*}
over $\A^1/\G_m$. The \emph{Hodge-filtered de Rham stack} $X^{\dR, +}$ of $X$ is the stack $\pi_{X^{\dR, +}}: X^{\dR, +}\rightarrow\A^1/\G_m$ defined by
\begin{equation*}
X^{\dR, +}(\Spec S\rightarrow\A^1/\G_m)\coloneqq \Map(\Spec\G_a^{\dR, +}(S), X)\;,
\end{equation*}
where the mapping space is computed in derived algebraic geometry. If $X=\Spf R$ is affine, we also write $R^{\dR, +}$ in place of $X^{\dR, +}$.
\end{defi}

\begin{rem}
\label{rem:fildrstack-unfiltereddr}
Observe that the preimage of $\G_m/\G_m\subseteq\A_1/\G_m$ under $\pi_{X^{\dR, +}}$ recovers the de Rham stack $X^\dR$. The preimage of $B\G_m\subseteq \A^1/\G_m$ is called the \emph{Hodge stack} of $X$ and denoted $X^\Hod$ or also \EDIT{$R^\Hod$} if $X=\Spf R$ is affine.
\end{rem}

\begin{ex}
\label{ex:fildrstack-perfd}
Generalising \cref{ex:drstack-perfd}, we claim that, for a perfectoid ring $R$, there is an isomorphism of stacks
\begin{equation}
\label{eq:fildrstack-perfd}
R^{\dR, +}\cong\Spf(A_\inf(R)[\tfrac{u^n}{n!}, t: n\geq 1]_{(p)}^\wedge/(ut-\xi))/\G_m
\end{equation}
over $\A^1/\G_m$, \EDIT{where $A=A_\inf(R)[\tfrac{u^n}{n!}, t: n\geq 1]_{(p)}^\wedge/(ut-\xi)$ is equipped with the $(p, \xi)$-adic topology (note that this agrees with the $p$-adic topology due to the existence of divided powers for $u$ and hence also for $\xi=ut$) and, as usual, $t$ has degree \EEDIT{$-1$} while $u$ has degree \EEDIT{$1$}. We note here that while $A$ is clearly derived $p$-complete, it is also classically $p$-complete since it is $p$-adically separated.} 

To prove (\ref{eq:fildrstack-perfd}), as before, it suffices to describe this isomorphism on $S$-valued points for $\G_a^\sharp$-acyclic $p$-nilpotent rings $S$ equipped with a morphism $\Spec S\rightarrow\A^1/\G_m$ by \cref{lem:gasharp-acyclic}. Then $R^{\dR, +}(S)$ is the groupoid of maps of animated rings \EEDIT{$R\rightarrow\Cone(\V(\O(1))^\sharp(S)\rightarrow S)$} by $\G_a^\sharp$-acyclicity of $S$ (as any \EEDIT{$\V(\O(1))^\sharp$-torsor} over $S$ is also a $\G_a^\sharp$-torsor) and if the morphism $\Spec S\rightarrow\A^1/\G_m$ is classified by the generalised Cartier divisor $t: L\rightarrow S$, then \EEDIT{$\V(\O(1))^\sharp(S)$} identifies with the set of global sections of $L$ equipped with a system of divided powers. Now given any such map \EEDIT{$R\rightarrow\Cone(\V(\O(1))^\sharp(S)\rightarrow S)$}, we obtain a unique lift
\begin{equation*}
\EEDIT{
\begin{tikzcd}[ampersand replacement=\&]
A_\inf(R)\ar[r, dotted]\ar[d] \& S\ar[d] \\
R=A_\inf(R)/(\xi)\ar[r] \& \Cone(\V(\O(1))^\sharp(S)\rightarrow S)\nospacepunct{\;.}
\end{tikzcd}
}
\end{equation*}
as in \cref{ex:drstack-perfd} and then the commutativity of the diagram provides a unique factorisation $S\xrightarrow{u} L\xrightarrow{t} S$ of the multiplication-by-$\xi$-map together with a system of divided powers for $u$ (i.e.\ maps $\frac{u^n}{n!}: S\rightarrow L^{\tensor n}$). \EEDIT{Together with the map $A_\inf(R)\rightarrow S$ we obtained,} this is precisely the datum of an $S$-valued point of $\Spf A/\G_m$ and the construction is clearly reversible, so we are done.
\end{ex}

\begin{ex}
\label{ex:fildrstack-perf}
In the previous example, let $R=k$ be a perfect field of characteristic $p$. Then we find that
\begin{equation*}
k^{\dR, +}\cong \Spf(W(k)[\tfrac{u^n}{n!}, t: n\geq 1]_{(p)}^\wedge/(ut-p))/\G_m\;. \qedhere
\end{equation*}
\end{ex}

As expected, coherent cohomology on the Hodge-filtered de Rham stack of $X$ computes \EDIT{the} Hodge-filtered de Rham cohomology of $X$ in good cases:

\begin{thm}
\label{thm:fildrstack-comparison}
Let $X$ be a \EEDIT{smooth qcqs} $p$-adic formal scheme and consider its Hodge-filtered de Rham stack $\pi_{X^{\dR, +}}: X^{\dR, +}\rightarrow\A^1/\G_m$. Then $\H_{\dR, +}(X)\coloneqq \pi_{X^{\dR, +}, *}\O_{X^{\dR, +}}$ identifies with $\Fil^\bullet_{\Hod} R\Gamma_\dR(X)$ in $\widehat{\DF}(\Z_p)$ under the Rees equivalence.
\end{thm}
\begin{proof}
This is \cite[Thm. 2.5.6]{FGauges}.
\end{proof}

Note that the above statement also implies that the pushforward of $\O_{X^\Hod}$ to $\Z_p^\Hod=B\G_m$ identifies as a graded object with \EDIT{the} Hodge cohomology of $X$. Again, we are led to the following definition:

\begin{defi}
Let $X$ be a \EEDIT{smooth qcqs} $p$-adic formal scheme. For a quasi-coherent complex $E\in\D(X^{\dR, +})$, we define the \emph{Hodge-filtered de Rham cohomology} of $X$ with coefficients in $E$ as
\begin{equation*}
\EDIT{\Fil^\bullet_\Hod R\Gamma_\dR(X, E)\coloneqq \pi_{X^{\dR, +}, *}(E)\;.}
\end{equation*}
\end{defi}

\begin{rem}
\label{rem:fildrstack-vect}
\EDIT{
As in the case of the de Rham stack, one can show that the above notion of coefficients for Hodge-filtered de Rham cohomology agrees with a more classical one, \EEDIT{see \cite[Rem. 2.5.8]{FGauges}}: Namely, a vector bundle on $X^{\dR, +}$ is the same as a vector bundle $E$ on $X$ equipped with a decreasing filtration $\Fil^\bullet E$ by subbundles and a flat connection $\nabla: E\rightarrow E\tensor\Omega_{X/\Z_p}^1$ which has nilpotent $p$-curvature mod $p$ and satisfies Griffiths transversality with respect to the filtration $\Fil^\bullet E$, i.e.\ we have $\nabla(\Fil^i E)\subseteq \Fil^{i-1} E\tensor\Omega_{X/\Z_p}^1$ for all $i\in\Z$.
}
\end{rem}

\subsection{Prismatisation and syntomification}
\label{subsect:prismsyn}

Having seen stacky formulations of de Rham cohomology and the Hodge filtration, we now describe a similar story for prismatic cohomology and the Nygaard filtration. As before, let $X$ be a bounded $p$-adic formal scheme. We roughly follow the treatment of \cite[Ch.s 4, 5]{FGauges}. 

\subsubsection{The prismatisation}

\begin{defi}
\label{defi:prismatisation-cwdiv}
For a $p$-nilpotent ring $S$, a \emph{Cartier--Witt divisor} on $S$ is a generalised Cartier divisor $\alpha: I\rightarrow W(S)$ on $W(S)$ satisfying the following two conditions:
\begin{enumerate}[label=(\roman*)]
\item The \EDIT{ideal generated by the} image of the map $I\xrightarrow{\alpha} W(S)\rightarrow S$ \EDIT{is nilpotent}.
\item The image of the map $I\xrightarrow{\alpha} W(S)\xrightarrow{\delta} W(S)$ generates the unit ideal.
\end{enumerate}
Here, $\delta: W(S)\rightarrow W(S)$ is the usual $\delta$-structure on $W(S)$.
\end{defi}

\begin{ex}
\label{ex:prismatisation-prismtocwdiv}
We show how to construct Cartier--Witt divisors from prisms. Namely, let $(A, I)$ be a prism and $S$ an $A$-algebra such that $(p, I)$ is nilpotent in $S$. Then the structure map $A\rightarrow S$ uniquely lifts to a map of $\delta$-rings $A\rightarrow W(S)$ and this produces a Cartier--Witt divisor $I\tensor_A W(S)\rightarrow W(S)$ on $S$. \EEDIT{Indeed, condition (i) from \cref{defi:prismatisation-cwdiv} is clear and condition (ii) is due to the fact that, pro-Zariski-locally, the ideal $I$ is generated by a distinguished element, see \cite[Lem. III.1.9]{PrismaticBhatt}.}
\end{ex}

\begin{ex}
\label{ex:prismatisation-cwdivtodiv}
We show how to construct generalised Cartier divisors from Cartier--Witt divisors. Namely, let $I\xrightarrow{\alpha} W(S)$ be a Cartier--Witt divisor on a $p$-nilpotent ring $S$. Then the induced map $I\tensor_{W(S)} S\rightarrow S$ is a generalised Cartier divisor on $S$ with the property that the \EEDIT{ideal generated by the} image of $I\tensor_{W(S)} S$ in $S$ is nilpotent.
\end{ex}

\begin{defi}
For a Cartier--Witt divisor $I\xrightarrow{\alpha} W(S)$ on $W(S)$, write $\ol{W(S)}$ for the 1-truncated animated ring $\cofib(I\xrightarrow{\alpha} W(S))$. Then the \emph{prismatisation} $X^\prism$ is the stack over $\Spf\Z_p$ given by assigning to a $p$-nilpotent ring $S$ the groupoid of pairs 
\begin{equation*}
\EDIT{
(I\xrightarrow{\alpha} W(S), \Spec \ol{W(S)}\rightarrow X)\;,
}
\end{equation*}
where $I\xrightarrow{\alpha} W(S)$ is a Cartier--Witt divisor on $S$ and $\Spec\ol{W(S)}\rightarrow X$ is a morphism of derived formal schemes. If $X=\Spf R$ is affine, we also write $R^\prism$ in place of $X^\prism$.
\end{defi}

\begin{rem}
Actually, one can show that the fibred category of triples 
\begin{equation*}
(S, I\xrightarrow{\alpha} W(S), \Spec\ol{W(S)}\rightarrow X)
\end{equation*}
as above is already fibred in groupoids: Indeed, maps between Cartier--Witt divisors satisfy a rigidity property analogous to the one for maps between prisms, i.e.\ they are always isomorphisms, see \cite[Lem. 5.1.5]{FGauges}. Moreover, as the maps $S\leftarrow W(S)\rightarrow\ol{W(S)}$ are pro-infinitesimal thickenings due to $\alpha(I)$ being nilpotent mod $p$, any endomorphism of $\ol{W(S)}$ lifting the identity on $S$ must be an isomorphism.
\end{rem}

\begin{ex}
\label{ex:prismatisation-perfd}
We claim that, for a perfectoid ring $R$, there is an isomorphism of stacks
\begin{equation}
\label{eq:prismatisation-perfd1}
R^\prism\cong\Spf A_\inf(R)\;,
\end{equation}
\EDIT{where $A_\inf(R)$ is endowed with the $(p, \xi)$-adic topology.} To prove this, consider an $S$-valued point of $R^\prism$, i.e.\ a Cartier--Witt divisor $I\xrightarrow{\alpha} W(S)$ on $S$ together with a map $R\rightarrow \ol{W(S)}$. Using the fact that $W(S)\rightarrow\ol{W(S)}$ is a pro-infinitesimal thickening, a similar deformation-theoretic argument as in \cref{ex:drstack-perfd} shows that we get a unique lift
\begin{equation}
\label{eq:prismatisation-perfd2}
\begin{tikzcd}
A_\inf(R)\ar[r, dotted]\ar[d] & W(S)\ar[d] \\
R=A_\inf(R)/(\xi)\ar[r] & \ol{W(S)}
\end{tikzcd}
\end{equation}
and hence an $S$-valued point $A_\inf(R)\rightarrow W(S)\rightarrow S$ of $\Spf A_\inf(R)$; \EDIT{indeed, note that the image of $\xi$ in $W(S)$ lies in $I$ by the commutativity of the above diagram, hence $\xi$ is nilpotent in $S$ due to the fact that the image of $I$ in $S$ generates a nilpotent ideal.} To see that this construction is reversible and thus prove (\ref{eq:prismatisation-perfd1}), first observe that, $W(S)\rightarrow S$, too, is a pro-infinitesimal thickening, so again by deformation theory, we see that there are equivalences
\begin{equation*}
\Map(A_\inf(R), W(S))\cong\Map(A_\inf(R), S)\cong\Map_\delta(A_\inf(R), W(S))\;,
\end{equation*}
where the last part follows from the adjunction between the Witt vector functor and the forgetful functor from $\delta$-rings. Thus, we learn that we can recover the map $A_\inf(R)\rightarrow W(S)$ from the map $A_\inf(R)\rightarrow S$ and, moreover, that the map $A_\inf(R)\rightarrow W(S)$ is automatically a $\delta$-map. Hence, we can base-change the Cartier--Witt divisor $A_\inf(R)\xrightarrow{\xi} A_\inf(R)$ along the map $A_\inf(R)\rightarrow W(S)$ and obtain a map of Cartier--Witt divisors 
\begin{equation*}
(W(S)\xrightarrow{\xi} W(S))\rightarrow (I\xrightarrow{\alpha} W(S))
\end{equation*}
from the commutativity of the diagram (\ref{eq:prismatisation-perfd2}). This must be an isomorphism by rigidity of maps, so we can recover the Cartier--Witt divisor $I\xrightarrow{\alpha} W(S)$.
\end{ex}

\begin{ex}
\label{ex:prismatisation-perf}
In the previous example, let $R=k$ be a perfect field of characteristic $p$. Then we find that
\begin{equation*}
k^\prism\cong\Spf W(k)\;. \qedhere
\end{equation*}
\end{ex}

\begin{ex}
\label{ex:prismatisation-prismaticsitetoxprism}
We explain how any object $(\Spf A\leftarrow\Spf \ol{A}\rightarrow X)\in X_\prism$ of the absolute prismatic site of $X$ gives rise to a point
\begin{equation*}
\rho_{A, X}: \Spf A\rightarrow X^\prism\;,
\end{equation*}
where we endow $A$ with the $(p, I)$-adic topology. Namely, for any $A$-algebra $S$ in which $(p, I)$ is nilpotent, we obtain from \cref{ex:prismatisation-prismtocwdiv} a Cartier--Witt divisor $I\tensor_A W(S)\rightarrow W(S)$ on $S$ and it automatically comes with a map $\Spec\ol{W(S)}\rightarrow\Spf\ol{A}\rightarrow X$ \EEDIT{since the morphism $A\rightarrow S$ lifts to a $\delta$-morphism $A\rightarrow W(S)$}; here, we use that $\pi_0(\ol{W(S)})$ is $p$-nilpotent by \cite[Lem. 3.3]{PFS}.
\end{ex}

\begin{ex}
\label{ex:prismatisation-xprismtoa1gm}
By \cref{ex:prismatisation-cwdivtodiv}, any Cartier--Witt divisor $I\xrightarrow{\alpha} W(S)$ on a $p$-nilpotent ring $S$ yields a generalised Cartier divisor $I\tensor_{W(S)} S\rightarrow S$ on $S$ such that the image of $I\tensor_{W(S)} S$ in $S$ is nilpotent. Thus, we obtain a morphism of stacks
\begin{equation*}
X^\prism\rightarrow\widehat{\A}^1/\G_m\;,
\end{equation*}
where $\widehat{\A}^1$ denotes the formal completion of $\A^1$ at the origin.
\end{ex}

\begin{ex}
\label{ex:prismatisation-frob}
For a $p$-nilpotent ring $S$ and any $S$-valued point $(I\xrightarrow{\alpha} W(S), \Spec \ol{W(S)}\xrightarrow{\eta} X)$ of $X^\prism$, we obtain another $S$-valued point
\begin{equation*}
(F^*I\xrightarrow{F^*\alpha} W(S), \Spec \ol{W(S)}\xrightarrow{\eta\circ\ol{F}} X) 
\end{equation*}
of $X^\prism$, where $F$ denotes the Frobenius of $W(S)$ and $\ol{F}$ is the induced map 
\begin{equation*}
\ol{F}: \cofib(F^*I\xrightarrow{F^*\alpha} W(S))\rightarrow\cofib(I\xrightarrow{\alpha} W(S))\;;
\end{equation*}
indeed, one can check that $F^*I\xrightarrow{F^*\alpha} W(S)$ is again a Cartier--Witt divisor, see \cite[Rem. 5.1.10]{FGauges}. Thus, we obtain an endomorphism
\begin{equation*}
F_X: X^\prism\rightarrow X^\prism
\end{equation*}
called the \emph{Frobenius} on $X^\prism$. If $X=\Spf R$ is affine, we also write $F_R$ in place of $F_X$.
\end{ex}

As before, in good situations, coherent cohomology of the structure sheaf on $X^\prism$ agrees with \EDIT{the} prismatic cohomology of $X$ computed via the absolute prismatic site:

\begin{thm}
\label{thm:prismatisation-comparison}
Let $X$ be a bounded $p$-adic formal scheme and \EEDIT{furthermore} assume that $X$ is $p$-quasisyntomic and qcqs. Then there is a natural isomorphism
\begin{equation*}
R\Gamma(X^\prism, \O_{X^\prism})\cong R\Gamma_\prism(X)\;.
\end{equation*}
\end{thm}
\begin{proof}
Combine \cite[Cor. 8.17]{PFS} with \cite[Thm. 4.4.30]{APC}.
\end{proof}

Even more strongly, one can show that $\D(X^\prism)$ is also the correct category for a notion of coefficients for prismatic cohomology in the following sense:

\begin{prop}
\label{prop:prismatisation-crystals}
Let $X$ be a bounded $p$-adic formal scheme \EDIT{and} assume that $X$ is $p$-quasisyntomic. Then the maps $\rho_{A, X}$ from \cref{ex:prismatisation-prismaticsitetoxprism} induce an equivalence
\begin{equation*}
\D(X^\prism)\cong \lim_{(A, I)\in X_\prism} \widehat{\D}(A)
\end{equation*}
of symmetric monoidal stable $\infty$-categories; \EDIT{here, $X_\prism$ denotes the absolute prismatic site of $X$.}
\end{prop}
\begin{proof}
See \cite[Prop. 8.15]{PFS}.
\end{proof}

In particular, note that the category $\D(\Z_p^\prism)$ is equivalent to the category of prismatic crystals.

\begin{ex}
By pushforward along the morphism $\pi_{X^\prism}: X^\prism\rightarrow\Z_p^\prism$ obtained by functoriality, we obtain a prismatic crystal $\H_\prism(X)\coloneqq \pi_{X^\prism, *}\O_{X^\prism}$.
\end{ex}

\begin{ex}
\label{ex:prismatisation-htstack}
By \cref{prop:prismatisation-crystals} and rigidity, assigning to any prism $(A, I)$ the ideal $I$ defines a line bundle $\cal{I}$ on $\Z_p^\prism$ called the \emph{Hodge--Tate ideal sheaf}. Alternatively, one may describe it by assigning to every Cartier--Witt divisor $I\xrightarrow{\alpha} W(S)$ on a $p$-nilpotent ring $S$ the line bundle $I\tensor_{W(S)} S$ \EEDIT{on $S$}. By pulling back to $X^\prism$, we obtain a line bundle $\cal{I}_X$ on $X^\prism$, which we also denote by $\cal{I}_R$ in the case where $X=\Spf R$ is affine. The line bundle $\cal{I}_X$ cuts out a closed \EDIT{substack} $X^\HT$ called the \emph{Hodge--Tate stack} of $X$, which may alternatively be obtained as a pullback
\begin{equation*}
\begin{tikzcd}
X^\HT\ar[d]\ar[r] & X^\prism\ar[d] \\
B\G_m\ar[r] & \widehat{\A}^1/\G_m\nospacepunct{\;,}
\end{tikzcd}
\end{equation*}
where $X^\prism\rightarrow\widehat{\A}^1/\G_m$ is the morphism from \cref{ex:prismatisation-xprismtoa1gm}. Again, if $X=\Spf R$ is affine, we denote $X^\HT$ also by $R^\HT$.
\end{ex}

\begin{ex}
Recall that for any prism $(A, I)$, there is a special invertible $A$-module $A\{1\}$ called the \emph{Breuil--Kisin twist}, see \cite[Def. 2.5.2]{APC}. It can be thought of heuristically as the \EDIT{infinite} tensor product $\bigotimes_{n\geq 0} (\phi^n)^* I$, where $\phi$ denotes the Frobenius lift associated to the prism $(A, I)$, and has the property that there are canonical isomorphisms $A\{1\}/IA\{1\}\cong I/I^2$ and $\phi^*A\{1\}\cong I^{-1}A\{1\}$, see \cite[Rem.s 2.5.7, 2.5.9]{APC}. As the construction $(A, I)\mapsto A\{1\}$ is compatible with base change by \cite[Rem. 2.5.5]{APC}, it determines a line bundle $\O_{\Z_p^\prism}\{1\}$ \EEDIT{on $\Z_p^\prism$} by \cref{prop:prismatisation-crystals} which we also call \emph{Breuil--Kisin twist} and whose $n$-th power we denote by $\O_{\Z_p^\prism}\{n\}$. Pullback to $X^\prism$ then yields line bundles $\O_{X^\prism}\{n\}$ with the property that \EEDIT{$\O_{X^\prism}\{1\}/\cal{I}_X\O_{X^\prism}\{1\}\cong \cal{I}_X/\cal{I}_X^2$ and} $F_X^*\O_{X^\prism}\{1\}\cong \cal{I}_X^{-1}\tensor\O_{X^\prism}\{1\}$.
\end{ex}

As before, this leads to a notion of prismatic cohomology with coefficients:

\begin{defi}
Let $X$ be a bounded $p$-adic formal scheme which is $p$-quasisyntomic and qcqs. For a quasi-coherent complex $E\in\D(X^\prism)$, we define the \emph{prismatic cohomology} of $X$ with coefficients in $E$ as
\begin{equation*}
R\Gamma_\prism(X, E)\coloneqq R\Gamma(X^\prism, E)\;.
\end{equation*}
\end{defi}

\subsubsection{The Nygaard-filtered prismatisation}

Let $W$ be the ring scheme of $p$-typical Witt vectors, i.e.\ $W\cong\prod_{n\in\mathbb{N}} \Spec\Z[t_1, t_2, \dots]$ as schemes \EDIT{so that the functor of points of $W$ is given by $S\mapsto W(S)$ for any $p$-nilpotent ring $S$ via the Witt components}; if no confusion arises, we will also denote the base change of $W$ to another ring $S$ by the same letter. As usual, we denote the Frobenius and the Verschiebung by $F$ and $V$, respectively. We will study the following classes of $W$-module schemes:

\begin{defi}
Let $S$ be a $p$-nilpotent ring and $M$ an affine $W$-module scheme over $S$. Then $M$ is called ...
\begin{enumerate}[label=(\roman*)]
\item ... an \emph{invertible $W$-module} if it is fpqc-locally on $R$ isomorphic to $W$ as a $W$-module. 

\item ... an \emph{invertible $F_*W$-module} if it is fpqc-locally on $R$ isomorphic to $F_*W$ as a $W$-module.

\item ... \emph{$\sharp$-invertible} if it is fpqc-locally on $R$ isomorphic to $\G_a^\sharp$ as a $W$-module.
\end{enumerate}
\end{defi}

In part (iii) of the above definition, $\G_a^\sharp$ is regarded as a $W$-module in the following manner:

\begin{lem}
\label{lem:admissible-fundamentalseq}
There is a unique isomorphism $W[F]\cong\G_a^\sharp$ lifting the composition $W[F]\subseteq W\rightarrow\G_a$, where $W[F]$ denotes the kernel of the Frobenius. In fact, there is a short exact sequence of $W$-module sheaves
\begin{equation*}
\begin{tikzcd}
0\ar[r] & \G_a^\sharp\ar[r] & W\ar[r] & F_*W\ar[r] & 0\nospacepunct{\;.}
\end{tikzcd}
\end{equation*}
\end{lem}
\begin{proof}
See \cite[Lem. 2.6.1, Rem. 2.6.2]{FGauges}.
\end{proof}

There is actually a very straightfoward classification of the three classes of $W$-modules introduced above:

\begin{lem}
\label{lem:admissible-classification}
Let $S$ be a $p$-nilpotent ring. 
\begin{enumerate}[label=(\roman*)]
\item The groupoid of invertible $W$-modules is equivalent to the Picard groupoid of $W(S)$ via the construction $L\mapsto L\tensor_{W(S)} W$ for $L\in\Pic(W(S))$, \EEDIT{where the sheaf $L\tensor_{W(S)} W$ is defined by
\begin{equation*}
(L\tensor_{W(S)} W)(T)\coloneqq L\tensor_{W(S)} W(T)
\end{equation*}
for any $S$-algebra $T$.}

\item The groupoid of invertible $F_*W$-modules is equivalent to the groupoid of invertible $W$-modules via the construction $M\mapsto F_*M$ for invertible $W$-modules $N$.

\item The category of $\sharp$-invertible modules is equivalent to the category of invertible $S$-modules via the construction $L\mapsto \V(L)^\sharp$ for $L\in\Pic S$.
\end{enumerate}
\end{lem}
\begin{proof}
See \cite[Constr. 5.2.2]{FGauges}.
\end{proof}

Note that, for $L\in\Pic(W(S))$, the sheaf $L\tensor_{W(S)} W$ is representable by the scheme $\Spec(A\tensor \Z[t_1, t_2, \dots])$, where $A$ is the ring of global sections of the vector bundle $\V(L)$; \EEDIT{indeed, for any $S$-algebra $T$, we have}
\begin{equation*}
\EEDIT{
\begin{split}
\Hom(A\tensor \Z[t_1, t_2, \dots], T)&\cong \Hom(A, \Hom(\Z[t_1, t_2, \dots], T))\cong \Hom(A, W(T)) \\
&\cong \Hom(\Spec W(T), \V(L))\cong L\tensor_{W(S)} W(T)\;.
\end{split}
}
\end{equation*}
Also note that, for $L\in\Pic S$, the scheme $\V(L)^\sharp$ is actually affine since it is a $\G_a^\sharp$-torsor over $\Spec S$ and $\G_a^\sharp$ is affine.

\begin{ex}
\label{ex:admissible-cwdivtoinvertible}
For a $p$-nilpotent ring $S$, any Cartier--Witt divisor $I\xrightarrow{\alpha} W(S)$ gives rise to a morphism of invertible $W$-module schemes $M\xrightarrow{d} W$ over $S$ by \cref{lem:admissible-classification}.
\end{ex}

\begin{defi}
Let $S$ be a $p$-nilpotent ring. An \emph{admissible} $W$-module is an affine $W$-module scheme $M$ which (as a sheaf) can be written as an extension of an invertible $F_*W$-module by a $\sharp$-invertible $W$-module.
\end{defi}

\begin{ex}
Any invertible $W$-module $M=L\tensor_{W(R)} W$ for $L\in\Pic(W(R))$ is admissible as the sequence
\begin{equation*}
\begin{tikzcd}
0\ar[r] & L\tensor_{W(R)} \G_a^\sharp\ar[r] & L\tensor_{W(R)} W\ar[r] & L\tensor_{W(R)} F_*W\ar[r] & 0
\end{tikzcd}
\end{equation*}
is exact by \cref{lem:admissible-fundamentalseq}.
\end{ex}

\begin{rem}
\label{rem:filprism-admissibleseq}
\EDIT{
Given two admissible $W$-modules $M_1$ and $M_2$ sitting in exact sequences
\begin{equation*}
\begin{tikzcd}[ampersand replacement=\&]
0\ar[r] \& \V(L_i)^\sharp\ar[r] \& M_i\ar[r] \& F_*M'_i\ar[r] \& 0\nospacepunct{\;,}
\end{tikzcd}
\end{equation*}
any map $M_1\rightarrow M_2$ lifts uniquely to a map
\begin{equation*}
\begin{tikzcd}[ampersand replacement=\&]
0\ar[r] \& \V(L_1)^\sharp\ar[r]\ar[d] \& M_1\ar[r]\ar[d] \& F_*M_1'\ar[r]\ar[d] \& 0 \\
0\ar[r] \& \V(L_2)^\sharp\ar[r] \& M_2\ar[r] \& F_*M_2'\ar[r] \& 0
\end{tikzcd}
\end{equation*}
since one can show that $\sHom_W(\G_a^\sharp, F_*W)=0$, see \cite[Prop. 5.2.1]{FGauges}. In particular, it follows that a sequence witnessing the admissibility of some $W$-module $M$ is unique up to unique isomorphism; in the sequel, we will call this the \emph{admissible sequence} associated to $M$ and denote it by
\begin{equation*}
\begin{tikzcd}[ampersand replacement=\&]
0\ar[r] \& \V(L_M)^\sharp\ar[r] \& M\ar[r] \& F_*M'\ar[r] \& 0\nospacepunct{\;.}
\end{tikzcd}
\end{equation*}
Note that the uniqueness of the admissible sequence implies that being admissible is an fpqc-local property.
}
\end{rem}

\begin{defi}
Let $S$ be a $p$-nilpotent ring. A \emph{filtered Cartier--Witt divisor} on $S$ consists of an admissible $W$-module scheme $M$ and a map $d: M\rightarrow W$ such that the induced map $F_*M'\rightarrow F_*W$ of associated invertible $F_*W$-modules comes from a Cartier--Witt divisor on $R$ via the construction of \cref{ex:admissible-cwdivtoinvertible}. We denote the groupoid of filtered Cartier--Witt divisors on $S$ by $\Z_p^\N(S)$ and this defines a stack $\Z_p^\N$ on $p$-nilpotent rings.
\end{defi}

\begin{ex}
We explain how to turn any Cartier--Witt divisor $I\xrightarrow{\alpha} W(S)$ on a $p$-nilpotent ring $S$ into a filtered Cartier--Witt divisor. Indeed, one can show that the induced map $I\tensor_{W(S)} W\rightarrow W$ is a filtered Cartier--Witt divisor and that any filtered Cartier--Witt divisor $M\xrightarrow{d} W$ on $S$ such that $M$ is invertible arises in this way, see \cite[Constr. 5.3.2]{FGauges}. This defines a map
\begin{equation*}
j_\HT: \Z_p^\prism\rightarrow\Z_p^\N
\end{equation*}
of stacks and one can in fact show that it is an open immersion.
\end{ex}

\begin{ex}
\label{ex:filteredprism-drmap}
We explain how to turn any generalised Cartier divisor $L\xrightarrow{t} S$ on a $p$-nilpotent ring $S$ into a filtered Cartier--Witt divisor. Namely, we can associate to $t$ the map 
\begin{equation*}
\V(L)^\sharp\oplus F_*W\xrightarrow{(t^\sharp, V)} W
\end{equation*}
of admissible $W$-modules and this is indeed a filtered Cartier--Witt divisor since the associated map $F_*W\rightarrow F_*W$ is multiplication by $p$. Thus, we obtain a morphism of stacks
\begin{equation*}
i_{\dR, +}: \A^1/\G_m\rightarrow\Z_p^\N
\end{equation*}
called the \emph{Hodge-filtered de Rham map} and restricting $i_{\dR, +}$ to $\Spf\Z_p=\G_m/\G_m$ yields a morphism of stacks
\begin{equation*}
i_\dR: \Spf\Z_p\rightarrow\Z_p^\N\;,
\end{equation*}
which we call the \emph{de Rham map} \EEDIT{and, similarly, restricting $i_{\dR, +}$ to $B\G_m=\{0\}/\G_m$, we obtain a morphism
\begin{equation*}
i_\dR: B\G_m\rightarrow\Z_p^\N\;,
\end{equation*}
which we call the \emph{Hodge map}.} We warn the reader that, despite the notation, the maps \EEDIT{$i_{\dR, +}, i_\dR$ and $i_\Hod$ are \emph{not} closed immersions.}
\end{ex}

Recall \EEDIT{from \cref{rem:filprism-admissibleseq}} that, given a filtered Cartier--Witt divisor $M\xrightarrow{d} W$ on a $p$-nilpotent ring $S$, we obtain an induced map of admissible sequences
\begin{equation*}
\begin{tikzcd}
0\ar[r] & \V(L_M)^\sharp\ar[d, "\sharp(d)"]\ar[r] & M\ar[d, "d"]\ar[r] & F_*M'\ar[d, "F_*(d')"]\ar[r] & 0 \\
0\ar[r] & \G_a^\sharp\ar[r] & W\ar[r] & F_*W\ar[r] & 0\nospacepunct{\;.}
\end{tikzcd}
\end{equation*}

\begin{defi}
Fix a $p$-nilpotent ring $S$ and a filtered Cartier--Witt divisor $M\xrightarrow{d} W$ on $S$.
\begin{enumerate}[label=(\roman*)]
\item As $d': M'\rightarrow W$ is a Cartier--Witt divisor, sending $d$ to $d'$ defines a map of stacks
\begin{equation*}
\pi: \Z_p^\N\rightarrow\Z_p^\prism
\end{equation*}
called the \emph{structure map}.

\item The map $\sharp(d)$ uniquely has the form $t(d)^\sharp$ for some $t(d): L_M\rightarrow S$ by \cref{lem:admissible-classification}. Hence, sending $d$ to $t(d)$ defines a map of stacks
\begin{equation*}
t: \Z_p^\N\rightarrow\A^1/\G_m
\end{equation*}
called the \emph{Rees map}.
\end{enumerate}
\end{defi}

\begin{rem}
\label{rem:filprism-jdr}
One can show that the preimage of $\G_m/\G_m\subseteq\A^1/\G_m$ under the Rees map identifies with $\Z_p^\prism$ via the structure map, see \cite[Constr. 5.3.5]{FGauges}, and hence we obtain an open immersion
\begin{equation*}
j_\dR: \Z_p^\prism\rightarrow\Z_p^\N\;.
\end{equation*}
\end{rem}

\begin{rem}
The images of the open immersions $j_\HT$ and $j_\dR$ are disjoint: Indeed, if a filtered Cartier--Witt divisor $M\xrightarrow{d} W$ on a $p$-nilpotent ring $S$ is in the image of $j_\dR$, then $\sharp(d)$ is an isomorphism; if, however, it is in the image of $j_\HT$, then $\sharp(d)$ has the form $t^\sharp$ for the map $t: I\tensor_{W(S)} \G_a\rightarrow \G_a$ coming from a Cartier--Witt divisor $I\xrightarrow{\alpha} W(S)$ and \EEDIT{and thus its image generates a nilpotent ideal}. Since $\sHom_W(\G_a^\sharp, \G_a^\sharp)\cong\G_a$ by \cref{lem:admissible-classification}, these two possibilities exclude each other.
\end{rem}

It turns out that one can show that any filtered Cartier--Witt divisor $M\xrightarrow{d} W$ on a $p$-nilpotent ring $S$ is a quasi-ideal and that $R\Gamma_\fl(\Spec S, \ol{W})$ is concentrated in degrees $-1$ and $0$, see \cite[Prop.s 5.3.8, 5.3.9]{FGauges}, where $\ol{W}$ denotes the sheaf of $\infty$-groupoids $\cofib(M\xrightarrow{d} W)$, hence $R\Gamma_\fl(\Spec S, \ol{W})$ is naturally a 1-truncated animated $W(S)$-algebra.

\begin{defi}
The construction carrying a filtered Cartier--Witt divisor $M\xrightarrow{d} W$ on a $p$-nilpotent ring $S$ to $R\Gamma_\fl(\Spec S, \ol{W})$ yields an animated $W$-algebra stack $\G_a^\N\rightarrow\Z_p^\N$. The \emph{Nygaard-filtered prismatisation} $X^\N$ of $X$ is the stack $\pi_{X^\N}: X^\N\rightarrow\Z_p^\N$ defined by
\begin{equation*}
X^\N(\Spec S\rightarrow\Z_p^\N)\coloneqq \Map(\Spec\G_a^\N(S), X)\;,
\end{equation*}
where the mapping space is computed in derived algebraic geometry. If $X=\Spf R$ is affine, we also write $R^\N$ in place of $X^\N$.
\end{defi}

\begin{defi}
\begin{enumerate}[label=(\roman*)]
\item For a filtered Cartier--Witt divisor $M\xrightarrow{d} W$ on a $p$-nilpotent ring $S$, the diagram
\begin{equation*}
\begin{tikzcd}
M\ar[r]\ar[d, "d"] & F_*M'\ar[d, "F_*(d')"] \\
W\ar[r] & F_*W
\end{tikzcd}
\end{equation*}
determines a map of animated $W$-algebra stacks $\G_a^\N\rightarrow\pi^*\G_a^\prism$ (i.e.\ the map is linear over $F: W\rightarrow W$). The \emph{structure map}
\begin{equation*}
\pi_X: X^\N\rightarrow X^\prism 
\end{equation*}
is the map of stacks induced by $\G_a^\N\rightarrow\pi^*\G_a^\prism$; it lives over $\pi: \Z_p^\N\rightarrow\Z_p^\prism$.

\item The \emph{Rees map} 
\begin{equation*}
t_X: X^\N\rightarrow\A^1/\G_m
\end{equation*}
is the composition of the map $X^\N\rightarrow\Z_p^\N$ with $t: \Z_p^\N\rightarrow\A^1/\G_m$.
\end{enumerate}
\end{defi}

Also note that the open immersions $j_\HT, j_\dR: \Z_p^\prism\rightarrow\Z_p^\N$ induce open immersions $j_\HT, j_\dR: X^\prism\rightarrow X^\N$ with disjoint image via pullback to $\G_a^\N$. Finally, observe that, over $\A^1/\G_m$, there is an equivalence
\begin{equation}
\label{eq:filprism-drmap}
\EEDIT{
\Cone(\V(\O(1))^\sharp\oplus F_*W\xrightarrow{(t^\sharp, V)} W)\cong\Cone(\V(\O(1))^\sharp\xrightarrow{t^\sharp} \G_a)
}
\end{equation}
and hence, we obtain a map of stacks
\begin{equation*}
i_{\dR, +}: X^{\dR, +}\rightarrow X^\N
\end{equation*}
living over the map $\A^1/\G_m\rightarrow\Z_p^\N$ from \cref{ex:filteredprism-drmap}, which we also call the \emph{Hodge-filtered de Rham map}. Again, restriction to $X^\dR$ yields a map
\begin{equation*}
i_\dR: X^\dR\rightarrow X^\N
\end{equation*}
called the \emph{de Rham map}; note that $i_\dR$ factors through $j_\dR$ by construction. \EEDIT{Finally, restriction to $X^\Hod$ yields a map
\begin{equation*}
i_\Hod: X^\Hod\rightarrow X^\N
\end{equation*}
called the \emph{Hodge map}.}

\EEDIT{The most important} maps we have constructed so far and their compatibilities are summarised by the following commutative diagram:
\begin{equation}
\label{eq:filprism-maps}
\EDIT{
\begin{tikzcd}[ampersand replacement=\&]
\&\& X^\prism \ar[drr, equals] \&\& \\
X^\prism\ar[rr, "j_\HT"] \ar[dd] \ar[urr, "F_X"] \&\& X^\N \ar[dd, "t_X" {yshift=10pt}] \ar[u, "\pi_X", swap] \&\& X^\prism \ar[ll, "j_\dR", swap] \ar[dd] \\
\& X^{\dR, +}\ar[ur, "i_{\dR, +}" {yshift=-2pt}]\ar[dr] \&\& X^\dR\ar[ur, "i_\dR" {yshift=-2pt}] \ar[dr]\ar[ll] \&  \\
\widehat{\A}^1/\G_m \ar[rr] \&\& \A^1/\G_m \&\& \G_m/\G_m \ar[ll]
\end{tikzcd}
}
\end{equation}

\begin{ex}
\label{ex:filprism-perfd}
Generalising \cref{ex:prismatisation-perfd}, we claim that, for a perfectoid ring $R$, there is an isomorphism of stacks
\begin{equation}
\label{eq:filprism-perfd}
R^\N\cong\Spf(A_\inf(R)\langle u, t\rangle/(ut-\phi^{-1}(\xi)))/\G_m\;,
\end{equation}
\EDIT{where $A\coloneqq A_\inf(R)\langle u, t\rangle/(ut-\phi^{-1}(\xi))$ is endowed with the $(p, \xi)$-adic topology and, as usual, $t$ has degree \EEDIT{$-1$} while $u$ has degree \EEDIT{$1$}. We also note that while $A$ is clearly derived $(p, \xi)$-complete, it is also classically $(p, \xi)$-complete since it is $(p, \xi)$-adically separated.} 

To prove (\ref{eq:filprism-perfd}), consider an $S$-valued point of $R^\N$, i.e.\ a filtered Cartier--Witt divisor $M\xrightarrow{d} W$ on $S$ together with a map $R\rightarrow R\Gamma_\fl(\Spec S, \ol{W})$ of animated rings. Considering the diagram of admissible sequences
\begin{equation*}
\begin{tikzcd}
0\ar[r] & \V(L_M)^\sharp\ar[d, "\sharp(d)"]\ar[r] & M\ar[d, "d"]\ar[r] & F_*M'\ar[d, "F_*(d')"]\ar[r] & 0 \\
0\ar[r] & \G_a^\sharp\ar[r] & W\ar[r] & F_*W\ar[r] & 0\nospacepunct{\;,}
\end{tikzcd}
\end{equation*}
we see that the image of the map $M\xrightarrow{d} W$ is locally nilpotent mod all powers of $p$ since the same is true for the maps $M'\xrightarrow{d'} W$ and $\G_a^\sharp\rightarrow W$, see \cref{ex:drstack-perfd} and \cref{ex:prismatisation-perfd}. Thus, a similar deformation-theoretic argument as in \cref{ex:drstack-perfd} shows that we obtain a unique lift
\begin{equation*}
\begin{tikzcd}
A_\inf(R)\ar[r, dotted]\ar[d] & W(S)\ar[d] \\
\phi^*R=A_\inf(R)/(\phi^{-1}(\xi))\ar[r] & R\Gamma_\fl(\Spec S, \ol{W})\nospacepunct{\;,}
\end{tikzcd}
\end{equation*}
\EEDIT{where $\phi$ denotes the Frobenius of $A_\inf(R)$} \EDIT{and we note that the composite map $A_\inf(R)\rightarrow W(S)\rightarrow F_*W(S)$ agrees with the lift we obtain in \cref{ex:prismatisation-perfd} from the $S$-valued point $\Spec S\rightarrow R^\N\xrightarrow{\pi_R} R^\prism$ after undoing the Frobenius twist; in particular, we again obtain that the image of $\xi$ in $S$ under the composite map $A_\inf(R)\rightarrow W(S)\rightarrow S$ is nilpotent. Now passing to fibres in the above diagram} yields a commutative diagram
\begin{equation*}
\begin{tikzcd}
A_\inf(R)\ar[r]\ar[d, "\phi^{-1}(\xi)"] & R\Gamma_\fl(\Spec S, M)\ar[d] \\
A_\inf(R)\ar[r] & W(S)\nospacepunct{\;,}
\end{tikzcd}
\end{equation*}
which in turn yields a map 
\begin{equation*}
\alpha: (W\xrightarrow{\phi^{-1}(\xi)} W)\rightarrow (M\xrightarrow{d} W)
\end{equation*}
of filtered Cartier--Witt divisors on $S$; \EDIT{indeed, the source is a filtered Cartier--Witt divisor on $S$ by \cref{lem:admissible-fundamentalseq} and the fact that the image of $\xi$ in $S$ is nilpotent.} In terms of admissible sequences, this means that we get a diagram
\begin{equation*}
\begin{tikzcd}
0\ar[r] & \G_a^\sharp\ar[r]\ar[d, "u^\sharp", swap]\ar[dd, "\phi^{-1}(\xi)" {yshift=12.2pt, xshift=5pt}, bend right=45, swap] & W\ar[r]\ar[d, "\alpha", swap]\ar[dd, "\phi^{-1}(\xi)" {yshift=20pt, xshift=5pt}, bend right=45, swap] & F_*W\ar[r]\ar[d, "\alpha'"]\ar[dd, "\phi^{-1}(\xi)" {yshift=21.1pt, xshift=5pt}, bend right=45, swap] & 0 \\
0\ar[r] & \V(L_M)^\sharp\ar[r]\ar[d, "t^\sharp", swap] & M\ar[r]\ar[d, "d", swap] & F_*M'\ar[r]\ar[d, "F_*(d')"] & 0 \\
0\ar[r] & \G_a^\sharp\ar[r] & W\ar[r] & F_*W\ar[r] & 0\;,
\end{tikzcd}
\end{equation*}
where we wrote the maps in the left column as $u^\sharp$ and $t^\sharp$ for some $u: R\rightarrow L_M$ and $t: L_M\rightarrow R$ using \cref{lem:admissible-classification}. Observing that $ut=\phi^{-1}(\xi)$, we see that, \EEDIT{together with the map $A_\inf(R)\rightarrow W(S)\rightarrow S$ obtained above,} the datum $R\xrightarrow{u} L_M\xrightarrow{t} R$ determines an $S$-valued point of $\Spf A/\G_m$. Moreover, $\alpha'$ is an isomorphism by rigidity and hence the construction is reversible (namely, we can obtain $M$ as the pushout of the upper left corner in the category of sheaves), so (\ref{eq:filprism-perfd}) is proved.
\end{ex}

\begin{ex}
\label{ex:filprism-perf}
In the previous example, let $R=k$ be a perfect field of characteristic $p$. Then we find that
\begin{equation*}
k^\N\cong\Spf(W(k)\langle u, t\rangle/(ut-p))/\G_m\;. \qedhere
\end{equation*}
\end{ex}

\begin{rem}
\label{rem:filprism-jdrjhtperfd}
\EEDIT{
Under the isomorphisms (\ref{eq:prismatisation-perfd1}) and (\ref{eq:filprism-perfd}), the open immersion $j_\dR$ identifies with the map
\begin{equation*}
\Spf A_\inf(R)\cong \Spf(A_\inf(R)\langle t, t^{-1}\rangle)/\G_m\rightarrow \Spf(A_\inf(R)\langle u, t\rangle/(ut-\phi^{-1}(\xi)))/\G_m
\end{equation*}
obtained by inverting $t$. Meanwhile, the open immersion $j_\HT$ identifies with the map
\begin{equation*}
\begin{split}
\Spf A_\inf(R)\cong \Spf(A_\inf(R)\langle u, u^{-1}\rangle)/\G_m&\rightarrow \Spf(A_\inf(R)\langle u, t\rangle/(ut-\xi))/\G_m \\
&\rightarrow \Spf(A_\inf(R)\langle u, t\rangle/(ut-\phi^{-1}(\xi)))/\G_m
\end{split}
\end{equation*}
obtained by composing the Frobenius on $A_\inf(R)$ with inverting $u$.
}
\end{rem}

\begin{rem}
\label{rem:filprism-regular}
\EDIT{
There is an fpqc cover $\Spf\Z_p\langle u, t\rangle\rightarrow\Z_p^\N$ and hence the stack $\Z_p^\N$ is noetherian and regular. This can be proved by reduction to the perfectoid case treated in \cref{ex:filprism-perfd}, see \cite[Ex. 5.5.20]{FGauges} for details.
}
\end{rem}

As one should have expected by now, in good situations, coherent cohomology on the \EEDIT{Nygaard-filtered prismatisation} of $X$ computes \EDIT{the} Nygaard-filtered prismatic cohomology of $X$ (here, we use the definition of the Nygaard filtration on absolute prismatic cohomology from \cite[Def. 5.5.3]{APC}):

\begin{thm}
\label{thm:filprism-comparison}
Let $X$ be a bounded $p$-adic formal scheme. Assume that $X$ is $p$-quasisyntomic and qcqs. Then $t_{X, *}\O_{X^\N}$ identifies with $\Fil^\bullet_\N R\Gamma_\prism(X)$ in $\widehat{\DF}(\Z_p)$ under the Rees equivalence.
\end{thm}
\begin{proof}
This follows from \cite[Cor. 5.5.11, Rem. 5.5.18]{FGauges} and \cite[Cor. 5.5.21]{APC} via quasisyntomic descent.
\end{proof}

Again, the category $\D(X^\N)$ should provide a sensible notion of coefficients for Nygaard-filtered prismatic cohomology.

\begin{defi}
The category $\D(X^\N)$ is called the category of \emph{gauges} on $X$ and denoted $\Gauge_\prism(X)$. If $X=\Spf R$ is affine, we also write $\Gauge_\prism(R)$ in place of $\Gauge_\prism(\Spf R)$.
\end{defi}

\begin{ex}
By pushforward along the morphism $\pi_{X^\N}: X^\N\rightarrow\Z_p^\N$, we obtain a gauge $\H_\N(X)\coloneqq \pi_{X^\N, *}\O_{X^\N}$, which we call the \emph{structure gauge} of $X$. 
\end{ex}

\begin{ex}
\label{ex:filprism-bktwist}
The line bundle
\begin{equation*}
\EEDIT{
\O_{\Z_p^\N}\{1\}\coloneqq \pi^*\O_{\Z_p^\prism}\{1\}\tensor t^*\O(1)\in \Gauge_\prism(\Z_p)
}
\end{equation*}
is called the \emph{Breuil--Kisin twist} and we denote its $n$-th tensor power by $\O_{\Z_p^\N}\{n\}$. Via pullback to $X^\N$, we obtain line bundles $\O_{X^\N}\{n\}\in\Gauge_\prism(X)$.
\end{ex}

\begin{defi}
Let $X$ be a bounded $p$-adic formal scheme which is $p$-quasisyntomic and qcqs. For a gauge $E$ on $X$, we define the \emph{Nygaard-filtered prismatic cohomology} of $X$ with coefficients in $E$ as
\begin{equation*}
\EDIT{\Fil^\bullet_\N R\Gamma_\prism(X, E)\coloneqq t_{X, *}(E)\;.}
\end{equation*}
\end{defi}

\EDIT{
Recalling that the stack $\Z_p^\N$ is noetherian by virtue of the cover $\Spf\Z_p\langle u, t\rangle\rightarrow \Z_p^\N$ from \cref{rem:filprism-regular}, we can also introduce a reasonable notion of coherent sheaves on $\Z_p^\N$:

\begin{defi}
A quasi-coherent complex on $\Z_p^\N$ is called \emph{coherent} if its pullback to $\Spf\Z_p\langle u, t\rangle$ along the map from \cref{rem:filprism-regular} is a finitely generated $\Z_p\langle u, t\rangle$-module. The full subcategory of $\Perf(\Z_p^\N)$ spanned by coherent sheaves is denoted $\Coh(\Z_p^\N)$.
\end{defi}

\begin{rem}
\label{rem:filprism-cohheartperf}
Since $\Z_p\langle u, t\rangle$ \EEDIT{is regular}, one may alternatively describe $\Coh(\Z_p^\N)$ as the heart of a canonical $t$-structure on $\Perf(\Z_p^\N)$ induced by the $t$-structure on $\Perf(\Z_p\langle u, t\rangle)$, see \cite[Rem. 5.5.19]{FGauges}.
\end{rem}
}

We conclude this section by giving an alternative description of the Nygaard-filtered prismatisation via descent which is sometimes helpful. For this, one first proves that, similarly to the case of perfectoid rings treated in \cref{ex:filprism-perfd}, one can explicitly describe the Nygaard-filtered prismatisation of a quasiregular semiperfectoid ring $R$ in the sense of \cite[Def. 4.20]{THHandPAdicHodgeTheory}.

\begin{thm}
\label{thm:filprism-qrsp}
For any quasiregular semiperfectoid ring $R$, there is a natural isomorphism 
\begin{equation*}
R^\N\cong \Spf(\Rees(\Fil^\bullet_\N\Prism_R))/\G_m
\end{equation*}
of stacks over $\A^1/\G_m$, \EDIT{where $\Rees(\Fil^\bullet_\N\Prism_R)$ is endowed with the $(p, I)$-adic topology and $(\Prism_R, I)$ is the initial object of the absolute prismatic site of $R$.}
\end{thm}
\begin{proof}
See \cite[Cor. 5.5.11]{FGauges}.
\end{proof}

Subsequently, one can use the fact that quasiregular semiperfectoid rings form a basis of the quasisyntomic site introduced by Bhatt--Morrow--Scholze in \cite[Def. 4.10]{THHandPAdicHodgeTheory} to obtain a description of $X^\N$ for any $p$-adic formal scheme $X$ which is $p$-quasisyntomic and qcqs. For this, one needs to use the following property of the Nygaard-filtered prismatisation:

\begin{lem}
\label{lem:filprism-quasisyntomiccover}
Assume that $f: X\rightarrow Y$ is a quasi-syntomic cover of $p$-adic formal schemes which are $p$-quasisyntomic and qcqs. Then the induced map $X^\N\rightarrow Y^\N$ is a flat cover.
\end{lem}
\begin{proof}
See \cite[Rem. 5.5.18]{FGauges}.
\end{proof}

\subsubsection{The syntomification}
\label{subsect:syntomification}

\begin{defi}
\label{defi:syntomification-def}
The \emph{syntomification} $X^\Syn$ of $X$ is defined as the pushout 
\begin{equation*}
\begin{tikzcd}[column sep=large]
X^\prism\sqcup X^\prism\ar[r, "j_\HT\sqcup j_\dR"]\ar[d] & X^\N\ar[d, "j_\N"] \\
X^\prism\ar[r, "j_\prism"] & X^\Syn\;.
\end{tikzcd}
\end{equation*}
If $X=\Spf R$ is affine, we also write $R^\Syn$ in place of $X^\Syn$.
\end{defi}

Note that the pushout above is really a coequaliser and that we can construct it by first taking coequalisers on the level of points and then sheafifying. By \cite[Thm. 8.2.2]{Prismatization}, it does not make a difference whether we perform this construction in the category of stacks in groupoids or in the category of stacks in $\infty$-groupoids. Also note that $j_\prism$ is an open immersion and $j_\N$ is étale.

\begin{defi}
The category $\D(X^\Syn)$ is called the category of \emph{$F$-gauges} on $X$ and denoted $\FGauge_\prism(X)$. If $X=\Spf R$ is affine, we also write $\FGauge_\prism(R)$ in place of $\FGauge_\prism(\Spf R)$.
\end{defi}

Observe that we have a functor 
\begin{equation*}
j_\N^*: \FGauge_\prism(X)\rightarrow\Gauge_\prism(X)\;,
\end{equation*}
which we call \emph{passage to the underlying gauge}, and that there is an equaliser diagram
\begin{equation}
\label{eq:syntomification-equaliser}
\begin{tikzcd}
\D(X^\Syn)\ar[r, "j_\N^*"] & \D(X^\N)\ar[r,shift left=.75ex,"j_\HT^*"]
  \ar[r,shift right=.75ex,swap,"j_\dR^*"] & \D(X^\prism)\nospacepunct{\;.}
\end{tikzcd}
\end{equation}
In particular, for any $E\in\D(X^\Syn)$, there is a fibre sequence
\begin{equation}
\label{eq:syntomification-fibreseq}
\begin{tikzcd}[column sep=large]
R\Gamma(X^\Syn, E)\ar[r] & R\Gamma(X^\N, j_\N^*E)\ar[r, "j_\HT^*-j_\dR^*"] & R\Gamma(X	^\prism, j_\prism^*E)\nospacepunct{\;.}
\end{tikzcd}
\end{equation}

\begin{ex}
Let $\pi_{X^\Syn}$ denote the morphism $X^\Syn\rightarrow\Z_p^\Syn$ obtained by functoriality. We obtain an $F$-gauge $\H_\Syn(X)\coloneqq \pi_{X^\Syn, *}\O_{X^\Syn}$, which we call the \emph{structure $F$-gauge} of $X$. Note that the underlying gauge of $\H_\Syn(X)$ is $\H_\N(X)$.
\end{ex}

\begin{ex}
\label{ex:syntomification-bktwists}
The Breuil--Kisin twist $\O_{X^\N}\{1\}$ descends to a line bundle $\O_{X^\Syn}\{1\}$ since
\begin{equation*}
j_\dR^*\O_{X^\N}\{1\}\cong \O_{X^\prism}\{1\}\cong F_X^*\O_{X^\prism}\tensor \cal{I}_X\cong j_\HT^*\O_{X^\N}\{1\}
\end{equation*}
by (\ref{eq:filprism-maps}). We also call $\O_{X^\Syn}\{1\}$ \emph{Breuil--Kisin twist} and again denote its $n$-th tensor power by $\O_{X^\Syn}\{n\}$.
\end{ex}

\begin{ex}
\label{ex:syntomification-fgaugesperfd}
Let $R$ be a perfectoid ring. By \cref{ex:filprism-perfd}, a gauge $E$ on $R$ carries the same data as a diagram
\begin{equation*}
\begin{tikzcd}
\dots \ar[r,shift left=.5ex,"t"]
  & \ar[l,shift left=.5ex, "u"] M^{i+1} \ar[r,shift left=.5ex,"t"] & \ar[l,shift left=.5ex, "u"] M^i \ar[r,shift left=.5ex,"t"] & \ar[l,shift left=.5ex, "u"] M^{i-1} \ar[r,shift left=.5ex,"t"] & \ar[l,shift left=.5ex, "u"] \dots
\end{tikzcd}
\end{equation*}
of $(p, \xi)$-complete $A_\inf(R)$-complexes such that $ut=tu=\phi^{-1}(\xi)$. Writing $M^{-\infty}=(\colim_i M^{-i})_{(p, \xi)}^\wedge$ and $M^\infty=(\colim_i M^i)_{(p, \xi)}^\wedge$ for the $(p, \xi)$-completed colimits along the $t$- and $u$-maps, respectively, we see \EEDIT{using \cref{rem:filprism-jdrjhtperfd}} that $j_\dR^*E$ identifies with $M^{-\infty}$ while $j_\HT^*E$ identifies with $\phi^*M^\infty$. Thus, specifying a descent datum of $E$ to $R^\Syn$ amounts to specifying an isomorphism $\tau: \phi^*M^\infty\cong M^{-\infty}$. In this situation, observing that $R\Gamma(R^\N, j_\N^* E)=M^0$, we see that the fibre sequence (\ref{eq:syntomification-fibreseq}) yields 
\begin{equation*}
R\Gamma(R^\Syn, E)=\fib(M^0\xrightarrow{t^\infty-\tau u^\infty} M^{-\infty})\;.
\end{equation*}
Here, $t^\infty: M^0\rightarrow M^{-\infty}$ and $u^\infty: M^0\rightarrow M^\infty$ denote the maps induced by the $t$- and $u$-maps, respectively, and to define the composition $\tau u^\infty$, we use the identification $M^\infty\cong\phi^*M^\infty$.
\end{ex}

Similar to what we have seen previously, the syntomification and the category of $F$-gauges provide a sensible notion of syntomic cohomology with coefficients:

\begin{defi}
Let $X$ be a bounded $p$-adic formal scheme which is $p$-quasisyntomic and qcqs. For an $F$-gauge $E$ on $X$, we define the \emph{syntomic cohomology} of $X$ with coefficients in $E$ as
\begin{equation*}
R\Gamma_\Syn(X, E)\coloneqq R\Gamma(X^\Syn, E)\;.
\end{equation*}
If $E=\O_{X^\Syn}\{n\}$ is a Breuil-Kisin twist, we also write $R\Gamma_\Syn(X, \Z_p(n))$ instead of $R\Gamma_\Syn(X, E)$ and we use $R\Gamma_\Syn(X, \Q_p(n))$ to denote $R\Gamma_\Syn(X, E)[\tfrac{1}{p}]$.
\end{defi}

In some sense, the category of $F$-gauges on $X$ should capture a universal notion of coefficients for various $p$-adic cohomology theories. Evidence for this is given by the existence of certain \emph{realisation functors} from $F$-gauges to coefficients for various other cohomology theories.

\begin{ex}
\label{ex:syntomification-tcrys}
By functoriality of the syntomification applied to $X_{p=0}\rightarrow X$, we obtain a pullback functor 
\begin{equation*}
T_\crys: \FGauge_\prism(X)\rightarrow\FGauge_\prism(X_{p=0})\;,
\end{equation*}
which we call the \emph{crystalline realisation}.
\end{ex}

\begin{ex}
\label{ex:syntomification-tdr+}
The Hodge-filtered de Rham map $i_{\dR, +}: X^{\dR, +}\rightarrow X^\N$ composed with the map $j_\N: X^\N\rightarrow X^\Syn$ induces a pullback functor
\begin{equation*}
T_{\dR, +}: \FGauge_\prism(X)\rightarrow \D(X^{\dR, +})\;,
\end{equation*}
which we call the \emph{Hodge-filtered de Rham realisation}. Pulling back further to $X^\dR$, we obtain a functor
\begin{equation*}
\EDIT{
T_\dR: \FGauge_\prism(X)\rightarrow \D(X^\dR)\;,
}
\end{equation*}
which we call the \emph{de Rham realisation}.
\end{ex}

There is also an \emph{étale realisation}
\begin{equation*}
T_\et: \Perf(X^\Syn)\rightarrow\D^b_\lisse(X_\eta, \Z_p)\;.
\end{equation*}
Here, $X_\eta$ denotes the generic fibre of $X$ regarded as an adic space over $\Q_p$ and $\D^b_\lisse(X_\eta, \Z_p)$ denotes the full subcategory of $\D(X_{\eta, \proet}, \Z_p)$ spanned by derived $p$-complete locally bounded objects whose mod $p$ reduction has locally constant cohomology sheaves with finitely generated stalks.

Using the arc-topology introduced in \cite{arcTopology}, it suffices to describe $T_\et$ in the case where $X=\Spf R$ is perfectoid by descent, see \cite[Constr. 6.3.2]{FGauges}. Here, in the notation of \cref{ex:syntomification-fgaugesperfd}, we see that $u^\infty$ and $t^\infty$ induce isomorphisms
\begin{equation*}
M^\infty[\tfrac{1}{\phi^{-1}(\xi)}]\xleftarrow{u^\infty} M^0[\tfrac{1}{\phi^{-1}(\xi)}]\xrightarrow{t^\infty} M^{-\infty}[\tfrac{1}{\phi^{-1}(\xi)}]\;.
\end{equation*}
Hence, $\tau$ induces an isomorphism
\begin{equation*}
\phi^*M^{-\infty}[\tfrac{1}{\xi}]\cong \phi^*M^\infty[\tfrac{1}{\xi}]\cong M^{-\infty}[\tfrac{1}{\xi}]
\end{equation*}
endowing $M^{-\infty}$ with the structure of a prismatic $F$-crystal on $R$. Thus, we obtain a natural map
\begin{equation}
\label{eq:syntomification-fgaugetolaurentfcrystal}
\Perf(R^\Syn)\rightarrow \Perf^\phi(A_\inf(R)[\tfrac{1}{\xi}]^\wedge_{(p)})\;,
\end{equation}
\EDIT{where the right-hand side denotes the category of perfect complexes over $A_\inf(R)[\tfrac{1}{\xi}]^\wedge_{(p)}$ equipped with an automorphism linear over the Frobenius of $A_\inf(R)[\tfrac{1}{\xi}]^\wedge_{(p)}$.} By \cite[Ex. 3.5]{FCrystals}, the latter category identifies with $\D^b_\lisse(\Spa(R[\tfrac{1}{p}], R), \Z_p)$ via the construction
\begin{equation}
\label{eq:syntomification-laurentfcrystaltolocsys}
M\mapsto (M\tensorL_{A_\inf(R)} W(R[\tfrac{1}{p}]^\flat))^{\phi=1}\;,
\end{equation}
where $\phi$ acts diagonally \EEDIT{and the $\phi$-invariants are taken in the derived sense, i.e.\ the right-hand side denotes the fibre of the map}
\begin{equation*}
\EEDIT{
\phi-\id: M\tensorL_{A_\inf(R)} W(R[\tfrac{1}{p}]^\flat)\rightarrow M\tensorL_{A_\inf(R)} W(R[\tfrac{1}{p}]^\flat)\;.
}
\end{equation*}

\begin{rem}
\label{rem:syntomification-etalerealisationglobal}
It is also possible to describe the étale realisation globally (i.e.\ without using descent): Namely, pullback along either one of the maps $X^\prism\rightarrow X^\N\rightarrow X^\Syn$ given by $j_\N\circ j_\HT$ or $j_\N\circ j_\dR$, respectively, induces a functor
\begin{equation}
\label{eq:syntomification-etalerealisationglobal}
\Perf(X^\Syn)\rightarrow \Perf^\phi(X_\prism)\;,
\end{equation}
where the right-hand side denotes the category of prismatic $F$-crystals in perfect complexes on $X$ as defined in \cite[Rem. 4.2]{FCrystals} and the Frobenius structure comes from the gluing description of $X^\Syn$ using the fact that $j_\HT$ and $j_\dR$ differ by a Frobenius twist, see \cite[Rem. 6.3.4]{FGauges} for details. Then the étale realisation is given by the composition
\begin{equation*}
\Perf(X^\Syn)\rightarrow \Perf^\phi(X_\prism)\rightarrow \Perf^\phi(X_\prism, \O_\prism[1/\cal{I}_\prism]_{(p)}^\wedge)\cong \D^b_\lisse(X_\eta, \Z_p)\;;
\end{equation*}
here, the last equivalence is \cite[Cor. 3.7]{FCrystals}.
\end{rem}

\begin{ex}
One can show that the étale realisation carries the Breuil--Kisin twists $\O_{X^\Syn}\{n\}$ to the usual Tate twist $\Z_p(n)$, see \cite[Ex. 4.9]{FCrystals}.
\end{ex}

In fact, in the case $X=\Spf\Z_p$, it turns out that the étale realisation defines an equivalence between a certain full subcategory of $\Perf(\Z_p^\Syn)$ and the category $\Rep_{\Z_p}^\crys(G_{\Q_p})$ of crystalline $G_{\Q_p}$-representations in finite free $\Z_p$-modules, \EDIT{as we will now recall; here, $G_{\Q_p}$ denotes the absolute Galois group of $\Q_p$.}

\begin{defi}
Let $R$ be a perfectoid valuation ring. Then the category $\Coh^\refl(R^\Syn)$ of \emph{reflexive} $F$-gauges on $R$ is the full subcategory of $\Perf(R^\Syn)$ spanned by $F$-gauges $E$ satisfying the following two properties:
\begin{enumerate}[label=(\roman*)]
\item $j_\prism^*E$ is a locally free sheaf.
\item $j_\N^*E\in\Perf(A^\N)\cong\Perf_{\gr}(A_\inf(R)\langle u, t\rangle/(ut-\phi^{-1}(\xi)))$ is $(u, t)$-regular, i.e.\ the Koszul complex $\Kos(j_\N^*E; u, t)$ is discrete.
\end{enumerate} 
\end{defi}

\begin{prop}
\label{prop:syntomification-refltofcrystal}
Let $R$ be a perfectoid valuation ring. Then the following categories are equivalent:
\begin{enumerate}[label=(\roman*)]
\item The category $\Vect^\phi(A_\inf(R))$ of prismatic $F$-crystals in vector bundles on $R$.
\item The category $\Coh^\refl(R^\Syn)$ of reflexive $F$-gauges on $R$.
\end{enumerate}
\end{prop}
\begin{proof}
See \cite[Cor. 6.6.5]{FGauges}.
\end{proof}

\begin{rem}
\label{rem:syntomification-descriptionrefltofcrystal}
The functor $\Vect^\phi(A_\inf(R))\rightarrow\Coh^\refl(R^\Syn)$ can actually be described as follows: To a prismatic $F$-crystal $M$ in vector bundles on $R$ with $\tau: \phi^*M[\tfrac{1}{\xi}]\cong M[\tfrac{1}{\xi}]$, we associate the $\phi^{-1}(\xi)$-adically filtered $A_\inf(R)$-module $\Fil^\bullet M$ given by $\Fil^i M=\phi_*(\xi^i M\cap \phi^*M)$, where the intersection takes place inside $\phi^*M[\tfrac{1}{\xi}]\cong M[\tfrac{1}{\xi}]$ via $\tau$; by \cref{ex:syntomification-fgaugesperfd}, this data will determine an $F$-gauge on $R$.
\end{rem}

\begin{defi}
An $F$-gauge $E\in\Perf(\Z_p^\Syn)$ is called \emph{reflexive} if its pullback to $\O_C^\Syn$ is reflexive, \EDIT{where we recall that $C$ denotes a fixed completed algebraic closure of $\Q_p$.} The full subcategory of $\Perf(\Z_p^\Syn)$ spanned by reflexive $F$-gauge is denoted $\Coh^\refl(\Z_p^\Syn)$.
\end{defi}

Note that the notation $\Coh^\refl(\Z_p^\Syn)$ is actually not misleading: Indeed, one can show that any reflexive $F$-gauge on $\Z_p$ is coherent, see \cite[Rem. 6.6.12]{FGauges}. Here, we call an $F$-gauge on $\Z_p$ \emph{coherent} if its pullback to $\Z_p^\N$ is coherent. The full subcategory of $\Perf(\Z_p^\Syn)$ spanned by coherent $F$-gauges is denoted $\Coh(\Z_p^\Syn)$.

\begin{thm}
\label{thm:syntomification-reflcrys}
There is an equivalence of categories
\begin{equation*}
\Coh^\refl(\Z_p^\Syn)\cong \Vect^\phi((\Z_p)_\prism)\cong\Rep_{\Z_p}^\crys(G_{\Q_p})
\end{equation*}
induced by étale realisation.
\end{thm}
\begin{proof}
See \cite[Thm. 6.6.13]{FGauges}.
\end{proof}

Actually, it turns out that this allows one to show that the étale realisation of all coherent sheaves on $\Z_p^\Syn$ is a crystalline $G_{\Q_p}$-representations as soon as one inverts $p$ and that, moreover, in this setting, cohomology of a coherent sheaf agrees with the crystalline part of the Galois cohomology of its étale realisation:

\begin{prop}
\label{prop:syntomification-cohcrystalline}
For any $M\in\Coh(\Z_p^\Syn)$, the $G_{\Q_p}$-representation $T_\et(M)[\tfrac{1}{p}]$ is crystalline and the map
\begin{equation*}
R\Gamma(\Z_p^\Syn, M)[\tfrac{1}{p}]\rightarrow R\Gamma(G_{\Q_p}, T_\et(M)[\tfrac{1}{p}])
\end{equation*}
induced by étale realisation has the following properties:
\begin{enumerate}[label=(\roman*)]
\item It induces an isomorphism on $H^0$.

\item It induces an injective map on $H^1$ with image the subspace of $H^1(G_{\Q_p}, T_\et(M)[\tfrac{1}{p}])$ spanned by crystalline extensions of $\Q_p$ by $T_\et(M)[\tfrac{1}{p}]$.

\item All $H^i(\Z_p^\Syn, M)[\tfrac{1}{p}]$ for $i\geq 2$ vanish.
\end{enumerate}
\end{prop}
\begin{proof}
See \cite[Prop. 6.7.3]{FGauges}.
\end{proof}

We now focus on studying $\Z_p^\Syn$ further and will hence omit the subscript $\Z_p^\Syn$ from the notation for the Breuil--Kisin twists for the purpose of simplification. Namely, we will describe a special locus $\Z_{p, \red}^\Syn$ on $\Z_p^\Syn$ called the \emph{reduced locus} and this will require us to construct a certain section $v_1\in H^0(\Z_p^\Syn, \O\{p-1\}/p)$. Here, we only briefly sketch the argument and refer to \cite[Constr. 6.2.1]{FGauges} for details. First, by quasisyntomic descent, see \cref{lem:filprism-quasisyntomiccover}, it suffices to construct compatible sections $v_{1, R}\in H^0(R^\Syn, \O_{R^\Syn}\{p-1\}/p)$ for all $p$-torsionfree quasiregular semiperfectoid rings $R$ and, by the gluing description of $R^\Syn$, this is the same as constructing sections $v_{1, R}\in H^0(R^\N, \O_{R^\N}\{p-1\}/p)$ satisfying a compatibility with respect to the Frobenius. However, using \cref{thm:filprism-qrsp}, we have
\begin{equation*}
H^0(R^\N, \O_{R^\N}\{p-1\}/p)=\Fil^{p-1}_\N\Prism_R\{p-1\}/p
\end{equation*}
and due to the isomorphism $\Prism_R\{p-1\}/p\cong I^{-1}/p$, we are reduced to constructing a nonzero map $I/p\rightarrow \Fil^{p-1}_\N\Prism_R/p$. However, one can show that $I/p\subseteq \Fil^p_\N\Prism_R/p$ as submodules of $\Prism_R/p$ and hence the map we seek is given by
\begin{equation*}
I/p\subseteq \Fil^p_\N\Prism_R/p\subseteq\Fil^{p-1}_\N\Prism_R/p\;.
\end{equation*}
Finally, the necessary compatibility with the Frobenius follows from $\phi(I)\tensor_{\Prism_R} I^{-(p-1)}/p\cong I/p$.

\begin{ex}
\label{ex:syntomification-v1perfd}
Perhaps most importantly for what follows, if $R$ is perfectoid, the section $v_{1, R}$ identifies with
\begin{equation}
\label{eq:syntomification-v1perfd}
u^pt=\xi t^{-p+1}\in A_\inf(R)\langle u, t\rangle/(ut-\phi^{-1}(\xi), p)\;,
\end{equation}
where we have used $\xi$ to trivialise the Breuil--Kisin twist. In particular, due to $ut=\xi^{1/p}$ in $A_\inf(R)\langle u, t\rangle/(ut-\phi^{-1}(\xi), p)$, the loci $\{v_1\neq 0\}$ and $\{\xi\neq 0\}$ in $\Spec A_\inf(R)\langle u, t\rangle/(ut-\phi^{-1}(\xi), p)$ are the same.
\end{ex}

\begin{defi}
The \emph{reduced locus} $\Z_{p, \red}^\Syn$ of $\Z_p^\Syn$ is the vanishing locus of $(p, v_1)$. Its pullback to $\Z_p^\N$ is called the \emph{reduced locus} of $\Z_p^\N$ and denoted $\Z_{p, \red}^\N$.
\end{defi}

It turns out that there is a very concrete way to describe $\Z_{p, \red}^\Syn$ and cohomology on it. To explain this, we need to consider several components of $\Z_{p, \red}^\Syn$. Note that we are now working over $\F_p$.

\begin{enumerate}[label=(\roman*)]
\item The \emph{conjugate-filtered Hodge--Tate component} $\Z_{p, \HT, c}^\N$ is the closed substack of $(\Z_p^\N)_{p=0}$ cut out by the equation $t=0$ and one can show that $\Z_{p, \HT, c}^\N\cong \G_a/(\G_a^\sharp\rtimes\G_m)$, see \cite[Prop. 5.3.7]{FGauges}. Thus, the stack $\Z_{p, \HT, c}^\N$ admits the following stratification:
\begin{enumerate}[label=(\arabic*)]
\item The substack $\Z_{p, \HT}^\N\coloneqq B\Stab_{\G_a^\sharp\rtimes\G_m}(1\in\G_a)\cong B\G_m^\sharp$ is open and one can show that it identifies with $j_\HT(\Z_p^\prism)\cap \Z_{p, \red}^\N$. Here, $\G_m^\sharp$ denotes the (base change to $\F_p$ of) the PD-hull of $\G_m$ at $1$, i.e.\ $\Spec\Z[t, t^{-1}, \frac{(t-1)}{2!}, \frac{(t-1)}{3!}, \dots]$.

\item The substack $\Z_{p, \Hod}^\N\coloneqq B\Stab_{\G_a^\sharp\rtimes\G_m}(0\in\G_a)\cong B(F_*\G_a^\sharp\rtimes\G_m)$ is closed and identifies with the reduced complement of $\Z_{p, \HT}^\N$.
\end{enumerate}

\item The \emph{Hodge-filtered de Rham component} $\Z_{p, \dR, +}^\N$ is the closed substack parametrising filtered Cartier--Witt divisors $M\xrightarrow{d} W$ where the admissible sequence of $M$ is locally split; \EEDIT{for simplicity, we will call such filtered Cartier--Witt divisors \emph{locally split} in the sequel.} Noting that $i_{\dR, +}: \A^1/\G_m\rightarrow \Z_{p, \dR, +}^\N$ is an fpqc cover splitting the restriction $t_{\dR, +}: \Z_{p, \dR, +}^\N\rightarrow\A^1/\G_m$ of the Rees map, we conclude that there is an isomorphism of stacks $\Z_{p, \dR, +}^\N\cong B_{\A^1/\G_m}\cal{G}$ over $\A^1/\G_m$, where $\cal{G}$ denotes the group scheme $\Aut(i_{\dR, +})$. One can explicitly compute $\cal{G}$, see \cite[Ch. 7]{Prismatization}; for us, however, it is only important that $\Z_{p, \dR, +}^\N$ admits the following stratification:
\begin{enumerate}[label=(\arabic*)]
\item The preimage of $\Spec\F_p\subseteq\A^1/\G_m$ under $t_{\dR, +}$ identifies with the open substack $j_\dR(\Z_p^\prism)\cap \Z_{p, \red}^\N$ and \EDIT{is} denoted $\Z_{p, \dR}^\N$. Up to a Frobenius twist, which is trivial here, it is isomorphic to $\Z_{p, \HT}^\N$.

\item The preimage of $B\G_m\subseteq\A^1/\G_m$ under $t_{\dR, +}$ is isomorphic to $\Z_{p, \Hod}^\N$.
\end{enumerate}
\end{enumerate}

To obtain $\Z_{p, \red}^\N$ from these components, one glues $\Z_{p, \HT, c}^\N$ and $\Z_{p, \dR, +}^\N$ along $\Z_{p, \Hod}^\N$. Further gluing $\Z_{p, \HT}^\N$ and $\Z_{p, \dR}^\N$ along a Frobenius twist, which is trivial in our situation, then yields $\Z_{p, \red}^\Syn$. Thus, to study cohomology on $\Z_{p, \red}^\Syn$, we need to study cohomology on the loci described above, starting with $\Z_{p, \dR}^\N$ and $\Z_{p, \Hod}^\N$:

\begin{enumerate}[label=(\arabic*)]
\item For $\Z_{p, \dR}^\N\cong B\G_m^\sharp$, one can show that
\begin{equation*}
\D(\Z_{p, \dR}^\N)\cong \D_{(\Theta^p-\Theta)-\nilp}(\F_p[\Theta])\;,
\end{equation*}
where the latter category consists of $\F_p$-complexes $V$ equipped an operator $\Theta$ such that the action of $\Theta^p-\Theta$ on cohomology is locally nilpotent, \EDIT{i.e.\ for every $i\in\Z$ and $x\in H^i(V)$, there is some $n\geq 0$ such that $(\Theta^p-\Theta)^n(x)=0$,} see \cite[Thm. 3.5.8]{APC}. Under this equivalence, the Breuil--Kisin twist $\O\{n\}$ corresponds to $(\F_p, \Theta=n)$ and if $E\in\D(\Z_{p, \dR}^\N)$ corresponds to $(V, \Theta)$, then
\begin{equation*}
R\Gamma(\Z_{p, \dR}^\N, E)=\fib(V\xrightarrow{\Theta} V)\;,
\end{equation*}
see \cite[Ex. 3.5.6, Prop. 3.5.11]{APC}.

\item For $\Z_{p, \Hod}^\N\cong B(F_*\G_a^\sharp\rtimes\G_m)$, one can show that
\begin{equation*}
\D(\Z_{p, \Hod}^\N)\cong \D_{\gr, \Theta-\nilp}(\F_p[\Theta])\;,
\end{equation*}
where $\Theta$ has degree $-p$, i.e.\ the latter category consists of $\Z$-indexed collections of $\F_p$-complexes $V_i$ equipped with operators $\Theta_i: V_i\rightarrow V_{i-p}$ such that the action of $\Theta=\bigoplus_i \Theta_i$ on cohomology is locally nilpotent. Under this equivalence, the Breuil--Kisin twist $\O\{n\}$ corresponds to the vector space $\F_p$ sitting in grading degree $-n$ and if $E\in\D(\Z_{p, \Hod}^\N)$ corresponds to \EEDIT{$(\{V_i\}, \{\Theta_i\})$}, then
\begin{equation*}
R\Gamma(\Z_{p, \Hod}^\N, E)=\fib(V_0\xrightarrow{\Theta} V_{-p})\;.
\end{equation*}
\end{enumerate}

On $\Z_{p, \dR, +}^\N$ and $\Z_{p, \HT, c}^\N$, we obtain the following descriptions of quasi-coherent complexes:

\begin{enumerate}[label=(\roman*)]
\item For $\Z_{p, \HT, c}^\N\cong \G_a/(\G_a^\sharp\rtimes\G_m)$, one can show that 
\begin{equation*}
\D(\Z_{p, \HT, c}^\N)\cong\D_{\gr, D-\nilp}(\cal{A}_1)\;,
\end{equation*}
where $\cal{A}_1\coloneqq\F_p\{x, D\}/(Dx-xD-1)$ with $D$ having degree \EEDIT{$-1$} and $x$ having degree \EEDIT{$1$}, \EEDIT{see also \cref{prop:fildhod-zpntzero}}; here, the notation $\F_p\{x, D\}$ refers to the free (non-commutative!) $\F_p$-algebra on the variables $x$ and $D$. Alternatively, via the Rees equivalence, we may describe any $E\in\D(\Z_{p, \HT, c}^\N)$ as an increasingly filtered $\F_p$-complex $\Fil_\bullet V$ together with operators $D: \Fil_\bullet V\rightarrow \Fil_{\bullet-1} V$ which are locally nilpotent in cohomology such that $Dx=xD+1$, where $x: \Fil_\bullet V\rightarrow\Fil_{\bullet+1} V$ denotes the transition maps. \EEDIT{We warn the reader that, in particular, this relation means that the operators $xD: \Fil_\bullet V\rightarrow\Fil_\bullet V$ are \emph{not} compatible with the filtration, but one rather has commutative diagrams
\begin{equation*}
\begin{tikzcd}[ampersand replacement=\&]
\Fil_i V\ar[r, "x"]\ar[d, "xD"] \& \Fil_{i+1} V\ar[d, "xD-1"] \\
\Fil_i V\ar[r, "x"] \& \Fil_{i+1} V\nospacepunct{\;;}
\end{tikzcd}
\end{equation*}
however, since we are working over $\F_p$ and hence $xD-p=xD$, this also means that the operator $xD$ \emph{is} compatible with the coarser filtration
\begin{equation*}
\Fil_{p\bullet} V = (\dots\rightarrow \Fil_{-p} V\rightarrow \Fil_0 V\rightarrow\Fil_p V\rightarrow\dots)\;.
\end{equation*}
Under the above description of $\D(\Z_{p, \HT, c}^\N)$, the Breuil--Kisin twist $\O\{n\}$ corresponds to the filtered complex $\Fil_\bullet V$ with $\Fil_\bullet V=\F_p$ if $\bullet\geq -n$ and zero else with the transition maps between nonzero terms being the identity; moreover, the operator $D: \Fil_i V\rightarrow\Fil_{i-1} V$ is given by multiplication by $i$ for all $i\in\Z$.}

If $E$ corresponds to a graded $\cal{A}_1$-module $M$, then restriction to $\Z_{p, \dR}^\N$ corresponds to passing to $(M[1/x^p]_{\deg=0}, \Theta=xD)$; alternatively, one may view this as passing from \EEDIT{$\Fil_{p\bullet} V$} to the underlying non-filtered complex and taking $\Theta=xD$. Restriction to $\Z_{p, \Hod}^\N$, however, corresponds to passing to $M/x^p$ and one can in fact show that $\D(\Z_{p, \Hod}^\N)\cong\D_{\gr, D-\nilp}(\cal{A}_1/x^p)$ as one can show that $\D(\Z_{p, \Hod}^\N)\cong\D_{\gr, D-\nilp}(\cal{A}_1/x^p)$ due to
\begin{equation*}
\Z_{p, \Hod}^\N\cong \G_a[F]/(\G_a^\sharp\rtimes\G_m)\subseteq\G_a/(\G_a^\sharp\rtimes\G_m)\cong\Z_{p, \HT, c}^\N\;,
\end{equation*}
where $\G_a[F]$ denotes the kernel of the Frobenius on $\G_a$; note that $x^p\in\cal{A}_1$ is central, so the notation $\cal{A}_1/x^p$ is unambiguous. \EEDIT{To reconcile this with the description of $\D(\Z_{p, \Hod}^\N)$ as $\D_{\gr, \Theta-\nilp}(\F_p[\Theta])$ from earlier, one has to observe that the map $\F_p[\Theta]\rightarrow \cal{A}_1/x^p$ given by $\Theta\mapsto D^p$ endows the target with the structure of a split Azumaya algebra over the source with splitting module $\cal{A}_1/\cal{A}_1x$ compatibly with the grading. In particular, there is an equivalence of categories
\begin{equation*}
\D_{\gr, D-\nilp}(\cal{A}_1/x^p)\xrightarrow{\cong} \D_{\gr, \Theta-\nilp}(\F_p[\Theta])\;, \hspace{0.3cm} N\mapsto N\tensor_{\cal{A}_1/x^p} \cal{A}_1/\cal{A}_1x
\end{equation*}
and thus, under the description $\D(\Z_{p, \Hod}^\N)\cong \D_{\gr, \Theta-\nilp}(\F_p[\Theta])$ from above, restriction to $\Z_{p, \Hod}^\N$ corresponds to passing to the associated graded $\gr_\bullet V$ of $\Fil_\bullet V$ and taking $\Theta=D^p$.}

Finally, we have
\begin{equation*}
R\Gamma(\Z_{p, \HT, c}^\N, E)=\fib(\Fil_0 V\xrightarrow{D} \Fil_{-1} V)\;.
\end{equation*}

\item For $\Z_{p, \dR, +}^\N\cong B_{\A^1/\G_m}\cal{G}$, one can show that any $E\in\D(\Z_{p, \dR, +}^\N)$ corresponds to a \EEDIT{decreasingly} filtered complex $\Fil^\bullet V\in\DF(\F_p)$ together with an operator $\Theta: \Fil^\bullet V\rightarrow\Fil^{\bullet-p} V$ such that the action of $\Theta^p-t^{p^2-p}\Theta$ on $\Rees(\Fil^\bullet V)$ is locally nilpotent in cohomology, see \cite[Prop. 6.5.6]{FGauges}. Under this correspondence, the Breuil--Kisin twist $\O\{n\}$ corresponds to the filtered complex $\Fil^\bullet V$ with $\Fil^\bullet V=\F_p$ if $\bullet\leq -n$ and zero else with the transition maps between nonzero terms being the identity and $\Theta=n$. For any $E\in\D(\Z_{p, \dR, +}^\N)$ corresponding to $(\Fil^\bullet V, \Theta)$, restriction to $\Z_{p, \dR}^\N$ corresponds to passage to the underlying non-filtered complex and restriction to $\Z_{p, \Hod}^\N$ corresponds to passage to the associated graded; moreover, we have
\begin{equation*}
R\Gamma(\Z_{p, \dR, +}^\N, E)=\fib(\Fil^0V\xrightarrow{\Theta} \Fil^{-p}V)\;.
\end{equation*}
\end{enumerate}

For $E\in\D(\Z_{p, \red}^\Syn)$, let us introduce the notation $F_\dR(E)$ for $R\Gamma(\Z_{p, \dR}^\N, E|_{\Z_{p, \dR}^\N})$ and define $F_\Hod(E), F_{\HT, c}(E)$ and $F_{\dR, +}(E)$ analogously. Then we have natural restriction maps
\begin{equation*}
\begin{split}
a_E: F_{\dR, +}(E)&\rightarrow F_\dR(E)\oplus F_\Hod(E)\;, \\
b_E: F_{\HT, c}(E)&\rightarrow F_\dR(E)\oplus F_\Hod(E)
\end{split}
\end{equation*}
\EEDIT{and} the gluing description of $\Z_{p, \red}^\Syn$ shows that we have 
\begin{equation}
\label{eq:syntomification-cohomologyreduced}
R\Gamma(\Z_{p, \red}^\Syn, E)=\fib(F_{\dR, +}(E)\oplus F_{\HT, c}(E)\xrightarrow{a_E-b_E} F_\dR(E)\oplus F_\Hod(E))\;.
\end{equation}

As we now have a good handle on cohomology on $\Z_{p, \red}^\Syn$, we want to make use of this to understand cohomology on $\Z_p^\Syn$.

\begin{defi}
\label{defi:syntomification-fil}
Let \EEDIT{$E\in\D((\Z_p^\Syn)_{p=0})$}. The filtration
\begin{equation*}
\begin{tikzcd}[column sep=scriptsize]
\dots\ar[r, "v_1"] & R\Gamma(\Z_p^\Syn, E\{-(p-1)\})\ar[r, "v_1"] & R\Gamma(\Z_p^\Syn, E)\ar[r, "v_1"] & R\Gamma(\Z_p^\Syn, E\{p-1\})\ar[r, "v_1"] & \dots  
\end{tikzcd}
\end{equation*}
is called the \emph{syntomic filtration} and denoted $\Fil_\bullet^\Syn R\Gamma(\Z_p^\Syn, E)[\frac{1}{v_1}]$; here, $E\{n\}$ denotes the tensor product $E\tensor\O\{n\}$ for any $n$. We denote the underlying unfiltered object by $R\Gamma(\Z_p^\Syn, E)[\frac{1}{v_1}]$.
\end{defi}

\EDIT{As one can show that $v_1$ is topologically nilpotent, the syntomic filtration is complete, see \cref{prop:syntomicetale-filtrationetale}. Moreover, notice} that, for any \EEDIT{$E\in\D((\Z_p^\Syn)_{p=0})$}, there is a canonical isomorphism
\begin{equation}
\label{eq:syntomification-grsyn}
\gr^\Syn_\bullet R\Gamma(\Z_p^\Syn, E)[\tfrac{1}{v_1}]\cong R\Gamma(\Z_{p, \red}^\Syn, E/v_1\{\bullet(p-1)\})\;,
\end{equation}
where $E/v_1\coloneqq\cofib(E\{-(p-1)\}\xrightarrow{v_1} E)$. Moreover, \EEDIT{if $E\in\Perf((\Z_p^\Syn)_{p=0})$}, one can check that there is a natural isomorphism
\begin{equation}
\label{eq:syntomification-synet}
R\Gamma(\Z_p^\Syn, E)[\tfrac{1}{v_1}]\cong R\Gamma(G_{\Q_p}, T_\et(E))\;,
\end{equation}
see \cite[Eq. (6.4.1)]{FGauges}.

\newpage

\section{The conjugate-filtered diffracted Hodge stack}
\label{sect:conjdhod}

\EDIT{In this section, we will develop a stacky formulation of diffracted Hodge cohomology in a similar spirit to what has been done in Section \ref{sect:stacks} for de Rham cohomology and prismatic cohomology.} Namely, we will introduce stacks $\pi_{X^\dHod}: X^\dHod\rightarrow\Spf\Z_p$ and $\pi_{X^\dHod, c}: X^{\dHod, c}\rightarrow\A^1/\G_m$ attached to any bounded $p$-adic formal scheme $X$ whose coherent cohomology computes \EDIT{the} diffracted Hodge cohomology of $X$ together with its conjugate filtration as introduced in \cite[§4.7]{APC} in good cases. While the diffracted Hodge stack $X^\dHod$ was already introduced by Bhatt--Lurie in \cite[Constr. 3.8]{PFS}, albeit without an explicit proof of the comparison with diffracted Hodge cohomology as stated below, its filtered refinement $X^{\dHod, c}$ does not seem to appear in the literature so far. More specifically, we are going to show the following:

\begin{thm}
\label{thm:fildhod-comparison}
Let $X$ be a bounded $p$-adic formal scheme and assume that $X$ is $p$-quasisyntomic and qcqs. Then $\pi_{X^\dHod, *}\O_{X^\dHod}$ identifies with $R\Gamma_\dHod(X)$ in $\widehat{\D}(\Z_p)$.
\end{thm}

\begin{thm}
\label{thm:fildhod-comparisonfiltered}
Let $X$ be a bounded $p$-adic formal scheme and assume that $X$ is $p$-quasisyntomic and qcqs. Then $\pi_{X^{\dHod, c}, *}\O_{X^{\dHod, c}}$ identifies with $\Fil_\bullet^\conj R\Gamma_\dHod(X)$ under the Rees equivalence. 
\end{thm}

Finally, we will also show that one can naturally incorporate the Sen operator $\Theta$ on diffracted Hodge cohomology into this picture in a way that encodes the divisibility properties of the Sen operator with respect to the conjugate filtration; \EDIT{namely, for any $n\in\Z$, the endomorphism 
\begin{equation*}
\Theta+n: \Fil^\conj_n R\Gamma_\dHod(X)\rightarrow\Fil^\conj_n R\Gamma_\dHod(X)
\end{equation*}
factors uniquely through $\Fil^\conj_{n-1} R\Gamma_\dHod(X)$, see \cite[Rem. 4.9.10]{APC}, and this will be reflected in the stacky picture. To this end,} first recall that, for the Hodge--Tate stack $\pi_{X^\HT}: X^\HT\rightarrow\Z_p^\HT$ \EEDIT{from \cref{ex:prismatisation-htstack}, the complex} $\pi_{X^\HT, *}\O_{X^\HT}$ identifies with $R\Gamma_\dHod(X)$ equipped with its Sen operator essentially by definition. The stack $X^{\HT, c}\coloneqq (X^\N)_{t=0}$ provides a filtered refinement of this picture in the following way:

\begin{thm}
\label{thm:fildhod-comparisonsen}
Let $X$ be a bounded $p$-adic formal scheme and assume that $X$ is $p$-quasisyntomic and qcqs. Under the equivalence
\begin{equation*}
\D(\Z_p^{\HT, c})\cong \widehat{\D}_{\gr, D-\nilp}(\Z_p\{x, D\}/(Dx-xD-1))
\end{equation*}
from \cref{lem:fildhod-complexeszpntzero}, the underlying graded \EEDIT{$\Z_p[x]$-complex} of $\pi_{X^{\HT, c}, *}\O_{X^{\HT, c}}$ identifies \EDIT{with} $\Fil_\bullet^\conj R\Gamma_\dHod(X)$ under the Rees equivalence \EEDIT{and, under this identification, the operator 
\begin{equation*}
xD-i: \Fil_i^\conj R\Gamma_\dHod(X)\rightarrow \Fil_i^\conj R\Gamma_\dHod(X)
\end{equation*}
identifies with the Sen operator on $\Fil_i^\conj R\Gamma_\dHod(X)$ for all $i\in\Z$.} Here, we write $\pi_{X^{\HT, c}}: X^{\HT, c}\rightarrow\Z_p^{\HT, c}$ for the map induced by functoriality.
\end{thm}

\begin{rem}
Despite the notation, coherent cohomology on $X^{\HT, c}$ does \emph{not} compute conjugate-filtered Hodge--Tate cohomology in the sense of \cite[Constr. 4.5.5]{APC} and we also do not see a way to endow $X^{\HT, c}$ with a map to $\A^1/\G_m$. Nevertheless, it seems as though there is some relation between $X^{\HT, c}$ and \EDIT{the} conjugate-filtered Hodge--Tate cohomology of $X$, which is why we still use the notation $X^{\HT, c}$. More specifically, as $R\Gamma(X^{\HT, c}, \O_{X^\HT, c})$ identifies \EEDIT{with} $\gr^0_\N R\Gamma_\prism(X)$, at least in favourable cases, see \cref{lem:nygaardhodge-rgammazpntzero}, the comparison 
\begin{equation*}
\EDIT{
R\Gamma(\Z_p^\prism, \gr^m_\N F^*\cal{H}_\prism(X))\cong \Fil^\conj_m R\Gamma_{\ol{\prism}}(X)\{m\}
}
\end{equation*}
from \cite[Rem. 5.7.2]{APC}, \EDIT{where $F: \Z_p^\prism\rightarrow\Z_p^\prism$ is the Frobenius from \cref{ex:prismatisation-frob} and $\gr^\bullet_\N F^*\cal{H}_\prism(X)$ is the associated graded of the Nygaard filtration on $F^*\cal{H}_\prism(X)$ as defined in \cite[Not. 5.5.1]{APC},} suggests that $X^{\HT, c}$ might differ from the ``correct'' definition of a conjugate-filtered Hodge--Tate stack by a Frobenius twist in some sense. It would be interesting to construct a stack $X^{\HT, c, '}$ equipped with a map $\pi_{X^{\HT, c, '}}: X^{\HT, c, '}\rightarrow\A^1/\G_m$ such that $\pi_{X^{\HT, c, '}, *}\O_{X^{\HT, c, '}}$ indeed identifies with $\Fil^\conj_\bullet R\Gamma_{\ol{\prism}}(X)$.
\end{rem}

\subsection{Recollections on diffracted Hodge cohomology}
\label{subsect:review-dhod}

We start by briefly recalling the definition of diffracted Hodge cohomology from \cite[§4.7]{APC}. To do this, first recall that, for a prism $(A, I)$, we may extend relative Hodge--Tate cohomology as a functor $R\mapsto R\Gamma_{\ol{\Prism}}(R/A)$ on $\ol{A}$-algebras \EEDIT{$R$} to all animated $\ol{A}$-algebras via left Kan extension as in \cite[Constr. 4.1.3]{APC}. Moreover, the conjugate filtration $\Fil_\bullet^\conj R\Gamma_{\ol{\Prism}}(R/A)$ \EEDIT{is an increasing filtration} on $R\Gamma_{\ol{\Prism}}(R/A)$ \EEDIT{and can be} obtained by left Kan extension from the full subcategory of smooth $\ol{A}$-algebras, where it is just defined to be the canonical filtration on $R\Gamma_{\ol{\Prism}}(R/A)$, see \cite[Rem. 4.1.7]{APC}.

Now recall that, analogously to the description of $\D(\Z_p^\prism)$ from \cref{prop:prismatisation-crystals}, we have
\begin{equation*}
\D(\Z_p^\HT)\cong \lim_{(A, I)\in (\Z_p)_\prism} \widehat{\D}(A/I)\;,
\end{equation*}
\EDIT{see \cite[Rem. 3.5.3]{APC}.} Thus, for any ring $R$, the association $(A, I)\mapsto \Fil_\bullet^\conj R\Gamma_{\ol{\Prism}}(\ol{A}\tensorL R/A)$ defines a filtered quasi-coherent complex $\Fil_\bullet^\conj\H_{\ol{\prism}}(R)$ on $\Z_p^\HT$, whose underlying unfiltered object we denote by $\H_{\ol{\prism}}(R)$. \EDIT{Essentially} by definition, we have $\H_{\ol{\prism}}(R)=\H_\prism(R)|_{\Z_p^\HT}$ and the Hodge--Tate comparison implies that there is an isomorphism
\begin{equation}
\label{eq:fildhod-htcomparison1}
\gr_n^\conj\H_{\ol{\prism}}(R)\cong L\widehat{\Omega}_R^n\tensor\O_{\Z_p^\HT}\{-n\}[-n]
\end{equation}
for any $n$.

Now recall from \cite[Thm. 3.4.13]{APC} that $\Z_p^\HT\cong B\G_m^\sharp$ and that the category $\D(B\G_m^\sharp)$ admits the following straightfoward linear algebraic description:

\begin{prop}
\label{prop:fildhod-bgmsharp}
There is an equivalence of categories
\begin{equation*}
\D(B\G_m^\sharp)\cong \widehat{\D}_{(\Theta^p-\Theta)-\nilp}(\Z_p[\Theta])
\end{equation*}
induced by pullback along the quotient map $\Spf\Z_p\rightarrow B\G_m^\sharp$. Here, we demand the \EDIT{local} nilpotence of $\Theta^p-\Theta$ only mod $p$; \EEDIT{in other words, the right-hand side denotes the full subcategory of all $M\in\widehat{\D}(\Z_p[\Theta])$ for which $\Theta^p-\Theta$ acts locally nilpotently on $H^i(M\tensorL_{\Z_p} \F_p)$ for all $i\in\Z$.}
\end{prop}
\begin{proof}
See \cite[Thm. 3.5.8]{APC}.
\end{proof}

Thus, we may identify $\Fil^\conj_\bullet \H_{\ol{\prism}}(R)$ with a filtered $p$-complete $\Z_p$-complex $\Fil^\conj_\bullet\widehat{\Omega}_R^\dHod$ equipped with an operator $\Theta: \Fil^\conj_\bullet\widehat{\Omega}_R^\dHod\rightarrow\Fil^\conj_\bullet\widehat{\Omega}_R^\dHod$ such that $\Theta^p-\Theta$ is \EEDIT{locally} nilpotent \EEDIT{in cohomology} mod $p$, \EDIT{which we call the Sen operator}; the underlying unfiltered object of $\Fil^\conj_\bullet\widehat{\Omega}_R^\dHod$ is denoted $\widehat{\Omega}_R^\dHod$. As the $n$-th power of the Breuil-Kisin twist identifies with the $\Z_p$-module $\Z_p$ with $\Theta$ acting by multiplication by $n$ under the equivalence from \cref{prop:fildhod-bgmsharp}, by virtue of (\ref{eq:fildhod-htcomparison1}), we have isomorphisms
\begin{equation}
\label{eq:fildhod-htcomparison2}
\gr_n^\conj\widehat{\Omega}_R^\dHod\cong L\widehat{\Omega}^n_R[-n]
\end{equation}
for all $n$ \EDIT{and the Sen operator acts by multiplication by $-n$ on $\gr_n^\conj\widehat{\Omega}_R^\dHod$}. 

\begin{rem}
\label{rem:fildhod-siftedcolims}
\EDIT{Using (\ref{eq:fildhod-htcomparison2}) and the fact that} the functor $R\mapsto L\widehat{\Omega}^n_R[-n]$ from commutative rings to \EDIT{$\widehat{\D}(\Z_p)$} commutes with sifted colimits, we conclude by induction that the functor $R\mapsto \Fil^\conj_\bullet\widehat{\Omega}^\dHod_R$ commutes with sifted colimits as well \EDIT{(note that $\Fil^\conj_\bullet\widehat{\Omega}^\dHod_R=0$ for $\bullet<0$).}
\end{rem}

Finally, as the construction $R\mapsto\Fil^\conj_\bullet\widehat{\Omega}_R^\dHod$ is compatible with $p$-complete étale localisation by \cite[Prop. 4.7.11]{APC}, the constructions above glue in the following sense: For any bounded $p$-adic formal scheme $X$, we obtain a filtered quasi-coherent complex $\Fil^\conj_\bullet\widehat{\Omega}_X^\dHod$ on $X$ with the property that
\begin{equation*}
\Fil^\conj_\bullet\widehat{\Omega}_X^\dHod|_{\Spf R}=\Fil^\conj_\bullet\widehat{\Omega}_R^\dHod
\end{equation*}
for any affine open $\Spf R\rightarrow X$. Again, we denote the underlying unfiltered object by $\widehat{\Omega}_X^\dHod$. Moreover, also the corresponding Sen operators glue to give a Sen operator $\Theta$ on $\Fil^\conj_\bullet\widehat{\Omega}_X^\dHod$. This finally leads to the following definition:

\begin{defi}
\label{defi:fildhod-defdhod}
Let $X$ be a bounded $p$-adic formal scheme. The \emph{diffracted Hodge cohomology} of $X$ is defined as
\begin{equation*}
R\Gamma_\dHod(X)\coloneqq R\Gamma(X, \widehat{\Omega}_X^\dHod)\;.
\end{equation*}
It is equipped with a \emph{conjugate filtration} given by 
\begin{equation*}
\Fil^\conj_\bullet R\Gamma_\dHod(X)\coloneqq R\Gamma(X, \Fil^\conj_\bullet\widehat{\Omega}_X^\dHod)
\end{equation*}
and a \emph{Sen operator} $\Theta: \Fil^\conj_\bullet R\Gamma_\dHod(X)\rightarrow\Fil^\conj_\bullet R\Gamma_\dHod(X)$ induced by the Sen operator on $\Fil^\conj_\bullet\widehat{\Omega}_X^\dHod$.
\end{defi}

\subsection{The diffracted Hodge stack}

We now examine the stack $X^\dHod$ as introduced by Bhatt--Lurie in \cite[Constr. 3.8]{PFS} \EDIT{and relate it to \cref{defi:fildhod-defdhod}}. For the rest of this section, let $X$ be a bounded $p$-adic formal scheme.

\begin{defi}
\label{defi:fildhod-dhod}
For a bounded $p$-adic formal scheme $X$, its \emph{diffracted Hodge stack} $X^\dHod$ is defined as the pullback
\begin{equation*}
\begin{tikzcd}
X^\dHod\ar[r, "i_\dHod"]\ar[d, "\pi_{X^\dHod}"] & X^\HT\ar[d] \\
\Spf\Z_p\ar[r, "\eta"] & \Z_p^\HT\nospacepunct{\;,}
\end{tikzcd}
\end{equation*}
where the map $\eta: \Spf\Z_p\rightarrow\Z_p^\HT$ identifies with the quotient map under the isomorphism $\Z_p^\HT\cong B\G_m^\sharp$.
\end{defi}

\begin{rem}
Despite the notation, the map $i_\dHod$ is generally \emph{not} a closed immersion. Moreover, recalling from \cite[Thm. 3.4.13]{APC} that the map $\eta: \Spf\Z_p\rightarrow \Z_p^\HT$ corresponds to assigning to any $p$-nilpotent ring $S$ the Cartier--Witt divisor $W(S)\xrightarrow{V(1)} W(S)$, we immediately see that
\begin{equation*}
X^\dHod(S)=\Map(\Spec(W(S)/V(1)), X)\;,
\end{equation*}
where the mapping space is computed in derived algebraic geometry and the quotient $W(S)/V(1)$ is to be taken in the derived sense. 
\end{rem}

We can now prove the first theorem announced in the beginning \EDIT{of this section}:

\begin{proof}[Proof of \cref{thm:fildhod-comparison}]
As the constructions $X\mapsto X^\dHod$ and $X\mapsto \widehat{\Omega}_X^\dHod$ are compatible with Zariski localisation by \cite[Rem. 3.9]{PFS} and \cite[Prop. 4.7.11]{APC}, we may assume that $X=\Spf R$ is affine. As $\pi_{R^\prism, *}\O_{R^\prism}=\H_\prism(R)$ by \cite[Cor.s 8.13, 8.17]{PFS} and $\H_{\ol{\prism}}(R)=\H_\prism(R)|_{\Z_p^\HT}$, we conclude that $\pi_{R^\HT, *}\O_{R^\HT}=\H_{\ol{\prism}}(R)$. Now the result follows from base change for the cartesian square
\begin{equation*}
\EDIT{
\raisebox{\depth}{
\begin{tikzcd}[baseline={(current bounding box.center)}, ampersand replacement=\&]
R^\dHod\ar[r, "i_{\dHod}"]\ar[d, "\pi_{R^\dHod}"] \& R^\HT\ar[d, "\pi_{R^\HT}"] \\
\Z_p^\dHod\ar[r] \& \Z_p^\HT\nospacepunct{\;.}
\end{tikzcd}
}
}
\qedhere
\end{equation*}
\end{proof}

In the sequel, we will also need a more concrete description of $X^\dHod$ in the case where $X=\Spf R$ for $R$ quasiregular semiperfectoid. 

\begin{rem}
\label{rem:fildhod-degzero}
\EEDIT{For $R$ quasiregular semiperfectoid,} the complex $\widehat{\Omega}_R^\dHod$ is concentrated in degree zero. Indeed, the conjugate filtration is complete and its graded pieces $L\widehat{\Omega}^n_R[-n]$ are all concentrated in degree zero, \EEDIT{see \cite[Rem. 4.21]{THHandPAdicHodgeTheory}}, \EEDIT{hence $\Fil^\conj_n \widehat{\Omega}_R^\dHod$ is concentrated in degree zero for any $n\in\Z$ and we conclude that the same is true for $\widehat{\Omega}_R^\dHod$.}
\end{rem}

\begin{prop}
\label{prop:fildhod-dhodqrsp}
Assume that $X=\Spf R$ for a quasiregular semiperfectoid ring $R$. Then there is an isomorphism $R^\dHod\cong \Spf\widehat{\Omega}^\dHod_R$.
\end{prop}
\begin{proof}
Let $(\Z_p[[\tilde{p}]], (\tilde{p}))$ denote the prism from \cite[Not. 3.8.9]{APC}. Then the map $\eta$ identifies with the base change of $\rho_{\Z_p[[\tilde{p}]]}: \Spf\Z_p[[\tilde{p}]]\rightarrow\Z_p^\prism$ to $\Z_p^\HT$ by \cite[Prop. 3.8.12]{APC} and hence $R^\dHod$ identifies with the relative \EDIT{Hodge--Tate stack} of $R$ with respect to $(\Z_p[[\tilde{p}]], (\tilde{p}))$ as defined in \cite[Var. 5.1]{PFS}. \EDIT{Thus, the stack $R^\dHod$ is actually an affine formal scheme by \cite[Cor. 7.18]{PFS} and the claim follows from \cref{thm:fildhod-comparison}.}
\end{proof}

\subsection{Incorporating the conjugate filtration}

We now construct the conjugate-filtered diffracted Hodge stack $X^{\dHod, c}$. For this, we first make some preliminary remarks. Namely, recall that, by the proof of \cite[Prop. 5.3.7]{FGauges}, there is an fpqc cover $\A^1/\G_m\rightarrow (\Z_p^\N)_{t=0}$, where we now take $\A^1/\G_m$ to be $\Spf\Z_p\langle u\rangle/\G_m$ with $u$ having degree \EEDIT{$1$} and consequently label the tautological section on $\A^1/\G_m$ by \EEDIT{$u: \O\rightarrow\O(1)$}. This cover is given by a filtered Cartier--Witt divisor $M\xrightarrow{d} W$ on $\A^1/\G_m$ constructed by virtue of the commutative diagram
\begin{equation*}
\EEDIT{
\begin{tikzcd}[ampersand replacement=\&]
0\ar[r] \& \G_a^\sharp\ar[r]\ar[d, "u^\sharp", swap] \& W\ar[r]\ar[d]\ar[dd, "V(1)" {yshift=20pt, xshift=5pt}, bend right=45, swap] \& F_*W\ar[r]\ar[d, equals] \& 0 \\
0\ar[r] \& \V(\O(1))^\sharp\ar[r]\ar[d, "0", swap] \& M\ar[r]\ar[d, "d", swap] \& F_*W\ar[r]\ar[d, "p"] \& 0 \\
0\ar[r] \& \G_a^\sharp\ar[r] \& W\ar[r] \& F_*W\ar[r] \& 0\nospacepunct{\;;}
\end{tikzcd}
}
\end{equation*}
here, $M$ is defined as the pushout of the upper left square and the maps out of $M$ are induced by the other maps in the diagram using the universal property.

\begin{defi}
\label{defi:fildhod-dhodconj}
In the situation above, we obtain a 1-truncated animated $W$-algebra stack
\begin{equation*}
\G_a^{\dHod, c}\coloneqq\Cone(M\xrightarrow{d} W)
\end{equation*}
over $\A^1/\G_m$. The \emph{conjugate-filtered diffracted Hodge stack} $X^{\dHod, c}$ of $X$ is the stack over $\A^1/\G_m$ defined by
\begin{equation*}
X^{\dHod, c}(\Spec S\rightarrow\A^1/\G_m)\coloneqq \Map(\Spec\G_a^{\dHod, c}(S), X)\;,
\end{equation*}
where the mapping space is computed in derived algebraic geometry. If $X=\Spf R$ is affine, we also write $R^{\dHod, c}$ in place of $X^{\dHod, c}$.
\end{defi}

\begin{rem}
\label{rem:fildhod-pullback}
It is immediate from the definition above that one may alternatively describe $X^{\dHod, c}$ as a pullback
\begin{equation*}
\begin{tikzcd}
X^{\dHod, c}\ar[r, "i_{\dHod, c}"]\ar[d, "\pi_{X^{\dHod, c}}"] & (X^\N)_{t=0}\ar[d] \\
\A^1/\G_m\ar[r] & (\Z_p^\N)_{t=0}\nospacepunct{\;,}
\end{tikzcd}
\end{equation*}
\EEDIT{where the bottom map is given by the filtered Cartier--Witt divisor on $\A^1/\G_m$ constructed above.} Also note that the preimage of $\G_m/\G_m\subseteq\A^1/\G_m$ under $\pi_{X^{\dHod, c}}$ recovers $X^\dHod$: Indeed, in this case, $u$ is an isomorphism and hence $M\xrightarrow{d} W$ identifies with $W\xrightarrow{V(1)} W$. Moreover, the preimage of $B\G_m\subseteq\A^1/\G_m$ under $\pi_{X^{\dHod, c}}$ recovers $X^\Hod$: This follows from the fact that $u=0$ in this case and hence $M\xrightarrow{d} W$ identifies with \EEDIT{$\V(\O(1))^\sharp\oplus F_*W\xrightarrow{(0, V)} W$}.
\end{rem}

To prove \cref{thm:fildhod-comparisonfiltered}, \EDIT{we first have to show the corresponding assertion in the special case where the $p$-adic formal scheme $X$ is smooth:}

\begin{lem}
\label{lem:fildhod-comparisonfilteredsmooth}
Let $X$ be a \EEDIT{smooth qcqs} $p$-adic formal scheme. Then $\pi_{X^{\dHod, c}, *}\O_{X^{\dHod, c}}$ identifies with $\Fil_\bullet^\conj R\Gamma_\dHod(X)$ under the Rees equivalence. 
\end{lem}
\begin{proof}
Similarly to the proof of \cref{thm:fildhod-comparison}, we may reduce to the case where $X=\Spf R$ is affine. As the underlying unfiltered object of $\pi_{X^{\dHod, c}, *}\O_{X^{\dHod, c}}$ identifies with $\widehat{\Omega}_R^\dHod$ by \cref{thm:fildhod-comparison} and \EDIT{$\Fil^\conj_\bullet\widehat{\Omega}_R^\dHod$ is the canonical filtration on $\widehat{\Omega}_R^\dHod$ in this case, all we have to show is that $\pi_{X^{\dHod, c}, *}\O_{X^{\dHod, c}}$ identifies with the canonical filtration on its underlying unfiltered object. Hence, by \cite[Ex. 2.2.3]{FGauges},} it suffices to show that $\pi_{X^{\dHod, c}, *}\O_{X^{\dHod, c}}$ is derived $u$-complete and that its $n$-th graded piece is concentrated in degree $n$ for any $n$. For the claim about the graded pieces, note that the pullback of $\pi_{X^{\dHod, c}, *}\O_{X^{\dHod, c}}$ to $B\G_m$ identifies with the Hodge cohomology $\bigoplus_n L\widehat{\Omega}^n_R[-n]$ of $X$ by virtue of \cref{thm:fildrstack-comparison}. To prove $u$-completeness, we use the strategy from the proof of \cite[Thm. 2.7.9]{FGauges}: First observe that $\pi_0(\G_a^{\dHod, c})\cong\G_a$ and \EEDIT{$\pi_1(\G_a^{\dHod, c})\cong\V(\O(1))^\sharp$}. Thus, the projection $\G_a^{\dHod, c}\rightarrow\G_a$ is a square-zero extension of the target by \EEDIT{$B\V(\O(1))^\sharp$}
over $\A^1/\G_m$ and hence the induced map
\begin{equation*}
\nu: X^{\dHod, c}\rightarrow X\times\A^1/\G_m
\end{equation*}
of stacks over $\A^1/\G_m$ is a \EEDIT{$\V(\T_{X/\Z_p}(1))^\sharp$-gerbe}, \EDIT{where $\T_{X/\Z_p}$ denotes the tangent sheaf of $X$ over $\Z_p$: Indeed, for any $p$-nilpotent test ring $S$ equipped with a map to $\A^1/\G_m$, the animated algebras $\G_a^{\dHod, c}(S)$ and $\G_a(S)$ are also $p$-nilpotent and hence we find some $n\geq 0$ such that
\begin{equation*}
X^{\dHod, c}(S)=\Map(\Spec \G_a^{\dHod, c}(S), X)=\Map(\Spec \G_a^{\dHod, c}(S), X_{p^n=0})
\end{equation*}
and similarly $(X\times\A^1/\G_m)(S)=\Map(\Spec \G_a(S), X_{p^n=0})$; now the claim follows by derived deformation theory.}\footnote{More precisely, we are using the following assertion in derived algebraic geometry: Let $X$ be a finite type $\Z_p$-scheme and $A'\rightarrow A$ a square-zero extension of animated $\Z_p$-algebras with fibre $N\in\D^{\leq 0}(A)$. Then the fibre of the map $X(A')\rightarrow X(A)$ over a point $\eta: \Spec A\rightarrow X$ is a torsor for $\operatorname{Der}_{\Z_p}(\O_X, \eta_*N)\cong\Map_A(\eta^*L_{X/\Z_p}, N)$; \EDIT{the proof is similar to \cite[Thm. 5.1.13]{FlatPurity}.} Note that, if furthermore $N=L[1]\in\D^{\leq -1}(A)$ and $X$ is smooth, we have $\Map_A(\eta^*L_{X/\Z_p}, N)\cong B(\eta^*\T_{X/\Z_p}\tensor_A L)$.} Finally, as the relative cohomology sheaves $H^i(\nu_*\O_{X^\dHod, c})$ of such a gerbe are independent of the gerbe structure by the proof of \cite[Cor. 2.7.2.(1)]{FGauges} and it is enough to check that each of them is $u$-complete, we may replace $\nu$ with the trivial gerbe \EEDIT{$\nu': B_{X\times\A^1/\G_m}\V(\T_{X/\Z_p}(1))^\sharp\rightarrow X\times\A^1/\G_m$}. \EEDIT{Now the conclusion follows as the $\G_m$-equivariant version of \cref{lem:syntomicetale-vesharpreps} implies that the pushforward of the structure sheaf along $\nu'$ is given by the quasi-coherent complex 
\begin{equation*}
\EEDIT{
\Tot(\O_X\rightarrow \widehat{\Omega}^1_{X/\Z_p}(-1)\rightarrow \widehat{\Omega}^2_{X/\Z_p}(-2)\rightarrow\dots)
}
\end{equation*}
on $X\times \A^1/\G_m$, which is clearly $u$-complete.}
\end{proof}

\EDIT{From the smooth case of \cref{thm:fildhod-comparisonfiltered} treated in the previous lemma, we can now infer an explicit description of the conjugate-filtered diffracted Hodge stack of a quasiregular semiperfectoid ring.}

\begin{thm}
\label{thm:fildhod-dhodconjqrsp}
Assume that $X=\Spf R$ for a quasiregular semiperfectoid ring $R$. Then there is an isomorphism 
\begin{equation*}
R^{\dHod, c}\cong \Spf\Rees(\Fil^\conj_\bullet\widehat{\Omega}^\dHod_R)/\G_m\;.
\end{equation*}
\end{thm}
\begin{proof}
We mimic the proof of \cite[Thm. 5.5.10]{FGauges}. As in loc.\ cit., we will work with derived stacks, i.e.\ stacks on animated rings, since we will have to ponder derived pullbacks, which are more natural in the world of derived stacks; however, it will turn out in the end that all stacks we consider were classical to begin with. For this, we extend all the stacks we have introduced so far to the category of animated rings by left Kan extension as in \cite[Rem. 5.5.13]{FGauges}. To begin the proof, observe that we can define the conjugate filtration on diffracted Hodge cohomology of $p$-complete animated rings via left Kan extension from the full subcategory of smooth $\Z_p$-algebras since it commutes with sifted colimits \EDIT{by \cref{rem:fildhod-siftedcolims}}. Thus, by \cref{lem:fildhod-comparisonfilteredsmooth}, there is a natural map
\begin{equation*}
\Rees(\Fil^\conj_\bullet\widehat{\Omega}^\dHod_R)\rightarrow \pi_{R^{\dHod, c}, *}\O_{R^{\dHod, c}}
\end{equation*}
of commutative algebra objects in $\D(\A^1/\G_m)$. \EEDIT{Recalling that the left-hand side is concentrated in degree zero by \cref{rem:fildhod-degzero}, we see that, by adjunction, the above defines a morphism}
\begin{equation*}
R^{\dHod, c}\rightarrow \Spf\Rees(\Fil^\conj_\bullet\widehat{\Omega}^\dHod_R)/\G_m
\end{equation*}
over $\A^1/\G_m$. To check that this is an isomorphism, we may pull back to $\G_m/\G_m\subseteq\A^1/\G_m$ and $B\G_m\subseteq\A^1/\G_m$ as these form a stratification of $\A^1/\G_m$.\footnote{\EEDIT{Here we are using that, given any animated ring $A$ and an element $f\in\pi_0(A)$, a map $\tau: M\rightarrow N$ in $\D(A)$ is an isomorphism if and only if both $\tau[1/f]$ and $\tau\tensorL_A A/f$ are isomorphisms.}} However, after pulling back to $\G_m/\G_m$, the above morphism becomes
\begin{equation*}
R^\dHod\rightarrow \Spf\widehat{\Omega}^\dHod_R\;,
\end{equation*}
which can be verified to agree with the isomorphism from \cref{prop:fildhod-dhodqrsp}, and after pulling back to $B\G_m$, we obtain
\begin{equation*}
R^\Hod\rightarrow \Spf\left(\bigoplus_n L\widehat{\Omega}^n_R[-n]\right)/\G_m\;,
\end{equation*}
which is also an isomorphism, see the proof of \cite[Thm. 5.5.10]{FGauges}.
\end{proof}

\EDIT{Finally, we have made all the necessary preparations to prove \cref{thm:fildhod-comparisonfiltered} as announced in the beginning of the section.

\begin{proof}[Proof of \cref{thm:fildhod-comparisonfiltered}]
First observe that, by base change, the analogue of \cref{lem:filprism-quasisyntomiccover} also holds true for the conjugate-filtered diffracted Hodge stack. Thus, by quasisyntomic descent for $\Fil^\conj_\bullet R\Gamma_\dHod(X)$, see \cite[Rem. 4.7.9]{APC}, we are reduced to the case where $X=\Spf R$ for a quasiregular semiperfectoid ring $R$. However, in this case, the claim follows from \cref{thm:fildhod-dhodconjqrsp}.
\end{proof}
}

\begin{rem}
Observe that the mod $p$ reduction $(X^{\dHod, c})_{p=0}$ of the conjugate-filtered diffracted Hodge stack agrees with the conjugate-filtered de Rham stack of $X_{p=0}$ defined in \cite[§2.7]{FGauges} up to a Frobenius twist. This reflects the fact that, for smooth $X$, the isomorphism
\begin{equation*}
\widehat{\Omega}_{X^{(1)}}^\dHod\tensor_{\Z_p} \F_p\cong F_*\widehat{\Omega}_{X_{p=0}/\F_p}^\bullet\;,
\end{equation*}
coming from the crystalline comparison theorem for prismatic cohomology, see \cite[Rem. 4.7.18]{FGauges}, is compatible with the respective conjugate filtrations on either side. Here, $X^{(1)}\coloneqq \phi^* X$ denotes the pullback of $X$ along the Frobenius of $\Z_p$ and $F: X_{p=0}\rightarrow X_{p=0}^{(1)}$ is the relative Frobenius for $X_{p=0}$ over $\F_p$.
\end{rem}

\subsection{The Sen operator}

Finally, \EDIT{we construct the stack $X^{\HT, c}$ as follows}:

\begin{defi}
For a bounded $p$-adic formal scheme $X$, its \emph{conjugate-filtered Hodge--Tate stack} $X^{\HT, c}$ is defined as the pullback
\begin{equation*}
\begin{tikzcd}
X^{\HT, c}\ar[r, "i_{\HT, c}"]\ar[d, "\pi_{X^{\HT, c}}"] & X^\N\ar[d] \\
(\Z_p^\N)_{t=0}\ar[r] & \Z_p^\N\nospacepunct{\;.}
\end{tikzcd}
\end{equation*}
In other words, the stack $X^{\HT, c}$ is just the locus inside $X^\N$ cut out by the equation $t=0$. 
\end{defi}

To begin our analysis \EDIT{of the stack $X^{\HT, c}$ and to relate it to the Sen operator on the conjugate-filtered diffracted Hodge cohomology of $X$}, we first study the stack \EDIT{$\Z_p^{\HT, c}=(\Z_p^\N)_{t=0}$} and its category of quasi-coherent complexes.

\begin{prop}
\label{prop:fildhod-zpntzero}
There is a natural identification
\begin{equation*}
(\Z_p^\N)_{t=0}\cong \G_a^\dR/\G_m=\G_a/(\G_a^\sharp\rtimes\G_m)\;.
\end{equation*}
\end{prop}
\begin{proof}
See \cite[Prop. 5.3.7]{FGauges}.
\end{proof}

\begin{lem}
\label{lem:fildhod-complexeszpntzero}
There is an equivalence of categories
\begin{equation*}
\D((\Z_p^\N)_{t=0})\cong \widehat{\D}_{\gr, D-\nilp}(\Z_p\{x, D\}/(Dx-xD-1))\;,
\end{equation*}
where $x$ and $D$ have grading degree \EEDIT{$1$} and \EEDIT{$-1$}, respectively, and the \EDIT{local} nilpotence of $D$ is only demanded mod $p$ here.
\end{lem}
\begin{proof}[Proof]
Using \cref{prop:fildhod-zpntzero} to identify $(\Z_p^\N)_{t=0}$ as $\G_a^\dR/\G_m$, the statement is completely analogous to the description of $\D(\Z_{p, \HT, c}^\N)$ from our discussion of the reduced locus of $\Z_p^\Syn$. Nevertheless, we briefly \EDIT{lay out} the argument and first show that
\begin{equation}
\label{eq:fildhod-complexeszpntzero}
\D(\G_a^\dR)=\D(\G_a/\G_a^\sharp)\cong \D_{D-\nilp}(\Z_p\{x, D\}/(Dx-xD-1))\;;
\end{equation}
then the claim will follow as specifying a $\G_m$-equivariant structure \EEDIT{on a quasi-coherent complex on $\G_a^\dR$} then amounts to specifying a grading \EEDIT{on its image under the equivalence above which is} compatible with the \EEDIT{grading on} $\Z_p\{x, D\}/(Dx-xD-1)$.

To prove (\ref{eq:fildhod-complexeszpntzero}), first recall that the Cartier dual of $\G_a^\sharp$ is the formal completion $\widehat{\G}_a$ of $\G_a$ at the origin, \EDIT{i.e.\ $\widehat{\G}_a\cong\sHom(\G_a^\sharp, \G_m)$}, see also \cref{lem:syntomicetale-vesharpreps}, and that an endomorphism $D: M\rightarrow M$ of a $p$-complete $\Z_p$-module $M$ which is locally nilpotent mod $p$ corresponds to the $\O(\G_a^\sharp)$-comodule structure on $M$ given by
\begin{equation*}
M\rightarrow M\widehat{\tensor}_{\Z_p} \O(\G_a^\sharp)\;, \hspace{0.5cm} m\mapsto \sum_{i\geq 0} D^i(m)\tensor\frac{t^i}{i!}\;,
\end{equation*}
see \cite[Prop. 2.4.4]{FGauges}; here, $\O(\G_a^\sharp)$ denotes the $p$-adic completion of $\Z_p[t, \tfrac{t^2}{2!}, \tfrac{t^3}{3!}, \dots]$. Thus, equipping a $p$-complete $\Z_p\langle x\rangle$-module $M$ with an $\O(\G_a^\sharp)$-equivariant structure means specifiying an endomorphism $D: M\rightarrow M$ which is locally nilpotent mod $p$ such that the diagram
\begin{equation*}
\begin{tikzcd}
M\widehat{\tensor}_{\Z_p} \Z_p\langle x\rangle\ar[r]\ar[d] & M\ar[r] & M\widehat{\tensor}_{\Z_p} \O(\G_a^\sharp) \\
(M\widehat{\tensor}_{\Z_p} \O(\G_a^\sharp))\widehat{\tensor}_{\Z_p} (\Z_p\langle x\rangle\widehat{\tensor}_{\Z_p} \O(\G_a^\sharp))\ar[rr, "\cong"] && (M\widehat{\tensor}_{\Z_p} \Z_p\langle x\rangle)\widehat{\tensor}_{\Z_p} (\O(\G_a^\sharp)\widehat{\tensor}_{\Z_p} \O(\G_a^\sharp))\ar[u]
\end{tikzcd}
\end{equation*}
commutes. Now observing that the endomorphism of $\Z_p\langle x\rangle$ corresponding to the action of $\G_a^\sharp$ on $\G_a$ is given by the usual derivative $\frac{\d}{\d x}: \Z_p\langle x\rangle\rightarrow\Z_p\langle x\rangle$, we see that the condition above translates to
\begin{equation*}
D^i(fm)=\sum_{k+\ell=i} \binom{i}{k}\left(\frac{\d^\ell}{\d^\ell x} f\right) D^k(m)
\end{equation*}
for any $i\geq 0$ and $f\in\Z_p\langle x\rangle, m\in M$. \EDIT{As this is clearly an additive condition, it suffices to demand this for monomials $f=x^n$, where $n\geq 0$. However, a straightforward induction on $n$ then shows} that the relation $D(xm)=xD(m)+m$ already implies all the others, hence specifying a $\G_a^\sharp$-equivariant structure on $M$ amounts to specifying a lift of $M$ to a $\Z_p\{x, D\}/(Dx-xD-1)$-module such that $D$ acts locally nilpotently on $M$ mod $p$ and we are done.
\end{proof}

\begin{rem}
\label{rem:fildhod-complexeszpntzeroexplicit}
In light of \cref{lem:fildhod-complexeszpntzero}, the datum of a quasi-coherent complex $E$ on $(\Z_p^\N)_{t=0}$ may be thought of as follows: Via the Rees equivalence, $E$ corresponds to an increasing filtration $\Fil_\bullet V$ of $p$-complete $\Z_p$-complexes equipped with an endomorphism $D: \Fil_\bullet V\rightarrow\Fil_{\bullet-1} V$ such that the diagram
\begin{equation}
\label{eq:fildhod-complexeszpntzeroexplicit}
\begin{tikzcd}
\dots\ar[r] & \Fil_{-1} V\ar[r, "x"]\ar[d, "xD+1"] & \Fil_0 V\ar[r, "x"]\ar[d, "xD"] & \Fil_1 V\ar[r, "x"]\ar[d, "xD-1"] & \dots\ar[r, "x"] & \Fil_i V\ar[r]\ar[d, "xD-i"] & \dots \\
\dots\ar[r] & \Fil_{-1} V\ar[r, "x"] & \Fil_0 V\ar[r, "x"] & \Fil_1 V\ar[r, "x"] & \dots\ar[r, "x"] & \Fil_i V\ar[r] & \dots
\end{tikzcd}
\end{equation}
commutes. Note that the filtration $\Fil_\bullet V$ precisely corresponds to the pullback of $E$ to $\A^1/\G_m$ under the Rees equivalence.
\end{rem}

\begin{proof}[Proof of \cref{thm:fildhod-comparisonsen}]
By \cref{thm:fildhod-comparisonfiltered}, we already know that the underlying filtered objects agree, so it remains to \EEDIT{prove the assertion about the Sen operator. However, as the conjugate filtration on the diffracted Hodge cohomology of $X$ is complete and the diagram (\ref{eq:fildhod-complexeszpntzeroexplicit}) is commutative, we may reduce to checking the corresponding statement on the associated graded level by induction. Here, we observe that 
\begin{equation*}
xD-i: \Fil_i^\conj R\Gamma_\dHod(X)\rightarrow \Fil_i^\conj R\Gamma_\dHod(X)
\end{equation*}
acts by $-i$ on $\gr_i^\conj R\Gamma_\dHod(X)$ since $xD$ factors through $\Fil_{i-1}^\conj R\Gamma_\dHod(X)$. Recalling that we deduced from (\ref{eq:fildhod-htcomparison1}) that the Sen operator also acts by $-i$ on $\gr_i^\conj R\Gamma_\dHod(X)$, we obtain the result.}
\comment{
and this may be checked on the unfiltered objects, i.e.\ after pullback to $\Z_p^\HT$. \EEDIT{However, here we note that the operator $xD$ on a $\Z_p[x]$-complex $M$ coming from a quasi-coherent complex on $\Z_p^{\HT, c}$ pulls back to the operator $\Theta$ from \cref{prop:fildhod-bgmsharp} as the operator $D$ arises from the exact sequence
\begin{equation*}
\begin{tikzcd}[ampersand replacement=\&]
0\ar[r] \& \Z\ar[r] \& \O(\G_a^\sharp)\ar[r, "\frac{\d}{\d t}"] \& \O(\G_a^\sharp)\ar[r] \& 0\nospacepunct{\;,}
\end{tikzcd}
\end{equation*}
where $\O(\G_a^\sharp)=\Z_p\langle t, \tfrac{t^2}{2!}, \tfrac{t^3}{3!}, \dots\rangle$ as before, see the proof of \cite[Prop. 2.4.4]{FGauges}, while $\Theta$ similarly arises from
\begin{equation*}
\begin{tikzcd}[ampersand replacement=\&]
0\ar[r] \& \Z\ar[r] \& \O(\G_m^\sharp)\ar[r, "t\frac{\d}{\d t}"] \& \O(\G_m^\sharp)\ar[r] \& 0\nospacepunct{\;,}
\end{tikzcd}
\end{equation*}
where $\O(\G_m^\sharp)=\Z_p\langle t, t^{-1}, \tfrac{(t-1)^2}{2!}, \tfrac{(t-1)^3}{3!}, \dots\rangle$, see the proof of \cite[Prop. 3.5.10]{APC}. Now the claim} follows from base change for the cartesian square
\begin{equation*}
\raisebox{\depth}{
\begin{tikzcd}[baseline={(current bounding box.center)}, ampersand replacement=\&]
X^\HT\ar[r, "j_\HT"]\ar[d] \& X^{\HT, c}\ar[d] \\
\Z_p^\HT\ar[r, "j_\HT"] \& \Z_p^{\HT, c}\nospacepunct{\;.}
\end{tikzcd}
}
\qedhere
\end{equation*}
}
\end{proof}

\EDIT{
Finally, we note that, as promised in the beginning of this section, by \cref{rem:fildhod-complexeszpntzeroexplicit}, the $\Z_p\{x, D\}/(xD-Dx-1)$-module structure on conjugate-filtered diffracted Hodge cohomology supplied by \cref{thm:fildhod-comparisonsen} not only encodes the Sen operator $\Theta$, but also the fact that, for all $n\in\Z$, the endomorphism
\begin{equation*}
\Theta+n: \Fil^\conj_n R\Gamma_\dHod(X)\rightarrow \Fil^\conj_n R\Gamma_\dHod(X)
\end{equation*}
admits a factorisation through $\Fil^\conj_{n-1} R\Gamma_\dHod(X)$ as remarked in \cite[Rem. 4.9.10]{APC}. \EEDIT{This is due to the fact that $\Theta+n$ identifies with $(xD-n)+n=xD$, which factors canonically through $x: \Fil^\conj_{n-1} R\Gamma_\dHod(X)\rightarrow \Fil^\conj_n R\Gamma_\dHod(X)$ via $D$.} 
}

\newpage

\section{A stacky comparison of the Nygaard and Hodge filtrations}
\label{sect:nygaardhodge}

In this section, we prove a stacky version of the following theorem \EDIT{of Bhatt--Lurie} comparing the Nygaard filtration on prismatic cohomology with the Hodge filtration on de Rham cohomology. Note that the smoothness assumption can be removed via Kan extensions by \cite[Rem. 5.5.10]{APC}.

\begin{thm}
\label{thm:nygaardhodge-motivation}
Let $X$ be a smooth qcqs $p$-adic formal scheme. Then there is a natural filtered comparison map
\begin{equation*}
\Fil^\bullet_\N R\Gamma_\prism(X)\rightarrow\Fil^\bullet_\Hod R\Gamma_\dR(X)
\end{equation*}
with the property that the induced maps
\begin{equation}
\label{eq:nygaardhodge-motivation}
R\Gamma_\prism(X)/\Fil^i_\N R\Gamma_\prism(X)\rightarrow R\Gamma_\dR(X)/\Fil^i_\Hod R\Gamma_\dR(X)
\end{equation}
are $p$-isogenies for all $i\geq 0$. If \EEDIT{$0\leq i\leq p$}, they are in fact already isomorphisms integrally.
\end{thm}
\begin{proof}
The affine version is \cite[Prop. 5.5.12]{APC} \EDIT{and we deduce the general version using Zariski descent and \cref{lem:syntomicetale-colimtot}.} \EEDIT{Note that, in loc.\ cit., the integral version is only stated for $i<p$, but the proof given there in fact also works for the case $i=p$.}
\end{proof}

\EDIT{
To formulate our precise result, we will need the following definition:

\begin{defi}
\label{defi:nygaardhodge-htweights}
Let $X$ be a bounded $p$-adic formal scheme and $E\in\D(X^\N)$ a gauge on $X$. Then the pullback of $E$ along
\begin{equation*}
\EEDIT{
\begin{tikzcd}[ampersand replacement=\&, column sep=large]
X\times B\G_m\ar[r] \& X^{\Hod}\ar[r, "i_\Hod"] \& X^\N\;,
\end{tikzcd}
}
\end{equation*}
where the first map is induced by the canonical map \EEDIT{$\G_a\rightarrow B\V(\O(1))^\sharp\oplus \G_a=\G_a^\Hod$ of stacks over $B\G_m$,} identifies with a graded quasi-coherent complex $M^\bullet$ on $X$ and the set of integers $i$ such that \EEDIT{the $i$-th graded piece $M^i$} is nonzero is called the set of \emph{Hodge--Tate weights} of $E$. For an $F$-gauge $E\in\D(X^\Syn)$, we define its Hodge--Tate weights as the Hodge--Tate weights of $j_\N^*E$.
\end{defi}

\begin{ex}
For any bounded $p$-adic formal scheme $X$, the gauge $\O\{i\}$ only has a single Hodge--Tate weight which equals $-i$. To see this, recall from \cref{ex:filprism-bktwist} that 
\begin{equation*}
\EEDIT{
\O_{X^\N}\{1\}=\pi_{X^\N}^*\O_{\Z_p^\N}\{1\}\;, \hspace{0.5cm} \O_{\Z_p^\N}\{1\}=\pi^*\O_{\Z_p^\prism}\{1\}\tensor t^*\O(1)
}
\end{equation*}
and observe that there is a commutative diagram
\begin{equation*}
\begin{tikzcd}[ampersand replacement=\&]
X\times\A^1/\G_m\ar[r]\ar[rd] \& X^{\dR, +}\ar[d]\ar[r, "i_{\dR, +}"] \& X^\N\ar[d, "\pi_{X^\N}"] \\
\& \A^1/\G\mathrlap{_m}\ar[d]\ar[r, "i_{\dR, +}" {xshift=-1em}, out=25, in=150] \& \Z_p^\N\ar[d, "\pi"]\ar[l, "t", out=-145, in=-25] \\
\& \Spf\Z_p\ar[r] \& \Z_p^\prism\nospacepunct{\;,}
\end{tikzcd}
\end{equation*}
where the lower map is the one associated to the prism $(\Z_p, (p))$ as in \cref{ex:prismatisation-prismaticsitetoxprism}.
\end{ex}

Now the statement we are going to prove can be formulated as follows:}

\begin{thm}
\label{thm:nygaardhodge-main}
There is a commutative square of stacks
\begin{equation*}
\begin{tikzcd}
\Z_p^\dR\ar[r]\ar[d, "i_\dR"] & \Z_p^{\dR, +}\ar[d, "i_{\dR, +}"] \\
\Z_p^\prism\ar[r, "j_\dR"] & \Z_p^\N
\end{tikzcd}
\end{equation*}
which is an almost pushout up to $p$-isogeny. \EDIT{More precisely,} for any $E\in\Perf(\Z_p^\N)$, it induces a pullback diagram
\begin{equation*}
\begin{tikzcd}
R\Gamma(\Z_p^\N, E)[\frac{1}{p}]\ar[r]\ar[d] & R\Gamma(\Z_p^\prism, j_\dR^*E)[\frac{1}{p}]\ar[d] \\
R\Gamma(\Z_p^{\dR, +}, i_{\dR, +}^*E)[\frac{1}{p}]\ar[r] & R\Gamma(\Z_p^\dR, i_\dR^*E)[\frac{1}{p}]\nospacepunct{\;.}
\end{tikzcd}
\end{equation*}
If the Hodge--Tate weights of $E$ are all at least \EEDIT{$-p$}, then the statement already holds integrally.
\end{thm}

The above theorem then allows us to generalise \cref{thm:nygaardhodge-motivation} to accommodate gauge coefficients in the case where $X$ is proper in the following way:

\begin{cor}
\label{cor:nygaardhodge-coeffs}
Let $X$ be a $p$-adic formal scheme which smooth and proper over $\Spf\Z_p$. For any perfect gauge $E\in\Perf(X^\N)$, there is a natural pullback square
\begin{equation*}
\begin{tikzcd}
\EDIT{\Fil^0_\N R\Gamma_\prism(X, E)[\frac{1}{p}]}\ar[r]\ar[d] & R\Gamma_\prism(X, j_\dR^*E)[\frac{1}{p}]\ar[d] \\
\EDIT{\Fil^0_\Hod R\Gamma_\dR(X, i_{\dR, +}^*E)[\frac{1}{p}]}\ar[r] & R\Gamma_\dR(X, i_\dR^*E)[\frac{1}{p}]\nospacepunct{\;.}
\end{tikzcd}
\end{equation*}
If the Hodge--Tate weights of $E$ are all at least \EEDIT{$-p$}, then the statement already holds integrally.
\end{cor}
\begin{proof}[Proof of \cref{cor:nygaardhodge-coeffs}]
Using \cref{prop:finiteness-main}, \EDIT{the statement is an immediate consequence of} \cref{thm:nygaardhodge-main} once we know that pushforward along $X^\N\rightarrow\Z_p^\N$ does not decrease the smallest Hodge--Tate weight of a gauge; \EDIT{however, this follows from} \cref{prop:syntomicetale-cohomologyhod}.
\end{proof}

\EDIT{In particular, by the comparisons from \cref{thm:fildrstack-comparison} and \cref{thm:filprism-comparison}, the above recovers the result of \cref{thm:nygaardhodge-motivation} in the proper case by putting $E=\O\{i\}$.} In fact, we will deduce \cref{thm:nygaardhodge-main} from the following statement, which, in some sense, is the analogous assertion on the level of the associated graded:

\begin{thm}
\label{thm:nygaardhodge-graded}
The Hodge-filtered de Rham map 
\begin{equation*}
i_{\dR, +}: \Z_p^{\dR, +}\rightarrow\Z_p^\N
\end{equation*}
is an almost isomorphism over the locus $t=0$ up to $p$-isogeny. \EDIT{More precisely,} for any $E\in\Perf(\Z_p^\N)$, it induces an isomorphism
\begin{equation*}
R\Gamma(\Z_p^{\HT, c}, E|_{\Z_p^{\HT, c}})[\tfrac{1}{p}]\cong R\Gamma(\Z_p^\Hod, (i_{\dR, +}^*E)|_{\Z_p^\Hod})[\tfrac{1}{p}]\;.
\end{equation*}
If the Hodge--Tate weights of $E$ are all at least \EEDIT{$-(p-1)$}, then the statement already holds integrally.
\end{thm} 

Before we can give the proof of the theorems above, we need to prepare ourselves with two easy lemmas:

\begin{lem}
\label{lem:nygaardhodge-rgammazpntzero}
Let $E\in\Perf(\Z_p^\N)$ and identify its pushforward to $\A^1/\G_m$ \EEDIT{along} the Rees map with a filtered complex
\begin{equation*}
\begin{tikzcd}
\dots\ar[r] & M^{i+1}\ar[r] & M^i\ar[r] & M^{i-1} \ar[r] & \dots
\end{tikzcd}
\end{equation*}
Then we have
\begin{equation*}
R\Gamma(\Z_p^{\HT, c}, E)=M^0/M^1\;.
\end{equation*}
\end{lem}
\begin{proof}
\EDIT{As pullback along $B\G_m\rightarrow\A^1/\G_m$ corresponds to passage to the associated graded under the Rees equivalence}, we have 
\begin{equation*}
M^0/M^1\cong R\Gamma(B\G_m, (t_* E)|_{B\G_m})\;,
\end{equation*}
so we have to prove that the cartesian diagram
\begin{equation*}
\begin{tikzcd}
\Z_p^{\HT, c}\ar[r]\ar[d, "t"] & \Z_p^\N\ar[d, "t"] \\
B\G_m\ar[r] & \A^1/\G_m
\end{tikzcd}
\end{equation*}
satisfies base change for perfect complexes, i.e.\ that
\begin{equation*}
(t_*E)|_{B\G_m}\cong t_*E|_{\Z_p^{\HT, c}}
\end{equation*}
via the natural map. However, this follows from an application of \cref{prop:basechange-locclosed}.
\end{proof}

\begin{cor}
\label{cor:nygaardhodge-fibreseq}
In the situation of \cref{lem:nygaardhodge-rgammazpntzero}, identify $E|_{\Z_p^{\HT, c}}$ with a filtered object $\Fil_\bullet V$ together with an operator $D$ as in \cref{rem:fildhod-complexeszpntzeroexplicit}. Then there is a natural fibre sequence
\begin{equation*}
\begin{tikzcd}
M^0/M^1\ar[r] & \Fil_0 V\ar[r, "D"] & \Fil_{-1} V\;.
\end{tikzcd}
\end{equation*}
\end{cor}
\begin{proof}
By \cref{lem:fildhod-complexeszpntzero}, we have
\begin{equation*}
R\Gamma(\Z_p^{\HT, c}, E|_{\Z_p^{\HT, c}})\cong\fib(\Fil_0 V\xrightarrow{D}\Fil_{-1} V)\;.
\end{equation*}
Combining this with \cref{lem:nygaardhodge-rgammazpntzero} yields the claim.
\end{proof}

In light of the comparisons from Chapter \ref{sect:conjdhod}, one should view the above statement as a version of \cite[Rem. 5.5.8]{APC} with coefficients. We are now ready to prove \cref{thm:nygaardhodge-main} and \cref{thm:nygaardhodge-graded}:

\begin{proof}[Proof of \cref{thm:nygaardhodge-graded}]
Consider the commutative diagram
\begin{equation}
\label{eq:nygaardhodge-graded}
\begin{tikzcd}
\Z_p^\dR\ar[r]\ar[d, "i_\dR"] & \Z_p^{\dR, +}\ar[d, "i_{\dR, +}"] & \Z_p^\Hod\ar[d]\ar[l] \\
\Z_p^\prism\ar[r, "j_\dR"]\ar[d] & \Z_p^\N\ar[d, "t"] & \Z_p^{\HT, c}\ar[d]\ar[l] \\
\Spf\Z_p\ar[r] & \A^1/\G_m & B\G_m\ar[l]\nospacepunct{\;,}
\end{tikzcd}
\end{equation}
where the vertical compositions are all identities since the Hodge-filtered de Rham map $i_{\dR, +}$ is a section of the Rees map $t$. We identify $t_*E$ with a filtered complex
\begin{equation*}
\begin{tikzcd}
\dots\ar[r] & M^{i+1}\ar[r] & M^i\ar[r] & M^{i-1}\ar[r] & \dots
\end{tikzcd}
\end{equation*}
via the Rees equivalence and use \cref{lem:fildhod-complexeszpntzero} to identify $E|_{\Z_p^{\HT, c}}$ with a filtered complex $\Fil_\bullet V$ together with an operator $D: \Fil_\bullet V\rightarrow\Fil_{\bullet-1} V$ satisfying the \EEDIT{compatibility} properties discussed \EEDIT{in \cref{rem:fildhod-complexeszpntzeroexplicit}}; finally, we let $i_{\dR, +}^* E$ correspond to the filtered complex $\Fil^\bullet F$ under the Rees equivalence. Using \cref{lem:nygaardhodge-rgammazpntzero}, we see that
\begin{equation*}
R\Gamma(\Z_p^{\HT, c}, E|_{\Z_p^{\HT, c}})\cong M^0/M^1\;, \hspace{0.5cm} R\Gamma(\Z_p^\Hod, (i_{\dR, +}^*E)|_{\Z_p^\Hod})\cong \Fil^0 F/\Fil^1 F
\end{equation*}
and thus, we have to examine the fibre of the map $M^0/M^1\rightarrow \Fil^0 F/\Fil^1 F$ induced by the Hodge-filtered de Rham map $i_{\dR, +}$. To this end, first observe that the commutativity of the top right square of (\ref{eq:nygaardhodge-graded}) implies that there is an isomorphism $\Fil^0 F/\Fil^1 F\cong \Fil_0 V/\Fil_{-1} V$. Thus, using \cref{cor:nygaardhodge-fibreseq}, we obtain a commutative diagram of fibre sequences
\begin{equation*}
\begin{tikzcd}
& M^0/M^1\ar[r]\ar[d] & \Fil^0 F/\Fil^1 F\ar[d, "\cong"] \\
\Fil_{-1} V\ar[r, "x"]\ar[d, "Dx"] & \Fil_0 V\ar[d, "D"]\ar[r] & \Fil_0 V/\Fil_{-1} V\ar[d] \\
\Fil_{-1} V\ar[r, equals] & \Fil_{-1} V\ar[r] & 0\nospacepunct{\;,}
\end{tikzcd}
\end{equation*}
from which we conclude that
\begin{equation*}
\fib(M^0/M^1\rightarrow \Fil^0 F/\Fil^1 F)\cong\fib(\Fil_{-1} V\xrightarrow{Dx} \Fil_{-1} V)\cong \fib(\Fil_{-1} V\xrightarrow{xD+1} \Fil_{-1} V)\;,
\end{equation*}
\EEDIT{where the last step is due to $Dx=xD+1$.}
As $E$ is perfect, we have $\Fil_i V=0$ for $i\ll 0$ and thus, \EEDIT{by the commutative diagram (\ref{eq:fildhod-complexeszpntzeroexplicit}),} we are reduced to showing that \EEDIT{the operators}
\begin{equation*}
xD-i: \Fil_i V\rightarrow \Fil_i V
\end{equation*}
induce $p$-isogenies on $\gr_i V$ for all $i\leq -1$ and isomorphisms for \EEDIT{$-(p-1)\leq i\leq -1$}. However, this is clear since \EEDIT{$xD-i$} acts by multiplication by \EEDIT{$-i$} on $\gr_i V$.
\end{proof}

\begin{proof}[Proof of \cref{thm:nygaardhodge-main}]
We use the notations from the previous proof. As
\begin{equation*}
R\Gamma(\Z_p^\N, E)\cong M^0\;, \hspace{0.5cm} R\Gamma(\Z_p^\prism, j_\dR^*E)\cong M^{-\infty}\;,
\end{equation*}
where $M^{-\infty}$ as usual denotes the $p$-completion of $\colim_i M^{-i}$, and
\begin{equation*}
R\Gamma(\Z_p^{\dR, +}, i_{\dR, +}^*E)\cong \Fil^0 F\;, \hspace{0.5cm} R\Gamma(\Z_p^\dR, i_\dR^*E)\cong F\;,
\end{equation*}
where $F$ denotes the $p$-completion of $\colim_i \Fil^{-i} F$, we have to examine the map $F/\Fil^0 F\rightarrow M^{-\infty}/M^0$ induced by $i_{\dR, +}$. Since $E$ is perfect, the filtrations $\Fil^\bullet F$ and $M^\bullet$ eventually stabilise, i.e.\ we have $F=\Fil^i F$ and $M^{-\infty}=M^i$ for $i\ll 0$, and thus we are reduced to showing that the maps $\gr^i F\rightarrow M^i/M^{i+1}$ are $p$-isogenies for all $i\leq -1$ and isomorphisms if $E$ has Hodge--Tate weights all at least \EEDIT{$-p$}. However, note that, using \cref{lem:nygaardhodge-rgammazpntzero}, we have
\begin{equation*}
\EEDIT{
\begin{split}
R\Gamma(\Z_p^{\HT, c}, (E\tensor t^*\O(i))|_{\Z_p^{\HT, c}})&\cong M^i/M^{i+1} \\
R\Gamma(\Z_p^\Hod, (i_{\dR, +}^*(E\tensor t^*\O(i)))|_{\Z_p^\Hod})&\cong \gr^i F
\end{split}
}
\end{equation*}
and if $E$ has Hodge--Tate weights all at least \EEDIT{$-p$}, then \EEDIT{$E\tensor t^*\O(i)$} has Hodge--Tate weights all at least \EEDIT{$-(p+i)$}, so we are done by \cref{thm:nygaardhodge-graded}.
\end{proof}

Note that \cref{thm:nygaardhodge-main} and \cref{thm:nygaardhodge-graded} also apply to certain non-perfect gauges $E$: Namely, going through the proof, one sees that all we have used about $E$ is that it is bounded below with respect to the standard $t$-structure, that the associated decreasing filtrations $\Fil^\bullet F$ and $M^\bullet$ eventually stabilise and that the terms of the increasing filtration $\Fil_\bullet V$ vanish in sufficiently negative degrees. We will now use this observation to recover \cref{thm:nygaardhodge-motivation} from the beginning:

\begin{proof}[Proof of \cref{thm:nygaardhodge-motivation}]
Consider the gauge \EEDIT{$E=\H_\N(X)(i)$}, where the twist is pulled back from $\A^1/\G_m$ via the Rees map $t: \Z_p^\N\rightarrow\A^1/\G_m$. Noting that $i_{\dR, +}^*E, j_\dR^*E$ and $i_\dR^*E$ identify with \EEDIT{$\H_{\dR, +}(X)(i), \H_\prism(X)$} and $\H_\dR(X)$, respectively, we see that, by the comparisons from \cref{thm:drstack-comparison}, \cref{thm:fildrstack-comparison}, \cref{thm:prismatisation-comparison} and \cref{thm:filprism-comparison}, it suffices to show that \cref{thm:nygaardhodge-main} also applies to $E$. However, note that these comparison results also imply that the filtrations $\Fil^\bullet F$ and $M^\bullet$ identify (up to a shift of degrees) with the Hodge and Nygaard filtration, respectively, so they indeed eventually stabilise. Finally, to prove that the terms of $\Fil_\bullet V$ vanish in sufficiently negative degrees, we observe that this filtration identifies (up to a shift of degrees) with $\Fil_\bullet^\conj\Omega_X^\dHod$ by \cref{thm:fildhod-comparisonsen}, whose negative terms are all zero. As $E$ is clearly bounded below, we are done.
\end{proof}

\newpage

\section{The stacky Beilinson fibre square}
\label{sect:beilfibsq}

\EDIT{Recall the following theorem of Antieau--Mathew--Morrow--Nikolaus relating syntomic cohomology to Hodge-filtered de Rham cohomology:}

\begin{thm}
\label{thm:beilfibsq-motivation}
Let $X$ be a \EEDIT{smooth qcqs} $p$-adic formal scheme. For each $i\geq 0$, there is a functorial pullback square
\begin{equation*}
\begin{tikzcd}
R\Gamma_\Syn(X, \Z_p(i))[\frac{1}{p}]\ar[r]\ar[d] & R\Gamma_\Syn(X_{p=0}, \Z_p(i))[\frac{1}{p}]\ar[d] \\
\Fil^i_\Hod R\Gamma_\dR(X)[\frac{1}{p}]\ar[r] & R\Gamma_\dR(X)[\frac{1}{p}]\nospacepunct{\;.}
\end{tikzcd}
\end{equation*}
\end{thm}
\begin{proof}
\EDIT{The affine version is \cite[Thm. 6.17]{BeilFibSq} and we deduce the general version using Zariski descent and \cref{lem:syntomicetale-colimtot}.}
\end{proof}

\EDIT{We want to prove a stacky version of this result and thereby also generalise it to allow arbitrary coefficients, at least in the case where $X$ is proper.} Namely, the main theorem we are going to prove is the following:

\begin{thm}
\label{thm:beilfibsq-main}
There is a commutative square of stacks
\begin{equation*}
\begin{tikzcd}
\Z_p^\dR\ar[r]\ar[d] & \Z_p^{\dR, +}\ar[d, "i_{\dR, +}"] \\
\F_p^\Syn\ar[r] & \Z_p^\Syn
\end{tikzcd}
\end{equation*}
which is an almost pushout up to $p$-isogeny. \EDIT{More precisely,} for any $E\in\Perf(\Z_p^\Syn)$, it induces a pullback diagram
\begin{equation*}
\begin{tikzcd}
R\Gamma(\Z_p^\Syn, E)[\frac{1}{p}]\ar[r]\ar[d] & R\Gamma(\F_p^\Syn, T_\crys(E))[\frac{1}{p}]\ar[d] \\
R\Gamma(\Z_p^{\dR, +}, T_{\dR, +}(E))[\frac{1}{p}]\ar[r] & R\Gamma(\Z_p^\dR, T_\dR(E))[\frac{1}{p}]\nospacepunct{\;.}
\end{tikzcd}
\end{equation*}
\end{thm}

This allows us to generalise \cref{thm:beilfibsq-motivation} to accommodate $F$-gauge coefficients in the case where $X$ is proper in the following way:

\begin{cor}
\label{cor:beilfibsq-coeffs}
Let $X$ be a $p$-adic formal scheme which is smooth and proper over $\Spf\Z_p$. For any perfect $F$-gauge $E\in\Perf(X^\Syn)$, there is a natural pullback square
\begin{equation*}
\begin{tikzcd}
R\Gamma_\Syn(X, E)[\frac{1}{p}]\ar[r]\ar[d] & R\Gamma_\Syn(X_{p=0}, T_\crys(E))[\frac{1}{p}]\ar[d] \\
\EDIT{\Fil^0_\Hod R\Gamma_\dR(X, T_{\dR, +}(E))[\frac{1}{p}]}\ar[r] & R\Gamma_\dR(X, T_\dR(E))[\frac{1}{p}]\nospacepunct{\;.}
\end{tikzcd}
\end{equation*}
\end{cor}
\begin{proof}
Using \cref{prop:finiteness-main}, this follows immediately from \cref{thm:beilfibsq-main}.
\end{proof}

In particular, by the comparison from \cref{thm:fildrstack-comparison}, the above recovers the result of \cref{thm:beilfibsq-motivation} \EDIT{in the proper case} by putting $E=\O\{i\}$.

\subsection{Construction of the square}
\label{subsect:beilfibsq-construction}

We start by constructing the stacky Beilinson fibre square. Namely, using the open immersion $\Z_p^\dR\rightarrow\Z_p^{\dR, +}$ from \cref{rem:fildrstack-unfiltereddr} and the isomorphism $\F_p^\prism\cong\Spf\Z_p\cong\Z_p^\dR$ from \cref{ex:prismatisation-perf}, we obtain a square
\begin{equation}
\label{eq:beilfibsq-constructsquare}
\begin{tikzcd}
\F_p^\prism\cong\Z_p^\dR\ar[d, "j_\prism"]\ar[r] & \Z_p^{\dR, +}\ar[d, "i_{\dR, +}"] \\
\F_p^\Syn\ar[r] & \Z_p^\Syn\nospacepunct{\;.}
\end{tikzcd}
\end{equation}

\begin{prop}
\label{prop:beilfibsq-square}
The diagram (\ref{eq:beilfibsq-constructsquare}) is commutative.
\end{prop}
\begin{proof}
Clearly, it suffices to show that
\begin{equation*}
\begin{tikzcd}
\F_p^\prism\cong\Z_p^\dR\ar[d, "j_\dR"]\ar[r] & \Z_p^{\dR, +}\ar[d, "i_{\dR, +}"] \\
\F_p^\N\ar[r] & \Z_p^\N\nospacepunct{\;.}
\end{tikzcd}
\end{equation*}
commutes. By \cref{ex:filteredprism-drmap}, the composition $\Z_p^\dR\rightarrow\Z_p^{\dR, +}\rightarrow\Z_p^\N$ corresponds to the filtered Cartier--Witt divisor
\begin{equation*}
\G_a^\sharp\oplus F_*W\xrightarrow{(\incl, F)} W\;.
\end{equation*}
Meanwhile, by \cref{rem:filprism-jdrjhtperfd}, the map $j_\dR$ identifies with the open immersion
\begin{equation*}
\Spf W(k)\cong \Spf(W(k)\langle u, t, t^{-1}\rangle/(ut-p))/\G_m\rightarrow \Spf(W(k)\langle u, t\rangle/(ut-p))/\G_m\;,
\end{equation*}
under which $u$ is identified with $p$. As $p$ admits divided powers in $\Z_p$ and hence also in any $p$-nilpotent ring, tracing through the argument from \cref{ex:filprism-perfd} and using that $\sExt_W(F_*W, \G_a^\sharp)\cong\Cone(\G_a^\sharp\rightarrow\G_a)$, see \cite[Prop. 5.2.1]{FGauges}, thus shows that the filtered Cartier--Witt divisor corresponding to the composition $\F_p^\prism\rightarrow\F_p^\N\rightarrow\Z_p^\N$ is split and is therefore isomorphic to
\begin{equation*}
\G_a^\sharp\oplus F_*W\xrightarrow{(\incl, F)} W\;.
\end{equation*}
This shows that the diagram commutes.
\end{proof}

\subsection{A Beilinson-type fibre square for crystalline representations}
\label{subsect:beilfibsq-crys}

Our strategy to prove \cref{thm:beilfibsq-main} will be to make use of \cref{prop:syntomification-cohcrystalline}. Thus, we will start by quickly reviewing some of the theory of crystalline representations. For more details on this topic, we refer to \cite{Fontaine}.

Recall that we use $C$ to denote a \EDIT{fixed} completed algebraic closure of $\Q_p$ and recall the rings
\begin{equation*}
\EEDIT{
A_\inf\coloneqq A_\inf(\O_C)=W(\O_C^\flat)\;,\hspace{0.5cm} A_\crys\coloneqq A_\crys(\O_C)=A_\inf[\tfrac{\xi^n}{n!}: n\geq 0]_{(p)}^\wedge
}
\end{equation*}
from \cref{ex:drstack-perfd}, where \EEDIT{$\xi$} \EDIT{is a generator of} the kernel of the Fontaine map $\theta: A_\inf\rightarrow\O_C$. Choosing once and for all a compatible system $\epsilon=(1, \zeta_p, \zeta_{p^2}, \dots)\in\O_C^\flat$ of $p$-power roots of unity, one can actually take
\begin{equation*}
\EEDIT{
\xi=\frac{[\epsilon]-1}{[\epsilon^{1/p}]-1}=1+[\epsilon^{1/p}]+\dots+[\epsilon^{(p-1)/p}]\;,
}
\end{equation*}
where $[\hspace{2pt}\cdot\hspace{2pt}]: \O_C^\flat\rightarrow W(\O_C^\flat)$ is the Teichmüller representative. \EDIT{One can show that} $A_\crys$ contains \EDIT{the element}
\begin{equation}
\label{eq:beilfibsq-t}
t=\log[\epsilon]=-\sum_{k=1}^\infty \frac{(1-[\epsilon])^k}{k}\;,
\end{equation}
\EDIT{see \cite[Prop. 7.6]{Fontaine},} and \EDIT{thus} we define rings
\begin{equation*}
B_\crys^+\coloneqq A_\crys[\tfrac{1}{p}]\;,\hspace{0.5cm} B_\crys\coloneqq B_\crys^+[\tfrac{1}{t}]\;;
\end{equation*}
note that one can actually show that $B_\crys=A_\crys[\frac{1}{t}]$. The ring $B_\crys$ is equipped with a natural $G_{\Q_p}$-action and, perhaps most importantly, the Frobenius $\phi$ of $A_\inf$ can be continued to an endomorphism of $B_\crys$. We moreover define a ring
\begin{equation*}
B_\dR^+\coloneqq A_\inf[\tfrac{1}{p}]_{(\xi)}^\wedge\;,
\end{equation*}
which is a complete discrete valuation ring with uniformiser $\xi$ and hence its fraction field $B_\dR\coloneqq B_\dR[\tfrac{1}{\xi}]$ is equipped with a natural descending \EEDIT{$\xi$-adic} filtration
\begin{equation*}
\Fil^i B_\dR\coloneqq \xi^iB_\dR^+\;.
\end{equation*}
Note that $B_\crys^+\subseteq B_\dR^+$ and $B_\crys\subseteq B_\dR$.

\begin{defi}
A finite-dimensional $G_{\Q_p}$-representation $V$ over $\Q_p$ is called \emph{crystalline} if there is a $G_{\Q_p}$-equivariant isomorphism
\begin{equation*}
V\tensor_{\Q_p} B_\crys\cong B_\crys^{\dim V}\;.
\end{equation*}
We use $\Rep_{\Q_p}^\crys(G_{\Q_p})$ to denote the category of crystalline $G_{\Q_p}$-representations.
\end{defi}

For a crystalline representation $V$, we write
\begin{equation*}
D_\crys(V)\coloneqq (V\tensor_{\Q_p} B_\crys)^{G_{\Q_p}}\;,
\end{equation*}
where $G_{\Q_p}$ acts diagonally. \EDIT{From} $B_\dR^{G_{\Q_p}}=B_\crys^{G_{\Q_p}}=\Q_p$, \EDIT{see \cite[Prop. 6.26]{Fontaine}, we conclude that} $D_\crys(V)$ is a $\Q_p$-vector space of dimension $\dim V$. By virtue of the Frobenius on $B_\crys$, it comes equipped with a natural Frobenius action $\phi$ and due to 
\begin{equation*}
D_\crys(V)=(V\tensor_{\Q_p} B_\dR)^{G_{\Q_p}}\;,
\end{equation*}
\EDIT{which again comes from $B_\dR^{G_{\Q_p}}=B_\crys^{G_{\Q_p}}$ and $V$ being crystalline}, it also inherits a natural filtration from $B_\dR$. Thus, $D_\crys$ defines a functor
\begin{equation*}
D_\crys: \Rep_{\Q_p}^\crys(G_{\Q_p})\rightarrow\MF^\phi(\Q_p)
\end{equation*}
from crystalline representations into the category of filtered $\phi$-modules over $\Q_p$, which is defined as follows:

\begin{defi}
A \emph{$\phi$-module} over $\Q_p$ is a $\Q_p$-vector space $M$ together with the data of an isomorphism $\phi^*M\cong M$, where $\phi$ denotes the Frobenius on $\Q_p$, and we denote the category of $\phi$-modules over $\Q_p$ by $\Mod^\phi(\Q_p)$. A \emph{filtered $\phi$-module} over $\Q_p$ is a filtered $\Q_p$-vector space $D\in\MF(\Q_p)\coloneqq\Fun((\Z, \geq), \Mod(\Q_p))$ equipped with the data of a $\phi$-module structure on $\ul{D}=\colim_i D(i)$.
\end{defi}

Note that our definition of a filtered $\phi$-module differs slightly from the usual one in that we also allow non-honest filtrations (i.e.\ filtrations with non-injective transition maps) and that we furthermore do not require the filtration to be \EDIT{complete} and exhaustive, as one would usually do. This has the advantage that the category $\MF^\phi(\Q_p)$ is clearly abelian (while this is not the case with the usual definition) and it is also more convenient for our goal of proving \cref{thm:beilfibsq-main}.

\begin{ex}
Let $V=\Q_p(n)$ be the usual Tate twist. As $\sigma t=\chi(\sigma)t$ for any $\sigma\in G_{\Q_p}$, where $\chi: G_{\Q_p}\rightarrow\Z_p^\times$ is the cyclotomic character, we see that any element of the form $\lambda\tensor t^{-n} u$ for $\lambda, u\in\Q_p$ is a $G_{\Q_p}$-fixed point of $V\tensor_{\Q_p} B_\crys$. As one can show that, for any $G_{\Q_p}$-representation $V$, we have $\dim D_\crys(V)\leq\dim V$ with equality if and only if $V$ is crystalline, \EEDIT{see \cite[Thm. 3.14]{Fontaine},} we see that these are all the fixed points and that $V=\Q_p(n)$ is crystalline. Moreover, the filtration on $D_\crys(V)=\Q_p$ is given by $\Q_p$ in degrees $\leq -n$ and zero else, while the $\phi$-structure is given by multiplication by $p^{-n}$ due to $\phi(t)=pt$.
\end{ex}

It is a theorem of Fontaine that the functor $D_\crys$ defines an equivalence from $\Rep_{\Q_p}^\crys(G_{\Q_p})$ onto an abelian subcategory of $\MF^\phi(\Q_p)$ denoted $\MF^\phi_\adm(\Q_p)$, \EDIT{see \EEDIT{\cite[§5.5.2.(iii)]{Semistable}},} and a filtered $\phi$-module in the essential image of $D_\crys$ is called \emph{admissible}.

Note that we have exact forgetful functors
\begin{equation*}
\begin{split}
T_{\dR, +}: \MF^\phi(\Q_p)&\rightarrow \MF(\Q_p)\;, \\
T_\crys: \MF^\phi(\Q_p)&\rightarrow\Mod^\phi(\Q_p)\;, \\
T_\dR: \MF^\phi(\Q_p)&\rightarrow\Mod(\Q_p)
\end{split}
\end{equation*}
defined by forgetting the $\phi$-structure, the filtration or both, respectively. Note that we use the same notation as for the realisation functors for $F$-gauges, but it will always be clear from the context which functor we mean. Our main goal in this section is to prove the following statement:

\begin{prop}
\label{prop:beilfibsq-crys}
Let $D\in\D^b(\MF^\phi_\adm(\Q_p))$ be a bounded complex of admissible filtered $\phi$-modules. Then the forgetful functors above induce a pullback square
\begin{equation*}
\begin{tikzcd}
\RHom_{\MF^\phi_\adm(\Q_p)}(\Q_p, D)\ar[r]\ar[d] & \RHom_{\Mod^\phi(\Q_p)}(\Q_p, T_\crys(D))\ar[d] \\
\RHom_{\MF(\Q_p)}(\Q_p, T_{\dR, +}(D))\ar[r] & \RHom_{\Q_p}(\Q_p, T_\dR(D))\nospacepunct{\;.}
\end{tikzcd}
\end{equation*}
\end{prop}

Note that, as the category $\MF^\phi_\adm(\Q_p)$ does not have enough injectives \EDIT{(in fact, it does not have any nonzero injective object by \cite[Rem. 1.9]{Bannai})}, the notation $\RHom_{\MF^\phi_\adm(\Q_p)}(\Q_p, D)$ does not make sense a priori. However, whenever we have an abelian category $\cal{A}$, we can embed it into its Ind-category $\Ind(\cal{A})$ and this will have enough injectives. By \cite[Thm. 3.5]{Oort}, we have
\begin{equation*}
\Ext^i_{\cal{A}}(M, N)\cong H^i(\RHom_{\Ind(\cal{A})}(M, N))\;,
\end{equation*}
where we use Yoneda Ext on the left-hand side, and hence we may define 
\begin{equation*}
\RHom_{\cal{A}}(M, N)\coloneqq\RHom_{\Ind(\cal{A})}(M, N)
\end{equation*}
as this recovers the usual definition in the case where $\cal{A}$ has enough injectives.

\begin{proof}[Proof of \cref{prop:beilfibsq-crys}]
First observe that we can readily calculate
\begin{equation*}
\begin{split}
\RHom_{\Q_p}(\Q_p, T_\dR(D))&=\ul{D}\;, \\
\RHom_{\MF(\Q_p)}(\Q_p, T_{\dR, +}(D))&=\Fil^0\ul{D}\;, \\
\RHom_{\Mod^\phi(\Q_p)}(\Q_p, T_\crys(D))&=\ul{D}^{\phi=1}\;,
\end{split}
\end{equation*}
\EDIT{where 
\begin{equation*}
\ul{D}^{\phi=1}\coloneqq\fib(\ul{D}\xrightarrow{\phi-1}\ul{D})
\end{equation*}
denotes the derived $\phi$-invariants of $\ul{D}$,} and hence we only need to show that
\begin{equation*}
\RHom_{\MF^\phi_\adm(\Q_p)}(\Q_p, D)\cong \fib(\ul{D}^{\phi=1}\rightarrow\ul{D}/\Fil^0\ul{D})\;,
\end{equation*}
where the quotient $\ul{D}/\Fil^0\ul{D}$ is to be taken in the derived sense, i.e.\ 
\begin{equation*}
\ul{D}/\Fil^0\ul{D}\coloneqq\cofib(\Fil^0\ul{D}\rightarrow\ul{D})\;.
\end{equation*}
However, this follows from \cite[Cor. 2.4.4]{CrystallineExt}.
\end{proof}

\begin{rem}
\label{rem:beilfibsq-rhommfphionetrunc}
The above proof also shows that, if $D$ is an admissible filtered $\phi$-module, then $\RHom_{\MF^\phi_\adm(\Q_p)}(\Q_p, D)$ is concentrated in degrees $0$ and $1$: \EDIT{Indeed, in this case the map $\Fil^0\ul{D}\rightarrow\ul{D}$ is injective and hence $\ul{D}/\Fil^0\ul{D}$ is concentrated in degree zero.}
\end{rem}

\subsection{Étale realisation and the Beilinson fibre square}

\label{sect:etrealisationandbeilfibsq}

Recall from \cref{prop:syntomification-cohcrystalline} that, for any $E\in\Coh(\Z_p^\Syn)$, the Galois representation $T_\et(E)[\tfrac{1}{p}]$ is crystalline. Hence, we can associate to $E$ a filtered $\phi$-module $D_\crys(T_\et(E)[\tfrac{1}{p}])$ and, in this section, we will show how this filtered $\phi$-module can also be obtained in a geometric manner.

To explain this, recall from \cref{ex:syntomification-fgaugesperfd} that an $F$-gauge $E$ on $\F_p$ corresponds to a diagram 
\begin{equation}
\label{eq:beilfibsq-fgaugefp}
\begin{tikzcd}
\dots \ar[r,shift left=.5ex,"t"]
  & \ar[l,shift left=.5ex, "u"] M^{i+1} \ar[r,shift left=.5ex,"t"] & \ar[l,shift left=.5ex, "u"] M^i \ar[r,shift left=.5ex,"t"] & \ar[l,shift left=.5ex, "u"] M^{i-1} \ar[r,shift left=.5ex,"t"] & \ar[l,shift left=.5ex, "u"] \dots
\end{tikzcd}
\end{equation}
of $p$-complete $\Z_p$-complexes such that $ut=tu=p$ together with an isomorphism $\tau: \phi^*M^\infty\cong M^{-\infty}$. Base changing everything to $\Q_p$, we see that all maps in the diagram above become isomorphisms and the only remaining data is $M^{-\infty}$ together with the isomorphism $\phi^*M^{-\infty}[\frac{1}{p}]\cong M^{-\infty}[\frac{1}{p}]$ induced by $\tau$. Sending $E$ to $M^{-\infty}[\frac{1}{p}]$ equipped with its Frobenius structure determines a functor
\begin{equation}
\label{eq:beilfibsq-dfpsynisog}
(-)[\tfrac{1}{p}]: \D(\F_p^\Syn)\rightarrow\Mod_{\Q_p[x, x^{-1}]}(\D(\Q_p))\cong\D(\Mod^\phi(\Q_p))\;,
\end{equation}
where the last equivalence is due to \cite[Thm. 7.1.3.1]{HA}, and we remind the reader that \EDIT{the restriction of this functor to $\Perf(\F_p^\Syn)$ exactly recovers} the functor from (\ref{eq:syntomification-fgaugetolaurentfcrystal}). \EDIT{Finally, recalling that the pullback $j_\prism^*E\in\D(\F_p^\prism)\cong\widehat{\D}(\Z_p)$, where we have used \cref{ex:prismatisation-perf}, identifies with $M^{-\infty}$, we see that the functor (\ref{eq:beilfibsq-dfpsynisog}) yields a commutative diagram
\begin{equation}
\label{eq:beilfibsq-crysfrob}
\begin{tikzcd}[ampersand replacement=\&, column sep=large]
\D(\F_p^\Syn)\ar[r, "j_\prism^*"]\ar[dr, "{(-)[\tfrac{1}{p}]}", swap] \& \D(\F_p^\prism)\cong \widehat{\D}(\Z_p)\ar[r, "{(-)[\tfrac{1}{p}]}"] \& \D(\Q_p) \\
\&  \D(\Mod^\phi(\Q_p))\ar[ur] \& \nospacepunct{\;,}
\end{tikzcd}
\end{equation}
where $\D(\Mod^\phi(\Q_p))\rightarrow \D(\Q_p)$ is the forgetful functor.}

\begin{lem}
\label{lem:beilfibsq-rationalcohfpsyn}
For any $E\in\D(\F_p^\Syn)$, the functor $(-)[\frac{1}{p}]$ induces an isomorphism
\begin{equation*}
R\Gamma(\F_p^\Syn, E)[\tfrac{1}{p}]\cong \RHom_{\Mod^\phi(\Q_p)}(\Q_p, E[\tfrac{1}{p}])\;.
\end{equation*}
\end{lem}
\begin{proof}
Recall from \cref{ex:syntomification-fgaugesperfd} that
\begin{equation*}
R\Gamma(\F_p^\Syn, E)=\fib(M^0\xrightarrow{t^\infty-\tau u^\infty} M^{-\infty})
\end{equation*}
in the above notation. However, as $t^\infty$ and $u^\infty$ become isomorphisms upon inverting $p$, this yields
\begin{equation*}
R\Gamma(\F_p^\Syn, E)[\tfrac{1}{p}]\cong\fib(M^{-\infty}[\tfrac{1}{p}]\xrightarrow{1-\phi} M^{-\infty}[\tfrac{1}{p}])\cong\RHom_{\Mod^\phi(\Q_p)}(\Q_p, M^{-\infty}[\tfrac{1}{p}])\;,
\end{equation*}
so we are done.
\end{proof}

\begin{lem}
\label{lem:beilfibsq-etalephimod}
For any $E\in\Coh(\Z_p^\Syn)$, there is a natural isomorphism
\begin{equation*}
T_\crys(E)[\tfrac{1}{p}]\cong T_\crys(D_\crys(T_\et(E)[\tfrac{1}{p}]))\;.\footnote{\EDIT{Beware that, on the right-hand side, we use the functor $T_\crys: \MF^\phi(\Q_p)\rightarrow\Mod^\phi(\Q_p)$, while $T_\crys(E)$ on the left-hand side denotes the pullback of $E$ to $\F_p^\Syn$ as defined in \cref{ex:syntomification-tcrys}.}}
\end{equation*}
\end{lem}
\begin{proof}
Since any coherent $F$-gauge on $\Z_p$ is isomorphic to a reflexive $F$-gauge up to $p$-isogeny by \cite[Prop. 6.7.1]{FGauges}, we may assume that $E$ is reflexive. Let $\cal{E}\in\Vect^\phi((\Z_p)_\prism)$ denote the prismatic $F$-crystal in vector bundles corresponding to $E$ by \cref{thm:syntomification-reflcrys}. \EEDIT{Identifying $E$ with a diagram (\ref{eq:beilfibsq-fgaugefp}) and writing $M^{-\infty}\coloneqq(\colim_i M^{-i})_{(p)}^\wedge$ as usual, the module} $M^{-\infty}$ identifies with the pullback of $\cal{E}$ to the absolute prismatic site of $\F_p$, i.e.\ $T_\crys(E)[\tfrac{1}{p}]\cong\cal{E}(\Z_p)[\frac{1}{p}]$, so we have to show
\begin{equation}
\label{eq:beilfibsq-etalephimod}
\cal{E}(\Z_p)[\tfrac{1}{p}]\cong (T_\et(E)\tensor_{\Z_p} B_\crys)^{G_{\Q_p}}
\end{equation}
as $\phi$-modules. Recalling that
\begin{equation*}
T_\et(E)=(\cal{E}(A_\inf)\tensor_{A_\inf} W(C^\flat))^{\phi=1}
\end{equation*}
\EEDIT{from (\ref{eq:syntomification-laurentfcrystaltolocsys}),} we see that \cite[Lem. 4.26]{BMS} implies that there is a natural isomorphism
\begin{equation}
\label{eq:beilfibsq-tetcrystal}
T_\et(E)\tensor_{\Z_p} A_\inf[\tfrac{1}{\mu}]\cong \cal{E}(A_\inf)[\tfrac{1}{\mu}]\;,
\end{equation}
where $\mu\coloneqq [\epsilon^{1/p}]-1$; note that our definition of $\mu$ differs from the one in loc.\ cit.\ by a Frobenius twist since not $\cal{E}(A_\inf)$, but $\phi^*\cal{E}(A_\inf)$ is a Breuil--Kisin--Fargues module in the sense of \cite[Def. 4.22]{BMS}. \EEDIT{Using the maps $\O_C\rightarrow \O_C/p\leftarrow\F_p$ and the fact that $(A_\crys, (p))$ is the initial object of the absolute prismatic site of $\O_C/p$, see \cite[Not. 5.1]{FCrystals}, we obtain natural maps of prisms 
\begin{equation*}
(A_\inf, (\xi))\rightarrow (A_\crys, (p))\leftarrow (\Z_p, (p))
\end{equation*}
and hence $\cal{E}$ being a prismatic crystal yields isomorphisms}
\begin{equation}
\label{eq:beilfibsq-fcrystalbasechange}
\cal{E}(A_\inf)\tensor_{A_\inf} A_\crys\cong \cal{E}(A_\crys)\cong \cal{E}(\Z_p)\tensor_{\Z_p} A_\crys\;.
\end{equation}
\EEDIT{Finally, as} $\mu$ is invertible in $B_\crys$ (note that $t\in \xi\mu B_\crys^+$), we overall obtain
\begin{equation*}
\EEDIT{
T_\et(E)\tensor_{\Z_p} B_\crys\cong \cal{E}(A_\inf)\tensor_{A_\inf} B_\crys\cong \cal{E}(\Z_p)\tensor_{\Z_p} B_\crys\;.
}
\end{equation*}
As this isomorphism is both $\phi$- and $G_{\Q_p}$-equivariant, taking $G_{\Q_p}$-fixed points yields (\ref{eq:beilfibsq-etalephimod}).
\end{proof}

Observe that there is also a functor
\begin{equation}
\label{eq:beilfibsq-da1gmisog}
(-)[\tfrac{1}{p}]: \D(\Z_p^{\dR, +})\cong\D(\A^1/\G_m)\cong\widehat{\DF}(\Z_p)\rightarrow\DF(\Q_p)
\end{equation}
\EDIT{given by sending a $p$-complete filtered complex $M^\bullet$ over $\Z_p$ to the filtered complex $M^\bullet[\tfrac{1}{p}]$ over $\Q_p$ and that this yields an analogous result to \cref{lem:beilfibsq-rationalcohfpsyn}. Namely, we have the following:}

\begin{lem}
\label{lem:beilfibsq-rationalcoha1gm}
\EDIT{For any $E\in\D(\Z_p^{\dR, +})$, the functor $(-)[\tfrac{1}{p}]$ induces an isomorphism}
\begin{equation*}
\EDIT{R\Gamma(\Z_p^{\dR, +}, E)[\tfrac{1}{p}]\cong \RHom_{\MF(\Q_p)}(\Q_p, E[\tfrac{1}{p}])\;.}
\end{equation*}
\end{lem}
\begin{proof}
\EDIT{Identifying $E$ with a $p$-complete filtered complex $M^\bullet$ via the Rees equivalence, this follows by observing that}
\begin{equation*}
\EDIT{
R\Gamma(\Z_p^{\dR, +}, E)[\tfrac{1}{p}]\cong M^0[\tfrac{1}{p}]\cong\RHom_{\MF(\Q_p)}(\Q_p, E[\tfrac{1}{p}])\;.}
\qedhere
\end{equation*}
\end{proof}

Furthermore, we also get a statement analogous to \cref{lem:beilfibsq-etalephimod}:

\begin{lem}
\label{lem:beilfibsq-etalefiltered}
For any $E\in\Coh(\Z_p^\Syn)$, there is a natural isomorphism
\begin{equation*}
T_{\dR, +}(E)[\tfrac{1}{p}]\cong T_{\dR, +}(D_\crys(T_\et(E)[\tfrac{1}{p}]))\;.\footnote{\EDIT{Similarly to \cref{lem:beilfibsq-etalephimod}, beware that, on the right-hand side, we use the functor $T_{\dR, +}: \MF^\phi(\Q_p)\rightarrow\MF(\Q_p)$, while $T_{\dR, +}(E)$ on the left-hand side denotes the pullback of $E$ to $\A^1/\G_m$ as defined in \cref{ex:syntomification-tdr+}.}}
\end{equation*}
\end{lem}
\begin{proof}
\EDIT{As in the proof of of \cref{lem:beilfibsq-etalephimod}}, we may assume that $E$ is reflexive and let $\cal{E}\in\Vect^\phi((\Z_p)_\prism)$ denote the corresponding prismatic $F$-crystal in vector bundles. Recall that the filtration on
\begin{equation*}
D_\crys(T_\et(E)\EDIT{[\tfrac{1}{p}]})=(T_\et(E)\tensor_{\Z_p} B_\crys)^{G_{\Q_p}}=(T_\et(E)\tensor_{\Z_p} B_\dR)^{G_{\Q_p}}
\end{equation*}
is given by 
\begin{equation*}
\Fil^i D_\crys(T_\et(E))=(T_\et(E)\tensor_{\Z_p} \xi^i B_\dR^+)^{G_{\Q_p}}\;.
\end{equation*}
Furthermore, recall that the filtered object $T_{\dR, +}(E)$ is defined by descent from $T_{\dR, +}(E_{\O_C})$, where $E_{\O_C}$ denotes the pullback of $E$ to $\O_C^\Syn$. Now observe that, identifying $E_{\O_C}$ with a graded $A_\inf[u, t]/(ut-\phi^{-1}(\xi))$-module $M$ using \cref{ex:syntomification-fgaugesperfd}, the calculation from \cref{ex:fildrstack-perfd} implies that $T_{\dR, +}(E_{\O_C})$ corresponds to the graded $A_\inf[\frac{u^n}{n!}, t: n\geq 1]/(ut-\xi)$-module $\phi^*M\tensor_{A_\inf} A_\crys$. Hence, by \cref{prop:syntomification-refltofcrystal} and \cref{rem:syntomification-descriptionrefltofcrystal}, we conclude that
\begin{equation*}
T_{\dR, +}(E)[\tfrac{1}{p}]=(\phi^*\cal{E}(A_\inf)\tensor_{A_\inf} B_\crys^+)^{G_{\Q_p}}\;,
\end{equation*}
where \EDIT{the filtration on $\phi^*\cal{E}(A_\inf)\tensor_{A_\inf} B_\crys^+$ is the tensor product filtration and, in turn, the filtration on $\phi^*\cal{E}(A_\inf)$ is given by}
\begin{equation*}
\Fil^i\phi^*\cal{E}(A_\inf)=\xi^i\cal{E}(A_\inf)\cap \phi^*\cal{E}(A_\inf)\;;
\end{equation*}
\EDIT{here,} the intersection takes place inside $\phi^*\cal{E}(A_\inf)[\frac{1}{\xi}]\cong \cal{E}(A_\inf)[\frac{1}{\xi}]$ via the Frobenius structure on $\cal{E}$. Using (\ref{eq:beilfibsq-fcrystalbasechange}), however, we see that
\begin{equation*}
\begin{split}
(\phi^*\cal{E}(A_\inf)\tensor_{A_\inf} B_\crys^+)^{G_{\Q_p}}&\cong (\phi^*\cal{E}(\Z_p)\tensor_{\Z_p} B_\crys^+)^{G_{\Q_p}}\cong \phi^*\cal{E}(\Z_p)[\tfrac{1}{p}] \\
&\cong (\phi^*\cal{E}(\Z_p)\tensor_{\Z_p} B_\dR)^{G_{\Q_p}}\cong (\phi^*\cal{E}(A_\inf)\tensor_{A_\inf} B_\dR)^{G_{\Q_p}}
\end{split}
\end{equation*}
and hence we also have
\begin{equation*}
T_{\dR, +}(E)[\tfrac{1}{p}]=(\phi^*\cal{E}(A_\inf)\tensor_{A_\inf} B_\dR)^{G_{\Q_p}}\;.
\end{equation*}
We conclude that it suffices to show that the isomorphism
\begin{equation}
\label{eq:beilfibsq-tetcrystaldr}
T_\et(E)\tensor_{\Z_p} B_\dR\cong \cal{E}(A_\inf)\tensor_{A_\inf} B_\dR\cong\phi^*\cal{E}(A_\inf)\tensor_{A_\inf} B_\dR
\end{equation}
obtained from (\ref{eq:beilfibsq-tetcrystal}) using $\phi^*\cal{E}(A_\inf)[\frac{1}{\xi}]\cong \cal{E}(A_\inf)[\frac{1}{\xi}]$ is compatible with the filtrations, \EDIT{where again the right-hand side is equipped with the tensor product filtration.} As both sides are $\xi$-adically filtered, \EDIT{i.e.\ their filtrations have the property that $\Fil^i=\xi^i\Fil^0$ for all $i\in\Z$}, it suffices to check this in filtration degree $0$. \EDIT{In other words,} we have to show that the isomorphism above restricts to an isomorphism
\begin{equation*}
T_\et(E)\tensor_{\Z_p} B_\dR^+\cong \bigcup_{i\in\Z} (\xi^i\cal{E}(A_\inf)\cap \phi^*\cal{E}(A_\inf))\tensor_{A_\inf} \xi^{-i} B_\dR^+=\cal{E}(A_\inf)\tensor_{A_\inf} B_\dR^+\;,
\end{equation*}
where the right-hand side is a submodule of $\phi^*\cal{E}(A_\inf)\tensor_{A_\inf} B_\dR$ via the isomorphism $\phi^*\cal{E}(A_\inf)[\frac{1}{\xi}]\cong \cal{E}(A_\inf)[\frac{1}{\xi}]$. However, this again follows from (\ref{eq:beilfibsq-tetcrystal}): As \EDIT{$\mu=[\epsilon^{1/p}]-1$} maps to the nonzero element $\zeta_p-1$ in the residue field of $B_\dR^+$, it is already invertible in $B_\dR^+$ and hence the isomorphism (\ref{eq:beilfibsq-tetcrystaldr}) restricts to an isomorphism
\begin{equation*}
T_\et(E)\tensor_{\Z_p} B_\dR^+\cong \cal{E}(A_\inf)\tensor_{A_\inf} B_\dR^+\;,
\end{equation*}
as desired.
\end{proof}

\subsection{Proof of the main theorem}
\label{subsect:beilfibsq-main}

We are now ready to put everything together and prove \cref{thm:beilfibsq-main}. We begin with a general lemma about $t$-structures on stable $\infty$-categories, which is probably well-known, see also \cite[Lem. 3.2]{Bridgeland}. \EDIT{For this, first recall the following definition:}

\begin{defi}
\label{defi:beilfibsq-tstructbounded}
\EDIT{
Let $\cal{C}$ be a stable $\infty$-category equipped with a $t$-structure $(\cal{C}^{\geq 0}, \cal{C}^{\leq 0})$ and put
\begin{equation*}
\EEDIT{
\cal{C}^-\coloneqq \bigcup_{n\in\Z} \cal{C}^{\leq n}\;, \hspace{0.5cm} \cal{C}^+\coloneqq \bigcup_{n\in\Z} \cal{C}^{\geq -n}\;, \hspace{0.5cm} \cal{C}^b\coloneqq \cal{C}^-\cap\cal{C}^+\;.
}
\end{equation*}
Then the $t$-structure $(\cal{C}^{\geq 0}, \cal{C}^{\leq 0})$ is called \emph{bounded} if $\cal{C}^b=\cal{C}$. It is called \emph{non-degenerate} if}
\begin{equation*}
\EDIT{
\bigcap_{n\in\Z} \cal{C}^{\leq n}=\bigcap_{n\in\Z} \cal{C}^{\geq n}=0\;.
}
\end{equation*}
\end{defi}

\begin{rem}
\label{rem:beilfibsq-nondegeneratetstruct}
\EDIT{
Recall that a $t$-structure $(\cal{C}^{\geq 0}, \cal{C}^{\leq 0})$ on a stable $\infty$-category $\cal{C}$ induces canoncial truncation functors
\begin{equation*}
\tau^{\leq n}: \cal{C}\rightarrow\cal{C}^{\leq n}\;,\hspace{0.5cm} \tau^{\geq n}: \cal{C}\rightarrow\cal{C}^{\geq n}
\end{equation*}
for any $n\in\Z$ which are right-adjoint, respectively left-adjoint, to the inclusions $\cal{C}^{\leq n}\rightarrow\cal{C}$, respectively $\cal{C}^{\geq n}\rightarrow\cal{C}$. Moreover, we obtain cohomology functors
\begin{equation*}
H^n(-)=(\tau^{\geq n}\circ\tau^{\leq n})(-)[n]: \cal{C}\rightarrow \cal{C}^{\geq 0}\cap\cal{C}^{\leq 0}=\cal{C}^\heartsuit\;.
\end{equation*}
If the $t$-structure is non-degenerate, this implies that an object $C\in\cal{C}$ with $H^n(C)=0$ for all $n\in\Z$ must vanish. Moreover, it also implies that $C\in\cal{C}^{\geq 0}$ if and only if $H^n(C)=0$ for all $n<0$ and, similarly, $C\in\cal{C}^{\leq 0}$ if and only if $H^n(C)=0$ for all $n>0$, see \cite[Rem. 1.5]{tStructures}.
}
\end{rem}

\begin{lem}
\label{lem:beilfibsq-boundedtstructure}
Let $\cal{C}$ be a stable $\infty$-category equipped with a bounded non-degenerate $t$-structure. Then $\cal{C}$ is generated under shifts and \EDIT{fibres} by its heart $\cal{C}^\heartsuit$.
\end{lem}
\begin{proof}
\EDIT{By boundedness of the $t$-structure,} any $E\in\cal{C}$ has only finitely many non-vanishing cohomology groups and we use induction on the number of such nonzero $H^i(E)$. If all $H^i(E)$ vanish, then $E$ is itself zero as the $t$-structure is non-degenerate. Otherwise, let $k$ be maximal such that $H^k(E)\neq 0$ and put
\begin{equation*}
E'\coloneqq\fib(E\rightarrow\tau^{\leq k-1}E)\;.
\end{equation*}
As $H^i(\tau^{\leq k-1} E)=H^i(E)$ except for $i=k$, in which case $H^k(\tau^{\leq k-1}E)=0$, the long exact cohomology sequence shows that $H^i(E')=0$ for all $i\neq k$ and hence $E'[k]$ lies in $\cal{C}^\heartsuit$ by non-degeneracy of the $t$-structure. Thus, we are done by induction.
\end{proof}

\begin{proof}[Proof of \cref{thm:beilfibsq-main}]
The relevant square was already constructed in \cref{prop:beilfibsq-square}, so it remains to show that it is an almost pushout up to $p$-isogeny. \EDIT{First note that the $t$-structure on $\Perf(\Z_p^\Syn)$ from \cref{rem:filprism-cohheartperf} satisfies the hypotheses of \cref{lem:beilfibsq-boundedtstructure}: Indeed, this follows from the fact that this $t$-structure is induced from the one on $\Perf(\Z_p\langle u, t\rangle)$, which is clearly bounded and non-degenerate. As all terms of the square we claim to be a pullback commute with \EEDIT{shifts and fibres}, we may thus assume that $E$ is coherent by \cref{lem:beilfibsq-boundedtstructure}.} In this case, \EEDIT{recalling from \cref{prop:syntomification-cohcrystalline} that $T_\et(E)$ is crystalline and} writing $D=D_\crys(T_\et(E)[\frac{1}{p}])$, we know by \cref{prop:beilfibsq-crys} that there is a pullback square
\begin{equation*}
\begin{tikzcd}
\RHom_{\MF^\phi_\adm(\Q_p)}(\Q_p, D)\ar[r]\ar[d] & \RHom_{\Mod^\phi(\Q_p)}(\Q_p, T_\crys(D))\ar[d] \\
\RHom_{\MF(\Q_p)}(\Q_p, T_{\dR, +}(D))\ar[r] & \RHom_{\Q_p}(\Q_p, T_\dR(D))\nospacepunct{\;.}
\end{tikzcd}
\end{equation*}
\EEDIT{As} \cref{lem:beilfibsq-rationalcohfpsyn}, \cref{lem:beilfibsq-etalephimod}, \EEDIT{\cref{lem:beilfibsq-rationalcoha1gm}} and \cref{lem:beilfibsq-etalefiltered} allow us to rewrite this as
\begin{equation*}
\begin{tikzcd}
\RHom_{\MF^\phi_\adm(\Q_p)}(\Q_p, D)\ar[r]\ar[d] & R\Gamma(\F_p^\Syn, T_\crys(E))[\tfrac{1}{p}]\ar[d] \\
R\Gamma(\Z_p^{\dR, +}, T_{\dR, +}(E))[\tfrac{1}{p}]\ar[r] & R\Gamma(\Z_p^\dR, T_\dR(E))[\tfrac{1}{p}]\;,
\end{tikzcd}
\end{equation*}
it remains to show that $T_\et$ induces a quasi-isomorphism
\begin{equation*}
R\Gamma(\Z_p^\Syn, E)[\tfrac{1}{p}]\cong\RHom_{\MF^\phi_\adm(\Q_p)}(\Q_p, D)\;.
\end{equation*}
For this, first note that the embedding $\MF^\phi_\adm(\Q_p)\cong\Rep_{\Q_p}^\crys(G_{\Q_p})\hookrightarrow\Rep_{\Q_p}(G_{\Q_p})$ induces a morphism
\begin{equation*}
\RHom_{\MF^\phi_\adm(\Q_p)}(\Q_p, D)\rightarrow R\Gamma(G_{\Q_p}, T_\et(E)[\tfrac{1}{p}])
\end{equation*}
and let $Q$ denote its cofibre. We have to show that the composite
\begin{equation}
\label{eq:beilfibsq-compositezero}
R\Gamma(\Z_p^\Syn, E)[\tfrac{1}{p}]\rightarrow R\Gamma(G_{\Q_p}, T_\et(E)[\tfrac{1}{p}])\rightarrow Q\;,
\end{equation}
where the first map is induced by $T_\et$, is zero. From the long exact sequence
\begin{equation*}
\begin{tikzcd}
0\ar[r] & \Hom_{\MF^\phi_\adm(\Q_p)}(\Q_p, D)\ar[r, "\cong"] & \Hom_{G_{\Q_p}}(\Q_p, T_\et(E)[\tfrac{1}{p}])\ar[r] & H^0(Q) \ar[overlay, dll, out=-10, in=170] \\
& \Ext^1_{\MF^\phi_\adm(\Q_p)}(\Q_p, D)  \ar[r, hookrightarrow] & \Ext^1_{G_{\Q_p}}(\Q_p, T_\et(E)[\tfrac{1}{p}])\ar[r] & H^1(Q) \ar[r] & 0\nospacepunct{\;,}
\end{tikzcd}
\end{equation*}
where the last zero is due to the fact that $\RHom_{\MF^\phi_\adm(\Q_p)}(\Q_p, D)$ is concentrated in degrees $0$ and $1$ \EDIT{by \cref{rem:beilfibsq-rhommfphionetrunc},} we see that $H^0(Q)=0$ and that $H^1(Q)$ is the cokernel of the injection $\Ext^1_{\MF^\phi_\adm(\Q_p)}(\Q_p, D)\hookrightarrow \Ext^1_{G_{\Q_p}}(\Q_p, T_\et(E)[\tfrac{1}{p}])$. However, by \cref{prop:syntomification-cohcrystalline}, the complex $R\Gamma(\Z_p^\Syn, E)[\tfrac{1}{p}]$ is concentrated in degrees $0$ and $1$ and the map 
\begin{equation*}
H^1(\Z_p^\Syn, E)[\tfrac{1}{p}]\rightarrow\Ext^1_{G_{\Q_p}}(\Q_p, T_\et(E))
\end{equation*}
factors through $\Ext^1_{\MF^\phi_\adm(\Q_p)}(\Q_p, D)$, hence the map in (\ref{eq:beilfibsq-compositezero}) induces the zero map on all cohomology groups and thus is itself zero. This means that we obtain a unique factorisation
\begin{equation*}
R\Gamma(\Z_p^\Syn, E)[\tfrac{1}{p}]\rightarrow \RHom_{\MF^\phi_\adm(\Q_p)}(\Q_p, D)\rightarrow R\Gamma(G_{\Q_p}, T_\et(E)[\tfrac{1}{p}])
\end{equation*}
and, again by \cref{prop:syntomification-cohcrystalline}, the first map is a quasi-isomorphism, as desired.
\end{proof}

\subsection{Comparison with Fontaine--Messing syntomic cohomology}

As a corollary of \cref{thm:beilfibsq-main}, we can prove a comparison between syntomic cohomology in our sense and the syntomic cohomology introduced by Fontaine--Messing in \cite{FontaineMessing}.

\begin{defi}
Let $X$ be a \EEDIT{smooth qcqs} $p$-adic formal scheme. For $0\leq i\leq p-2$, we define the integral \emph{Fontaine--Messing syntomic cohomology} of $X$ in weight $i$ as
\begin{equation*}
R\Gamma_{\Syn, \FM}(X, \Z_p(i))\coloneqq \fib(\Fil^i_\Hod R\Gamma_\dR(X)\xrightarrow{1-\phi/p^i} R\Gamma_\dR(X))
\end{equation*}
and there is also a rational version defined for all $i\geq 0$ by 
\begin{equation*}
R\Gamma_{\Syn, \FM}(X, \Q_p(i))\coloneqq \fib(\Fil^i_\Hod R\Gamma_\dR(X)[\tfrac{1}{p}]\xrightarrow{\phi-p^i} R\Gamma_\dR(X)[\tfrac{1}{p}])\;.
\end{equation*}
Here, $\phi$ denotes the crystalline Frobenius on $R\Gamma_\dR(X)$.
\end{defi}

Note that the above definition indeed makes sense: By quasisyntomic descent and \cite[Prop. 6.8]{BeilFibSq}, the Frobenius on $\Fil^i_\Hod R\Gamma_\dR(X)$ is divisible by $p^i$ for $0\leq i\leq p-1$. In loc.\ cit., Antieau--Mathew--Morrow--Nikolaus prove the following comparison between syntomic cohomology and Fontaine--Messing syntomic cohomology:

\begin{thm}
\label{thm:beilfibsq-fmratmotivation}
Let $X$ be a \EEDIT{smooth qcqs} $p$-adic formal scheme. For $i\geq 0$, there is a natural isomorphism
\begin{equation*}
R\Gamma_{\Syn, \FM}(X, \Q_p(i))\cong R\Gamma_\Syn(X, \Q_p(i))\;.
\end{equation*}
\end{thm}
\begin{proof}
The affine version is \cite[Thm. 6.22]{BeilFibSq} \EDIT{and we deduce the general version using Zariski descent and \cref{lem:syntomicetale-colimtot}.}
\end{proof}

\EDIT{To generalise this result to arbitrary coefficients, at least in the case where $X$ is proper, we first observe that de Rham cohomology with $F$-gauge coefficients is also equipped with a crystalline Frobenius, just as in the case with trivial coefficients:}

\begin{rem}
\label{rem:beilfibsq-crysfrob}
\EDIT{
Let $X$ be a bounded $p$-adic formal scheme and $E\in\D(X^\Syn)$ an $F$-gauge on $X$. Then the rational de Rham cohomology $R\Gamma_\dR(X, T_\dR(E))[\tfrac{1}{p}]$ is equipped with an automorphism $\tau$ linear over the Frobenius of $\Z_p$ in the following way: By the stacky version of the crystalline comparison, i.e.\ by $\phi^*\F_p^\prism\cong \Z_p^\dR$, see \cite[Rem. 2.5.12, Constr. 3.1.1]{FGauges}, we have
\begin{equation*}
\EEDIT{
R\Gamma_\dR(X, T_\dR(E))[\tfrac{1}{p}]\cong \phi^*R\Gamma_\prism(X_{p=0}, j_\prism^*T_\crys(E))[\tfrac{1}{p}]=\phi^*j_\prism^*\pi_{(X_{p=0})^\Syn, *}T_\crys(E)[\tfrac{1}{p}]\;.
}
\end{equation*}
However, by (\ref{eq:beilfibsq-crysfrob}), the complex \EEDIT{$j_\prism^*\pi_{(X_{p=0})^\Syn, *}T_\crys(E)[\tfrac{1}{p}]$} is naturally equipped with a Frobenius automorphism since it is pulled back from $\F_p^\Syn$ along $j_\prism$. We call the corresponding automorphism $\tau$ of $R\Gamma_\dR(X, T_\dR(E))[\tfrac{1}{p}]$ the \emph{crystalline Frobenius}.
}
\end{rem}

\begin{thm}
\label{thm:beilfibsq-fmrat}
Let $E\in\Perf(\Z_p^\Syn)$. Then the Hodge-filtered de Rham map induces a fibre sequence
\begin{equation*}
R\Gamma(\Z_p^\Syn, E)[\tfrac{1}{p}]\rightarrow R\Gamma(\Z_p^{\dR, +}, T_{\dR, +}(E))[\tfrac{1}{p}]\xrightarrow{1-\tau} R\Gamma(\Z_p^\dR, T_\dR(E))[\tfrac{1}{p}]\;,
\end{equation*}
where $\tau$ denotes \EDIT{the crystalline Frobenius from \cref{rem:beilfibsq-crysfrob}.}
\end{thm}
\begin{proof}
Recall from \cref{lem:beilfibsq-rationalcohfpsyn} that there is a fibre sequence
\begin{equation*}
R\Gamma(\F_p^\Syn, T_\crys(E))[\tfrac{1}{p}]\rightarrow R\Gamma(\F_p^\prism, j_\prism^*T_\crys(E))[\tfrac{1}{p}]\xrightarrow{1-\tau} R\Gamma(\F_p^\prism, j_\prism^*T_\crys(E))[\tfrac{1}{p}]\;.
\end{equation*}
\EDIT{Again using the stacky version $\phi^*\F_p^\prism\cong \Z_p^\dR$ of the crystalline comparison as in \cref{rem:beilfibsq-crysfrob}}, the fibre sequence above becomes
\begin{equation*}
R\Gamma(\F_p^\Syn, T_\crys(E))[\tfrac{1}{p}]\rightarrow R\Gamma(\Z_p^\dR, T_\dR(E))[\tfrac{1}{p}]\xrightarrow{1-\tau} R\Gamma(\Z_p^\dR, T_\dR(E))[\tfrac{1}{p}]
\end{equation*}
and hence \cref{thm:beilfibsq-main} implies that it suffices to show that
\begin{equation*}
\begin{tikzcd}
\fib(R\Gamma(\Z_p^{\dR, +}, T_{\dR, +}(E))[\frac{1}{p}]\xrightarrow{1-\tau}R\Gamma(\Z_p^\dR, T_\dR(E))[\frac{1}{p}])\ar[r]\ar[d] & R\Gamma(\Z_p^\dR, T_\dR(E))[\frac{1}{p}]^{\tau=1}\ar[d] \\
R\Gamma(\Z_p^{\dR, +}, T_{\dR, +}(E))[\frac{1}{p}]\ar[r] & R\Gamma(\Z_p^\dR, T_\dR(E))[\frac{1}{p}]
\end{tikzcd}
\end{equation*}
is a pullback diagram. For this, let $M$ denote the term in the upper left corner and write
\begin{equation*}
N\coloneqq R\Gamma(\Z_p^\dR, T_\dR(E))[\tfrac{1}{p}]/R\Gamma(\Z_p^{\dR, +}, T_{\dR, +}(E))[\tfrac{1}{p}]\;.
\end{equation*}
\EDIT{Since we already know that} all rows and columns except the leftmost column in the commutative diagram
\begin{equation*}
\begin{tikzcd}
M\ar[r]\ar[d] & R\Gamma(\Z_p^{\dR, +}, T_{\dR, +}(E))[\frac{1}{p}]\ar[r, "1-\tau"]\ar[d] & R\Gamma(\Z_p^\dR, T_\dR(E))[\frac{1}{p}]\ar[d] \\
R\Gamma(\Z_p^\dR, T_\dR(E))[\frac{1}{p}]^{\tau=1}\ar[r]\ar[d] & R\Gamma(\Z_p^\dR, T_\dR(E))[\frac{1}{p}]\ar[r, "1-\tau"]\ar[d] & R\Gamma(\Z_p^\dR, T_\dR(E))[\frac{1}{p}]\ar[d] \\
N\ar[r] & N\ar[r] & 0
\end{tikzcd}
\end{equation*}
are fibre sequences, we conclude that also the left column is a fibre sequence and this finishes the proof.
\end{proof}

\begin{cor}
\EDIT{
Let $X$ be a $p$-adic formal scheme which is smooth and proper over $\Spf\Z_p$. For any perfect $F$-gauge $E\in\Perf(X^\Syn)$, there is a natural fibre sequence
\begin{equation*}
R\Gamma_\Syn(X, E)[\tfrac{1}{p}]\rightarrow \Fil^0_\Hod R\Gamma_\dR(X, T_{\dR, +}(E))[\tfrac{1}{p}]\xrightarrow{1-\tau} R\Gamma_\dR(X, T_\dR(E))[\tfrac{1}{p}]\;,
\end{equation*}
where $\tau$ again denotes the crystalline Frobenius from \cref{rem:beilfibsq-crysfrob}.
}
\end{cor}
\begin{proof}
\EDIT{Using \cref{prop:finiteness-main}, this follows immediately from \cref{thm:beilfibsq-fmrat}.}
\end{proof}

\EDIT{In particular, by the comparison from \cref{thm:fildrstack-comparison}, the above recovers the result of \cref{thm:beilfibsq-fmratmotivation} in the proper case by putting $E=\O\{i\}$.}

\begin{rem}
In fact, the result of \cref{thm:beilfibsq-fmratmotivation} holds already integrally for $i\leq p-2$. This integral version can also be obtained via $F$-gauges as shown in \cite[Rem. 5.21]{FontaineLaffaille} and \cite[Rem. 6.5.15]{FGauges}, hence \cref{thm:beilfibsq-fmrat} admits an integral analogue for perfect $F$-gauges $E$ with Hodge--Tate weights all at least $-(p-2)$.
\end{rem}

\subsection{The Beilinson fibre square as a categorical statement}

Taking the commutative square from \cref{prop:beilfibsq-square} and passing to categories of perfect complexes, we obtain a commutative diagram
\begin{equation*}
\begin{tikzcd}
\Perf(\Z_p^\Syn)\ar[r, "T_\crys"]\ar[d, "T_{\dR, +}"] & \Perf(\F_p^\Syn)\ar[d] \\
\Perf(\Z_p^{\dR, +})\ar[r] & \Perf(\Z_p^\dR)
\end{tikzcd}
\end{equation*}
of stable $\infty$-categories and hence a functor
\begin{equation*}
\Perf(\Z_p^\Syn)\rightarrow \Perf(\Z_p^{\dR, +})\times_{\Perf(\Z_p^\dR)} \Perf(\F_p^\Syn)\;.
\end{equation*}
Passing to isogeny categories and composing with the functors \EDIT{from (\ref{eq:beilfibsq-dfpsynisog}) and (\ref{eq:beilfibsq-da1gmisog}),} we obtain a functor
\begin{equation*}
\Perf(\Z_p^\Syn)[\tfrac{1}{p}]\rightarrow \DF(\Q_p)\times_{\D(\Q_p)} \D(\Mod^\phi(\Q_p))\;.
\end{equation*}
The category on the right-hand side (or, more precisely, a small variant thereof) has been called the \emph{category of $p$-adic Hodge complexes} in \cite{Bannai}. In this section, we will show that the Beilinson fibre square admits the following categorical extension:

\begin{thm}
\label{thm:beilfibsq-categorical}
The functor
\begin{equation*}
\Perf(\Z_p^\Syn)[\tfrac{1}{p}]\rightarrow \DF(\Q_p)\times_{\D(\Q_p)} \D(\Mod^\phi(\Q_p))
\end{equation*}
described above is a fully faithful embedding.
\end{thm}

On our way of proving the above result, \EDIT{we will also obtain the following statement, which we learnt from Bhargav Bhatt in private communication, see also \cite[Rem. 6.7.5]{FGauges}:}

\begin{prop}
\label{prop:beilfibsq-perfzpsyncrys}
The functor $T_\et$ induces an equivalence of categories
\begin{equation*}
\Perf(\Z_p^\Syn)[\tfrac{1}{p}]\cong \D^b(\Rep^\crys_{\Q_p}(G_{\Q_p}))
\end{equation*}
between the isogeny category of perfect complexes on $\Z_p^\Syn$ and the bounded derived category of crystalline $\Q_p$-representations of $G_{\Q_p}$.
\end{prop}
\begin{proof}
To prove full faithfulness, it suffices to show that, for any $E_1, E_2\in\Perf(\Z_p^\Syn)$, we have
\begin{equation*}
\RHom(E_1, E_2)[\tfrac{1}{p}]\cong \RHom_{\MF_\adm^\phi(\Q_p)}(D_1, D_2)\;,
\end{equation*}
where $D_i\coloneqq D_\crys(T_\et(E_i)[\tfrac{1}{p}])$ for $i=1, 2$. \EEDIT{For $E_1=\O$, we already know this: indeed, as all functors in sight commute with shifts and fibres, we may assume that $E_2$ is a coherent sheaf by an application of \cref{lem:beilfibsq-boundedtstructure} and then the claim follows from the proof of \cref{thm:beilfibsq-main}. However, the special case $E_1=\O$ now implies the general case: since} also $\underline{\RHom}(E_1, E_2)$ is a perfect complex, the tensor-hom adjunction yields
\begin{equation*}
\begin{split}
\RHom(E_1, E_2)[\tfrac{1}{p}]&\cong \RHom(\O, \sRHom(E_1, E_2))[\tfrac{1}{p}] \\
&\cong \RHom_{\MF^\phi_\adm(\Q_p)}(\Q_p, D_\crys(T_\et(\sRHom(E_1, E_2))[\tfrac{1}{p}])) \\
&\cong \RHom_{\MF^\phi_\adm(\Q_p)}(\Q_p, \sRHom_{\MF^\phi_\adm(\Q_p)}(D_1, D_2)) \\
&\cong \RHom_{\MF^\phi_\adm(\Q_p)}(D_1, D_2)\;,
\end{split}
\end{equation*}
as desired. Note that we have used that
\begin{equation*}
D_\crys(T_\et(\sRHom(E_1, E_2))[\tfrac{1}{p}])\cong \sRHom_{\MF^\phi_\adm(\Q_p)}(D_1, D_2)\;;
\end{equation*}
\EDIT{this} follows from the fact that $D_\crys$ is \EDIT{a symmetric monoidal equivalence} while the étale realisation is the composition of a pullback along an open immersion with a symmetric monoidal equivalence, \EDIT{both of which preserve internal hom} \EEDIT{between perfect complexes}. Finally, \EDIT{by \cref{lem:beilfibsq-boundedtstructure},} essential surjectivity is immediate from \cref{thm:syntomification-reflcrys} and the fact that the étale realisation preserves fibres \EDIT{(since it commutes with shifts and colimits)}.
\end{proof}

From here, it is quite straightforward to prove \cref{thm:beilfibsq-categorical}.

\begin{proof}[Proof of \cref{thm:beilfibsq-categorical}]
\EDIT{Since $\DF(\Q_p)\cong \D(\MF(\Q_p))$}, all we have is to show is that, for any $E_1, E_2\in\Perf(\Z_p^\Syn)$, the given functor induces a pullback diagram
\begin{equation*}
\begin{tikzcd}
\RHom(E_1, E_2)[\tfrac{1}{p}]\ar[r]\ar[d] & \RHom_{\Mod^\phi(\Q_p)}(T_\crys(E_1)[\tfrac{1}{p}], T_\crys(E_2)[\tfrac{1}{p}])\ar[d] \\
\RHom_{\MF(\Q_p)}(T_{\dR, +}(E_1)[\tfrac{1}{p}], T_{\dR, +}(E_2)[\tfrac{1}{p}])\ar[r] & \RHom_{\Q_p}(T_\dR(E_1)[\tfrac{1}{p}], T_\dR(E_2)[\tfrac{1}{p}])\nospacepunct{\;.}
\end{tikzcd}
\end{equation*}
However, \EEDIT{using \cref{prop:beilfibsq-perfzpsyncrys} and} writing $D_i=D_\crys(T_\et(E_i)[\tfrac{1}{p}])$ for $i=1, 2$ again, it follows from \cref{lem:beilfibsq-etalephimod} and \cref{lem:beilfibsq-etalefiltered} that it is equivalent to prove that the diagram
\begin{equation*}
\begin{tikzcd}
\RHom_{\MF^\phi_\adm(\Q_p)}(D_1, D_2)\ar[r]\ar[d] & \RHom_{\Mod^\phi(\Q_p)}(T_\crys(D_1), T_\crys(D_2))\ar[d] \\
\RHom_{\MF(\Q_p)}(T_{\dR, +}(D_1), T_{\dR, +}(D_2))\ar[r] & \RHom_{\Q_p}(T_\dR(D_1), T_\dR(D_2))
\end{tikzcd}
\end{equation*}
is cartesian. However, this follows from \cref{prop:beilfibsq-crys} using the tensor-hom adjunction and the fact that $T_\crys, T_{\dR, +}$ and $T_\dR$ commute with internal hom.
\end{proof}

\newpage

\section{Comparing syntomic cohomology and étale cohomology}
\label{sect:syntomicetale}

In this section, we prove a comparison theorem between syntomic cohomology and étale cohomology with coefficients in sufficiently negative Hodge--Tate weights. This generalises previous results by Abhinandan, Colmez--Niziol and Tsuji, see \EDIT{\cite[Cor. 1.11]{Abhinandan}}, \cite[Cor. 5.21]{ColmezNiziol} and \cite[Thm. 3.3.4]{Tsuji}. Most notably, while the aforementioned results all concern the rational case, i.e.\ one needs to invert $p$, this is not necessary in our setting and our results hold integrally.

\begin{thm}
\label{thm:syntomicetale-maincoarse}
Let $X$ be a smooth qcqs $p$-adic formal scheme of relative dimension $d$ over $\Spf\Z_p$. For any $F$-gauge $E\in\Perf(X^\Syn)$ with Hodge--Tate weights all at most $-d-2$, there is a natural isomorphism
\begin{equation*}
R\Gamma_\Syn(X, E)\cong R\Gamma_\proet(X_\eta, T_\et(E))\;.
\end{equation*}
\end{thm}

If we restrict to vector bundle $F$-gauges, we can obtain an even finer comparison. However, note that the range of our result differs slightly from those of the results of Colmez--Niziol and Tsuji mentioned above, which is due to the fact that they are working with log structures; see also the discussion in \cite[Rem.s 1.3, 1.12.(i)]{Abhinandan}.

\begin{thm}
\label{thm:syntomicetale-mainfine}
Let $X$ be a smooth qcqs $p$-adic formal scheme. For any vector bundle $F$-gauge $E\in\Vect(X^\Syn)$ with Hodge--Tate weights all at most $-i-1$ for some $i\geq 0$, the natural morphism
\begin{equation*}
R\Gamma_\Syn(X, E)\rightarrow R\Gamma_\proet(X_\eta, T_\et(E))
\end{equation*}
induces an isomorphism
\begin{equation*}
\tau^{\leq i} R\Gamma_\Syn(X, E)\cong\tau^{\leq i} R\Gamma_\proet(X_\eta, T_\et(E))
\end{equation*}
and an injection on $H^{i+1}$.
\end{thm}

\comment{
In the geometric case, i.e.\ if $X$ admits a morphism to $\Spf\O_C$, we can obtain the comparison above in an even larger range. In particular, using $p\geq 2$, this recovers the result from \cite[Thm. 10.1]{THHandPAdicHodgeTheory}.

\begin{thm}
\label{thm:syntomicetale-mainfineoc}
Let $X$ be a smooth qcqs $p$-adic formal scheme over $\Spf\O_C$. For any vector bundle $F$-gauge $E\in\Vect(X^\Syn)$ with Hodge--Tate weights all at most $-i+p-2$ for some $i\geq 0$, the natural morphism
\begin{equation*}
R\Gamma_\Syn(X, E)\rightarrow R\Gamma_\proet(X_\eta, T_\et(E))
\end{equation*}
induces an isomorphism
\begin{equation*}
\tau^{\leq i} R\Gamma_\Syn(X, E)\cong\tau^{\leq i} R\Gamma_\proet(X_\eta, T_\et(E))
\end{equation*}
and an injection on $H^{i+1}$.
\end{thm}

As in the arithmetic case, one can also formulate a version that works for arbitrary perfect complexes on $X^\Syn$ instead of just vector bundles:

\begin{thm}
\label{thm:syntomicetale-maincoarseoc}
Let $X$ be a smooth qcqs $p$-adic formal scheme of relative dimension $d$ over $\Spf\O_C$. For any $F$-gauge $E\in\Perf(X^\Syn)$ with Hodge--Tate weights all at most $-d+p-2$ for some $i\geq 0$, there is a natural isomorphism
\begin{equation*}
R\Gamma_\Syn(X, E)\cong R\Gamma_\proet(X_\eta, T_\et(E))\;.
\end{equation*}
\end{thm}
}

\subsection{The reduced locus of $X^\Syn$}
\label{subsect:syntomicetale-red}

To prove \cref{thm:syntomicetale-maincoarse} and \cref{thm:syntomicetale-mainfine}, we will define and describe a certain locus inside $X^\N$ and $X^\Syn$, similarly to what we have done for the reduced locus of $\Z_p^\N$ and $\Z_p^\Syn$. For the rest of this section, let $X$ be a bounded $p$-adic formal scheme.

\begin{defi}
The \emph{reduced locus} $X_\red^\Syn$ of $X^\Syn$ is defined as the pullback
\begin{equation*}
\begin{tikzcd}
X^\Syn_\red\ar[d]\ar[r] & X^\Syn\ar[d] \\
\Z_{p, \red}^\Syn\ar[r] & \Z_p^\Syn\nospacepunct{\;.}
\end{tikzcd}
\end{equation*}
Similarly, we define the reduced locus $X_\red^\N$ of $X^\N$.
\end{defi}

Alternatively, one may describe the reduced locus of $X^\Syn$ as the locus cut out by the equations $p=v_{1, X}=0$, where $v_{1, X}\in H^0(X^\Syn, \O\{p-1\}/p)$ denotes the pullback of the section $v_1\in H^0(\Z_p^\Syn, \O\{p-1\}/p)$ to $X^\Syn$. Then $X_\red^\N$ can also be obtained as the pullback of $X_\red^\Syn$ to $X^\N$.

To pass between cohomology on $X_\red^\Syn$ and $X^\Syn$, we again introduce a syntomic filtration similarly to the case $X=\Spf\Z_p$ \EEDIT{treated in \cref{defi:syntomification-fil}:}

\begin{defi}
Let \EEDIT{$E\in\D((X^\Syn)_{p=0})$}. The filtration
\begin{equation*}
\begin{tikzcd}[column sep=scriptsize]
\dots\ar[r, "v_{1, X}"] & R\Gamma(X^\Syn, E\{-(p-1)\})\ar[r, "v_{1, X}"] & R\Gamma(X^\Syn, E)\ar[r, "v_{1, X}"] & R\Gamma(X^\Syn, E\{p-1\})\ar[r, "v_{1, X}"] & \dots  
\end{tikzcd}
\end{equation*}
is called the \emph{syntomic filtration} and denoted $\Fil_\bullet^\Syn R\Gamma(X^\Syn, E)[\frac{1}{v_1}]$. We denote the underlying unfiltered object by $R\Gamma(X^\Syn, E)[\frac{1}{v_1}]$.
\end{defi}

\EDIT{Note that, similarly to (\ref{eq:syntomification-grsyn}), for any \EEDIT{$E\in\D((X^\Syn)_{p=0})$}, there is a canonical isomorphism 
\begin{equation}
\label{eq:syntomicetale-grsyn}
\gr^\Syn_\bullet R\Gamma(X^\Syn, E)[\tfrac{1}{v_1}]\cong R\Gamma(X_\red^\Syn, E/v_{1, X}\{\bullet(p-1)\})\;,
\end{equation}
where $E/v_{1, X}\coloneqq\cofib(E\{-(p-1)\}\xrightarrow{v_{1, X}} E)$. Moreover, as in the case $X=\Spf\Z_p$, \EEDIT{see (\ref{eq:syntomification-synet})}, we have the following result relating the syntomic filtration to étale cohomology:
}

\begin{prop}
\label{prop:syntomicetale-filtrationetale}
Let $X$ be a bounded \EEDIT{qcqs} $p$-adic formal scheme. For any \EEDIT{$E\in\D((X^\Syn)_{p=0})$}, \EDIT{the syntomic filtration is complete}. \EEDIT{If $X$ is moreover $p$-quasisyntomic and $E\in\Perf((X^\Syn)_{p=0})$, then} there is a natural isomorphism
\begin{equation*}
R\Gamma(X^\Syn, E)[\tfrac{1}{v_1}]\cong R\Gamma_\proet(X_\eta, T_\et(E))\;.
\end{equation*}
\end{prop}
\begin{proof}
\EDIT{We first claim that the section $v_{1, X}$ is topologically nilpotent, i.e.\ its pullback to any scheme $\Spec S\rightarrow X^\Syn$ is nilpotent. To prove this, first note that, as $v_{1, X}$ arises from $v_1$ via pullback, it suffices to prove this for $X=\Spf\Z_p$ and subsequently we may reduce to the case $X=\Spf R$ for a $p$-torsionfree quasiregular semiperfectoid ring $R$ using quasisyntomic descent and \cref{lem:filprism-quasisyntomiccover}. However, here we know from \cref{thm:filprism-qrsp} that
\begin{equation*}
R^\N\cong \Spf(\Rees(\Fil^\bullet_\N\Prism_R))/\G_m\;,
\end{equation*}
where $\Rees(\Fil^\bullet_\N\Prism_R)$ is endowed with the $(p, I)$-adic topology, and, after pullback to $R^\N$, the section $v_{1, R}$ is given by the composition
\begin{equation*}
\begin{split}
\Prism_R/p\cong I^{-1}/p\tensor_{\Prism_R/p} I/p&\rightarrow I^{-1}/p\tensor_{\Prism_R/p} \Fil^p_\N\Prism_R/p \\
&\rightarrow I^{-1}/p\tensor_{\Prism_R/p} \Fil^{p-1}_\N\Prism_R/p\cong\Fil^{p-1}_\N\Prism_R\{p-1\}/p
\end{split}
\end{equation*}
by construction. However, this implies that $v_{1, R}^p$ factors through a map
\begin{equation*}
\begin{split}
\Prism_R/p\rightarrow I^{-(p-1)}/p\tensor_{\Prism_R/p} \Fil^{p(p-1)}_\N\Prism_R/p\cong I/p\tensor_{\Prism_R/p} \Fil^{p(p-1)}_\N\Prism_R\{p(p-1)\}/p
\end{split}
\end{equation*}
and hence $v_{1, R}$ is topologically nilpotent (recall that we are working in the $I$-adic topology), as claimed. This now shows that the syntomic filtration is \EDIT{complete}: Indeed, as the global sections functor commutes with limits, it suffices to show that the limit of the diagram
\begin{equation}
\label{eq:syntomicetale-synfilcomp}
\begin{tikzcd}[ampersand replacement=\&]
\dots\ar[r, "v_{1, X}"] \& E\{-(p-1)\}\ar[r, "v_{1, X}"] \& E\ar[r, "v_{1, X}"] \& E\{p-1\}\ar[r, "v_{1, X}"] \& \dots  
\end{tikzcd}
\end{equation}
vanishes. \EEDIT{However, for any $E'\in\D(\Z_p^\Syn)$ and any family of maps $\{E'\rightarrow E\{\bullet(p-1)\}\}$ which is compatible with the transition maps in the diagram (\ref{eq:syntomicetale-synfilcomp}), each of the maps $E'\rightarrow E\{\bullet(p-1)\}$ becomes zero after pullback to any scheme $\Spec S\rightarrow X^\Syn$ since the limit of the pullback of the diagram (\ref{eq:syntomicetale-synfilcomp}) to $\Spec S$ vanishes by topological nilpotency of $v_{1, X}$. Thus, each of the maps $E'\rightarrow E\{\bullet(p-1)\}$ must have been zero to begin with and this shows that the limit of the diagram (\ref{eq:syntomicetale-synfilcomp}) is zero, as desired.}

To prove the second statement, first observe that there is a natural map
\begin{equation*}
R\Gamma(X^\Syn, E)[\tfrac{1}{v_1}]\rightarrow R\Gamma_\proet(X_\eta, T_\et(E))
\end{equation*}
as the étale realisation carries $v_{1, X}$ to an isomorphism by \cite[Ex. 6.3.3]{FGauges}. To show that this map is an isomorphism, \cref{lem:syntomicetale-colimtot} allows us to reduce to the case where $X=\Spf R$ is quasiregular semiperfectoid using quasisyntomic descent and \cref{lem:filprism-quasisyntomiccover}. However, in this case, the claim follows from the description of the étale realisation functor from \cite[Constr. 6.3.1]{FGauges}. 
}
\end{proof}

\subsection{The components of the reduced locus of $X^\Syn$}
\label{subsect:syntomicetale-compred}

We now explain how to obtain the reduced locus $X^\Syn_\red$ by gluing several components which have a fairly explicit description.  \EDIT{Before proceeding, we encourage the reader to revisit the description of the reduced locus $\Z_{p, \red}^\Syn$ of $\Z_p^\Syn$ given towards the end of Section \ref{subsect:syntomification} as the case $X=\Spf\Z_p$ will be prototypical.} We let $Y\coloneqq X_{p=0}$ denote the special fibre of $X$.

\subsubsection{The de Rham and Hodge--Tate components}

\begin{defi}
The \emph{de Rham component} $X^\N_\dR$ of $X^\N$ is defined as the pullback
\begin{equation*}
\begin{tikzcd}
X^\N_\dR\ar[r]\ar[d] & X^\N\ar[d] \\
\Z_{p, \dR}^\N\ar[r] & \Z_p^\N\nospacepunct{\;.}
\end{tikzcd}
\end{equation*}
\end{defi}

Note that, by definition of $\Z_{p, \dR}^\N$, there is an fpqc cover $\Spec\F_p\rightarrow\Z_{p, \dR}^\N$ given by the filtered Cartier--Witt divisor
\begin{equation*}
\G_a^\sharp\oplus F_*W\xrightarrow{(\incl, V)} W
\end{equation*}
on $\Spec\F_p$. Due to the equivalence
\begin{equation*}
\Cone(\G_a^\sharp\oplus F_*W\xrightarrow{(\incl, V)} W)\cong\Cone(\G_a^\sharp\xrightarrow{\can}\G_a)
\end{equation*}
from \EDIT{(\ref{eq:filprism-drmap})}, we see that this cover induces a pullback square
\begin{equation}
\label{eq:syntomicetale-coverdr}
\begin{tikzcd}
(X^\dR)_{p=0}\ar[r]\ar[d] & X^\N_\dR\ar[d] \\
\Spec\F_p\ar[r] & \Z_{p, \dR}^\N\nospacepunct{\;.}
\end{tikzcd}
\end{equation}

\begin{ex}
\label{ex:syntomicetale-perfddr}
\EEDIT{
We claim that, for a perfectoid ring $R$, there is an isomorphism
\begin{equation*}
R^\N_\dR\cong \Spec R/p
\end{equation*}
compatible with the isomorphism 
\begin{equation*}
R^\N\cong \Spec(A_\inf(R)\langle u, t\rangle/(ut-\phi^{-1}(\xi)))/\G_m
\end{equation*}
from \cref{ex:filprism-perfd}. Indeed, recall that $v_1=u^pt$ from \cref{ex:syntomification-v1perfd} and that we have $\Z_{p, \dR}^\N=j_\dR(\Z_p^\prism)\cap \Z_{p, \red}^\N$. As $j_\dR(\Z_p^\prism)$ identifies with the locus $\{t\neq 0\}$ in $\Z_p^\N$ by \cref{rem:filprism-jdr}, we obtain
\begin{equation*}
\begin{split}
R^\N_\dR=(R^\N_\red)_{t\neq 0}&\cong \Spec(A_\inf(R)\langle u, t, t^{-1}\rangle/(ut-\phi^{-1}(\xi), p, u^pt))/\G_m \\
&\cong \Spec(A_\inf(R)\langle t, t^{-1}\rangle/(p, \phi^{-1}(\xi)^p))/\G_m \\
&\cong \Spec(A_\inf(R)\langle t, t^{-1}\rangle/(p, \xi))/\G_m \\
&\cong \Spec((R/p)[t, t^{-1}])/\G_m\cong\Spec(R/p)\;,
\end{split}
\end{equation*}
as desired.}
\end{ex}

\begin{rem}
Note that, if $X$ is smooth and qcqs, using (\ref{eq:syntomicetale-coverdr}) \EEDIT{and \cref{rem:drstack-vect}}, one can give an explicit description of vector bundles on $X^\N_\dR$ in terms of vector bundles on $Y$ with a flat connection \EEDIT{having locally nilpotent $p$-curvature} equipped with an additional Sen operator $\Theta$ satisfying certain compatibilities.
\end{rem}

Although we will not need it in the sequel, we also shortly want to discuss the Hodge--Tate component $X^\N_\HT$ defined as the pullback
\begin{equation*}
\begin{tikzcd}
X^\N_\HT\ar[r]\ar[d] & X^\N\ar[d] \\
\Z_{p, \HT}^\N\ar[r] & \Z_p^\N\nospacepunct{\;.}
\end{tikzcd}
\end{equation*}
To this end, recall that $\Z_{p, \HT}^\N$ can be described as \EEDIT{$j_\HT(\Z_p^\prism)\cap\Z_{p, \HT, c}^\N=j_\HT(\Z_p^\prism)\cap (\Z_p^\N)_{p=t=0}$} and that the map $j_\HT$ sends a Cartier--Witt divisor $I\xrightarrow{\alpha} W(S)$ on $S$ to the filtered Cartier--Witt divisor $I\tensor_{W(S)} W\xrightarrow{\alpha} W$ on $S$. In this situation, we have
\begin{equation*}
R\Gamma_\fl(S, \ol{W})\cong \ol{W(S)}
\end{equation*}
and thus we conclude that $X^\N_\HT$ identifies with $(X^\HT)_{p=0}$. \EEDIT{Recalling \cref{defi:fildhod-dhod}}, this means that there is a pullback square
\begin{equation}
\label{eq:syntomicetale-coverht}
\begin{tikzcd}
(X^\dHod)_{p=0}\ar[r]\ar[d] & X^\N_\HT\ar[d] \\
\Spec\F_p\ar[r] & \Z_{p, \HT}^\N\nospacepunct{\;,}
\end{tikzcd}
\end{equation}
where the lower map corresponds to the filtered Cartier--Witt divisor $W\xrightarrow{p} W$ over $\F_p$.

\begin{ex}
\label{ex:syntomicetale-perfdht}
\EEDIT{
We claim that, for a perfectoid ring $R$, there is an isomorphism
\begin{equation*}
R^\N_\HT\cong \Spec \phi^*(R/p)
\end{equation*}
compatible with the isomorphism from \cref{ex:filprism-perfd}, where $\phi$ denotes the Frobenius on $A_\inf(R)$. Indeed, recall $v_1=u^pt$ from \cref{ex:syntomification-v1perfd} and that $\Z_{p, \HT}^\N=j_\HT(\Z_p^\prism)\cap \Z_{p, \red}^\N$, hence $R_\HT^\N=j_\HT(R^\prism)\cap R_\red^\N$. As $j_\HT(R^\prism)$ identifies with the locus $\{u\neq 0\}$ in $R^\N$ by \cref{rem:filprism-jdrjhtperfd}, we obtain
\begin{equation*}
\begin{split}
R^\N_\HT=(R^\N_\red)_{u\neq 0}&\cong \Spec(A_\inf(R)\langle u, u^{-1}, t\rangle/(ut-\phi^{-1}(\xi), p, u^pt))/\G_m \\
&\cong \Spec(A_\inf(R)\langle u, u^{-1}\rangle/(p, \phi^{-1}(\xi)))/\G_m \\
&\cong \Spec (\phi^*(R/p)[t, t^{-1}])/\G_m\cong \Spec \phi^*(R/p)\;,
\end{split}
\end{equation*}
as desired.}
\end{ex}

\subsubsection{The Hodge-filtered de Rham component}

\begin{defi}
The \emph{Hodge-filtered de Rham component} $X^\N_{\dR, +}$ of $X^\N$ is defined as the pullback
\begin{equation*}
\begin{tikzcd}
X^\N_{\dR, +}\ar[r]\ar[d] & X^\N\ar[d] \\
\Z_{p, \dR, +}^\N\ar[r] & \Z_p^\N\nospacepunct{\;.}
\end{tikzcd}
\end{equation*}
\end{defi}

Note that, by definition of \EDIT{$\Z_{p, \dR, +}^\N$}, there is an fpqc cover \EDIT{$\A^1/\G_m\rightarrow\Z_{p, \dR, +}^\N$} given by the filtered Cartier--Witt divisor
\begin{equation*}
\EEDIT{
\V(\O(1))^\sharp\oplus F_*W\xrightarrow{(t^\sharp, V)} W
}
\end{equation*}
on $\A^1/\G_m$. Due to the equivalence
\begin{equation*}
\EEDIT{
\Cone(\V(\O(1))^\sharp\oplus F_*W\xrightarrow{(t^\sharp, V)} W)\cong \Cone(\V(\O(1))^\sharp\xrightarrow{t^\sharp} \G_a)
}
\end{equation*}
from \EDIT{(\ref{eq:filprism-drmap}),} we see that this cover induces a pullback square
\begin{equation}
\label{eq:syntomicetale-coverdr+}
\begin{tikzcd}
(X^{\dR, +})_{p=0}\ar[r]\ar[d] & X^\N_{\dR, +}\ar[d] \\
\A^1/\G_m\ar[r] & \Z_{p, \dR, +}^\N\nospacepunct{\;.}
\end{tikzcd}
\end{equation}

\begin{ex}
\label{ex:syntomicetale-perfddr+}
\EEDIT{
We claim that, for a perfectoid ring $R$, there is an isomorphism
\begin{equation*}
R^\N_{\dR, +}\cong \Spf(A_\inf(R)\langle u, t\rangle/(ut-\phi^{-1}(\xi), p, u^p))/\G_m
\end{equation*}
compatible with the isomorphism from \cref{ex:filprism-perfd}. Indeed, recall that $\Z_{p, \dR, +}^\N$ is the closed substack of $(\Z_p^\N)_{p=0}$ parametrising locally split filtered Cartier--Witt divisors. Tracing through the argument of \cref{ex:filprism-perfd} and using that $\sExt_W(F_*W, \G_a^\sharp)\cong\Cone(\G_a^\sharp\rightarrow\G_a)$, see \cite[Prop. 5.2.1]{FGauges}, we see that the filtered Cartier--Witt divisor constructed in \cref{ex:filprism-perfd} is locally split if and only if $u$ admits divided powers and, since we are working over $\F_p$ now, this means demanding that $u^p=0$. Thus, we conclude that
\begin{equation*}
R^\N_{\dR, +}=(R^\N)_{p=u^p=0}\cong \Spf(A_\inf(R)\langle u, t\rangle/(ut-\phi^{-1}(\xi), p, u^p))/\G_m\;,
\end{equation*}
as desired.}
\end{ex}

\begin{rem}
If $X$ is smooth and qcqs, using (\ref{eq:syntomicetale-coverdr+}) \EEDIT{and \cref{rem:fildrstack-vect}}, one can describe vector bundles on $X_{\dR, +}^\N$ explicitly in terms of filtered vector bundles on $Y$ equipped with a flat connection satisfying Griffiths transversality and a Sen operator.
\end{rem}

\subsubsection{The conjugate-filtered Hodge--Tate component}

\begin{defi}
The \emph{conjugate-filtered Hodge--Tate component} $X^\N_{\HT, c}$ of $X^\N$ is defined as the pullback
\begin{equation*}
\begin{tikzcd}
X^\N_{\HT, c}\ar[r]\ar[d] & X^\N\ar[d] \\
\Z_{p, \HT, c}^\N\ar[r] & \Z_p^\N\nospacepunct{\;.}
\end{tikzcd}
\end{equation*}
\end{defi}

Recall \EEDIT{from \cref{rem:fildhod-pullback}} that there is an fpqc cover $\A^1/\G_m\rightarrow\Z_{p, \HT, c}^\N$ \EEDIT{which induces} a pullback square
\begin{equation*}
\begin{tikzcd}
(X^{\dHod, c})_{p=0}\ar[r]\ar[d] & X^\N_{\HT, c}\ar[d] \\
\A^1/\G_m\ar[r] & \Z_{p, \HT, c}^\N\nospacepunct{\;.}
\end{tikzcd}
\end{equation*}

\begin{ex}
\label{ex:syntomicetale-perfdhtc}
\EEDIT{
We claim that, for a perfectoid ring $R$, there is an isomorphism
\begin{equation*}
R^\N_{\HT, c}\cong \Spec(\phi^*(R/p)[u])/\G_m
\end{equation*}
compatible with the isomorphism from \cref{ex:filprism-perfd}, where $\phi$ denotes the Frobenius of $A_\inf(R)$. Indeed, recall that $\Z_{p, \HT, c}^\N$ identifies with the locus $\{p=t=0\}$ in $\Z_p^\N$ and hence
\begin{equation*}
\begin{split}
R^\N_{\HT, c}=(R^\N)_{p=t=0}&\cong \Spf(A_\inf(R)\langle u, t\rangle/(ut-\phi^{-1}(\xi), p, t))/\G_m \\
&\cong \Spf(A_\inf(R)\langle u\rangle/(\phi^{-1}(\xi), p))/\G_m\cong \Spec(\phi^*(R/p)[u])/\G_m\;,
\end{split}
\end{equation*}
as desired.}
\end{ex}

\subsubsection{The Hodge component}
\label{sect:hodcomp}

\begin{defi}
The \emph{Hodge component} $X^\N_\Hod$ of $X^\N$ is defined as the pullback
\begin{equation*}
\begin{tikzcd}
X^\N_\Hod\ar[r]\ar[d] & X^\N\ar[d] \\
\Z_{p, \Hod}^\N\ar[r] & \Z_p^\N\nospacepunct{\;.}
\end{tikzcd}
\end{equation*}
\end{defi}

Note that, by definition of $\Z_{p, \Hod}^\N$, there is an fpqc cover $B\G_m\rightarrow\Z_{p, \Hod}^\N$ given by the filtered Cartier--Witt divisor
\begin{equation*}
\EEDIT{
\V(\O(1))^\sharp\oplus F_*W\xrightarrow{(0, V)} W
}
\end{equation*}
on $B\G_m$. Due to the equivalence
\begin{equation*}
\EEDIT{
\Cone(\V(\O(1))^\sharp\oplus F_*W\xrightarrow{(0, V)} W)\cong \G_a\oplus B\V(\O(1))^\sharp
}
\end{equation*}
from \EDIT{(\ref{eq:filprism-drmap})}, we see that this cover induces a pullback square
\begin{equation}
\label{eq:syntomicetale-coverhod}
\begin{tikzcd}
(X^\Hod)_{p=0}\ar[r]\ar[d] & X^\N_\Hod\ar[d] \\
B\G_m\ar[r] & \Z_{p, \Hod}^\N\nospacepunct{\;.}
\end{tikzcd}
\end{equation}

\begin{ex}
\label{ex:syntomicetale-perfdhod}
\EEDIT{
We claim that, for a perfectoid ring $R$, there is an isomorphism
\begin{equation*}
R^\N_\Hod\cong \Spec(\phi^*(R/p)[u]/(u^p))/\G_m
\end{equation*}
compatible with the isomorphism from \cref{ex:filprism-perfd}, where $\phi$ denotes the Frobenius of $A_\inf(R)$. Indeed, recall that $\Z_{p, \Hod}^\N$ is the intersection of $\Z_{p, \dR, +}^\N$ and $\Z_{p, \HT, c}$ and that $R^\N_{\dR, +}=(R^\N)_{p=u^p=0}$ and $R^\N_{\HT, c}=(R^\N)_{p=t=0}$ by \cref{ex:syntomicetale-perfddr+} and \cref{ex:syntomicetale-perfdhtc}. Thus, we have
\begin{equation*}
\begin{split}
R^\N_\Hod=(R^\N)_{p=t=u^p=0}&\cong \Spf(A_\inf(R)\langle u, t\rangle/(ut-\phi^{-1}(\xi), p, t, u^p))/\G_m \\
&\cong \Spf(A_\inf(R)\langle u\rangle/(\phi^{-1}(\xi), p, u^p))/\G_m\cong \Spec(\phi^*(R/p)[u]/(u^p))/\G_m\;,
\end{split}
\end{equation*}
as desired.
}
\end{ex}

For our analysis of the syntomic cohomology of $X$, it will be vital to have a good understanding of vector bundles on $X_\Hod^\N$ and their cohomology. For this, we will first need to recall some facts about Higgs bundles. Thus, let $Y$ be a scheme or formal scheme which is smooth over a base \EDIT{(formal) scheme} $S$. For us, the case $S=\Spec\F_p$ will be the most relevant as we will want to apply the following definitions in our context of $Y$ being the the special fibre of some smooth $p$-adic formal scheme $X$.

\begin{defi}
Let $V\in\Vect(Y)$ be a vector bundle on $Y$. A \emph{Higgs field} on \EEDIT{$V$} is an $\O_Y$-linear morphism $\Phi: V\rightarrow V\tensor\Omega_{Y/S}^1$ such that $\Phi\wedge\Phi=0$. Here, $\Phi\wedge\Phi: V\rightarrow V\tensor\Omega_{Y/S}^2$ denotes the composition
\begin{equation*}
V\xrightarrow{\mathrlap{\hspace{0.25cm}\Phi}\hphantom{\Phi\tensor\id}} V\tensor\Omega^1_{Y/S}\xrightarrow{\Phi\tensor\id} V\tensor\Omega^1_{Y/S}\tensor\Omega^1_{Y/S}\xrightarrow{\id\tensor\wedge}V\tensor\Omega^2_{Y/S}\;,
\end{equation*}
\EDIT{where $\wedge: \Omega^1_{Y/S}\tensor\Omega^1_{Y/S}\rightarrow\Omega^1_{Y/S}\wedge\Omega^1_{Y/S}\cong \Omega^2_{Y/S}$ is the canonical map.}
\end{defi}

From a Higgs field $\Phi$ on some $V\in\Vect(Y)$, we obtain induced morphisms $\Phi: V\tensor\Omega_{Y/S}^i\rightarrow V\tensor\Omega_{Y/S}^{i+1}$ given by
\begin{equation*}
V\tensor\Omega_{Y/S}^i\xrightarrow{\Phi\tensor\id}V\tensor\Omega_{Y/S}^1\tensor\Omega_{Y/S}^i\xrightarrow{\id\tensor\wedge} V\tensor\Omega_{Y/S}^{i+1}
\end{equation*}
and thus a variant of the Hodge complex 
\begin{equation*}
V\xrightarrow{\Phi} V\tensor\Omega^1_{Y/S}\xrightarrow{\Phi} V\tensor\Omega^2_{Y/S}\xrightarrow{\Phi}\dots\;.
\end{equation*}

By the tensor-hom adjunction, for any vector bundle $V$ on $Y$, an $\O_Y$-linear morphism $\Phi: V\rightarrow V\tensor\Omega_{Y/S}^1$ is equivalent to the data of an $\O_Y$-linear morphism
\begin{equation*}
\Phi_{(-)}: \T_{Y/S}\rightarrow \sEnd_{\O_Y}(V)\;,
\end{equation*}
which can explicitly be described as sending any local section $\theta\in\T_{Y/S}(U)$ to the endomorphism
\begin{equation*}
\Phi_\theta: V|_U\xrightarrow{\mathrlap{\hspace{0.3cm}\Phi}\hphantom{\id\tensor\theta^\vee}} (V\tensor\Omega^1_{Y/S})|_U\xrightarrow{\id\tensor\theta^\vee} V|_U\;;
\end{equation*}
\EDIT{here, $\T_{Y/S}$ denotes the tangent bundle of $Y$ over $S$, as usual.} One can then check that the condition $\Phi\wedge\Phi=0$ is exactly equivalent to requiring that $\Phi_\theta$ and $\Phi_{\theta'}$ commute for any local sections $\theta, \theta'\in\T_{Y/S}(U)$. Thus, we conclude that equipping $V$ with a Higgs field is equivalent to equipping $V$ with the structure of a $\Sym^\bullet_{\O_Y}\T_{Y/S}$-module, \EDIT{where we remind the reader that $\Sym^\bullet_{\O_Y}\T_{Y/S}$ denotes the symmetric algebra of the $\O_Y$-module $\T_{Y/S}$}. Finally, we say that a Higgs field $\Theta$ on $V\in\Vect(Y)$ is \emph{locally nilpotent} if for any local sections $\theta\in\T_{Y/S}(U), s\in V(U)$, we have $\Phi_\theta^n(s)=0$ for some $n>0$.

Finally, we note that the above definitions immediately generalise to the case of arbitrary quasi-coherent complexes $V\in\D(Y)$ instead of vector bundles: A Higgs bundle on $Y$ is an $\O_Y$-linear morphism $\Phi: V\tensor\Omega_{Y/S}^1\rightarrow\Omega_{Y/S}^1$ and this is equivalent to the datum of equipping $V$ with the structure of a $\Sym^\bullet_{\O_Y}\T_{Y/S}$-complex; it is locally nilpotent if for any local section $\theta\in\T_{Y/S}(U)$ and any map $s: \O_U\rightarrow V|_U$, we have $\Phi_\theta^n\circ s=0$ for some $n>0$. Again, we obtain an analogue of the Hodge complex given by the total complex
\begin{equation*}
\Tot(V\xrightarrow{\Phi} V\tensor\Omega^1_{Y/S}\xrightarrow{\Phi} V\tensor\Omega^2_{Y/S}\xrightarrow{\Phi}\dots)\;.
\end{equation*}

We now return to our particular case of $Y=X_{p=0}$ for a smooth $p$-adic formal scheme $X$. Using the pullback square (\ref{eq:syntomicetale-coverhod}), it turns out that the main part of the work lies in analysing vector bundles on $(X^\Hod)_{p=0}$ and their cohomology.

\begin{lem}
\label{lem:syntomicetale-identifyhod}
Let $X$ be a smooth qcqs $p$-adic formal scheme \EDIT{and put $Y=X_{p=0}$}. Then there is an isomorphism
\begin{equation*}
\EEDIT{
(X^\Hod)_{p=0}\cong B_{Y\times B\G_m}\V(\T_{Y/\F_p}(1))^\sharp
}
\end{equation*}
over $Y\times B\G_m$.
\end{lem}
\begin{proof}
First recall from \cref{lem:gasharp-acyclic} that there is a basis of the fpqc topology consisting of $\G_a^\sharp$-acyclic rings. For such a ring $S$ of characteristic $p$, we have
\begin{equation*}
\EEDIT{
(X^\Hod)_{p=0}(S)=Y(S\oplus \V(\O(1))^\sharp(S)[1])
}
\end{equation*}
by definition of $X^\Hod$. As \EEDIT{$S\oplus\V(\O(1))^\sharp(S)[1]\rightarrow S$} is a square-zero extension, the same deformation-theoretic argument as in the proof of \cref{lem:fildhod-comparisonfilteredsmooth} shows that the induced map $(X^\Hod)_{p=0}\rightarrow Y\times B\G_m$ is a \EEDIT{$\V(\T_{Y/\F_p}(1))^\sharp$-gerbe}. However, as the square-zero extension above splits via the canonical injection, this gerbe is also split and hence trivial.
\end{proof}

Due to the description of $(X^\Hod)_{p=0}$ provided by the lemma above, we will make use of the following general statement about $\V(E)^\sharp$-representations for a vector bundle $E$:

\begin{lem}
\label{lem:syntomicetale-vesharpreps}
Let $X$ be a scheme and $E\in\Vect(X)$ a vector bundle on $X$. Then there is a natural equivalence
\begin{equation*}
\D(B\V(E)^\sharp)\cong\D(\widehat{\V(E^\vee)})
\end{equation*}
\EEDIT{of stable $\infty$-categories} compatible with the forgetful functors on both sides; here, $\widehat{\V(E^\vee)}$ denotes the completion of $\V(E^\vee)$ at the zero section. Moreover, if $V$ is a quasi-coherent complex on $B\V(E)^\sharp$ and $\pi: B\V(E)^\sharp\rightarrow X$ is the projection, then
\begin{equation*}
\pi_*V\cong \Tot(V\rightarrow V\tensor E^\vee\rightarrow V\tensor \wedge^2 E^\vee\rightarrow\dots)\;,
\end{equation*}
where the maps are induced by the natural $\Sym^\bullet_{\O_X} E$-module structure on $V$ via the equivalence above. Note that we abuse notation and also use $V$ to denote the pullback of $V$ to $X$.
\end{lem}
\begin{proof}
The first assertion is \cite[Prop. 2.4.5]{FGauges}. For the second one, we reduce to the case where $X=\Spec R$ is affine. Then we have
\begin{equation*}
\pi_*V=R\Gamma(B\V(E)^\sharp, V)\cong\RHom_{B\V(E)^\sharp}(\O, V)\cong \RHom_{\Sym^\bullet_R E}(R, V)\;,
\end{equation*}
where the action of $E$ on $R$ is trivial in the rightmost term. Using the Koszul resolution
\begin{equation*}
(\dots\rightarrow\wedge^2 E\tensor_R \Sym^\bullet_R E\rightarrow E\tensor_R \Sym^\bullet_R E\rightarrow\Sym^\bullet_R E)\rightarrow R
\end{equation*}
to compute the RHom above, we obtain the result.
\end{proof}

\begin{prop}
\label{prop:syntomicetale-vecthod}
Let $X$ be a smooth qcqs $p$-adic formal scheme \EDIT{and put $Y=X_{p=0}$}. Then giving a quasi-coherent complex on $(X^\Hod)_{p=0}$ amounts to specifying a graded quasi-coherent complex $V=\bigoplus_i V_i$ on $Y$ equipped with a Higgs field $\Phi: V\rightarrow V\tensor\Omega_{Y/\F_p}^1$ which is locally nilpotent and decreases degree by $1$, i.e.\ the restriction of $\Phi$ to $V_i$ comes with a factorisation through $V_{i-1}\tensor\Omega_{Y/\F_p}^1$.
\end{prop}
\begin{proof}
This now follows by combining \cref{lem:syntomicetale-identifyhod} \EDIT{with the $\G_m$-equivariant version of} \cref{lem:syntomicetale-vesharpreps}.
\end{proof}

\begin{prop}
\label{prop:syntomicetale-cohomologyhod}
Let $X$ be a smooth qcqs $p$-adic formal scheme \EDIT{and put $Y=X_{p=0}$}. If $E\in\D((X^\Hod)_{p=0})$ corresponds to the data $(V_\bullet, \Phi)$ via \cref{prop:syntomicetale-vecthod}, then the pushforward of $E$ to $B\G_m$ identifies with the graded complex whose degree $i$ term is given by
\begin{equation*}
R\Gamma(Y, \Tot(V_i\xrightarrow{\Phi}V_{i-1}\tensor\Omega_{Y/\F_p}^1\xrightarrow{\Phi} V_{i-2}\tensor\Omega_{Y/\F_p}^2\xrightarrow{\Phi}\dots))\;.
\end{equation*}
\end{prop}
\begin{proof}
Immediate from \cref{lem:syntomicetale-identifyhod} and \EDIT{the $\G_m$-equivariant version of} \cref{lem:syntomicetale-vesharpreps}.
\end{proof}

Finally, let us explain how to obtain $X^\N_\red$ and $X^\Syn_\red$ from the components above by gluing: From the analogous statements for $X=\Spf\Z_p$, we see that gluing $X_{\dR, +}^\N$ and $X_{\HT, c}^\N$ along $X_\Hod^\N$ yields the stack $X^\N_\red$. To obtain $X^\Syn_\red$, we then have to further glue $X_\dR^\N$ and $X_\HT^\N$ along a Frobenius twist.

\subsection{Proof of the main theorems}

We now move on to proving our main statements concerning the comparison between syntomic cohomology and étale cohomology announced in the beginning. As in the case $X=\Spf\Z_p$, we use the notation $F_\dR(-), F_\Hod(-), F_{\dR, +}(-)$ and $F_{\HT, c}(-)$ to denote cohomology on the components of $X^\Syn_\red$.

\begin{proof}[Proof of \cref{thm:syntomicetale-mainfine}]
Let $d$ be the relative dimension of $X$ over $\Spf\Z_p$. We are going to show that $\cofib(R\Gamma_\Syn(X, E)\rightarrow R\Gamma_\proet(X, T_\et(E)))$ is concentrated in degrees at least $i+1$, \EEDIT{which suffices to establish the statement}. By $p$-completeness, it suffices to check this after reducing mod $p$, i.e.\ we may assume that $E$ is a vector bundle on $(X^\Syn)_{p=0}$. Subsequently, \cref{prop:syntomicetale-filtrationetale} reduces us to proving that the complexes \EDIT{$\gr^\Syn_k R\Gamma(X^\Syn, E)[\tfrac{1}{v_1}]$} are concentrated in \EDIT{cohomological} degrees at least $i+1$ for all $k\geq 1$ \EDIT{as this will imply that the cofibre of the map
\begin{equation*}
R\Gamma(X^\Syn, E)\rightarrow R\Gamma(X^\Syn, E)[\tfrac{1}{v_1}]
\end{equation*}
is concentrated in cohomological degrees at least $i+1$, from which our result follows.} \EDIT{Using the isomorphism
\begin{equation*}
\gr^\Syn_k R\Gamma(X^\Syn, E)[\tfrac{1}{v_1}]\cong R\Gamma(X^\Syn_\red, E/v_{1, X}\{k(p-1)\})
\end{equation*}
from (\ref{eq:syntomicetale-grsyn}), we now replace $E$ by $E/v_{1, X}\{k(p-1)\}$; then our assumption becomes that $E$ is a vector bundle on $X^\Syn_\red$ with Hodge--Tate weights all at most $-i-p$ and we have to} prove that $R\Gamma(X^\Syn_\red, E)$ is concentrated in \EDIT{cohomological} degrees at least $i+1$. For this, we identify the pushforward $\pi_{X^\Syn, *} E$ of $E$ to $\Z_{p, \red}^\Syn$ with the data $(V, \Fil^\bullet V, \Fil_\bullet V)$ with $\gr^\bullet V=\gr_\bullet V$ equipped with operators $D$ and $\Theta$ according to the description of the reduced locus of $\Z_p^\Syn$. Now we make the following observation: by \cref{prop:syntomicetale-cohomologyhod}, the Hodge--Tate weights of $\pi_{X^\Syn, *} E$ are all at most $-i-p+d$; moreover, for $0\leq k\leq d$, the complex $\gr^{-i-p+d-k} V$ is concentrated in degrees at least $d-k$. In particular, we deduce that
\begin{equation*}
F_\Hod(E)=\fib(\gr^0 V\xrightarrow{\Theta} \gr^{-p} V)
\end{equation*}
is concentrated in degrees at least $i+1$. As the filtration $\Fil^\bullet V$ is \EDIT{complete by \cref{lem:finiteness-pushforwardcomplete} and base change for the cartesian square (\ref{eq:syntomicetale-coverdr+})}, we may also conclude that $\Fil^0 V$ is concentrated in degrees at least $i+p$ while $\Fil^{-p} V$ is concentrated in degrees at least $i$. Thus, the complex
\begin{equation*}
F_{\dR, +}(E)=\fib(\Fil^0 V\xrightarrow{\Theta} \Fil^{-p} V)
\end{equation*}
is concentrated in degrees at least $i+1$. Finally, we may also infer that the natural map $\Fil_{-1} V\rightarrow V$ is an isomorphism in cohomological degrees at most \EEDIT{$i+p-1$ while} $\Fil_0 V\rightarrow V$ \EEDIT{is an isomorphism in cohomological degrees at most $i+p$}. Thus, the natural map
\begin{equation*}
\EEDIT{
F_{\HT, c}(E)=\fib(\Fil_0 V\xrightarrow{D} \Fil_{-1} V)\rightarrow\fib(V\xrightarrow{\Theta} V)=F_\dR(E)
}
\end{equation*}
is also an isomorphism in cohomological degrees at most $i+p$. Overall, using the formula
\begin{equation*}
R\Gamma(X^\Syn_\red, E)=\fib(F_{\dR, +}(E)\oplus F_{\HT, c}(E)\xrightarrow{a_E-b_E} F_\dR(E)\oplus F_\Hod(E))
\end{equation*} 
from (\ref{eq:syntomification-cohomologyreduced}), we see that $R\Gamma(X^\Syn_\red, E)$ is indeed concentrated in degrees at least $i+1$, as wanted.
\end{proof}

\begin{proof}[Proof of \cref{thm:syntomicetale-maincoarse}]
As in the \EDIT{proof of \cref{thm:syntomicetale-mainfine}, we may assume that $E$ is a vector bundle on $(X^\Syn)_{p=0}$ by $p$-completeness of both sides and then \cref{prop:syntomicetale-filtrationetale} reduces us to proving that
\begin{equation*}
\gr^\Syn_k R\Gamma(X^\Syn, E)[\tfrac{1}{v_1}]\cong R\Gamma(X^\Syn_\red, E/v_{1, X}\{k(p-1)\})
\end{equation*}
vanishes for all $k\geq 1$, where we have again used (\ref{eq:syntomicetale-grsyn}). Replacing $E$ by $E/v_{1, X}\{k(p-1)\}$, we thus have to show that,} for $E\in\Perf(X^\Syn_\red)$ with Hodge--Tate weights all at most $-p-d-1$, the cohomology $R\Gamma(X^\Syn_\red, E)$ vanishes. As before, we identify $\pi_{X^\Syn, *} E$ with the data $(V, \Fil^\bullet V, \Fil_\bullet V)$ with $\gr^\bullet V=\gr_\bullet V$ equipped with operators $D$ and $\Theta$. From \cref{prop:syntomicetale-cohomologyhod}, we see that $\pi_{X^\Syn, *} E$ has Hodge--Tate weights all at most $-p-1$. Thus, we conclude
\begin{equation*}
F_\Hod(E)=\fib(\gr^0 V\xrightarrow{\Theta} \gr^{-p} V)=0\;.
\end{equation*}
Moreover, as the filtration $\Fil^\bullet V$ is again \EDIT{complete by \cref{lem:finiteness-pushforwardcomplete}}, we also have $\Fil^0 V=\Fil^{-p} V=0$ and hence also
\begin{equation*}
F_{\dR, +}(E)=\fib(\Fil^0 V\xrightarrow{\Theta} \Fil^{-p} V)=0\;.
\end{equation*}
Finally, we see that $\Fil_{-1} V=\Fil_0 V=V$ and thus
\begin{equation*}
F_{\HT, c}(E)=\fib(\Fil_{-1} V\xrightarrow{D}\Fil_0 V)=\fib(V\xrightarrow{\Theta} V)=F_\dR(E)\;.
\end{equation*}
Overall, using (\ref{eq:syntomification-cohomologyreduced}), we obtain
\begin{equation*}
R\Gamma(X^\Syn_\red, E)=\fib(F_{\dR, +}(E)\oplus F_{\HT, c}(E)\xrightarrow{a_E-b_E} F_\dR(E)\oplus F_\Hod(E))=0\;,
\end{equation*} 
as wanted.
\end{proof}

\comment{
Finally, we show how to obtain the claimed improvement upon the result from \cref{thm:syntomicetale-mainfine} in the case where $X$ admits a morphism to $\Spf\O_C$:

\begin{proof}[Proof of \cref{thm:syntomicetale-mainfineoc}]
Let $d$ be the relative dimension of $X$ over $\Spf\O_C$. As in the proof of \cref{thm:syntomicetale-mainfine}, \EDIT{we may first assume that $E\in\Vect((X^\Syn)_{p=0})$ and are then reduced to showing that
\begin{equation*}
\gr^\Syn_k R\Gamma(X^\Syn, E)[\tfrac{1}{v_1}]\cong R\Gamma(X^\Syn_\red, E/v_{1, X}\{k(p-1)\})
\end{equation*}
is concentrated in cohomological degrees at least $i+1$ for all $k\geq 1$. Replacing $E$ by $E/v_{1, X}\{k(p-1)\}$, our assumption thus becomes that $E$ is a vector bundle on $X^\Syn_\red$ with Hodge--Tate weights all at most $-i-1$ and we have to show that the complex $R\Gamma(X^\Syn_\red, E)$ is concentrated in cohomological degrees at least $i+1$. Denoting the map $X^\Syn\rightarrow\O_C^\Syn$ by $\pi'_{X^\Syn}$ and} recalling that 
\begin{equation*}
\O_C^\N\cong \Spf (A_\inf(\O_C)\langle u, t\rangle/(ut-\phi^{-1}(\xi)))/\G_m
\end{equation*}
from \cref{ex:filprism-perfd} and $v_{1, \O_C}=u^pt\in A_\inf(\O_C)[u, t]/(ut-\phi^{-1}(\xi), p)$, we may identify the pushforward \EDIT{$\pi'_{X^\Syn, *}E$} of $E$ to $\O_{C, \red}^\Syn$ with a diagram
\begin{equation*}
\begin{tikzcd}
\dots \ar[r,shift left=.5ex,"t"]
  & \ar[l,shift left=.5ex, "u"] M^{i+1} \ar[r,shift left=.5ex,"t"] & \ar[l,shift left=.5ex, "u"] M^i \ar[r,shift left=.5ex,"t"] & \ar[l,shift left=.5ex, "u"] M^{i-1} \ar[r,shift left=.5ex,"t"] & \ar[l,shift left=.5ex, "u"] \dots
\end{tikzcd}
\end{equation*}
of $\xi$-complete $A_\inf(\O_C)/p$-complexes such that $ut=tu=\phi^{-1}(\xi)$ and $u^pt=0$ together with an identification $\tau: \phi^*M^\infty\cong M^{-\infty}$ similarly to \cref{ex:syntomification-fgaugesperfd}. Consider the filtration $\Fil^\bullet V\coloneqq M^\bullet/M^{\bullet-1}$ induced by the $t$-maps with associated graded $\gr^\bullet V$ \EDIT{and note that the filtered complex $\Fil^\bullet V$ identifies with the pullback of $\pi'_{X^\Syn, *}E$ to} 
\begin{equation*}
\EDIT{
(\O_{C, \red}^\Syn)_{u=0}=\O_{C, \dR, +}^\N\cong \Spec(\O_C/p)[t]/\G_m\;.
}
\end{equation*}
From the analogue of \cref{prop:syntomicetale-cohomologyhod} over $\Spf\O_C$, we may conclude that $\gr^\bullet V=0$ for $\bullet\geq -i+d$ and that, moreover, for $0\leq k\leq d$, the complex $\gr^{-i+d-k-1} V$ is concentrated in degrees at least $d-k$. In particular, as the filtration $\Fil^\bullet V$ is again \EDIT{complete by \cref{lem:finiteness-pushforwardcomplete}}, we conclude that $\Fil^\bullet V=0$ for $\bullet\geq -i+d$ and that $\Fil^{-i+d-k-1} V$ is concentrated in degrees at least $d-k$ for $0\leq k\leq d$. \EDIT{In other words,} the map $u: M^\bullet\rightarrow M^{\bullet+1}$ is an isomorphism for $\bullet\geq -i+d-1$ and, for $0\leq k\leq d$, the map $u: M^{-i+d-k-2}\rightarrow M^{-i+d-k-1}$ induces an isomorphism in cohomological degrees at most $d-k-1$ and an injection in degree $d-k$ for $0\leq k\leq d$. In particular, we conclude that the map $u^\infty: M^0\rightarrow M^\infty$ induces an isomorphism in cohomological degrees at most $i+1$ and an injection in degree $i+2$. Moreover, the map $u^p: M^{-1}\rightarrow M^{p-1}$ induces an isomorphism in cohomological degrees at most $i$ and an injection in degree $i+1$, so, due to $u^pt=0$, we conclude that the map $t: M^0\rightarrow M^{-1}$ induces the zero map in cohomological degrees at most $i+1$ and hence the same is true for $t^\infty: M^0\rightarrow M^{-\infty}$. By virtue of
\begin{equation*}
R\Gamma(\O_{C, \red}^\Syn, E)=\fib(M^0\xrightarrow{t^\infty-\tau u^\infty} M^{-\infty})\;,
\end{equation*} 
see \cref{ex:syntomification-fgaugesperfd}, we conclude that $R\Gamma(\O_{C, \red}^\Syn, E)$ is concentrated in degrees at least $i+1$ and this finishes the proof.
\end{proof}

\begin{proof}[Proof of \cref{thm:syntomicetale-maincoarseoc}]
We use the notations from the \EDIT{proof of \cref{thm:syntomicetale-mainfineoc} and can again reduce to showing that
\begin{equation*}
\gr^\Syn_k R\Gamma(X^\Syn, E)[\tfrac{1}{v_1}]\cong R\Gamma(X^\Syn_\red, E/v_{1, X}\{k(p-1)\})
\end{equation*}
vanishes in the case $E\in\Vect((X^\Syn)_{p=0})$. As before, we replace $E$ by $E/v_{1, X}\{k(p-1)\}$; then our assumption becomes that $E$ is a vector bundle} on $X^\Syn_\red$ with Hodge--Tate weights all at most $-d-1$ \EDIT{and we have to prove that} the complex $R\Gamma(X^\Syn_\red, E)$ vanishes. Again, from \cref{prop:syntomicetale-cohomologyhod}, we see that $\gr^\bullet V=0$ for $\bullet\geq 0$. As before, we now successively conclude that $\Fil^\bullet V=0$ for $\bullet\geq 0$ and hence the maps $u: M^\bullet\rightarrow M^{\bullet+1}$ are isomorphisms for $\bullet\geq -1$. In particular, the maps $u^\infty: M^0\rightarrow M^\infty$ and $u^p: M^{-1}\rightarrow M^{p-1}$ are isomorphisms and thus $t: M^0\rightarrow M^{-1}$ must be zero due to $u^pt=0$. By virtue of
\begin{equation*}
R\Gamma(\O_{C, \red}^\Syn, E)=\fib(M^0\xrightarrow{t^\infty-\tau u^\infty} M^{-\infty})\;,
\end{equation*} 
see \cref{ex:syntomification-fgaugesperfd}, we obtain $R\Gamma(\O_{C, \red}^\Syn, E)=0$, as required.
\end{proof}

\begin{rem}
\EDIT{
Note that the proofs of \cref{thm:syntomicetale-mainfineoc} and \cref{thm:syntomicetale-maincoarseoc} only make use of the fact that $\O_C$ is perfectoid, but do not rely on $C$ being algebraically closed.
}
\end{rem}
}

\newpage

\section{A comparison of crystalline cohomology and étale cohomology}
\label{sect:cryset}

In this final section, we establish a generalisation of a comparison theorem of Colmez--Niziol between rational crystalline cohomology and étale cohomology, see \cite[Cor. 1.4]{ColmezNiziol}, which allows for coefficients. 

\begin{thm}
\label{thm:cryset-main}
Let $X$ be a $p$-adic formal scheme \EEDIT{which is smooth and proper over $\Spf\Z_p$}. For any crystalline local system $L$ on $X_\eta$ with Hodge--Tate weights all at most $-i-1$ for some $i\geq 0$, let $\cal{E}$ be its associated $F$-isocrystal. Then there is a natural morphism
\begin{equation*}
R\Gamma_\crys(X_{p=0}, \cal{E})[\tfrac{1}{p}][-1]\rightarrow R\Gamma_\proet(X_\eta, L)[\tfrac{1}{p}]\;,
\end{equation*}
which induces an isomorphism
\begin{equation*}
\tau^{\leq i} R\Gamma_\crys(X_{p=0}, \cal{E})[\tfrac{1}{p}][-1]\cong\tau^{\leq i} R\Gamma_\proet(X_\eta, L)[\tfrac{1}{p}]
\end{equation*}
and an injection on $H^{i+1}$.
\end{thm}

\EDIT{In order to clear up all the terms appearing in the statement, we will first recall the basic definitions surrounding crystalline local systems. Afterwards, we will go on to proving \cref{thm:cryset-main}, for which we will roughly proceed as follows: First, we show how to associate to any crystalline local system $L$ on $X_\eta$ an $F$-gauge $E$ on $X$ with the property that $T_\et(E)\cong L$ and $T_\dR(E)\cong \cal{E}$ under the identification $X^\dR\cong (X_{p=0})^\prism$ from \cite[Constr. 3.1.1]{FGauges}. Then we prove that
\begin{equation*}
\tau^{\leq i} R\Gamma_\dR(X, T_\dR(E))[\tfrac{1}{p}][-1]\cong\tau^{\leq i} R\Gamma_\proet(X_\eta, T_\et(E))[\tfrac{1}{p}]
\end{equation*}
for this $F$-gauge $E$ using the results from Section \ref{sect:beilfibsq} and Section \ref{sect:syntomicetale}.}

\subsection{Recollections on crystalline local systems}
\label{subsect:review-cryslocsys}

\EDIT{We briefly recap the theory of crystalline local systems along the lines of \cite[Sect. 2]{GuoReinecke} and start by recalling the definition of an $F$-isocrystal. For this, we let $Y$ be an $\F_p$-scheme.

\begin{defi}
Let $Y$ be an $\F_p$-scheme and consider the standard divided power structure on $\Z_p$ (i.e.\ the one given by $\gamma_n(p)=\tfrac{p^n}{n!}$ for $n\geq 0$). A \emph{crystal} (in coherent sheaves) on $Y$ is a sheaf $\cal{E}$ of $\O_{Y/\Z_p}$-modules on the big crystalline site $(Y/\Z_p)_\crys$ of $Y$ satisfying the following two properties:
\begin{enumerate}[label=(\roman*)]
\item For each object $(U, T, \gamma)$ of $(Y/\Z_p)_\crys$, the restriction $\cal{E}|_T$ of $\cal{E}$ to the Zariski-site of $T$ is a coherent sheaf.

\item For each map $\alpha: (U', T', \gamma')\rightarrow (U, T, \gamma)$ in $(Y/\Z_p)_\crys$, the induced map $\alpha^*(\cal{E}|_{T'})\rightarrow \cal{E}|_T$ is an $\O_T$-linear isomorphism.
\end{enumerate}
We denote the category of crystals on $Y$ by $\Crys(Y/\Z_p)$ and call the isogeny category
\begin{equation*}
\Isoc(Y/\Z_p)\coloneqq \Crys(Y/\Z_p)[\tfrac{1}{p}]
\end{equation*}
the category of \emph{isocrystals} on $Y$.
\end{defi}

The absolute Frobenius on $Y$ and the Witt vector Frobenius on $\Z_p$ (which is the identity in our case) are compatible and hence we obtain a morphism of sites
\begin{equation*}
F: (Y/\Z_p)_\crys\rightarrow (Y/\Z_p)_\crys
\end{equation*}
inducing a morphism of the corresponding topoi.

\begin{defi}
Let $Y$ be an $\F_p$-scheme. An \emph{$F$-isocrystal} on $Y$ is an isocrystal $\cal{E}$ on $Y$ equipped with an isomorphism $\phi: F^*\cal{E}\rightarrow\cal{E}$. We use $\Isoc^\phi(Y/\Z_p)$ to denote the category of $F$-isocrystals on $Y$. 
\end{defi}

To proceed further, we need to introduce certain period rings on the pro-étale site of any locally noetherian adic space $W$ over $\Spa\Q_p$. As affinoid perfectoid objects of $W_\proet$ form a basis of the topology by \cite[Lem. 4.6, Prop. 4.8]{PAdicHodgeTheory}, \EEDIT{a sheaf on $W_\proet$ can be defined by giving its values on affinoid perfectoid objects; since $\Q_p\rightarrow C$ is pro-étale, we may furthermore restrict to affinoid perfectoids over $\Spa C$, where we recall that $C$ denotes a fixed completed algebraic closure of $\Q_p$.}

\begin{defi}
Let $W$ be a locally noetherian adic space \EEDIT{over $\Spa\Q_p$}. The sheaves $\mathbb{A}_\crys, \mathbb{B}_\crys^+$ and $\mathbb{B}_\crys$ on $W_\proet$ are defined as follows: For any affinoid perfectoid $\Spa(R, R^+)\in W_\proet$ \EEDIT{over $\Spa C$}, we define
\begin{enumerate}[label=(\roman*)]
\item $\mathbb{A}_\crys(R, R^+)\coloneqq A_\crys(R^+)$, where we recall that the right-hand side was defined in \cref{ex:drstack-perfd},
\item $\mathbb{B}_\crys^+(R, R^+)\coloneqq \mathbb{A}_\crys(R, R^+)[\tfrac{1}{p}]$,
\item $\mathbb{B}_\crys(R, R^+)\coloneqq \mathbb{B}_\crys^+(R, R^+)[\tfrac{1}{\mu}]$, where $\mu\coloneqq [\epsilon^{1/p}]-1\in A_\inf(\O_C)$ is defined as in the proof of \cref{lem:beilfibsq-etalephimod}.
\end{enumerate}
\end{defi}

Now assume that $W=X_\eta$ is the generic fibre of a smooth $p$-adic formal scheme $X$. In this case, we can define a variant of the sheaf $\mathbb{B}_\crys$ with coefficients in any $F$-isocrystal on the special fibre $X_{p=0}$:

\begin{defi}
Let $X$ be a smooth $p$-adic formal scheme and $\cal{E}\in\Isoc^\phi(X_{p=0}/\Z_p)$ an $F$-isocrystal on the special fibre of $X$. Then we define sheaves $\mathbb{B}_\crys^+(\cal{E})$ and $\mathbb{B}_\crys(\cal{E})$ on $X_{\eta, \proet}$ as follows: For an affinoid perfectoid $\Spa(R, R^+)\in X_{\eta, \proet}$ \EEDIT{over $\Spa C$}, we define
\begin{enumerate}[label=(\roman*)]
\item $\mathbb{B}_\crys^+(\cal{E})(R, R^+)\coloneqq \cal{E}(A_\crys(R^+))[\tfrac{1}{p}]$, where
\begin{equation*}
\cal{E}(A_\crys(R^+))\coloneqq \lim_n \cal{E}(A_\crys(R^+)/p, A_\crys(R^+)/p^n, \gamma)
\end{equation*}
and $\gamma$ denotes the canonical divided power structure,

\item $\mathbb{B}_\crys(\cal{E})(R, R^+)\coloneqq \mathbb{B}_\crys^+(\cal{E})(R, R^+)[\tfrac{1}{\mu}]$.
\end{enumerate}
\end{defi}

Note that the Frobenius on $A_\crys(R^+)$ for any affinoid perfectoid $\Spa(R, R^+)$ over $\Spa C$ induces a Frobenius $F: \mathbb{B}_\crys\rightarrow\mathbb{B}_\crys$ on the sheaf $\mathbb{B}_\crys$ and we have an isomorphism $F^*\mathbb{B}_\crys(\cal{E})\cong\mathbb{B}_\crys(F^*\cal{E})$ for any $F$-isocrystal $\cal{E}$ on $X_{p=0}$. Thus, we obtain a Frobenius
\begin{equation*}
\phi: F^*\mathbb{B}_\crys(\cal{E})\cong\mathbb{B}_\crys(F^*\cal{E})\xrightarrow{\cong}\mathbb{B}_\crys(\cal{E})
\end{equation*}
on $\mathbb{B}_\crys(\cal{E})$ by functoriality of the assignment $\cal{E}\mapsto \mathbb{B}_\crys(\cal{E})$.

By \cite[Prop. 2.38]{GuoReinecke}, we may now define a crystalline local system as follows:

\begin{defi}
Let $X$ be a smooth $p$-adic formal scheme. A $\widehat{\Z}_p$-local system $L$ on $X_{\eta, \proet}$ is called \emph{crystalline} if there exists an $F$-isocrystal $\cal{E}$ on $X_{p=0}$ and an isomorphism
\begin{equation*}
\mathbb{B}_\crys(\cal{E})\cong \mathbb{B}_\crys\tensor_{\widehat{\Z}_p} L
\end{equation*}
compatible with the Frobenius on both sides. In this case, we call $\cal{E}$ the \emph{associated $F$-isocrystal} of $L$. Here, $\widehat{\Z}_p$ denotes the sheaf $\lim_n \Z/p^n\Z$ on $X_{\eta, \proet}$, as usual.
\end{defi}

\begin{rem}
Intuitively, one should think of a crystalline local system on $X_\eta$ as a continuous family of crystalline representations. Indeed, in yet unpublished work, Guo--Yang show that a local system $L$ on $X_\eta$ is crystalline if and only if its restrictions to all closed points of $X_\eta$ are crystalline representations, see \cite[Rem. 2.32]{GuoReinecke}. In particular, this also implies that the notion of a local system on $X_\eta$ being crystalline is independent of the choice of the smooth integral model $X$.
\end{rem}

Finally, we need to establish the notion of \emph{Hodge--Tate weights} of a crystalline local system $L$ on $X_\eta$. For this, we first need to define even more period sheaves.

\begin{defi}
\EEDIT{Let $W$ be a locally noetherian adic space over $\Spa\Q_p$. The sheaves $\mathbb{B}_\dR^+$ and $\mathbb{B}_\dR$ on $W_\proet$ are defined as follows: For any affinoid perfectoid $\Spa(R, R^+)\in W_\proet$ over $\Spa C$, we define}
\begin{enumerate}[label=(\roman*)]
\item \EEDIT{$\mathbb{B}_\dR^+(R, R^+)\coloneqq A_\inf(R^+)[\tfrac{1}{p}]^\wedge_{\Ker\theta}$ equipped with the $(\Ker\theta)$-adic filtration, where $\theta: A_\inf(R^+)[\tfrac{1}{p}]\rightarrow R$ is induced by the usual Fontaine map,}
\item \EEDIT{$\mathbb{B}_\dR(R, R^+)\coloneqq \mathbb{B}_\dR^+(R, R^+)[\tfrac{1}{t}]$ with the filtration
\begin{equation*}
\Fil^n \mathbb{B}_\dR\coloneqq \sum_{m\geq -n} t^{-m}\Fil^{m+n}\mathbb{B}_\dR^+(R, R^+)\;,
\end{equation*}
where $t\coloneqq \log [\epsilon]\in A_\crys(R, R^+)$ is defined as in (\ref{eq:beilfibsq-t}).}
\end{enumerate}
\end{defi}

\begin{defi}
Let $W$ be a locally noetherian adic space \EEDIT{over $\Spa\Q_p$}. The sheaf $\O\mathbb{B}_\dR^+$ on $W_\proet$ is defined as follows: For any affinoid perfectoid $\Spa(R, R^+)=\lim_i\Spa(R_i, R_i^+)\in W_\proet$ \EEDIT{over $\Spa C$}, we define
\begin{equation*}
\EEDIT{
\O\mathbb{B}_\dR^+(R, R^+)\coloneqq \colim_i (R_i^+\widehat{\tensor}_{\Z_p} A_\inf(R^+))[\tfrac{1}{p}]^\wedge_{\Ker\theta}\;,
}
\end{equation*}
where we abuse notation and denote by $\theta: (R_i^+\widehat{\tensor}_{\Z_p} A_\inf(R^+))[\tfrac{1}{p}]\rightarrow R$ the map induced by the two maps $R_i\rightarrow R$ and $A_\inf(R^+)\rightarrow R$. We equip $\O\mathbb{B}_\dR^+(R, R^+)$ with the $(\Ker\theta)$-adic filtration and hence obtain a decreasing filtration on $\O\mathbb{B}_\dR^+(R, R^+)[\tfrac{1}{t}]$ given by
\begin{equation*}
\Fil^n \O\mathbb{B}_\dR^+(R, R^+)[\tfrac{1}{t}]\coloneqq \sum_{m\geq -n} t^{-m}\Fil^{m+n}\O\mathbb{B}_\dR^+(R, R^+)\;.
\end{equation*}
This defines a filtered sheaf on \EEDIT{$W_\proet$} whose completion with respect to the filtration we denote by $\O\mathbb{B}_\dR$ and call the \emph{structural de Rham sheaf}. Its associated graded is the \emph{structural Hodge--Tate sheaf} denoted by $\O\mathbb{B}_\HT$.
\end{defi}

Now recall that being crystalline implies being de Rham in the sense of \cite[Def. 8.3]{PAdicHodgeTheory} by \cite[Cor. 2.37]{GuoReinecke} and thus, \emph{a fortiori}, also being Hodge--Tate in the sense of \cite[Def. 5.6]{HTWeights}, see \cite[Cor. 3.12]{LiuZhu}. In particular, this means that the sheaves
\begin{equation}
\label{eq:cryset-ddr}
D_\dR(L)\coloneqq \nu_*(L\tensor_{\widehat{\Z}_p} \O\mathbb{B}_\dR)
\end{equation}
and 
\begin{equation}
\label{eq:cryset-dht}
D_\HT(L)\coloneqq \nu_*(L\tensor_{\widehat{\Z}_p} \O\mathbb{B}_\HT)
\end{equation}
are vector bundles on $X_{\eta, \et}$. Here, the map $\nu: X_{\eta, \proet}\rightarrow X_{\eta, \et}$ denotes the canonical projection; however, we warn the reader that, contrary to our usual convention that all pushforwards are automatically derived, the pushforward $\nu_*$ appearing in (\ref{eq:cryset-dht}) and (\ref{eq:cryset-ddr}) is \emph{underived}.}

\begin{defi}
\EDIT{
Let $X$ be a smooth $p$-adic formal scheme and $L$ a crystalline local system on the generic fibre $X_\eta$. The \emph{Hodge--Tate weights} of $L$ are those integers $i\in\Z$ such that the graded vector bundle $D_\HT(L)$ is non-trivial in grading degree $i$.
}
\end{defi}

\begin{rem}
\label{rem:cryset-htweightsddr}
\EDIT{
Note that, by definition, the sheaf $D_\HT(L)$ is the associated graded of the filtered vector bundle $D_\dR(L)$. Thus, one may alternatively define the Hodge--Tate weights of $L$ as the integers $i\in\Z$ for which the filtration on $D_\dR(L)$ jumps.
}
\end{rem}

\begin{rem}
\EDIT{
If $X_\eta$ is geometrically connected, by \cite[Thm. 1.1]{HTWeights}, we may alternatively define the Hodge--Tate weights of $L$ as the Hodge--Tate weights of the $p$-adic Galois representation $L_{\ol{x}}$ of $k(x)$, where $x$ denotes any classical point of $X_\eta$ (i.e.\ any point of $X_\eta$ seen as a rigid-analytic variety) and $k(x)$ is its residue field.
}
\end{rem}

\subsection{\EEDIT{$F$-gauges associated to crystalline local systems}}

We first show how to associate an $F$-gauge to any crystalline local system on the generic fibre of a smooth formal scheme $X$ over $\Z_p$. To this end, recall that, \EDIT{in independent work, both} Guo--Reinecke \EDIT{and Du--Liu--Moon--Shimizu} have recently given a description of crystalline local systems on $X_\eta$ in terms of the prismatic site of $X$, which we review here briefly, \EDIT{see \cite{GuoReinecke} and \cite{DuLiuMoonShimizu}.}

\begin{defi}
For any prism $(A, I)$, let $\Vect^\phi(\Spec(A)\setminus V(p, I))$ denote the category of vector bundles $M$ on $\Spec(A)\setminus V(p, I)$ together with an $A$-linear isomorphism $\phi^*M[1/I]\cong M[1/I]$, where $\phi$ is the Frobenius of $A$. Then the category of \emph{analytic prismatic $F$-crystals} on $X$ is defined as
\begin{equation*}
\Vect^{\an, \phi}(X_\prism)\coloneqq \lim_{(A, I)\in X_\prism} \Vect^\phi(\Spec(A)\setminus V(p, I))\;.
\end{equation*}
\end{defi}

In our setting, it turns out that crystalline local systems and analytic prismatic $F$-crystals are the same:

\begin{prop}
\label{prop:cryset-locsysvectan}
Let $X$ be a smooth $p$-adic formal scheme. Then there is a natural equivalence of categories
\begin{equation*}
\Vect^{\an, \phi}(X_\prism)\cong \Loc^\crys_{\Z_p}(X_\eta)\;.
\end{equation*}
\end{prop}
\begin{proof}
See \cite[Thm. A]{GuoReinecke} \EDIT{or \cite[Thm. 1.3, Rem. 1.5]{DuLiuMoonShimizu}.}
\end{proof}

Next, given an analytic prismatic $F$-crystal on $X$, we extend it to a prismatic $F$-crystal in perfect complexes on $X$ by (underived) pushforward along the open immersion 
\begin{equation*}
\EEDIT{
\Spec(A)\setminus V(p, I)\rightarrow\Spec A
}
\end{equation*}
for all $(A, I)\in X_\prism$.

\begin{prop}
\label{prop:cryset-vectanpushforward}
Let $X$ be a smooth $p$-adic formal scheme. Then \EEDIT{(underived)} pushforward along the open immersions $\Spec(A)\setminus V(p, I)\rightarrow \Spec A$ for all $(A, I)\in X_\prism$ induces a fully faithful \EDIT{embedding}
\begin{equation*}
\Vect^{\an, \phi}(X_\prism)\hookrightarrow \Perf^\phi(X_\prism)
\end{equation*}
from the category of analytic prismatic $F$-crystals to the category of prismatic $F$-crystals in perfect complexes.
\end{prop}
\begin{proof}
See \cite[Thm. 5.10]{GuoReinecke}.
\end{proof}

To finish the construction, we have to show how to extend a prismatic $F$-crystal in perfect complexes $M$, i.e.\ a perfect complex on $X^\prism$ together with a Frobenius structure, to a perfect complex on $X^\N$. The Frobenius structure will ensure that the result descends to $X^\Syn$, see \cite[Ex. 6.1.7]{FGauges}. By quasisyntomic descent, see \cref{lem:filprism-quasisyntomiccover}, it suffices to describe the procedure for $X=\Spf R$ with a quasiregular semiperfectoid ring $R$. Here, we have
\begin{equation*}
R^\N\cong \Spf\Rees(\Fil^\bullet_\N \Prism_R)/\G_m
\end{equation*}
by \cite[Cor. 5.5.11]{FGauges} and thus we need to equip $M$ with a decreasing filtration $\Fil^\bullet M$ compatible with the Nygaard filtration on $\Prism_R$. This construction is supplied by the following definition from \cite[Def. 6.6.6]{FGauges}:

\begin{defi}
Let $R$ be a quasiregular semiperfectoid ring and $(M, \tau)\in\Mod^\phi(\Prism_R)$ be a prismatic $F$-crystal in finitely presented modules on $\Spf R$. A \emph{Nygaardian filtration} on $M$ is a filtration $\Fil^\bullet M$ of $M$ in $(p, I)$-complete $\Prism_R$-modules such that the map
\begin{equation*}
M\xrightarrow{\can} \phi^*M\hookrightarrow \phi^*M[1/I]\overset{\tau}{\cong} M[1/I]
\end{equation*}
carries $\Fil^i M$ into $I^i M$ for all $i\in\Z$; here, $I$ is the effective Cartier divisor on $\Prism_R$ defining the initial prism of $R_\prism$. A Nygaardian filtration is called \emph{saturated} if it is maximal, i.e.\ if $\Fil^\bullet M$ is the preimage of $I^\bullet M$ under the map above. 
\end{defi}

Indeed, this construction gives what we want provided we impose a small restriction on the class of prismatic $F$-crystals in perfect complexes we are considering.

\begin{defi}
Let $X$ be a \EDIT{$p$-quasisyntomic} $p$-adic formal scheme. We define the category $\Perf^\phi_{I-\mathrm{tf}}(X_\prism)$ to be the full subcategory of $\Perf^\phi(X_\prism)$ consisting of prismatic $F$-crystals $\cal{E}$ satisfying the following property: For any quasiregular semiperfectoid ring $R$ such that $\Spf R\rightarrow X$ is \EDIT{$p$-quasisyntomic}, the complex $\cal{E}(\Prism_R)$ is concentrated in cohomological degree zero and has no $I$-torsion, where again $(\Prism_R, I)$ is the inital object of $R_\prism$.
\end{defi}

\begin{prop}
\label{prop:cryset-fcrysfgauge}
Let $X$ be a smooth $p$-adic formal scheme. Then there is a fully faithful embedding
\begin{equation*}
\Perf^\phi_{I-\mathrm{tf}}(X_\prism)\hookrightarrow \Perf(X^\Syn)
\end{equation*}
characterised by the requirement that its composition with pullback along $R^\Syn\rightarrow X^\Syn$ for any quasiregular semiperfectoid ring $R$ such that $\Spf R\rightarrow X$ is $p$-completely flat sends a prismatic $F$-crystal $\cal{E}$ to the module $\cal{E}(\Prism_R)$ equipped with its saturated Nygaardian filtration.
\end{prop}
\begin{proof}
See \cite[Thm. 2.31]{GuoLi}.
\end{proof}

To be able to apply the above proposition in our situation, we only need to check that the essential image of the functor from \cref{prop:cryset-vectanpushforward} is contained in $\Perf^\phi_{I-\mathrm{tf}}(X_\prism)$. However, for any quasiregular semiperfectoid ring $R$, any object in the essential image of 
\begin{equation*}
\Vect^{\an, \phi}(R_\prism)\rightarrow \Perf^\phi(R_\prism)\cong \Perf^\phi(\Prism_R)
\end{equation*}
identifies with the global sections of a vector bundle on $\Spec(\Prism_R)\setminus V(p, I)$ and this immediately implies the claim as $I$ is an effective Cartier divisor on $\Prism_R$ and hence $\Prism_R$ has no $I$-torsion.

We summarise the results we have collected so far:

\begin{thm}
\label{thm:cryset-locsysfgauges}
Let $X$ be a smooth $p$-adic formal scheme. Then there is a fully faithful embedding
\begin{equation*}
\Loc^\crys_{\Z_p}(X_\eta)\hookrightarrow \Perf(X^\Syn)
\end{equation*}
which is a weak right inverse to the étale realisation. The image of a crystalline local system $L$ under this embedding is called the \emph{associated $F$-gauge} $E$ of $L$ and has the property that $T_\dR(E)$ identifies with the associated $F$-isocrystal of $L$.
\end{thm}
\begin{proof}
The existence of the embedding follows by combining \cref{prop:cryset-locsysvectan}, \cref{prop:cryset-vectanpushforward} and \cref{prop:cryset-fcrysfgauge}; moreover, \EDIT{by the description of the étale realisation from \cref{rem:syntomification-etalerealisationglobal},} the relation with the étale realisation is clear from the construction: \EDIT{indeed, the functor from \cref{prop:cryset-fcrysfgauge} is a weak right inverse to the functor $\Perf(X^\Syn)\rightarrow\Perf^\phi(X_\prism)$ from (\ref{eq:syntomification-etalerealisationglobal}) and this implies the claim.} Finally, the last \EDIT{statement about compatibility with the de Rham realisation} follows from the commutativity of the diagram
\begin{equation*}
\begin{tikzcd}[row sep=tiny]
X^\dR\ar[dr, "i_\dR"]\ar[dd, "\cong", swap] \\
& X^\prism \\
(X_{p=0})^\prism\ar[ur]
\end{tikzcd}
\end{equation*}
and the fact that the associated $F$-isocrystal of $L$ can be obtained from the associated analytic prismatic $F$-crystal via \EDIT{the composite functor
\begin{equation*}
\Vect^{\an, \phi}(X_\prism)\rightarrow\Vect^{\an, \phi}((X_{p=0})_\prism)\cong \Isoc^\phi(X_{p=0}/\Z_p)
\end{equation*}
from} \cite[Constr. 3.9.(ii)]{GuoReinecke}, see the proof of \cite[Thm. 3.12]{GuoReinecke}. 
\end{proof}

\begin{defi}
The essential image of the embedding 
\begin{equation*}
\Loc^\crys_{\Z_p}(X_\eta)\hookrightarrow \Perf(X^\Syn)
\end{equation*}
is called the category of \emph{reflexive $F$-gauges} on $X$ and denoted by $\Coh^\refl(X^\Syn)$.
\end{defi}

\begin{rem}
We shortly discuss why the notions of Hodge--Tate weights of a crystalline local system $L$ and its associated $F$-gauge $E$ are compatible. By \cref{rem:cryset-htweightsddr}, this would follow from the fact that $D_\dR(L)$ can be identified with (the pullback to $X_\eta$ of) $T_{\dR, +}(E)$. Using \cite[Cor. 6.19]{PAdicHodgeTheory}, it then suffices to prove that there is a filtered isomorphism
\begin{equation*}
L\tensor_{\widehat{\Z}_p} \O\mathbb{B}_\dR\cong T_{\dR, +}(E)\tensor_{\O_{X_\eta}} \O\mathbb{B}_\dR\;.
\end{equation*}
\EEDIT{However, the unfiltered isomorphism exists by \cite[Prop. 2.36]{GuoReinecke} and the fact that $T_\dR(E)$ identifies with the associated $F$-isocrystal of $L$. To check that this isomorphism respects the filtrations, we may now proceed locally on $X_{\eta, \proet}$ and hence assume $\O\mathbb{B}_\dR^+\cong \mathbb{B}_\dR^+[[X_1, \dots, X_n]]$ by \cite[Prop. 6.10]{PAdicHodgeTheory}. Then we are reduced to showing that 
\begin{equation*}
L\tensor_{\widehat{\Z}_p} \mathbb{B}_\dR\cong T_{\dR, +}(E)\tensor_{\O_{X_\eta}} \mathbb{B}_\dR
\end{equation*}
compatibly with the filtrations on either side, which may be checked on affinoid perfectoid objects $\Spa(R, R^+)\in X_{\eta, \proet}$ over $\Spa C$ such that $L|_{\Spa(R, R^+)}$ is trivial since such objects form a basis of $X_{\eta, \proet}$ -- however, in this case, the claim follows by virtually the same proof as \cref{lem:beilfibsq-etalefiltered} with the use of \cite[Lem. 4.26]{BMS} replaced by \cite[Lem. 3.13]{GuoReinecke}.}
\end{rem}

\subsection{Proof of the main theorem}

We now move on to prove the necessary comparison results for reflexive $F$-gauges. However, for simplicity, we start with the case of vector bundles:

\begin{prop}
\label{prop:cryset-fine}
Let $X$ be a $p$-adic formal scheme \EEDIT{which is smooth and proper over $\Spf\Z_p$}. For any vector bundle $F$-gauge $E\in\Vect(X^\Syn)$ with Hodge--Tate weights all at most $-i-1$ for some $i\geq 0$, there is a natural morphism
\begin{equation*}
R\Gamma_\dR(X, T_\dR(E))[\tfrac{1}{p}][-1]\rightarrow R\Gamma_\proet(X_\eta, T_\et(E))[\tfrac{1}{p}]\;,
\end{equation*}
which induces an isomorphism
\begin{equation*}
\tau^{\leq i} R\Gamma_\dR(X, T_\dR(E))[\tfrac{1}{p}][-1]\cong\tau^{\leq i} R\Gamma_\proet(X_\eta, T_\et(E))[\tfrac{1}{p}]
\end{equation*}
and an injection on $H^{i+1}$.
\end{prop}
\begin{proof}
Combining \cref{thm:beilfibsq-fmrat} and \cref{prop:finiteness-main}, we obtain a fibre sequence
\begin{equation*}
R\Gamma_\dR(X, T_\dR(E))[\tfrac{1}{p}][-1]\rightarrow R\Gamma_\Syn(X, E)[\tfrac{1}{p}]\rightarrow \EDIT{\Fil^0_\Hod R\Gamma_\dR(X, T_{\dR, +}(E))[\tfrac{1}{p}]}\;.
\end{equation*}
In view of \cref{thm:syntomicetale-mainfine}, it thus suffices to show that \EDIT{$\Fil^0_\Hod R\Gamma_\dR(X, T_{\dR, +}(E))$} is concentrated in degrees at least $i+1$. To this end, identify $T_{\dR, +}(\pi_{X^\Syn, *}E)$ with a filtered object $\Fil^\bullet V$ \EEDIT{and let $d$ be the relative dimension of $X$ over $\Spf\Z_p$}. Then \cref{prop:syntomicetale-cohomologyhod} implies that $\gr^\bullet V=0$ for $\bullet\geq d-i$ and that, moreover, $\gr^{d-i-1-k} V$ is concentrated in degrees at least $d-k$ for $0\leq k\leq d$. As the filtration $\Fil^\bullet V$ is \EDIT{complete by \cref{lem:finiteness-pushforwardcomplete}}, we may \EDIT{thus} conclude that $\Fil^0 V$ is concentrated in degrees at least $i+1$, which yields the claim.
\end{proof}

We can also formulate an analogous result for arbitrary perfect $F$-gauges, albeit in a slightly weaker range:

\begin{prop}
\label{prop:cryset-coarse}
Let $X$ be a $p$-adic formal scheme \EEDIT{which is proper and smooth} of relative dimension $d$ over $\Spf\Z_p$. For any perfect $F$-gauge $E\in\Perf(X^\Syn)$ with Hodge--Tate weights all at most $-d-2$ for some $i\geq 0$, there is a natural isomorphism
\begin{equation*}
R\Gamma_\dR(X, T_\dR(E))[\tfrac{1}{p}][-1]\cong R\Gamma_\proet(X_\eta, T_\et(E))[\tfrac{1}{p}]\;.
\end{equation*}
\end{prop}
\begin{proof}
As in the proof of \cref{prop:cryset-fine}, now using \cref{thm:syntomicetale-maincoarse} in place of \cref{thm:syntomicetale-mainfine}, we are reduced to showing that \EDIT{$\Fil^0_\Hod R\Gamma_\dR(X, T_{\dR, +}(E))=0$}. However, \EEDIT{by \cref{lem:finiteness-pushforwardcomplete}}, the filtered object $T_{\dR, +}(\pi_{X^\Syn, *}E)\in\D(\A^1/\G_m)$ is \EDIT{complete} and \EEDIT{\cref{prop:syntomicetale-cohomologyhod} shows that} $\pi_{X^\Syn, *}E$ has Hodge--Tate weights all at most $-2$, hence $T_{\dR, +}(\pi_{X^\Syn, *}E)$ is zero in filtration degree zero and this implies the claim.
\end{proof}

Now observe that, by \cref{thm:cryset-locsysfgauges}, having the conclusion of \cref{prop:cryset-fine} for reflexive $F$-gauges instead of just vector bundle $F$-gauges would immediately imply \cref{thm:cryset-main}. \EDIT{This will be a consequence of the following lemma:}

\begin{lem}
\label{lem:cryset-degreezero}
Let $X$ be a smooth $p$-adic formal scheme and $E\in\Perf(X^\Syn)$ a reflexive $F$-gauge. Then the pullback of $E$ to $X^\N_\Hod$ is concentrated in cohomological degree zero. 
\end{lem}
\begin{proof}
By quasisyntomic descent, see \cref{lem:filprism-quasisyntomiccover}, we may check this in the case where $X=\Spf R$ for a $p$-torsionfree quasiregular semiperfectoid ring $R$. \EEDIT{For this, choose a perfectoid ring $R_0$ mapping to $R$; then also the initial object $(\Prism_{R_0}, (\xi))$ of the absolute prismatic site of $R_0$ maps to the initial object $(\Prism_R, I)$ of the absolute prismatic site of $R$ and, in particular, we conclude that $I=\xi\Prism_R$ is principal by rigidity. Now recall that $R^\N\cong \Spf\Rees(\Fil^\bullet_\N\Prism_R)/\G_m$ by \cref{thm:filprism-qrsp} and that} the $F$-gauge $E$ arises by endowing the module $M$ of global sections of a vector bundle on $\Spec(\Prism_R)\setminus V(p, I)$ equipped with a Frobenius structure $\tau: \phi^*M[1/I]\cong M[1/I]$ with its saturated Nygaardian filtration $\Fil^\bullet M$ \EEDIT{by \cref{prop:cryset-fcrysfgauge}.} We begin by noting that $M$ is $p$-torsionfree since $\Prism_R$ is and therefore the pullback $\Fil^\bullet M/p$ of $E$ to $(R^\N)_{p=0}$ is concentrated in degree zero. Now further observe that $M/I$ is $p$-torsionfree by virtue of the exact sequence
\begin{equation*}
\begin{tikzcd}
0\ar[r] & M\ar[r] & M[1/I]\oplus M[1/p]\ar[r] & M[1/pI]\nospacepunct{\;;}
\end{tikzcd}
\end{equation*}
\EEDIT{indeed, if $pm=am'$ for some $m, m'\in M$ and $a\in I$, then we would have $m/a=m'/p$, i.e.\ $m/a\in M$, which would in turn imply that $m$ is already zero in $M/I$. This now} implies that, \EEDIT{for all $i\in\Z$, the map} $\Fil^i M/p\rightarrow\Fil^{i-1} M/p$ is injective: If $m=pm'$ for some $m\in\Fil^i M$ and $m'\in\Fil^{i-1} M$, then $p\tau(m')=\tau(m)\in I^i M$; \EEDIT{as $\tau(m')\in I^{i-1}$ and $I$ is generated by $\xi$, we obtain the equation $p\tfrac{\tau(m')}{\xi^{i-1}}=\tfrac{\tau(m)}{\xi^{i-1}}\in IM$ and and hence $\tfrac{\tau(m')}{\xi^{i-1}}\in IM$, i.e.\ $\tau(m')\in I^iM$,} as $M/I$ is $p$-torsionfree -- however, this implies $m'\in\Fil^i M$ by definition of the saturated Nygaardian filtration and hence already $m=0\in\Fil^i M/p$. Therefore, also the pullback $\gr^\bullet M/p=(\Fil^\bullet M/p)/(\Fil^{\bullet+1} M/p)$ of $E$ to $(R^\N)_{p=t=0}=R^\N_{\HT, c}$ is concentrated in degree zero. 

\EEDIT{Finally, to compute the further pullback of $E$ to $R^\N_\Hod$, observe that there is a commutative diagram
\begin{equation*}
\begin{tikzcd}[ampersand replacement=\&]
R^\N_{\dR, +}\ar[r]\ar[d] \& R^\N\ar[d] \\
R_{0, \dR, +}^\N\ar[r]\ar[d] \& R_0^\N\ar[d] \\
\Z_{p, \dR, +}^\N\ar[r] \& \Z_p^\N\nospacepunct{\;,}
\end{tikzcd}
\end{equation*}
in which the large rectangle and the bottom square are pullbacks by definition. Thus, also the top square is a pullback and recalling that $R_{0, \dR, +}=(R_0^\N)_{p=u^p=0}$ from \cref{ex:syntomicetale-perfddr+}, where we use the isomorphism
\begin{equation*}
R_0^\N\cong \Spf(A_\inf(R_0)\langle u, t\rangle/(ut-\phi^{-1}(\xi)))/\G_m
\end{equation*}
from \cref{ex:filprism-perfd}, we learn that $R^\N_{\dR, +}$ identifies with the locus 
\begin{equation*}
\{p=\phi^{-1}(\xi)^pt^{-p}=0\}=\{p=\xi t^{-p}=0\}\subseteq R^\N\;.
\end{equation*}
Thus, we conclude that the pullback of $\gr^\bullet M/p$ from $R^\N_{\HT, c}$ to $R^\N_\Hod=R^\N_{\HT, c}\cap R^\N_{\dR, +}$ identifies with the graded object given by} the cofibre of the map
\begin{equation*}
(\Fil^{\bullet-p} M/p)/(\Fil^{\bullet-p+1} M/p)\tensor_{\Prism_R/p} \xi\Prism_R/p\rightarrow (\Fil^\bullet M/p)/(\Fil^{\bullet+1} M/p)\;.
\end{equation*}
However, \EEDIT{recalling that $I=\xi\Prism_R$,} by definition of the saturated Nygaardian filtration, this is clearly an injection since multiplication by $\xi$ corresponds to multiplication by $\xi^p$ after applying $\tau$ by $\phi$-linearity \EEDIT{and $\xi$ is a non-zerodivisor.}
\end{proof}

\begin{proof}[Proof of \cref{thm:cryset-main}]
\EDIT{As observed previously, \cref{thm:cryset-locsysfgauges} reduces us to having the conclusion of \cref{prop:cryset-fine} for reflexive $F$-gauges instead of just vector bundle $F$-gauges. However, note} that all we really need \EDIT{in order} to copy the proofs of \cref{prop:cryset-fine} and \cref{thm:syntomicetale-mainfine} verbatim for a reflexive $F$-gauge $E$ is that the pullback of $E$ to $(X^\Hod)_{p=0}$ is concentrated in cohomological degree zero. \EDIT{As this follows from \cref{lem:cryset-degreezero} and the flatness of the cover $(X^\Hod)_{p=0}\rightarrow X^\N_\Hod$ from (\ref{eq:syntomicetale-coverhod}), we are done.}
\end{proof}

\newpage
\appendix

\section{Some base change results}
\label{sect:basechange}

Throughout this thesis, we often use base change results for cartesian squares of formal stacks like
\begin{equation*}
\begin{tikzcd}
X^\prism\ar[r, "j_\dR"]\ar[d] & X^\N\ar[d] \\
\Z_p^\prism\ar[r, "j_\dR"] & \Z_p^\N\nospacepunct{\;.}
\end{tikzcd}
\end{equation*}
However, as we are working with formal stacks instead of algebraic (or derived) stacks, such base change compatibilities are usually not automatic. Thus, in this appendix, we indicate how to prove these statements -- in the main body of the thesis, we will then often use them without further mention.

\begin{defi}
We say that a cartesian diagram
\begin{equation*}
\begin{tikzcd}
\cal{X}'\ar[r, "g'"]\ar[d, "f'"] & \cal{X}\ar[d, "f"] \\
\cal{Y}'\ar[r, "g"] & \cal{Y}
\end{tikzcd}
\end{equation*}
of stacks satisfies \emph{base change for bounded below complexes} if, for any \EEDIT{$F\in\D^+(X)$, see \cref{defi:beilfibsq-tstructbounded},} there is an isomorphism
\begin{equation*}
g^*f_*F\cong f'_*{g'}^*F
\end{equation*}
given by the natural map.
\end{defi}

As it turns out, \EEDIT{many} of the base change results we need are an application of the following general statement:

\begin{prop}
\label{prop:basechange-locclosed}
Let $f: \cal{X}\rightarrow\cal{Y}$ be a morphism of stacks over $\Spf\Z_p$ such that $\cal{X}$ can be written as the colimit of a simplicial flat affine formal $\cal{Y}$-scheme $U^\bullet$ and assume that $\cal{Y}$ admits a flat cover from a regular affine formal scheme. \EEDIT{Then,} for any locally closed immersion $i: \cal{Z}\rightarrow\cal{Y}$, the cartesian diagram
\begin{equation*}
\begin{tikzcd}
\cal{X}_{\cal{Z}}\ar[r, "i'"]\ar[d, "f'"] & \cal{X}\ar[d, "f"] \\
\cal{Z}\ar[r, "i"] & \cal{Y}
\end{tikzcd}
\end{equation*}
satisfies base change for bounded below complexes. 
\end{prop}

\EDIT{
Before we can prove the above proposition, we need to establish the following technical statement about the commutation of filtered colimits with cosimplicial totalisations:}

\begin{lem}
\EDIT{
\label{lem:syntomicetale-colimtot}
Let $\cal{X}$ be a stack admitting a flat cover from an affine scheme or a noetherian affine formal scheme. Assume we are given a filtered diagram of cosimplicial objects 
\begin{equation*}
F: I\times\Delta^\op\rightarrow \D^{\geq 0}(\cal{X})\subseteq \D(\cal{X})
\end{equation*}
in the category $\D(\cal{X})$ whose terms are all cosimplicial objects of the full subcategory $\D^{\geq 0}(\cal{X})$, i.e.\ $I$ is a filtered category and $\Delta$ denotes the simplex category, as usual. Then taking colimits over $I$ in $\D(\cal{X})$ commutes with cosimplicial totalisation, i.e.\ there is a natural isomorphism}
\begin{equation*}
\EDIT{
\colim_{i\in I} \Tot(F(i, \bullet))\cong \Tot(\colim_{i\in I} F(i, \bullet))
}
\end{equation*}
\EDIT{in the category $\D(\cal{X})$.}
\end{lem}
\begin{proof}
\EDIT{
We first claim that filtered colimits in $\D(\cal{X})$ have finite cohomological dimension, i.e.\ that there is some $r\geq 0$ with the property that, if $G, G': K\rightarrow \D(\cal{X})$ are two filtered diagrams of the same shape together with a natural transformation $G\rightarrow G'$ with the property that $\tau^{\leq -1}G(k)\cong \tau^{\leq -1}G'(k)$ for all $k\in K$, then there is a natural isomorphism
\begin{equation}
\label{eq:syntomicetale-colimfindim}
\tau^{\leq -r-1}\colim_{k\in K} G(k)\cong\tau^{\leq -r-1}\colim_{k\in K} G'(k)\;.
\end{equation} 
Indeed, as colimits commute with pullbacks, this may be checked on a flat cover. Now there are two cases to consider: If we are given a flat cover $\Spec A\rightarrow\cal{X}$, the claim is clear as filtered colimits of $A$-modules are exact. If, however, we are given a flat cover $\Spf A\rightarrow\cal{X}$ with $A$ noetherian, then colimits in $\D(\Spf A)\cong \widehat{\D}(A)$ are computed by first taking the colimit in $\D(A)$ and then applying derived completion, so in this case the claim follows from the fact that derived completion has finite cohomological dimension in our setting due to the ideal of definition of $A$ being finitely generated, see \cite[Tag 0AAJ]{Stacks}. Thus, the claim is proved. In particular, note that, by picking $G'(k)=0$ for all $k$, we have also shown that, if $G(k)\in\D^{\geq 0}(\cal{X})$ for all $k\in K$, then
\begin{equation*}
\colim_{k\in K} G(k)\in \D^{\geq -r}(\cal{X})\;.
\end{equation*}

We now prove the claim from the statement. To check that the natural map 
\begin{equation*}
\colim_{i\in I} \Tot(F(i, \bullet))\rightarrow \Tot(\colim_{i\in I} F(i, \bullet))
\end{equation*}
is an isomorphism, it suffices to check that it induces an isomorphism on $\tau^{\leq n}$ for all $n$. As $\colim_{i\in I} F(i, \bullet)\in\D^{\geq -r}(\cal{X})$ by the first paragraph, the totalisation on the right-hand side may be replaced by a partial totalisation if we are only interested in the result after $n$-truncation; more precisely, we have
\begin{equation*}
\tau^{\leq n} \Tot(\colim_{i\in I} F(i, \bullet))\cong \tau^{\leq n}\lim_{j\in\Delta_{\leq n+r+1}} \colim_{i\in I} F(i, j)\;,
\end{equation*}
where $\Delta_{\leq n+r+1}$ denotes the full subcategory of $\Delta$ spanned by the objects $j$ with $j\leq n+r+1$. Now using that finite limits commute with filtered colimits and the finite cohomological dimension of filtered colimits, we obtain
\begin{equation*}
\begin{split}
\tau^{\leq n}\lim_{j\in\Delta_{\leq n+r+1}} \colim_{i\in I} F(i, j)&\cong \tau^{\leq n}\colim_{i\in I} \lim_{j\in\Delta_{\leq n+r+1}} F(i, j) \\
&\cong \tau^{\leq n} \colim_{i\in I} \tau^{\leq n+r}\lim_{j\in\Delta_{\leq n+r+1}} F(i, j)\;.
\end{split}
\end{equation*}
Finally, we use that $F(i, j)\in\D^{\geq 0}(\cal{X})$ to conclude that the $(n+r)$-truncation of the partial totalisation above actually agrees with the $(n+r)$-truncation of the full totalisation and make use of the finite cohomological dimension of filtered colimits once again to conclude
\begin{equation*}
\begin{split}
\tau^{\leq n} \colim_{i\in I} \tau^{\leq n+r}\lim_{j\in\Delta_{\leq n+r+1}} F(i, j)&\cong \tau^{\leq n} \colim_{i\in I} \tau^{\leq n+r}\Tot(F(i, \bullet)) \\
&\cong \tau^{\leq n} \colim_{i\in I} \Tot(F(i, \bullet))\;,
\end{split}
\end{equation*}
so we are done.
}
\end{proof}

\begin{proof}[Proof of \cref{prop:basechange-locclosed}]
We follow the argument given in the proof of \cite[Thm. 3.3.5]{FGauges}. Namely, using the presentations $U^\bullet\rightarrow\cal{X}$ and $U^\bullet_{\cal{Z}}\rightarrow\cal{X}_{\cal{Z}}$, we are reduced to checking that the formation of cosimplicial totalisations in $\D^{\geq n}(\cal{Y})$ commutes with pullback along $i$. However, as $i$ is a locally closed immersion, pullback along $i$ is merely tensoring with $i_*\O_{\cal{Z}}$ and this is a perfect complex over some localisation of $\cal{Y}$ by regularity of $\cal{Y}$. Thus, the claim now follows from the fact that cosimplicial totalisations in $\D^{\geq n}(\cal{Y})$ commute with filtered colimits, \EEDIT{see \cref{lem:syntomicetale-colimtot},} as well as tensoring with perfect complexes.
\end{proof}

By way of example, we will now indicate how to use \cref{prop:basechange-locclosed} to prove many of the base change results we need.

\begin{cor}
\label{cor:basechange-nygaardlocclosed}
Let $X$ be a bounded $p$-adic formal scheme which is $p$-quasisyntomic and qcqs. Then the cartesian diagram
\begin{equation*}
\begin{tikzcd}
X^\prism\ar[r, "j_\dR"]\ar[d] & X^\N\ar[d] \\
\Z_p^\prism\ar[r, "j_\dR"] & \Z_p^\N
\end{tikzcd}
\end{equation*}
satisfies base change for bounded below complexes. 
\end{cor}
\begin{proof}
To apply \cref{prop:basechange-locclosed}, first \EDIT{recall} that $j_\dR: \Z_p^\prism\rightarrow\Z_p^\N$ is an open immersion and that \EDIT{there is a flat cover} $\Spf\Z_p\langle u, t\rangle\rightarrow\Z_p^\N$, \EDIT{see \cref{rem:filprism-regular}}. To provide the necessary simplicial presentation of $X^\N$, first observe that, by quasisyntomic descent, see \cref{lem:filprism-quasisyntomiccover}, we may assume that $X=\Spf R$ for a quasiregular semiperfectoid ring $R$. However, then the claim is clear by the explicit description of $R^\N$ from \cite[Cor. 5.5.11]{FGauges}.
\end{proof}

\begin{cor}
\label{cor:basechange-dr}
Let $X$ be a bounded $p$-adic formal scheme which is smooth and qcqs. Then the cartesian diagram
\begin{equation*}
\begin{tikzcd}
X^\Hod\ar[r]\ar[d] & X^{\dR, +}\ar[d] \\
\Z_p^\Hod\ar[r] & \Z_p^{\dR, +}
\end{tikzcd}
\end{equation*}
satisfies base change for bounded below complexes.
\end{cor}
\begin{proof}
As $\Z_p^{\dR, +}=\A^1/\G_m$ \EDIT{admits a flat cover from the regular affine formal scheme $\A^1$} and $\Z_p^\Hod=B\G_m\rightarrow\A^1/\G_m=\Z_p^{\dR, +}$ is a closed immersion, we again want to apply \cref{prop:basechange-locclosed}. To obtain the necessary simplicial presentation of $X^{\dR, +}$, first observe that the functor $X\mapsto X^{\dR, +}$ preserves (affine) étale maps by the proof of \cite[Thm. 2.3.6]{FGauges}. As $X$ Zariski-locally admits an affine étale map $X\rightarrow\A^n$, we may thus assume that we are given an affine étale map \EEDIT{$X^{\dR, +}\rightarrow (\A^n)^{\dR, +}=(\G_a/\V(\O(1))^\sharp)^n$} of stacks over $\A^1/\G_m$. \EEDIT{Then the pullback $\widetilde{X}$ of the ${\V(\O(1))^\sharp}^n$-torsor $\G_a^n\rightarrow (\G_a/\V(\O(1))^\sharp)^n$ to $X^{\dR, +}$ is a flat affine formal scheme over $\A^1/\G_m$ and we have $X^{\dR, +}\cong \widetilde{X}/{\V(\O(1))^\sharp}^n$, which yields the desired presentation.}
\end{proof}

The second strategy we can employ to prove base change results is to explicitly identify descent data. Let us again explain what we mean by this using an example:

\begin{lem}
\label{lem:basechange-nygaardsyn}
Let $X$ be a bounded $p$-adic formal scheme which is $p$-quasisyntomic and qcqs. Then the cartesian diagram
\begin{equation*}
\begin{tikzcd}
X^\N\ar[r, "j_\N"]\ar[d, "\pi_{X^\N}"] & X^\Syn\ar[d, "\pi_{X^\Syn}"] \\
\Z_p^\N\ar[r, "j_\N"] & \Z_p^\Syn
\end{tikzcd}
\end{equation*}
satisfies base change for bounded below complexes. 
\end{lem}
\begin{proof}
\EDIT{Given any $F\in\D^+(X^\Syn)$,} by base change for the cartesian diagrams
\begin{equation*}
\begin{tikzcd}
X^\prism\ar[r, "j_\dR"]\ar[d] & X^\N\ar[d] \\
\Z_p^\prism\ar[r, "j_\dR"] & \Z_p^\N
\end{tikzcd}\;,\hspace{0.5cm}
\begin{tikzcd}
X^\prism\ar[r, "j_\HT"]\ar[d] & X^\N\ar[d] \\
\Z_p^\prism\ar[r, "j_\HT"] & \Z_p^\N
\end{tikzcd}\;,
\end{equation*}
see \cref{cor:basechange-nygaardlocclosed}, \EDIT{we have
\begin{equation*}
j_\dR^*\pi_{X^\N, *}j_\N^*F\cong \pi_{X^\prism, *}j_\dR^*j_\N^*F\cong \pi_{X^\prism, *}j_\HT^*j_\N^*F\cong j_\HT^*\pi_{X^\N, *}j_\N^*F\;,
\end{equation*}
where $j_\dR^*j_\N^*F\cong j_\HT^*j_\N^*F$ is due to the fact that $j_\N^*F$ is pulled back from $X^\Syn$. Thus,} we see that $\pi_{X^\N, *}j_\N^* F$ is equipped with a canonical descent datum to $\Z_p^\Syn$ and this descent datum corresponds to the complex $\pi_{X^\Syn, *} F$.
\end{proof}

The only base change claim which cannot be readily proved with either of the two techniques above is the following:

\begin{lem}
\label{lem:basechange-nygaarddr}
Let $X$ be a bounded $p$-adic formal scheme which is smooth and qcqs. Then the cartesian diagram
\begin{equation*}
\begin{tikzcd}
X^{\dR, +}\ar[r, "i_{\dR, +}"]\ar[d, "\pi_{X^{\dR, +}}"] & X^\N\ar[d, "\pi_{X^\N}"] \\
\Z_p^{\dR, +}\ar[r, "i_{\dR, +}"] & \Z_p^\N
\end{tikzcd}
\end{equation*}
satisfies base change for bounded below complexes.
\end{lem}
\begin{proof}
To check that the base change map $i_{\dR, +}^*\pi_{X^\N, *} F\rightarrow \pi_{X^{\dR, +}, *}i_{\dR, +}^* F$ is an isomorphism for any $F\in\D^+(X^\N)$, we may reduce mod $p$ by $p$-completeness, i.e.\ we may pull back to the stack $(\A^1/\G_m)_{p=0}$. Moreover, as the square
\begin{equation*}
\begin{tikzcd}
(X^{\dR, +})_{p=0}\ar[r]\ar[d] & X^{\dR, +}\ar[d] \\
(\A^1/\G_m)_{p=0}\ar[r] & \A^1/\G_m
\end{tikzcd}
\end{equation*}
satisfies base change for bounded below complexes by the proof of \cref{cor:basechange-dr}, we are reduced to proving that the square
\begin{equation*}
\begin{tikzcd}
(X^{\dR, +})_{p=0}\ar[r]\ar[d] & X^\N\ar[d] \\
\A^1/\G_m\ar[r] & \Z_p^\N
\end{tikzcd}
\end{equation*}
satisfies base change for bounded below complexes; note that we have now simply written $\A^1/\G_m$ instead of $(\A^1/\G_m)_{p=0}$ as any instance of the stack $\A^1/\G_m$ will be over $\F_p$ from now on. However, note that the above square admits a factorisation as 
\begin{equation*}
\begin{tikzcd}
(X^{\dR, +})_{p=0}\ar[r]\ar[d] & X^\N_{\dR, +}\ar[r]\ar[d] & X^\N\ar[d] \\
\A^1/\G_m\ar[r] & \Z_{p, \dR, +}^\N\ar[r] & \Z_p^\N\nospacepunct{\;,}
\end{tikzcd}
\end{equation*}
so it suffices to show that both the left and the right square satisfy base change for bounded below complexes. For the right square, as $\Z_{p, \dR, +}\rightarrow \Z_p^\N$ is a closed immersion \EDIT{by the description of the reduced locus}, this follows from an application of \cref{prop:basechange-locclosed}; for the left square, we recall that $\A^1/\G_m\rightarrow\Z_{p, \dR, +}$ and hence also $(X^{\dR, +})_{p=0}\rightarrow X^\N_{\dR, +}$ is a $\cal{G}$-torsor, where $\cal{G}$ is the group scheme described in Section \ref{subsect:syntomification}, so here the claim follows by identifying descent data.
\end{proof}

\newpage

\section{Finiteness of pushforwards along $X^\Syn\rightarrow\Z_p^\Syn$}
\label{sect:finiteness}

In this appendix, we will prove a finiteness result for pushforwards of perfect complexes along $X^\Syn\rightarrow\Z_p^\Syn$. Namely, the result we want to \EEDIT{establish} goes as follows:

\begin{prop}
\label{prop:finiteness-main}
Let $X$ be a $p$-adic formal scheme which is smooth and proper \EEDIT{over $\Spf\Z_p$}. For any perfect $F$-gauge $E\in\Perf(X^\Syn)$, the pushforward $\pi_{X^\Syn, *} E$ is perfect.
\end{prop}

\EDIT{
Before we can proceed to the proof, we first need to establish the following two lemmas:
}

\begin{lem}
\label{lem:finiteness-perfa1gm}
\EEDIT{We work over $\F_p$.}
\begin{enumerate}[label=(\roman*)]
\EEDIT{
\item A complex $E\in\D(B\G_m)$ is perfect if and only if it corresponds to a graded complex consisting of only finitely many nonzero graded pieces, which are all bounded complexes of finite-dimensional $\F_p$-vector spaces.
\item For a complex $E\in\D(\A^1/\G_m)$, the following are equivalent:}
\begin{enumerate}[label=(\arabic*)]
\EEDIT{
\item $E$ is perfect.
\item $E$ is $t$-complete and its pullback to $B\G_m$ is perfect (recall that $t$ denotes the coordinate on $\A^1$).
\item Under the Rees equivalence, $E$ corresponds to a finite filtration $\Fil^\bullet V$ such that all graded pieces $\gr^\bullet V$ are bounded complexes of finite-dimensional $\F_p$-vector spaces; here, the filtration $\Fil^\bullet V$ being finite means that $\Fil^\bullet V=0$ for $\bullet\gg 0$ and that the transition maps $\Fil^\bullet V\rightarrow \Fil^{\bullet-1} V$ are isomorphisms for $\bullet\ll 0$.
}
\end{enumerate}
\end{enumerate}
\end{lem}
\begin{proof}
\EEDIT{
Part (i) is clear as perfectness may be checked after pullback along $\Spec\F_p\rightarrow B\G_m$. For part (ii), first observe that, if $E\in\D(\A^1/\G_m)$ is perfect, then its pullback to $B\G_m$ must be perfect as well and \cite[Rem. 2.2.7]{FGauges} shows that $E$ is $t$-complete; hence, (1) $\Rightarrow$ (2). The implication (2) $\Rightarrow$ (3) is clear by (i) as the Rees equivalence matches $t$-completeness with completeness as a filtered object, see \cref{prop:rees-main}. Finally, to prove (3) $\Rightarrow$ (1), note that (3) ensures that the pullback of $E$ to $\A^1$ is a bounded complex of finitely generated $\F_p[t]$-modules -- however, as $\F_p[t]$ is regular, this is the same as being perfect.
}
\end{proof}

\begin{lem}
\label{lem:finiteness-pushforwardcomplete}
\EDIT{
Let $X$ be a \EEDIT{smooth qcqs} $p$-adic formal scheme. Then for any perfect complex $E\in\Perf((X^{\dR, +})_{p=0})$ on $(X^{\dR, +})_{p=0}$, the pushforward $\pi_{X^{\dR, +}, *}E\in\D(\A^1/\G_m)$ is $t$-complete (where we recall that $t$ denotes the coordinate on $\A^1$).
}
\end{lem}
\begin{proof}
\EDIT{We employ a similar strategy to the proof of \cite[Thm. 2.3.6]{FGauges}: One first observes that if $X$ is written as a colimit of a finite diagram $U^\bullet$ of affine open subschemes, then $X^{\dR, +}$ is the colimit of the diagram $U^{\bullet, \dR, +}$ and thus, by smoothness and the qcqs assumption, we may assume that \EEDIT{$X=\Spf R$ is affine and that} there is an étale map \EEDIT{$\Spf R\rightarrow \A^n$}. Then the diagram
\begin{equation*}
\EEDIT{
\begin{tikzcd}[ampersand replacement=\&]
\Spf R\times\A^1/\G_m\ar[r]\ar[d] \& R^{\dR, +}\ar[d] \\
\A^n\times\A^1/\G_m\ar[r] \& (\A^n)^{\dR, +}
\end{tikzcd}
}
\end{equation*}
is cartesian and the bottom map is a torsor for the group scheme \EEDIT{${\V(\O(1))^\sharp}^n$} over $\A^1/\G_m$, hence the same is true for the top row. This means that the map \EEDIT{$(R^{\dR, +})_{p=0}\rightarrow\A^1/\G_m$} admits a factorisation
\begin{equation}
\label{eq:finiteness-factorpidr+}
\EEDIT{
(R^{\dR, +})_{p=0}\cong (\Spec R/p\times\A^1/\G_m)/{\V(\O(1))^\sharp}^n\rightarrow B_{\A^1/\G_m}{\V(\O(1))^\sharp}^n\rightarrow \A^1/\G_m
}
\end{equation}
and hence the pushforward of a quasi-coherent complex along this map admits a similar description to the one in \cref{lem:syntomicetale-vesharpreps}: Namely, \EEDIT{we can identify} a quasi-coherent complex $E$ on \EEDIT{$(R^{\dR, +})_{p=0}$} with a complex $V$ \EEDIT{of $R/p$-modules} equipped with a decreasing filtration $\Fil^\bullet V$ and a flat connection \EEDIT{$\nabla: V\rightarrow V\tensor\Omega_{(R/p)/\F_p}^1$} having locally nilpotent $p$-curvature \EEDIT{such that} the restriction of $\nabla$ to $\Fil^i V$ comes with a factorisation through \EEDIT{$\Fil^{i-1} V\tensor\Omega_{(R/p)/\F_p}^1$} along the lines of \cref{rem:fildrstack-vect}. \EEDIT{Now observing that $\Omega_{(R/p)/\F_p}^1\cong (R/p)^n$ since $\Spf R$ is étale over $\A^n$ and using the $\G_m$-equivariant version of \cref{lem:syntomicetale-vesharpreps}, we find that $\pi_{R^{\dR, +}, *}E$} identifies with the filtered complex whose degree $i$ term is given by
\begin{equation}
\label{eq:finiteness-filcohdr+}
\Tot(\Fil^i V\xrightarrow{\nabla} \Fil^{i-1} V\tensor\Omega_{(R/p)/\F_p}^1\xrightarrow{\nabla} \Fil^{i-2} V\tensor\Omega_{(R/p)/\F_p}^2\xrightarrow{\nabla}\dots)
\end{equation}
\EEDIT{by virtue of the factorisation (\ref{eq:finiteness-factorpidr+}).}
However, if $E$ is perfect, then the filtration $\Fil^\bullet V$ is \EEDIT{finite by \cref{lem:finiteness-perfa1gm}} and hence also the filtration given by (\ref{eq:finiteness-filcohdr+}) is \EEDIT{finite and, in particular, complete}, so we are done as the Rees equivalence matches completeness as a filtered object with $t$-completeness, see \cref{prop:rees-main}.
}
\end{proof}

\begin{proof}[Proof of \cref{prop:finiteness-main}]
\EDIT{Recalling from the proof of \cref{prop:syntomicetale-filtrationetale} that $v_1$ is topologically nilpotent on $\Z_p^\Syn$}, we may check perfectness after pulling back to the reduced locus $\Z_{p, \red}^\Syn$ by $(p, v_1)$-completeness; as the natural map $\Z_{p, \red}^\N\rightarrow\Z_{p, \red}^\Syn$ is a cover, we may even pull back to $\Z_{p, \red}^\N$. By the gluing description of $\Z_{p, \red}^\N$, we first observe that the reasoning of \cite[Fn. 44]{FGauges} shows that it suffices to check perfectness after pullback to $\Z_{p, \HT, c}^\N$ and $\Z_{p, \dR, +}^\N$. Finally, recall that there are covers $\A^1/\G_m\rightarrow\Z_{p, \dR, +}$ and $\A^1/\G_m\rightarrow\Z_{p, \HT, c}$, so we may check perfectness after pullback along those.
\EEDIT{Using \cref{lem:finiteness-perfa1gm} and} some base change results analogous to the ones from Appendix \ref{sect:basechange}, we are thus reduced to showing the following assertions:
\begin{enumerate}[label=(\alph*)]
\item Pushforward along $(X^\Hod)_{p=0}\rightarrow B\G_m$ preserves perfect complexes;
\item pushforward along $(X^{\dHod, c})_{p=0}\rightarrow \A^1/\G_m$ carries perfect complexes to $t$-complete objects;
\item pushforward along $(X^{\dR, +})_{p=0}\rightarrow \A^1/\G_m$ carries perfect complexes to $t$-complete objects.
\end{enumerate}
For (a), we observe that this follows immediately from \EEDIT{\cref{lem:finiteness-perfa1gm}} together with \cref{prop:syntomicetale-cohomologyhod}; here, we make use of the properness assumption in order to conclude that cohomology on $X_{p=0}$ preserves perfect complexes. Moreover, (b) follows immediately by the same reasoning as in the proof of \cref{lem:fildhod-comparisonfilteredsmooth}. Finally, (c) is exactly \cref{lem:finiteness-pushforwardcomplete}.
\end{proof}

\newpage
\bibliographystyle{alpha}
\bibliography{References_arXivVersion}
\end{document}